\documentclass[11pt]{amsart}
\usepackage{geometry}                
\usepackage[latin1]{inputenc}   
\usepackage{mathpazo}  
\usepackage[english]{babel}
\usepackage[usenames,dvipsnames,cmyk]{xcolor}
\usepackage{url}
\usepackage{amsmath,amssymb,latexsym,mathrsfs,amssymb,amsthm}
\usepackage{enumerate}
\usepackage{mathtools} 
\usepackage[all]{xy}
\usepackage{setspace}
\usepackage{tikz, tkz-euclide}

\usepackage{algorithm2e,algpseudocode}
\RestyleAlgo{ruled}

\usepackage{listings}
\definecolor{codegreen}{rgb}{0,0.6,0}
\definecolor{codegray}{rgb}{0.5,0.5,0.5}
\definecolor{codepurple}{rgb}{0.58,0,0.82}
\definecolor{backcolour}{rgb}{0.95,0.95,0.95}
\lstdefinestyle{mystyle}{
    backgroundcolor=\color{backcolour},   
    commentstyle=\color{codegreen},
    keywordstyle=\color{magenta},
    numberstyle=\tiny\color{codegray},
    stringstyle=\color{codepurple},
    basicstyle=\ttfamily\small,
    breakatwhitespace=false,         
    breaklines=true,                 
    captionpos=b,                    
    keepspaces=true,                 
    numbers=left,                    
    numbersep=5pt,                  
    showspaces=false,                
    showstringspaces=false,
    showtabs=false,                  
    tabsize=2
}
\numberwithin{equation}{section}

\usepackage[colorlinks=true, linkcolor=MidnightBlue, citecolor=MidnightBlue, urlcolor=MidnightBlue]{hyperref}

\newtheorem{theorem}{Theorem}[section]

\newtheorem{proposition}[theorem]{Proposition}
\newtheorem{corollary}[theorem]{Corollary}

\theoremstyle{definition}

\theoremstyle{remark}
\newtheorem{remark}[theorem]{Remark}
\newtheorem{rhp}{Riemann-Hilbert Problem}
\newcommand{\rhref}[1]{Riemann-Hilbert Problem~\ref{#1}}

\let\Im=\undefined\DeclareMathOperator{\Im}{Im}

\newcommand{\ii}{\ensuremath{\mathrm{i}}}
\newcommand{\ee}{\ensuremath{\mathrm{e}}}
\newcommand*\dd{\mathop{}\!\mathrm{d}}
\newcommand{\bg}{\ensuremath{B}}
\newcommand{\newf}{\ensuremath{f}}
\newcommand{\newk}{\ensuremath{\ell}}

\newcommand{\tz}{Z}
\newcommand{\transconf}{\varphi}
\newcommand{\transr}{t}
\newcommand{\transs}{u}

\makeatletter
\renewcommand*\env@matrix[1][\arraystretch]{%
  \edef\arraystretch{#1}%
  \hskip -\arraycolsep
  \let\@ifnextchar\new@ifnextchar
  \array{*\c@MaxMatrixCols c}}
\makeatother

\let\originalleft\left
\let\originalright\right
\renewcommand{\left}{\mathopen{}\mathclose\bgroup\originalleft}
\renewcommand{\right}{\aftergroup\egroup\originalright}

\title[General rogue waves of infinite order]{General rogue waves of infinite order:  Exact properties, asymptotic behavior, and effective numerical computation}

\author{Deniz Bilman}
\address{Deniz Bilman:  Department of Mathematical Sciences, University of Cincinnati, Cincinnati, OH, USA}
\email{bilman@uc.edu}
\author{Peter D.~Miller}
\address{Peter D. Miller:  Department of Mathematics, University of Michigan, Ann Arbor, MI, USA}
\email{millerpd@umich.edu}

\date{\today}

\begin{document}
\begin{abstract}
This paper is devoted to a comprehensive analysis of a family of solutions of the focusing nonlinear Schr\"odinger equation called general rogue waves of infinite order.  These solutions have recently been shown to describe various limit processes involving large-amplitude waves, and they have also appeared in some physical models not directly connected with nonlinear Schr\"odinger equations.  We establish the following key property of these solutions:  they are all in $L^2(\mathbb{R})$ with respect to the spatial variable but they exhibit anomalously slow temporal decay.  In this paper we define general rogue waves of infinite order, establish their basic exact and asymptotic properties, and provide computational tools for calculating them accurately.  
\end{abstract}
\maketitle
\section{Introduction}
\label{s:introduction}
%
The focusing nonlinear Schr\"odinger equation is a universal evolution equation governing the complex amplitude of a weakly nonlinear, strongly dispersive wave packet over long time scales in very general settings.  In one space dimension, this equation can be written in normalized form as
\begin{equation}
\ii q_t + \frac{1}{2}q_{xx} + |q|^2q = 0,\quad q=q(x,t),\quad (x,t)\in\mathbb{R}^2.
\label{nls-general}
\end{equation}
For instance, in 1969, V. E. Zakharov studied the surface elevation of water wave packets over deep water in the classical setting of plane-parallel irrotational and incompressible (potential) flow below the free surface, which is subject to kinematic and pressure-balance boundary conditions \cite{Zakharov69}.  He gave a derivation of \eqref{nls-general} based on the formalism  of the method of multiple scales, with the wave packet amplitude being the fundamental small parameter.  This derivation has more recently been made fully rigorous \cite{TotzW12}.  Being as the multiple-scale argument is based on Taylor expansions of nonlinear terms and of the linearized dispersion relation, the derivation of \eqref{nls-general} as a model equation applies in far more settings than surface water waves \cite{BenneyN67}.  For example, it is also a fundamental model in nonlinear optics \cite{NewellM92} and in the theory of Bose-Einstein condensation (where it is known as the Gross-Pitaevskii equation) \cite{Ueda2010}.  

In 1983, D. H. Peregrine found a compelling exact solution of \eqref{nls-general} for which $q(x,t)$ is not of constant modulus, but nonetheless decays to the exact solution $q(x,t)=\ee^{\ii t}$ of \eqref{nls-general} uniformly in all directions of space-time \cite{Peregrine83}.  Peregrine's solution is
\begin{equation}
q(x,t)=\ee^{\ii t}\left[1-4\frac{1+2\ii t}{1+4x^2+4t^2}\right] = \ee^{\ii t}\left[1+O\left(\frac{1}{\sqrt{x^2+t^2}}\right)\right],\quad (x,t)\to\infty.
\label{eq:Peregrine}
\end{equation}
Since in the setting of water waves, $q(x,t)=\ee^{\ii t}$ is the complex amplitude of a uniform periodic wavetrain (a \emph{Stokes wave}), Peregrine's solution describes a space-time localized fluctuation of a Stokes wave, and as such it is a model for a \emph{rogue wave}.   

The focusing nonlinear Schr\"odinger equation \eqref{nls-general} was shown to be a completely integrable system in the work of Zakharov and Shabat \cite{ZakharovS72}.  This means that the methods of soliton theory apply, including tools for deriving numerous exact solutions such as \eqref{eq:Peregrine}.  These tools have a recursive nature, allowing for a given exact solution to be generalized to a whole infinite family by means of iterated B\"acklund transformations.  Thus, one sees that the Peregrine solution \eqref{eq:Peregrine} is by no means the only solution of \eqref{nls-general} that has the character of a rogue wave.  Indeed, there exist algebraic representations in terms of determinants of exact solutions of \eqref{nls-general} that for any $N\in\mathbb{N}$ can be viewed as a \emph{rogue wave of order $N$}.  At each subsequent value of $N$, new parameters enter into the algebraic solution formula that affect the details of the solution without influencing its fundamental property of decay to the background $q(x,t)=\ee^{\ii t}$.  If these parameters are scaled suitably, the rogue wave of order $N$ can resemble an array of a triangular number of distant copies of the Peregrine solution on the same background, and it has been shown \cite{YangY21} that the locations of the Peregrine peaks in space-time are correlated with the complex zeros of the Yablonskii-Vorob'ev polynomials.  

However, if the parameters are chosen in a highly-correlated way at each order, then the numerous peaks all combine and form a rogue wave of significantly higher amplitude, termed a \emph{fundamental rogue wave}.  For instance one can see that the Peregrine solution \eqref{eq:Peregrine} corresponding to $N=1$ has an amplitude $|q(x,t)|$ that grows to a maximum value of $3$ times the (unit) background level.  At the level $N=2$, the maximum amplitude obtainable is actually $5$ times the background level.  As such, the $N=2$ fundamental rogue wave is a better model than the Peregrine solution for the famous Draupner event \cite{Sunde95} in the North Sea that is frequently cited as the first quantitative observation of sea-surface rogue waves.  

To study large-amplitude rogue waves it then becomes of some interest to allow the order $N$ to grow and seek an asymptotic description of fundamental rogue waves as $N\to\infty$.  This limit became tractable with the introduction of a modified form of the inverse scattering transform for \eqref{nls-general} with nonzero boundary conditions at infinity, which yielded for the first time a Riemann-Hilbert representation of rogue wave solutions of arbitrary order \cite{BilmanM19}.  In \cite{BilmanLM2020}, this representation was used to analyze the fundamental rogue wave of order $N$ in the large-order/near-field limit that $N\to\infty$ while simultaneously the independent variables are rescaled near the peak $(x,t)=(0,0)$ so that $x=2X/N$ and $t=4T/N^2$ for fixed $(X,T)\in\mathbb{R}^2$.  It was found that a limiting profile $\Psi(X,T)$ of $2q/N$ exists as $N\to\infty$ that was called the \emph{rogue wave of infinite order}, and was shown (see Theorem~\ref{t:nls} below) to be a global solution of the focusing nonlinear Schr\"odinger (NLS) equation in the form
\begin{equation}
\ii \Psi_{T}+\frac{1}{2} \Psi_{X X}+|\Psi|^{2} \Psi=0.
\label{nls}
\end{equation}
It turns out the same solution also appeared recently in the physical literature \cite{Suleimanov17} to describe a universal dispersive regularization of an anomalously catastrophic self-focusing effect predicted by the geometrical optics approximation in self-focusing Kerr media as noted in the 1960s by Talanov \cite{Talanov65}, and there is also a rigorous proof that the solution arises in the semiclassical limit scaling of \eqref{nls-general} when it is taken with real semicircle-profile initial data matching the Talanov form \cite{BuckinghamJM}.  Indeed, several of the properties of $\Psi(X,T)$ that were proven in \cite{BilmanLM2020} had been also noted independently in the paper of Suleimanov \cite{Suleimanov17}.

The methodology developed in \cite{BilmanM19} also allowed for a streamlined analysis of multisoliton solutions of \eqref{nls-general} on the zero background, and in \cite{BilmanB2019} the soliton analogue of high-order fundamental rogue waves was analyzed in a similar near-field limit.  Here one considers reflectionless potentials corresponding to a transmission coefficient with a single pole of arbitrarily high order in the upper half-plane.  Unlike the case of fundamental rogue waves, two additional parameters appear in the iterated Darboux transformation that influence the shape of the limiting wave profile.  Thus one sees that the rogue wave of infinite order $\Psi(X,T)$ is a special case of a more general family of special solutions of \eqref{nls}.  We call these solutions \emph{general rogue waves of infinite order}.  

For rogue wave solutions, the iterated Darboux transformations are applied to a ``seed solution'' that is the uniform plane wave $q(x,t)=\ee^{\ii t}$, while for high-order soliton solutions the seed is instead the vacuum solution $q(x,t)=0$.  In both cases, the same type of limiting object was observed to appear in the high-order/near-field limit.  This observation was first generalized in \cite{BilmanM2021} in which a family of exact solutions of \eqref{nls-general} was described involving a continuous order parameter $M$ that when discretized in two different ways yielded both the fundamental rogue waves and also the arbitrary order solitons, with the same near-field limit appearing no matter how the continuous order was allowed to grow without bound, suggesting a type of universality of the limiting general rogue wave of infinite order.  This notion of universality was fully generalized in \cite{BilmanM2022}, where it was shown that the seed solution could be completely arbitrary (and need not represent any type of explicit solution at all), and the same family of general rogue wave solutions always appears in the high-order/near-field limit.  

General rogue waves of infinite order also have other applications.  For one thing, the initial condition for $\Psi$ at $T=0$ would be expected to generate corresponding integrable dynamics in any evolution equation that commutes with \eqref{nls}, i.e., in other equations of the same integrable hierarchy associated with the Zakharov-Shabat operator.  One such system is the sharp-line Maxwell-Bloch system, and in \cite{LiM24} the initial profiles of the general rogue waves of infinite order are identified with a family of self-similar solutions of the Maxwell-Bloch system that describe an important boundary-layer phenomenon.  There are also analogues of general rogue waves of infinite order in the modified Korteweg-de Vries equation (some constraints on the parameters are required to ensure reality of the solution) \cite{BilmanBMY}, and in simultaneous solutions of arbitrarily many commuting flows in the focusing NLS hierarchy \cite{BuckinghamJM}.  There is also some recent interest in general rogue waves of infinite order in the analysis community due to the fact that they lie in $L^2(\mathbb{R})$ for each fixed $T\in\mathbb{R}$ (see Theorem~\ref{t:L2-norm} below) and while \eqref{nls} is globally well-posed on this space \cite{Tsutsumi87}, these solutions neither generate any coherent structures (solitons) for large time $T$ nor do they exhibit the expected $O(T^{-\frac{1}{2}})$ decay consistent with solitonless initial data in smaller spaces such as $H^{1,1}(\mathbb{R})$ \cite{BorgheseJM2018}.  In fact, they decay at the anomalously slow rate of $O(T^{-\frac{1}{3}})$ (see Theorem~\ref{t:large-T} below) and there is no reason to expect that notions such as ``soliton content'' from inverse-scattering theory apply (see Remark~\ref{rem:no-IST}).

The main purpose of this paper is to present in one place all of the important properties of the family of general rogue waves of infinite order along with related computational methods.  We therefore begin by properly defining these solutions.

\subsection{Mathematical definition of general rogue waves of infinite order}
General rogue waves of infinite order are defined in terms of the following Riemann-Hilbert problem.

\begin{rhp}
Let $(X,T)\in\mathbb{R}^2$ and $\bg>0$ be fixed 
and let $\mathbf{G}$ be a $2\times 2$ matrix satisfying $\det(\mathbf{G})=1$ and $\mathbf{G}^*=\sigma_2\mathbf{G}\sigma_2$.
 Find a $2\times 2$ matrix $\mathbf{P}(\Lambda;X,T,\mathbf{G},\bg)$ with the following properties:
\begin{itemize}
\item[]\textbf{Analyticity:}  $\mathbf{P}(\Lambda;X,T,\mathbf{G},\bg)$ is analytic in $\Lambda$ for $|\Lambda|\neq 1$, and it takes continuous boundary values on the clockwise-oriented unit circle from the interior and exterior.
\item[]\textbf{Jump condition:} The boundary values\footnote{We use the standard convention that a subscript $+$ (resp., $-$) denotes a boundary value taken on an oriented contour arc from the left (resp., right).} on the unit circle are related as follows:
\begin{equation}
\mathbf{P}_+(\Lambda;X,T,\mathbf{G},\bg)=\mathbf{P}_-(\Lambda;X,T,\mathbf{G},\bg)
\ee^{-\ii(\Lambda X+\Lambda^2T+2 \bg \Lambda^{-1})\sigma_3}\mathbf{G}
\ee^{\ii(\Lambda X+\Lambda^2T+2\bg \Lambda^{-1})\sigma_3},\quad |\Lambda|=1.
\label{P-jump}
\end{equation}
\item[]\textbf{Normalization:}  $\mathbf{P}(\Lambda;X,T,\mathbf{G},\bg)\to\mathbb{I}$ as $\Lambda\to\infty$.
\end{itemize}
\label{rhp:near-field}
\end{rhp}
In general any matrix $\mathbf{G}$ satisfying $\det(\mathbf{G})=1$ and $\sigma_2 \mathbf{G}^* \sigma_2 = \mathbf{G}$ as in \rhref{rhp:near-field} can be written as
\begin{equation}
\mathbf{G}=\mathbf{G}(a,b)=\frac{1}{\sqrt{|a|^{2}+|b|^{2}}}\begin{bmatrix}
a & b^{*} \\
-b & a^{*}
\end{bmatrix}
\label{G-form}
\end{equation}
for complex numbers $a,b$ not both zero. 

It is a consequence of the conditions on $\mathbf{G}$ and the analytic dependence of the jump matrix on $(X,T)$ that the following holds.
\begin{proposition}[Global existence]
For each $(X,T)\in\mathbb{R}^2$ there exists a unique solution to \rhref{rhp:near-field}, and the solution depends real-analytically on $(X,T)\in\mathbb{R}^2$ and the real and imaginary parts of the parameters $a,b$ of the elements of $\mathbf{G}$.
\label{prop:RHP-EU}
\end{proposition}
This follows from Zhou's vanishing lemma \cite[Theorem 9.3]{Zhou1989} and the application of analytic Fredholm theory.  
The special solution $\Psi(X,T)=\Psi(X,T;\mathbf{G}, \bg)$ of \eqref{nls} is defined in terms of the solution of \rhref{rhp:near-field} by
\begin{equation}
\Psi(X,T;\mathbf{G},\bg):= 2\ii \lim_{\Lambda\to\infty}\Lambda P_{12}(\Lambda;X,T,\mathbf{G},\bg).
\label{Psi-def}
\end{equation}
This is in general a transcendental solution of the NLS equation; therefore its quantitative properties and the qualitative the nature of its profile (for instance what boundary conditions are satisfied as $X\to\pm\infty$), and how these depend on parameters are not immediately clear. 


The function $\Psi(X,T;\mathbf{G},\bg)$ was first studied by Suleimanov \cite{Suleimanov17} and independently by the authors with L.\@ Ling \cite{BilmanLM2020} for the special case of 
\begin{equation}
\mathbf{G}=\mathbf{Q}^{-1},\qquad \mathbf{Q}:=\frac{1}{\sqrt{2}}\begin{bmatrix}1& -1 \\ 1 & 1 \end{bmatrix},
\label{G-Qi}
\end{equation}
which corresponds to the choice $a=b=1$ (or any positive number) in \eqref{G-form}. 

In order to understand all important properties of the special solution $\Psi(X,T;\mathbf{G},\bg)$ and how they depend on parameters, there are three approaches one can take: (i) investigate its exact properties including symmetries, special values, differential equations satisfied, and equivalent representations; (ii) work in a variety of interesting asymptotic regimes to obtain rigorous approximations to $\Psi(X,T;\mathbf{G},\bg)$; and (iii) compute $\Psi(X,T;\mathbf{G},\bg)$ accurately in the $(X,T)$-plane by a suitable numerical method. In this paper, we achieve all three of these. 

In the rest of this introduction section, we summarize our results in the three areas mentioned above.  To set the scene, plots of $\Psi(X,T;\mathbf{G},\bg)$ computed with \texttt{RogueWaveInfiniteNLS.jl} with $a=b=\bg=1$ are shown in Figure~\ref{f:Psi-a1-b1}. \texttt{RogueWaveInfiniteNLS.jl} is a software package for the \texttt{Julia} programming language developed as part of this work to compute rogue waves of infinite order through numerical solution of suitable Riemann-Hilbert problems; see Section~\ref{s:numerics-intro} below.

\begin{figure}
\includegraphics[width=0.75\linewidth]{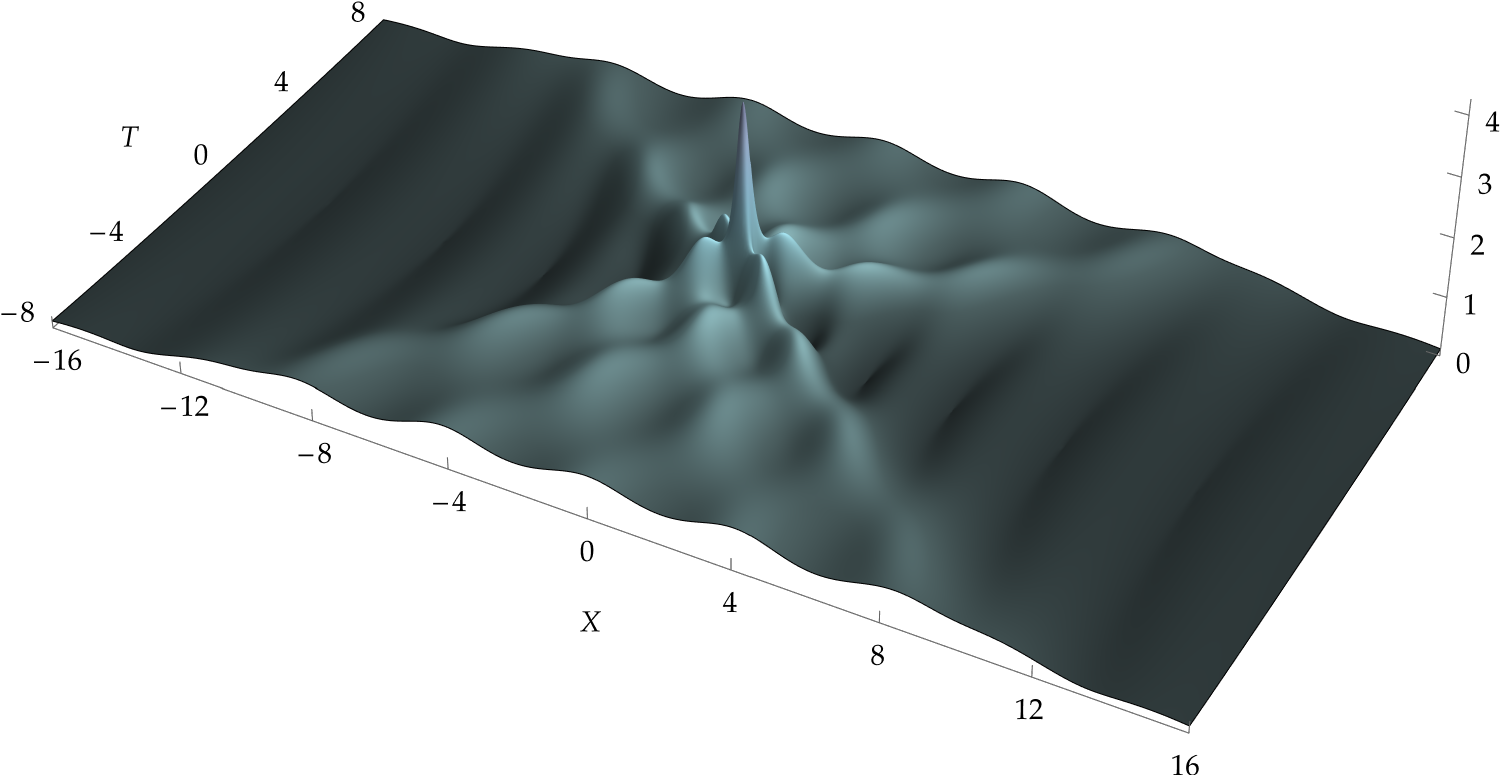}\\
\includegraphics[width=0.75\linewidth]{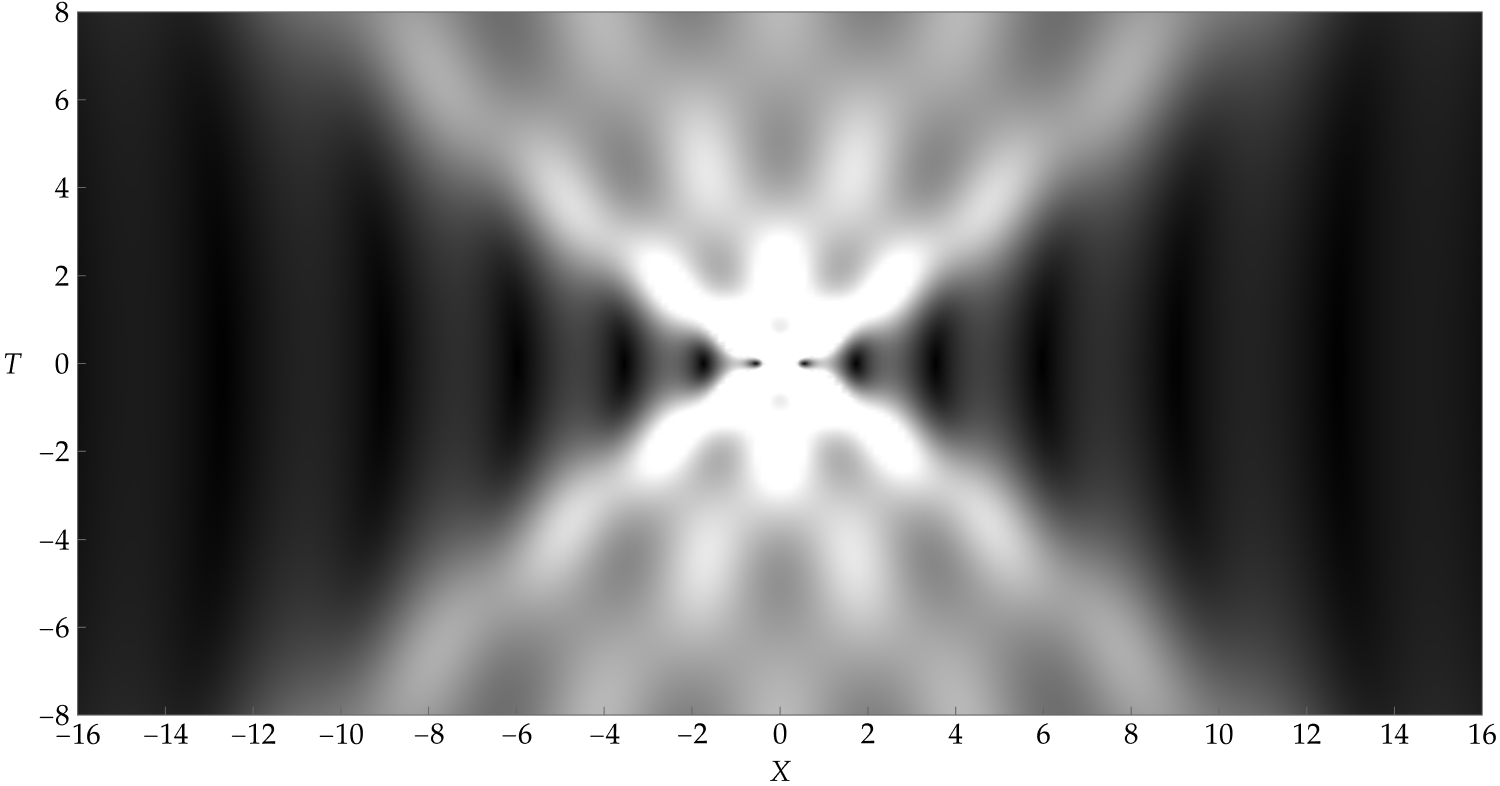}
\caption{The solution $\Psi(X,T;\mathbf{G},\bg)$ computed with \texttt{RogueWaveInfiniteNLS.jl} with $a=b=\bg=1$. \texttt{RogueWaveInfiniteNLS.jl} is a software package developed in this work for the \texttt{Julia} programming language to compute rogue waves of infinite order through numerical solution of suitable Riemann-Hilbert problems.}
\label{f:Psi-a1-b1}
\end{figure}

\subsection{Exact properties of $\Psi$}
Here we describe the symmetries of $\Psi(X,T;\mathbf{G},\bg)$ (Section~\ref{s:symmetries-intro}), evaluate it and its derivative $\Psi_X$ at $(X,T)=(0,0)$ (Section~\ref{s:origin-intro}) and give its $L^2$-norm (Section~\ref{s:norm-intro}), give partial and ordinary differential equations satisfied by $\Psi(X,T;\mathbf{G},\bg)$ (Section~\ref{s:DE-intro}), and give a new Fredholm determinant formula for the initial condition (Section~\ref{s:FredholmDeterminant-intro}).
\subsubsection{Symmetries}
\label{s:symmetries-intro}
In the setting that $\Psi(X,T;\mathbf{G},\bg)$ arises from the joint near-field/high-order limit of rogue wave solutions of \eqref{nls-general}, the parameter $\bg>0$ has the interpretation of the amplitude of the background wave supporting the rogue waves.  However, it is not hard to see that the dependence on $\bg>0$ can be scaled out of $\Psi$ by the scaling invariance $\Psi(X,T;\mathbf{G},\bg) \mapsto \bg^{-1} \Psi(\bg^{-1}X, \bg^{-2} T;\mathbf{G},\bg)$ of the focusing NLS equation \eqref{nls}. 
\begin{proposition}[Scaling symmetry]
\label{p:scaling}
Given $\mathbf{G}$ with $\det(\mathbf{G})=1$ and $\mathbf{G}^*=\sigma_2\mathbf{G}\sigma_2$, for each $\bg>0$ and $(X,T)\in\mathbb{R}^2$, we have
\begin{equation}
\Psi(X,T; \mathbf{G}, \bg) = \bg \Psi(\bg X, \bg^2 T; \mathbf{G}, 1).
\end{equation}
\end{proposition}
We give a proof in Appendix~\ref{a:elementary}. In this paper we make use of Proposition~\ref{p:scaling} and take $\bg=1$. Accordingly, we write
\begin{equation}
\Psi(X,T;\mathbf{G}):=  \Psi(X,T;\mathbf{G},\bg=1)
\end{equation}
to denote the special solution under study, and similarly we generally omit $\bg=1$ from the argument list of $\mathbf{P}$ going forward. 
On the other hand, the dependence of $\Psi(X,T;\mathbf{G})$ on the $2\times 2$ matrix $\mathbf{G}$ is nontrivial. 

We proceed with two observations that concern the symmetries with respect to reflections in the $X$ variable and in the $T$ variable.

\begin{proposition}[Reflection in $X$]
\label{prop:X-symmetry}
 $\Psi(X,T; \mathbf{G}(a,b)) = \Psi(-X,T; \mathbf{G}(b,a))$.
\end{proposition}

Similarly, we have
\begin{proposition}[Reflection in $T$]
\label{prop:T-symmetry}
$\Psi(X,-T; \mathbf{G}(a,b)) = \Psi(X,T; \mathbf{G}(a,b)^*)^*$.
\end{proposition}

The proofs of Proposition~\ref{prop:X-symmetry} and Proposition~\ref{prop:T-symmetry} are in Appendix~\ref{a:elementary}.
The next observation we make concerns a useful normalization of the parameters $a,b$. Indeed, we have the following result, which is also proved in Appendix~\ref{a:elementary}.
\begin{proposition}[Normalized parameters]
For all $(X,T)\in\mathbb{R}^2$ and $a,b\in\mathbb{C}$ with $ab\neq 0$,
\begin{equation}
\Psi(X,T;\mathbf{G}(a,b))=\ee^{-\ii\arg(ab)}\Psi\left(X,T;\mathbf{G}\left(\mathfrak{a},\mathfrak{b}\right)\right),
\label{ab-symmetry}
\end{equation}
where
\begin{equation}
\mathfrak{a}:=\frac{|a|}{\sqrt{|a|^2+|b|^2}}\quad\text{and}\quad
\mathfrak{b}:=\frac{|b|}{\sqrt{|a|^2+|b|^2}}
\label{eq:normalized-ab}
\end{equation}
satisfy $\mathfrak{a},\mathfrak{b}>0$ with $\mathfrak{a}^2+\mathfrak{b}^2=1$.
\label{prop:a-b-scaling}
\end{proposition}
Therefore, up to a phase factor, there is just one real parameter in the family of solutions $\Psi(X,T;\mathbf{G})$ with $ab\neq 0$, which one could take as $\mathfrak{a}\in (0,1)$, or equivalently as an angle $\eta\in (0,\frac{1}{2}\pi)$ for which $\mathfrak{a}=\cos(\eta)$ and $\mathfrak{b}=\sin(\eta)$.  The coordinate $\eta$ was used, for example, in the analysis of \cite[Section 2.3]{LiM24}.  Combining Proposition~\ref{prop:T-symmetry} and Proposition~\ref{prop:a-b-scaling} for $T=0$ shows that 
\begin{equation}
\begin{split}
\Psi(X,0;\mathbf{G}(a,b))&=\ee^{-\ii\arg(ab)}\Psi(X,0;\mathbf{G}(\mathfrak{a},\mathfrak{b}))\\ &=\ee^{-\ii\arg(ab)}\Psi(X,0;\mathbf{G}(\mathfrak{a},\mathfrak{b}))^*\\ &=\ee^{-2\ii\arg(ab)}\Psi(X,0;\mathbf{G}(a,b))^*
\end{split}
\label{eq:Psi-T0-symmetry}
\end{equation}
because $\mathbf{G}(\mathfrak{a},\mathfrak{b})$ is a real matrix.  It follows that $X\mapsto \ee^{\ii\arg(ab)}\Psi(X,0;\mathbf{G}(a,b))$ is a real-valued function of $X\in\mathbb{R}$.  


The \emph{normalized parameters} $\mathfrak{a}>0$ and $\mathfrak{b}>0$ with $\mathfrak{a}^2+\mathfrak{b}^2=1$ will be used in the proofs of our asymptotic results to be described in Section~\ref{s:asymptotics-intro} below.  So that they are available later we record here the following four standard matrix factorizations of the central factor $\mathbf{G}(\mathfrak{a},\mathfrak{b})$ in the jump matrix \eqref{P-jump}, which have been further manipulated to have a diagonal matrix as the leftmost factor:
\begin{equation}
\mathbf{G}(\mathfrak{a},\mathfrak{b})=\begin{bmatrix}\mathfrak{a} & \mathfrak{b}\\-\mathfrak{b} & \mathfrak{a}\end{bmatrix}=\mathfrak{a}^{\sigma_3}\begin{bmatrix}1 & 0\\-\mathfrak{ab} & 1\end{bmatrix}\begin{bmatrix}1 & \displaystyle\frac{\mathfrak{b}}{\mathfrak{a}}\\0 & 1\end{bmatrix},\quad \text{(``LDU'')},
\label{Gnorm-LDU}
\end{equation}
\begin{equation}
\mathbf{G}(\mathfrak{a},\mathfrak{b})=\begin{bmatrix}\mathfrak{a} & \mathfrak{b}\\-\mathfrak{b} & \mathfrak{a}\end{bmatrix}=\mathfrak{a}^{-\sigma_3}\begin{bmatrix}1 & \mathfrak{ab}\\0 & 1\end{bmatrix}\begin{bmatrix}1 & 0\\\displaystyle-\frac{\mathfrak{b}}{\mathfrak{a}} & 1\end{bmatrix},\quad\text{(``UDL'')},
\label{Gnorm-UDL}
\end{equation}
\begin{equation}
\mathbf{G}(\mathfrak{a},\mathfrak{b})=\begin{bmatrix}\mathfrak{a} & \mathfrak{b}\\-\mathfrak{b} & \mathfrak{a}\end{bmatrix}=\mathfrak{a}^{\sigma_3}\begin{bmatrix}1&0\\
\displaystyle\frac{\mathfrak{a}^3}{\mathfrak{b}} & 1\end{bmatrix}
\begin{bmatrix}0&\displaystyle \frac{\mathfrak{b}}{\mathfrak{a}}\\\displaystyle -\frac{\mathfrak{a}}{\mathfrak{b}} & 0\end{bmatrix}
\begin{bmatrix}1 & 0\\\displaystyle\frac{\mathfrak{a}}{\mathfrak{b}} & 1\end{bmatrix},\quad
\text{(``LTL'')},
\label{Gnorm-LTL}
\end{equation}
\begin{equation}
\mathbf{G}(\mathfrak{a},\mathfrak{b})=\begin{bmatrix}\mathfrak{a} & \mathfrak{b}\\-\mathfrak{b} & \mathfrak{a}\end{bmatrix}=\mathfrak{a}^{-\sigma_3}\begin{bmatrix}1 & \displaystyle -\frac{\mathfrak{a}^3}{\mathfrak{b}}\\0 & 1\end{bmatrix}
\begin{bmatrix}0 & \displaystyle\frac{\mathfrak{a}}{\mathfrak{b}}\\\displaystyle-\frac{\mathfrak{b}}{\mathfrak{a}} & 0\end{bmatrix}\begin{bmatrix}1 & \displaystyle-\frac{\mathfrak{a}}{\mathfrak{b}}\\0 & 1\end{bmatrix},\quad \text{(``UTU'')}.
\label{Gnorm-UTU}
\end{equation}

\subsubsection{$\Psi(X,0;\mathbf{G})$ near $X=0$}
\label{s:origin-intro}
It is straightforward to solve Riemann-Hilbert Problem~\ref{rhp:near-field} explicitly when $(X,T)=(0,0)$.  Indeed, one can verify directly that the solution is:
\begin{equation}
\mathbf{P}(\Lambda;0,0,\mathbf{G})
=\begin{cases}
\mathbf{G}^{-1},& |\Lambda|<1,\\
\mathbf{G}^{-1}\ee^{-2\ii\Lambda^{-1}\sigma_3}\mathbf{G}\ee^{2\ii\Lambda^{-1}\sigma_3},&|\Lambda|>1.
\end{cases}
\label{eq:P-at-0-0}
\end{equation}
Then, it follows from this formula assuming $|\Lambda|>1$ that
$\mathbf{P}(\Lambda;0,0,\mathbf{G})=\mathbb{I} + (2\ii\sigma_3 - 2\ii\mathbf{G}^{-1}\sigma_3\mathbf{G})\Lambda^{-1} + O(\Lambda^{-2})$ as $\Lambda\to\infty$.  Therefore \eqref{Psi-def} yields the following.
\begin{theorem}[Value at the origin]
\begin{equation}
\Psi(0,0;\mathbf{G})=4\left(\mathbf{G}^{-1}\sigma_3\mathbf{G}\right)_{12} = 8\mathfrak{a}\mathfrak{b}\ee^{-\ii\arg(ab)}=\frac{8a^*b^*}{|a|^2+|b|^2}.
\label{eq:Psi-at-0-0}
\end{equation}
\label{t:Psi-at-0-0}
\end{theorem}

Following \cite[Section 2.3.3]{LiM24}, it is then systematic to calculate derivatives of $\mathbf{P}(\Lambda;X,0,\mathbf{G})$ with respect to $X$ at $X=0$.  For instance, setting $\mathbf{F}(\Lambda;X):=\mathbf{P}_X(\Lambda;X,0,\mathbf{G})\mathbf{P}(\Lambda;X,0,\mathbf{G})^{-1}$, one sees that $\Lambda\mapsto\mathbf{F}(\Lambda;X)$ is analytic for $|\Lambda|\neq 1$, that $\mathbf{F}(\Lambda;X)\to \mathbf{0}$ as $\Lambda\to\infty$, and that $\mathbf{P}_+(\Lambda;X,0,\mathbf{G})=\mathbf{P}_-(\Lambda;X,0,\mathbf{G})\mathbf{V}(\Lambda;X)$ implies that also 
\begin{equation}
\mathbf{F}_+(\Lambda;X)-\mathbf{F}_-(\Lambda;X)=\mathbf{P}_-(\Lambda;X,0,\mathbf{G})\mathbf{V}_X(\Lambda;X)\mathbf{V}(\Lambda;X)^{-1}\mathbf{P}_-(\Lambda;X,0,\mathbf{G})^{-1},\quad |\Lambda|=1.
\end{equation}
It follows that (using the Plemelj formula and taking into account the clockwise orientation of the jump contour)
\begin{equation}
\begin{split}
\mathbf{F}(\Lambda;X)&=-\frac{1}{2\pi\ii}\oint_{|\mu|=1}\frac{\mathbf{P}_-(\mu;X,0,\mathbf{G})\mathbf{V}_X(\mu;X)\mathbf{V}(\mu;X)^{-1}\mathbf{P}_-(\mu;X,0,\mathbf{G})^{-1}}{\mu-\Lambda}\,\dd\mu\\
&=\frac{1}{2\pi\ii\Lambda}\oint_{|\mu|=1}\mathbf{P}_-(\mu;X,0,\mathbf{G})\mathbf{V}_X(\mu;X)\mathbf{V}(\mu;X)^{-1}\mathbf{P}_-(\mu;X,0,\mathbf{G})^{-1}\,\dd\mu + O(\Lambda^{-2})
\end{split}
\end{equation}
as $\Lambda\to\infty$, 
where on both lines the integration contour has counterclockwise orientation, and where $\mathbf{P}_-(\mu;X,0,\mathbf{G})$ refers to the boundary value taken from the interior of the unit circle.  Since according to \eqref{P-jump} the jump matrix is given by $\mathbf{V}(\Lambda;X):=\ee^{-\ii(\Lambda X + 2\Lambda^{-1})\sigma_3}\mathbf{G}\ee^{\ii (\Lambda X+2\Lambda^{-1})\sigma_3}$, we get that
\begin{equation}
\mathbf{V}_X(\Lambda;0)\mathbf{V}(\Lambda;0)^{-1}=-\ii\Lambda\sigma_3 +\ii\Lambda \ee^{-2\ii\Lambda^{-1}\sigma_3}\mathbf{G}\sigma_3\mathbf{G}^{-1}\ee^{2\ii\Lambda^{-1}\sigma_3},
\end{equation}
and according to \eqref{eq:P-at-0-0} we have $\mathbf{P}_-(\Lambda;0,0,\mathbf{G})=\mathbf{G}^{-1}$.  Therefore, as $\Lambda\to\infty$,
\begin{equation}
\mathbf{F}(\Lambda;0)=\frac{1}{2\pi\ii\Lambda}\oint_{|\mu|=1}\left[-\ii\mu\mathbf{G}^{-1}\sigma_3\mathbf{G} + \ii\mu \mathbf{G}^{-1}\ee^{-2\ii\mu^{-1}\sigma_3}\mathbf{G}\sigma_3\mathbf{G}^{-1}\ee^{2\ii\mu^{-1}\sigma_3}\mathbf{G}\right]\,\dd\mu + O(\Lambda^{-2}).
\end{equation}
The first term vanishes by Cauchy's theorem, and the second term can be evaluated by residues at $\mu=\infty$ using the expansion $\ee^{\pm 2\ii\mu^{-1}\sigma_3}=\mathbb{I} \pm 2\ii\sigma_3\mu^{-1}-2\mathbb{I}\mu^{-2} + O(\mu^{-3})$ as $\mu\to\infty$.  The result is that
\begin{equation}
\mathbf{F}(\Lambda;0)=\left[4\ii\mathbf{G}^{-1}\sigma_3\mathbf{G}\sigma_3\mathbf{G}^{-1}\sigma_3\mathbf{G}-4\ii\sigma_3\right]\Lambda^{-1}+O(\Lambda^{-2}),\quad\Lambda\to\infty.
\end{equation}
Differentiation of \eqref{Psi-def} then yields
\begin{equation}
\Psi_X(0,0;\mathbf{G}) = 2\ii\lim_{\Lambda\to\infty}\Lambda\frac{\partial P_{12}}{\partial X}(\Lambda;0,0,\mathbf{G}) = 2\ii\lim_{\Lambda\to\infty}\Lambda F_{12}(\Lambda;0) = -8\left[\mathbf{G}^{-1}\sigma_3\mathbf{G}\sigma_3\mathbf{G}^{-1}\sigma_3\mathbf{G}\right]_{12}.
\end{equation}
Explicit evaluation using \eqref{G-form} then yields the following result.
\begin{theorem}[Derivative at the origin]
\begin{equation}
\Psi_X(0,0;\mathbf{G})=32\ee^{-\ii\arg(ab)}\mathfrak{a}\mathfrak{b} (\mathfrak{b}^2-\mathfrak{a}^2)=32a^*b^*\frac{|b|^2-|a|^2}{(|a|^2+|b|^2)^2}.
\label{eq:PsiX-at-0-0}
\end{equation}
\label{t:PsiX-at-0-0}
\end{theorem}


\subsubsection{Exceptional parameter values and $L^2$ norm}
\label{s:norm-intro}
Another interesting result for the family $\Psi(X,T;\mathbf{G})$ of exact solutions of \eqref{nls} has to do with special values of the parameters.
\begin{proposition}[Degeneration property] If either $a=0$ or $b=0$, then $\Psi(X,T;\mathbf{G})\equiv 0$.
\label{p:degeneration}
\end{proposition} 
We give a proof of Proposition~\ref{p:degeneration} in Appendix~\ref{a:elementary}. The proof relies on the fact that \rhref{rhp:near-field} can be solved explicitly in either of the cases $a=0$ or $b=0$. It follows from Proposition~\ref{prop:RHP-EU} 
that the function $\Psi(X,T;\mathbf{G})$ depends continuously on the parameters $(a,b)$ for fixed $(X,T)\in\mathbb{R}^2$, so Proposition~\ref{p:degeneration} also implies the pointwise limit $\Psi(X,T;\mathbf{G})\to 0$ as $a\to 0$ or $b\to 0$ for each $(X,T)\in\mathbb{R}^2$, and this convergence can be generalized to be uniform over $(X,T)$ ranging over any given compact set in $\mathbb{R}^2$.  To avoid trivial cases, from this point on in the paper we therefore assume that $a,b$ are complex numbers with $ab\neq 0$. 

Another result is that for general $\mathbf{G}(a,b)$ with $ab\neq 0$, $\Psi(X,T;\mathbf{G})$ lies in $L^2(\mathbb{R})$ as a function of $X$ with an $L^2(\mathbb{R})$-norm that is \emph{independent of the parameter matrix $\mathbf{G}=\mathbf{G}(a,b)$}. Namely, we prove the following theorem.  
\begin{theorem}[$L^2$-norm of $\Psi(\diamond,T;\mathbf{G})$]
Let $\mathbf{G}=\mathbf{G}(a,b)$ be as in \eqref{G-form} with $ab \neq 0$. We have that $\Psi(\diamond, T; \mathbf{G})\in L^2(\mathbb{R})$ for all $T\in\mathbb{R}$ with $\| \Psi(\diamond, T;\mathbf{G} )\|_{L^2(\mathbb{R})}=\sqrt{8}$.
\label{t:L2-norm}
\end{theorem}

When combined with the degeneration property given in Proposition~\ref{p:degeneration}, the independence of the $L^2$-norm of $\Psi(X,T,\mathbf{G})$ from the matrix $\mathbf{G}=\mathbf{G}(a,b)$ asserted in Theorem~\ref{t:L2-norm} leaves us with an interesting conundrum! While $\|\Psi(\diamond,T;\mathbf{G}(a,b))\|_{L^2(\mathbb{R})}=\sqrt{8}$ for any nonzero $a,b\in\mathbb{C}$, we have
\begin{equation}
\lim_{a\to 0}\Psi(X,T;\mathbf{G}(a,b)) = 0\quad\text{and}\quad \lim_{b\to 0}\Psi(X,T;\mathbf{G}(a,b)) = 0
\end{equation}
pointwise for any $(X,T)\in\mathbb{R}^2$. 
A mechanism for this limiting behavior could be that the $L^2$ mass of the wave packet $\Psi(X,T;\mathbf{G}(a,b))$ at any given time $T$ escapes to $X=\pm\infty$ in the limit $a\to 0$ or $b\to 0$, or alternately, the wave packet spreads in the same limit so as to preserve the $L^2$-norm while still decaying pointwise or perhaps even uniformly to zero.  In fact, it turns out that a combination of both of these mechanisms is at play.
The explanation of this phenomenon lies in a double-scaling limit in which $X\to+\infty$ while also $a\to 0$ or $b\to 0$ at suitably-related rates.  The details can be found in our next paper on the subject, \cite{BilmanM2024b}.

\subsubsection{Differential equations}
\label{s:DE-intro}
It is straightforward to derive three different first-order systems of differential equations satisfied by the matrix function 
\begin{equation}
\mathbf{W}(\Lambda;X,T,\mathbf{G}):=\mathbf{P}(\Lambda;X,T,\mathbf{G})\ee^{-\ii (\Lambda X+\Lambda^2T+2\Lambda^{-1})\sigma_3}
\end{equation}
by following the procedure in \cite[Section 3.2.1]{BilmanLM2020}, which though written with the special case \eqref{G-Qi} in mind, in fact does not depend at all on the details of the matrix $\mathbf{G}(a,b)$ and hence applies to general rogue waves of infinite order.  It follows that the matrix $\mathbf{W}(\Lambda;X,T,\mathbf{G})$ satisfies the three Lax equations
\begin{equation}
\frac{\partial \mathbf{W}}{\partial X} = \mathbf{X}\mathbf{W},\quad\frac{\partial \mathbf{W}}{\partial T}=\mathbf{T}\mathbf{W},\quad\frac{\partial \mathbf{W}}{\partial\Lambda}=\mathbf{\Lambda}\mathbf{W},
\label{eq:Lax-equations}
\end{equation}
wherein the coefficient matrices $\mathbf{X}$, $\mathbf{T}$, and $\mathbf{\Lambda}$ are explicitly represented in terms of the coefficients $\mathbf{P}^{[j]}=\mathbf{P}^{[j]}(X,T;\mathbf{G})$ in the convergent Laurent expansion
\begin{equation}
\mathbf{P}(\Lambda;X,T,\mathbf{G})=\mathbb{I}+\sum_{j=1}^\infty\mathbf{P}^{[j]}(X,T;\mathbf{G})\Lambda^{-j},\quad |\Lambda|>1
\end{equation}
or alternatively in terms of the Taylor coefficients $\mathbf{P}_0:=\mathbf{P}(0;X,T,\mathbf{G})$ etc., of $\mathbf{P}(\Lambda;X,T,\mathbf{G})$ at $\Lambda=0$.  Thus:
\begin{equation}
\mathbf{X} = -\ii\Lambda\sigma_3 + \ii [\sigma_3,\mathbf{P}^{[1]}] = \begin{bmatrix}-\ii\Lambda & \Psi\\-\Psi^* & \ii\Lambda\end{bmatrix},
\label{eq:X-matrix}
\end{equation}
\begin{equation}
\mathbf{T}=-\ii\Lambda^2\sigma_3 + \ii [\sigma_3,\mathbf{P}^{[1]}]\Lambda + \ii[\mathbf{P}^{[1]},\sigma_3\mathbf{P}^{[1]}] + \ii[\sigma_3,\mathbf{P}^{[2]}] = \begin{bmatrix}
-\ii\Lambda^2 +\frac{1}{2}\ii |\Psi|^2 & \Lambda\Psi + \frac{1}{2}\ii\Psi_X\\
-\Lambda\Psi^* +\frac{1}{2}\ii\Psi^*_X & \ii\Lambda^2 -\frac{1}{2}\ii |\Psi|^2\end{bmatrix},
\label{eq:T-matrix}
\end{equation}
and
\begin{equation}
\mathbf{\Lambda}=\begin{bmatrix}-2\ii T\Lambda -\ii X +\ii T|\Psi|^2\Lambda^{-1} & 2T\Psi +(X\Psi+\ii T\Psi_X)\Lambda^{-1}\\-2T\Psi^*+(-X\Psi^*+\ii T\Psi_X^*)\Lambda^{-1} & 2\ii T\Lambda +\ii X-\ii T|\Psi|^2\Lambda^{-1}\end{bmatrix} + 2\ii \mathbf{P}_0\sigma_3\mathbf{P}_0^{-1}\Lambda^{-2}.
\label{eq:Lambda-matrix}
\end{equation}
The global existence of $\mathbf{P}(\Lambda;X,T,\mathbf{G})$ from Riemann-Hilbert Problem~\ref{rhp:near-field} recorded in Proposition~\ref{prop:RHP-EU} then guarantees that the three Lax equations \eqref{eq:Lax-equations} are mutually compatible.  In particular, the compatibility condition $\mathbf{X}_T-\mathbf{T}_X + [\mathbf{X},\mathbf{T}]=\mathbf{0}$ implies the following basic result which has already been mentioned.
\begin{theorem}[$\Psi(X,T)$ solves NLS]
The function $\mathbb{R}^2\ni (X,T)\mapsto \Psi(X,T;\mathbf{G})\in\mathbb{C}$ obtained from Riemann-Hilbert Problem~\ref{rhp:near-field} via \eqref{Psi-def} is a global solution of the focusing NLS equation in the form \eqref{nls}.
\label{t:nls}
\end{theorem}
On the other hand, the compatibility condition $\mathbf{X}_\Lambda-\mathbf{\Lambda}_X + [\mathbf{X},\mathbf{\Lambda}]=\mathbf{0}$ is a system of ordinary differential equations with respect to $X$ (in which $T\in\mathbb{R}$ plays the role of a parameter) that was shown in \cite[Section 3.2.1]{BilmanLM2020} to be related to the second equation in the Painlev\'e-III hierarchy of Sakka \cite{Sakka2009} when $T\neq 0$ and to the Painlev\'e-III (D$_6$) equation itself when $T=0$.  
More explicitly, we have the following.
\begin{theorem}[$\Psi(X,0)$ and the Painlev\'e-III (D$_6$) equation \protect{\cite[Corollary 4]{BilmanLM2020}}]
The function $u(x)$ defined for $x\in\mathbb{R}\cup(\ii\mathbb{R})$ by
\begin{equation}
u(x):=2\left(\frac{\dd}{\dd x}\log\left(x^2\Psi(-\tfrac{1}{8}x^2,0;\mathbf{G})\right)\right)^{-1}
\label{eq:u-Psi}
\end{equation}
is a solution of the Painlev\'e-III (D$_6$) equation 
\begin{equation}
\frac{\dd^2u}{\dd x^2}=\frac{1}{u}\left(\frac{\dd u}{\dd x}\right)^2-\frac{1}{x}\frac{\dd u}{\dd x} + \frac{4\Theta_0 u^2 + 4(1-\Theta_\infty)}{x}+4u^3-\frac{4}{u}
\label{eq:PIII-D6}
\end{equation}
in the case that both formal monodromy parameters vanish:  $\Theta_0=\Theta_\infty=0$.
\end{theorem}

\begin{corollary}[Behavior of $u(x)$ near $x=0$]
The function $u(x)$ defined by \eqref{eq:u-Psi} is an odd function of $x$ that is analytic at the origin with Taylor expansion
\begin{equation}
u(x)=x + \frac{u'''(0)}{3!}x^3 + O(x^5),\quad x\to 0,\quad u'''(0)=3(\mathfrak{b}^2-\mathfrak{a}^2)=3\frac{|b|^2-|a|^2}{|a|^2+|b|^2}.
\end{equation}
\end{corollary}
\begin{proof}
Combining Proposition~\ref{prop:RHP-EU} with \eqref{eq:u-Psi} shows that $u(x)$ is an odd function having a Taylor expansion about $x=0$ with $u'''(0)=\frac{3}{4}\Psi_X(0,0;\mathbf{G})/\Psi(0,0;\mathbf{G})$.  Theorems~\ref{t:Psi-at-0-0} and \ref{t:PsiX-at-0-0} then yield the claimed value of $u'''(0)$.
\end{proof}

Since the value of $\Psi(0,0;\mathbf{G})$ is known from Theorem~\ref{t:Psi-at-0-0}, it is straightforward to invert \eqref{eq:u-Psi} to explicitly express $\Psi(X,0;\mathbf{G})$ in terms of $u$:
\begin{equation}
\Psi(X,0;\mathbf{G})=\Psi(0,0;\mathbf{G})\exp\left(2\int_0^x\left[\frac{1}{u(y)}-\frac{1}{y}\right]\,\dd y\right),\quad X=-\frac{1}{8}x^2.
\label{eq:u-Psi-invert}
\end{equation}
We note that $-3<u'''(0)<3$.  In fact, when $\Theta_0=\Theta_\infty=0$ there is for each $\omega\in (-3,3)$ a unique solution of \eqref{eq:PIII-D6} analytic at the origin with $u(0)=0$, $u'(0)=1$, $u''(0)=0$, and $u'''(0)=\omega$.  This family of solutions of the Painlev\'e-III (D$_6$) equation has not only been associated with limits of sequences of solutions of the focusing NLS equation \cite{BilmanB2019,BilmanLM2020}, but has also appeared in the description of self-similar boundary layers in the sharp-line Maxwell-Bloch equations \cite{LiM24}.

\begin{remark}
The parametrization of solutions of \eqref{eq:PIII-D6} is discussed for example in \cite[Section 4.6]{BarhoumiLMP23} (see also \cite{PutS09}).  Solutions are parametrized by points on a certain \emph{monodromy manifold} characterized by a cubic equation in three variables.  In particular, the monodromy manifold for the Painlev\'e-III (D$_6$) equation \eqref{eq:PIII-D6} with parameters $\Theta_0=\Theta_\infty=0$ consists of those $(x_1,x_2,x_3)\in \mathbb{C}^3$ for which
\begin{equation}
x_1x_2x_3+x_1^2+x_2^2+2x_1+2x_2+1=0.
\label{eq:PIII-monodromy-cubic}
\end{equation}
This is a smooth complex surface except at two points obtained by also enforcing that the gradient of the left-hand side vanishes:  $(x_1,x_2,x_3)=(0,-1,2)$ and $(x_1,x_2,x_3)=(-1,0,2)$.  Because the Stokes phenomenon is trivial both at $\Lambda=0$ and $\Lambda=\infty$ (nontrivial Stokes phenomenon would be evident in additional jump conditions for \rhref{rhp:near-field} on two contours approaching the origin and two contours approaching $\Lambda=\infty$), 
the coordinates of the Painlev\'e-III (D$_6$) solution arising for $T=0$  from Riemann-Hilbert Problem~\ref{rhp:near-field} are determined from the matrix $\mathbf{G}$ alone (which plays the role of a \emph{connection matrix} in the isomonodromy theory) by $(x_1,x_2,x_3)=(G_{11}G_{22}-1,-G_{11}G_{22},2)=(\mathfrak{a}^2-1,-\mathfrak{a}^2,2)$.  Since $0<\mathfrak{a}^2<1$, this is a segment of a line in $\mathbb{C}^3$ parametrized by $G_{11}G_{22}=\mathfrak{a}^2$ that is fully contained within the monodromy manifold \eqref{eq:PIII-monodromy-cubic}.  The line passes through both singular points at parameter values $G_{11}G_{22}=\mathfrak{a}^2=0$ and $G_{11}G_{22}=\mathfrak{a}^2=1$, which are the endpoints of the relevant segment. Thus, the endpoints of the segment correspond to normalized parameters $(\mathfrak{a},\mathfrak{b})=(0,1)$ or $(\mathfrak{a},\mathfrak{b})=(1,0)$ respectively.  The singular points on the monodromy manifold then both correspond to the trivial solution $\Psi(X,T;\mathbf{G})\equiv 0$ according to Proposition~\ref{p:degeneration}.  However, each interior point of the segment yields a distinct nontrivial solution $u(x)=u(x;\mathfrak{a})$ of the Painlev\'e-III (D$_6$) equation \eqref{eq:PIII-D6} with $\Theta_0=\Theta_\infty=0$ related to $\Psi(X,0;\mathbf{G})$ via \eqref{eq:u-Psi} and its inverse \eqref{eq:u-Psi-invert}.  Note that the presence of the logarithmic derivative in \eqref{eq:u-Psi} cancels the constant phase factor $\ee^{-\ii\arg(ab)}$ in $\Psi(X,0;\mathbf{G}(a,b))$ (see \eqref{eq:Psi-T0-symmetry}) so that while $\Psi$ depends on the parameters $(a,b)$, $u$ indeed depends only on the normalized parameters $(\mathfrak{a},\mathfrak{b})$.
\label{rem:singular-monodromy}
\end{remark}

The compatibility condition $\mathbf{T}_\Lambda-\mathbf{\Lambda}_T + [\mathbf{T},\mathbf{\Lambda}]=\mathbf{0}$ is a system of ordinary differential equations with respect to $T$ in which $X\in\mathbb{R}$ is a parameter.  This system was written out in general in \cite[Eqn.\@ (119)]{BilmanLM2020}, and here we expand on a remark made in that paper concerning a symmetric special case (see also \cite{Suleimanov17}).  Suppose that $b=a$.  Then according to Proposition~\ref{prop:X-symmetry}, $X\mapsto\Psi(X,T;\mathbf{G}(a,a))$ is an even function for all $T\in\mathbb{R}$, which in light of real analyticity at $X=0$  implies also that 
\begin{equation}
\Psi_X(0,T;\mathbf{G}(a,a))=0
\label{eq:Psi-X-b=a}
\end{equation}
as is consistent with the conclusion of Theorem~\ref{t:PsiX-at-0-0} when also $T=0$.  Thus $X=0$ is an axis of symmetry of $\Psi(X,T;\mathbf{G}(a,a))$, and one may expect some simplification of the compatibility condition $\mathbf{T}_\Lambda-\mathbf{\Lambda}_T+[\mathbf{T},\mathbf{\Lambda}]=\mathbf{0}$ yielding ordinary differential equations satisfied by $T\mapsto \Psi(X,T;\mathbf{G})$.  To see this, we remark that Proposition~\ref{P-symmetry-X} in Appendix~\ref{a:elementary} implies that $\mathbf{P}(0;0,T,\mathbf{G}(a,a))=-\sigma_3\mathbf{P}(0;0,T,\mathbf{G}(a,a))\sigma_1$, and therefore for some function $s(T)\neq 0$ the unit-determinant matrix $\mathbf{P}_0=\mathbf{P}(0;0,T,\mathbf{G}(a,a))$ has the form
\begin{equation}
\mathbf{P}_0=\begin{bmatrix} s(T) & -s(T)\\(2s(T))^{-1} & (2s(T))^{-1}\end{bmatrix}\implies 2\ii\mathbf{P}_0\sigma_3\mathbf{P}_0^{-1}=2\ii\begin{bmatrix}0 & n(T)\\n(T)^{-1} & 0\end{bmatrix},\quad n(T):=2s(T)^2.
\end{equation}
Using this and \eqref{eq:Psi-X-b=a}, and setting $X=0$, the matrix coefficients $\mathbf{T}$ and  $\mathbf{\Lambda}$ defined in \eqref{eq:T-matrix} and \eqref{eq:Lambda-matrix} respectively take the simplified form
\begin{equation}
\mathbf{T}=-\ii\Lambda^2\sigma_3 +\begin{bmatrix}0 & \Psi\\-\Psi^* & 0\end{bmatrix}\Lambda +\frac{1}{2}\ii |\Psi|^2\sigma_3
\label{eq:T-D7}
\end{equation}
and
\begin{equation}
\mathbf{\Lambda}=-2\ii T\Lambda\sigma_3 +\begin{bmatrix}0 & 2T\Psi\\-2T\Psi^* & 0\end{bmatrix}+\ii T|\Psi|^2\Lambda^{-1}\sigma_3 +\frac{2\ii}{\Lambda^2}\begin{bmatrix} 0 & n(T)\\n(T)^{-1} & 0\end{bmatrix},
\label{eq:Lambda-D7}
\end{equation}
where $\Psi=\Psi(T):=\Psi(0,T;\mathbf{G}(a,a))$.  In \cite{KitaevV04} a Lax pair is presented for the partially-degenerate Painlev\'e-III equation (D$_7$ type, with parameters $\mathscr{A},\mathscr{B}\in\mathbb{C}$ and $\varepsilon=\pm 1$) 
\begin{equation}
\frac{\dd^2\mathfrak{u}}{\dd t^2}=\frac{1}{\mathfrak{u}}\left(\frac{\dd\mathfrak{u}}{\dd t}\right)^2 -\frac{1}{t}\frac{\dd\mathfrak{u}}{\dd t} + \frac{-8\varepsilon\mathfrak{u}^2+2\mathscr{A}\mathscr{B}}{t}+\frac{\mathscr{B}^2}{\mathfrak{u}},\quad\mathfrak{u}=\mathfrak{u}(t),
\label{eq:PIIID7}
\end{equation}
that, after a Fabry transformation, resembles the compatible linear system $\mathbf{W}_T=\mathbf{T W}$ and $\mathbf{W}_\Lambda=\mathbf{\Lambda W}$ with the coefficient matrices given in \eqref{eq:T-D7} and \eqref{eq:Lambda-D7}.  Hence we may expect that $\Psi(T)$ is related to a solution of the Painlev\'e-III (D$_7$) equation.  To complete the correspondence, we change variables and make a gauge transformation by 
\begin{equation}
\lambda:= T^\frac{1}{4}\Lambda,\quad t:=T^{\frac{1}{2}},\quad \mathbf{W}=t^{\frac{1}{2}\sigma_3}\mathbf{V},
\end{equation}
and then the resulting system takes the form
\begin{equation}
\frac{\partial\mathbf{V}}{\partial\lambda}=\left(-2\ii t\lambda\sigma_3 + 2t\begin{bmatrix}
0 & t^{-\frac{1}{2}}\Psi\\-t^\frac{3}{2}\Psi^* & 0\end{bmatrix}-\frac{1}{\lambda}\left(-\ii t^2|\Psi|^2\right)\sigma_3
+\frac{1}{\lambda^2}\begin{bmatrix}0 & 2\ii t^{-\frac{1}{2}}n\\
2\ii t^\frac{3}{2}n^{-1} & 0\end{bmatrix}\right)\mathbf{V}
\label{eq:V-mu}
\end{equation}
and
\begin{equation}
\frac{\partial\mathbf{V}}{\partial t}=\left(-\ii\lambda^2\sigma_3 + \lambda\begin{bmatrix}0 & t^{-\frac{1}{2}}\Psi\\
-t^\frac{3}{2}\Psi^* & 0\end{bmatrix} + \left(\frac{1}{2}\ii t|\Psi|^2-\frac{1}{2t}\right)\sigma_3
-\frac{1}{2t\lambda}\begin{bmatrix}0 & 2\ii t^{-\frac{1}{2}}n\\2\ii t^\frac{3}{2}n^{-1} & 0\end{bmatrix}\right)\mathbf{V},
\label{eq:V-tau}
\end{equation}
in which we now view $\Psi$, $\Psi^*$, and $n$ as functions of $t$.  This system has exactly the form of \cite[Eqn.\@ (12)]{KitaevV04} provided we make the correspondences
\begin{equation}
D=t^{\frac{3}{2}}\Psi^*,\quad \frac{2\ii A}{\sqrt{-AB}}=t^{-\frac{1}{2}}\Psi,
\label{eq:match1}
\end{equation}
\begin{equation}
\ii \mathscr{A}+\frac{1}{2}+\frac{2t AD}{\sqrt{-AB}}=-\ii t^2|\Psi|^2,\quad \frac{\ii \mathscr{A}}{2t}-\frac{AD}{\sqrt{-AB}}=\frac{1}{2}\ii t|\Psi|^2-\frac{1}{2t},
\label{eq:match2}
\end{equation}
\begin{equation}
\widetilde{\alpha}=2\ii t^{-\frac{1}{2}}n,\quad \ii t B=2\ii t^\frac{3}{2}n^{-1}.
\label{eq:match3}
\end{equation}
where the notation of \cite{KitaevV04} is on the left-hand side\footnote{In \cite{KitaevV04}, the symbols $\mathscr{A},\mathscr{B},t,\lambda$  are written as $a,b,\tau,\mu$ respectively, but the latter have other meanings in our paper.} and $A,B,C,D,\widetilde{\alpha}$ are functions of $t$ while $\mathscr{A}$ is a constant.  Using the product of the identities in \eqref{eq:match1} to eliminate $AD/\sqrt{-AB}$ shows that the two equations in \eqref{eq:match2} actually coincide, and yield $\mathscr{A}=\frac{1}{2}\ii$.  Then, according to \cite[Lemma 2.1]{KitaevV04}, the second parameter $\mathscr{B}$ is determined from \eqref{eq:match3} up to a sign by $\mathscr{B}=4\varepsilon$.  The corresponding solution of \eqref{eq:PIIID7} is then given by $\mathfrak{u}(t)=\varepsilon t\sqrt{-A(t)B(t)}$.  This proves the following, which is also easy to verify directly from the compatibility condition for the system \eqref{eq:V-mu}--\eqref{eq:V-tau}.
\begin{theorem}[$\Psi(0,T)$ for $b=a$ and the Painlev\'e-III (D$_7$) equation]
Fix $a\in\mathbb{C}$ nonzero, and let $\varepsilon=\pm 1$.  The function 
\begin{equation}
\mathfrak{u}(t):=\frac{1}{4}\ii\varepsilon t\Psi\cdot\left(\Psi^*+t\frac{\dd\Psi^*}{\dd t}\right),\quad t\in\mathbb{R},\quad\Psi=\Psi(0,t^2;\mathbf{G}(a,a))
\label{eq:u-D7-def}
\end{equation}
satisfies the Painlev\'e-III (D$_7$) equation in the form \eqref{eq:PIIID7} with parameters $\mathscr{A}=\frac{1}{2}\ii$ and $\mathscr{B}=4\varepsilon$.
\end{theorem}
This result was known to Suleimanov \cite{Suleimanov17}, although to extract it from his paper one must take the independent variable to be $t=T^\frac{1}{2}$ instead of $T$ and correct some constants.  Another simple identity satisfied by $\Psi$ that follows from the compatibility condition for \eqref{eq:V-mu}--\eqref{eq:V-tau} is 
\begin{equation}
\left|\Psi+t\frac{\dd\Psi}{\dd t}\right|^2=16,\quad t\in\mathbb{R},\quad\Psi=\Psi(0,t^2;\mathbf{G}(a,a)).
\label{eq:simple-ODE-T}
\end{equation}
\begin{corollary}[Behavior of $\mathfrak{u}(t)$ near $t=0$]
The function $\mathfrak{u}(t)$ defined by \eqref{eq:u-D7-def} is an odd function of $t$ that is analytic at the origin with Taylor expansion
\begin{equation}
\mathfrak{u}(t)=\mathfrak{u}'(0)t + O(\tau^3),\quad t\to 0,\quad \mathfrak{u}'(0)=4\ii\varepsilon.
\end{equation}
\label{cor:D7}
\end{corollary}
\begin{proof}
Both $\Psi$ and $\Psi^*$ are even functions of $t$ with real analytic real and imaginary parts, so $\mathfrak{u}(t)$ is odd and analytic at the origin.  The value of $\mathfrak{u}'(0)$ follows from Theorem~\ref{t:Psi-at-0-0} using $b=a$, or alternately, from \eqref{eq:simple-ODE-T}.
\end{proof}
According to \cite{KitaevV23}, the conditions on $\mathfrak{u}(t)$ asserted in Corollary~\ref{cor:D7} actually uniquely determine the solution of \eqref{eq:PIIID7} because the formal monodromy parameter has the special value $\mathscr{A}=\frac{1}{2}\ii$.  This value of $\mathscr{A}$ is special (so is $\mathscr{A}=\pm\frac{1}{2}\ii + k$ for any $k\in\mathbb{Z}$) in that there is a one-parameter family of solutions that are analytic and vanishing at $t=0$, but the solution becomes unique once oddness is asserted. On the other hand, for general $\mathscr{A}$ there is only one solution analytic and vanishing at $t=0$.

\subsubsection{Fredholm determinant representation of $\Psi(X,0;\mathbf{G})$}
\label{s:FredholmDeterminant-intro}
In \cite[Section 3.2.1]{BilmanLM2020} various transformations of \eqref{eq:PIII-D6} to other forms of the Painlev\'e-III equation were noted, including a transformation to the parameter-free and fully-degenerate Painlev\'e-III (D$_8$) equation, which takes the form
\begin{equation}
\frac{\dd^2U}{\dd z^2} = \frac{1}{U}\left(\frac{\dd U}{\dd z}\right)^2-\frac{1}{z}\frac{\dd U}{\dd z}+\frac{4U^2+4}{z}.
\label{eq:PIIID8}
\end{equation}
Here we report a new piece of information, which is that the relevant solution of \eqref{eq:PIIID8} is expressible in terms of a Fredholm determinant of an integrable operator \cite{ItsIKS1990,Deift1999}.  This in turn leads to an explicit representation of $\Psi(X,0;\mathbf{G})$ in terms of the same determinant.


Provided that $a>0$ and $b\in\ii\mathbb{R}$, the matrix $\hat{\mathbf{S}}(\xi;z)$ defined in terms of $\mathbf{P}(\Lambda;X,0,\mathbf{G})$ by
\begin{equation}
\hat{\mathbf{S}}(\xi;z):=\mathbf{P}(\Lambda;X,0,\mathbf{G})\ee^{(2z)^{1/2}(\ii\xi +\xi^{-1})\sigma_3},\quad X=-\ii z,\quad \Lambda=-2\ii (2z)^{-\frac{1}{2}}\xi,
\end{equation}
satisfies a related Riemann-Hilbert problem associated with the Painlev\'e-III ($D_8$) equation with specialized monodromy parameters $y_1=b/\sqrt{|a|^2+|b|^2}$, $y_2=\ii a/\sqrt{|a|^2+|b|^2}$, and $y_3=0$ (see \cite[Riemann-Hilbert Problem~9.2]{BarhoumiLMP23} and the discussion in Section~9.3 of that paper).  Comparing \eqref{Psi-def} with \cite[Eqns.\@ (9.11) and (9.14)]{BarhoumiLMP23} shows that 
\begin{equation}
\Psi(X,0;\mathbf{G})=\frac{\ii}{2}\frac{U'(z)}{U(z)},\quad z=\ii X
\label{eq:Psi-U}
\end{equation}
where $U(z)$ is a solution of the Painlev\'e-III ($D_8$) equation \eqref{eq:PIIID8}.
The parameters $y_1$ and $y_2$ are purely imaginary numbers constrained by $y_1^2+y_2^2+1=0$ (a reduction for $y_3=0$ of the monodromy cubic $y_1y_2y_3+y_1^2+y_2^2+1=0$ parametrizing solutions of \eqref{eq:PIIID8}), and they may be further parametrized by a single quantity $m\in\ii\mathbb{R}+\mathbb{Z}$ according to
\begin{equation}
y_1=\frac{\ii\ee^{\ii\pi m}}{\sqrt{1+\ee^{2\pi\ii m}}},\quad y_2=\frac{\ii}{\sqrt{1+\ee^{2\pi\ii m}}}.
\end{equation}
The generalization of this family of solutions $U(z)$ to allow for arbitrary $m\in\mathbb{C}\setminus (\mathbb{Z}+\frac{1}{2})$ was proven in \cite{BarhoumiLMP23} to correspond to a certain double-scaling limit of high-even-order rational solutions of the Painlev\'e-III ($D_6$) equation when examined near the origin in the complex plane (the only fixed singularity of the equation).  The Painlev\'e-III ($D_6$) equation has two essential parameters, one of which is quantized for rational solutions ($n=\frac{1}{2}(\Theta_0-\Theta_\infty+1)\in\mathbb{Z}$) and the other of which is arbitrary and corresponds to the value of $m=\frac{1}{2}(\Theta_0+\Theta_\infty-1)$.

Going further, the particular solution $U(z)$ with parameter $m\in\mathbb{C}\setminus(\mathbb{Z}+\frac{1}{2})$ is shown in \cite[Corollary 3.4]{BarhoumiLMP23} to be explicitly related to the Fredholm determinant of the scalar Bessel kernel:
\begin{equation}
U(z)-\frac{1}{U(z)}=R(z):=-2\ii-\frac{1}{2}\frac{\dd}{\dd z}z\frac{\dd}{\dd z}\log \left(D_{\varkappa(m)}(32\ii z)\right),
\label{eq:U-and-Uinverse}
\end{equation}
wherein $\varkappa(m):=(1+\ee^{2\pi\ii m})^{-1}=-y_2^2=a^2/(|a|^2+|b|^2)$ and 
\begin{equation}
D_\varkappa(r):=\det(\mathbb{1}-\varkappa \mathcal{K}_r)
\label{eq:FredholmDeterminant}
\end{equation}
is the Fredholm determinant of the integral operator $\mathcal{K}_r:L^2[0,r]\to L^2[0,r]$ with Bessel kernel 
\begin{equation}
K(x,y):=\frac{\sqrt{x}J_1(\sqrt{x})J_0(\sqrt{y})-J_0(\sqrt{x})\sqrt{y}J_1(\sqrt{y})}{2(x-y)}.
\label{eq:BesselKernel}
\end{equation}
It can be shown that $D_\varkappa(r)$ is entire in $\varkappa$ and analytic for $r\in\mathbb{C}$ of sufficiently small modulus.  To get a representation of $\Psi(X,0;\mathbf{G})$, we first solve \eqref{eq:U-and-Uinverse} for $U(z)$:
\begin{equation}
U(z)=\frac{1}{2}\left(R(z)\pm\sqrt{R(z)^2+4}\right),
\end{equation}
and hence from \eqref{eq:Psi-U}
\begin{equation}
\Psi(X,0;\mathbf{G})=\frac{\ii}{2}\frac{U'(z)}{U(z)}=\pm\frac{\ii R'(z)}{2\sqrt{R(z)^2+4}}.
\label{eq:Psi-Fredholm}
\end{equation}
We now have to resolve the sign of the square root; first, according to \cite[Eqn.\@ (3.22)]{BarhoumiLMP23}, one has
\begin{equation}
\log \left(D_\varkappa(r)\right)=-\frac{\varkappa}{4}r +\frac{\varkappa-\varkappa^2}{32}r^2+O(r^3),\quad r\to 0,
\end{equation}
which implies that $R(z)=-2\ii+4\ii\varkappa + 64(\varkappa-\varkappa^2)z + O(z^2)$ and $R'(z)=64(\varkappa-\varkappa^2)+O(z)$ as $z\to 0$.  Hence $R'(0)=64(\varkappa-\varkappa^2)$ and $R(0)^2+4=16(\varkappa-\varkappa^2)$.  Since $\varkappa=a^2/(|a|^2+|b|^2)\in (0,1)$, we obtain (using $z=0$ implies $X=0$)
\begin{equation}
\Psi(0,0;\mathbf{G})=\pm 8\ii\sqrt{\varkappa-\varkappa^2}=\pm\frac{8\ii a|b|}{a^2+|b|^2}.
\end{equation}
Comparing this with Theorem~\ref{t:Psi-at-0-0} and recalling that $a>0$ while $b$ is purely imaginary shows that we should choose the $+$ (resp., $-$) sign in \eqref{eq:Psi-Fredholm} when $b$ is negative (resp., positive) imaginary, taking the square root as positive when $X=0$.  

The representation of $\Psi(X,0;\mathbf{G})$ in terms of a Fredholm determinant is easily generalized to arbitrary $(a,b)\in\mathbb{C}^2$ with $ab\neq 0$ using Proposition~\ref{prop:a-b-scaling}, since
\begin{multline}
\Psi(X,T;\mathbf{G}(a,b))=\ee^{-\ii\arg(ab)}\Psi(X,T;\mathbf{G}(\mathfrak{a},\mathfrak{b})) \;\text{and}\; \Psi(X,T;\mathbf{G}(\mathfrak{a},\ii\mathfrak{b}))=-\ii\Psi(X,T;\mathbf{G}(\mathfrak{a},\mathfrak{b}))\\\implies \Psi(X,T;\mathbf{G}(a,b))=\ii\ee^{-\ii\arg(ab)}\Psi(X,T;\mathbf{G}(\mathfrak{a},\ii\mathfrak{b})),
\end{multline}
and the latter has parameters $a=\mathfrak{a}$ positive and $b=\ii\mathfrak{b}$ positive imaginary.

This proves the following.
\begin{theorem}[Painlev\'e-III (D$_8$) and Bessel kernel determinant formula for $\Psi(X,0)$]
Let $(a,b)\in\mathbb{C}^2$ with $ab\neq 0$ correspond to normalized parameters $(\mathfrak{a},\mathfrak{b})$ by \eqref{eq:normalized-ab}.  Then the function $\Psi(X,0;\mathbf{G}(\mathfrak{a},\ii\mathfrak{b}))=-\ii\ee^{\ii\arg(ab)}\Psi(X,0;\mathbf{G}(a,b))$ is expressible by \eqref{eq:Psi-U} in terms of a solution $U(z)$ of the Painlev\'e-III (D$_8$) equation \eqref{eq:PIIID8} having unit modulus for $z\in\ii\mathbb{R}$, and moreover for $X$ in a neighborhood of the origin,
\begin{equation}
\Psi(X,0;\mathbf{G}(a,b))=\left.\frac{\ee^{-\ii\arg(ab)}R'(z)}{2\sqrt{R(z)^2+4}}\right|_{z=\ii X},
\label{eq:Psi-det}
\end{equation}
where 
\begin{equation}
R(z):=-2\ii -\frac{1}{2}\frac{\dd}{\dd z}z\frac{\dd}{\dd z}\log \left(D_{\mathfrak{a}^2}(32\ii z)\right)
\end{equation}
and $D_\varkappa(r)$ denotes the Fredholm determinant \eqref{eq:FredholmDeterminant} of the Bessel kernel \eqref{eq:BesselKernel}.  Here the square root is taken to be positive when $X=0$ and the formula \eqref{eq:Psi-det} admits real analytic continuation to $X\in\mathbb{R}$.
\label{thm:FredholmDeterminant}
\end{theorem}

\begin{remark}
There is a general method \cite{Bertola17} based on factorization of jump matrices into a product of upper- and lower-triangular nilpotent perturbations of the identity matrix, allowing one to associate a Fredholm determinant to virtually any Riemann-Hilbert problem with jump contour being a closed curve in the plane.  The method is applicable in particular to Riemann-Hilbert Problem~\ref{rhp:near-field}, in which case the resulting Fredholm determinant depends on the variables $(X,T)$ generally allowed to vary in $\mathbb{C}^2$, and its vanishing detects precisely the values of $(X,T)$ for which the solution fails to exist (we already know that this is impossible on the real subspace $(X,T)\in\mathbb{R}^2$).  There are multiple admissible factorizations of the jump matrix, each of which gives rise to a different Fredholm determinant ($\tau$-function) with the same zero divisor.  It might be expected for $\Psi(X,T;\mathbf{G})$ or its square modulus $|\Psi(X,T;\mathbf{G})|^2$ to be expressible explicitly in terms of invariant expressions formed by suitable derivatives of any of these Fredholm determinants, which would extend Theorem~\ref{thm:FredholmDeterminant} to arbitrary $T\in\mathbb{R}$.  We hope to address this question in future work.  
\label{rem:FD-T-nonzero}
\end{remark}

\subsection{Asymptotic properties of $\Psi$}
\label{s:asymptotics-intro}
Now we describe the asymptotic properties of $\Psi(X,T;\mathbf{G})$, i.e., how a general rogue wave of infinite order behaves for large values of the independent variables $(X,T)$, and how this behavior depends on the parameters in $\mathbf{G}$.  
\subsubsection{Behavior of $\Psi(X,T;\mathbf{G})$ for $X$ large}
\label{s:large-X-results}
Our first result concerns the behavior of $\Psi(X,T;\mathbf{G})$ as $X\to\pm\infty$.   Let $v_\mathrm{c}:=54^{-\frac{1}{2}}$, and define two functions of $v$ on $(-v_\mathrm{c},v_\mathrm{c})$ by
\begin{equation}
z_1(v):=\begin{cases}\displaystyle 
\frac{1}{6v}\left(-1+2\cos\left(\frac{1}{3}\arccos\left(
2\frac{v^2}{v_\mathrm{c}^2}
-1\right)\right)\right),&\quad -v_\mathrm{c}<v<0,\\\\
\displaystyle \frac{1}{6v}\left(-1+2\cos\left(\frac{1}{3}\arccos\left(
2\frac{v^2}{v_\mathrm{c}^2}
-1\right)-\frac{2}{3}\pi\right)\right),&\quad 0<v<v_\mathrm{c},
\end{cases}
\label{eq:Cardano-1}
\end{equation}
and
\begin{equation}
z_2(v):=\begin{cases}\displaystyle 
\frac{1}{6v}\left(-1+2\cos\left(\frac{1}{3}\arccos\left(
2\frac{v^2}{v_\mathrm{c}^2}-1\right)-\frac{2}{3}\pi\right)\right),&\quad -v_\mathrm{c}<v<0,\\\\
\displaystyle \frac{1}{6v}\left(-1+2\cos\left(\frac{1}{3}\arccos\left(
2\frac{v^2}{v_\mathrm{c}^2}-1\right)\right)\right),&\quad 0<v<v_\mathrm{c},
\end{cases}
\label{eq:Cardano-2}
\end{equation}
which are equivalent for $|v|<v_\mathrm{c}/\sqrt{2}$ to
\begin{equation}
\begin{split}
z_1(v)&=\frac{1}{6v}\left(-1+2\cos\left(-\frac{1}{3}\arcsin\left(v\sqrt{216\left(1-
\frac{v^2}{v_\mathrm{c}^2}\right)}\right)-\frac{1}{3}\pi\right)\right)\\
z_2(v)&=\frac{1}{6v}\left(-1+2\cos\left(-\frac{1}{3}\arcsin\left(v\sqrt{216\left(1-
\frac{v^2}{v_\mathrm{c}^2}\right)}\right)+\frac{1}{3}\pi\right)\right).
\end{split}
\label{eq:Cardano-3}
\end{equation}
In each case, the singularity at $v=0$ is removable, and with the definition $z_1(0)=-\sqrt{2}$ and $z_2(0)=\sqrt{2}$, we see that  $z_j(v)$, $j=1,2$, are both analytic functions of $v\in (-v_\mathrm{c},v_\mathrm{c})$ satisfying $z_1(v)<0<z_2(v)$.   Introducing the function 
\begin{equation}
\vartheta(z;v):=z+vz^2+2z^{-1},
\label{eq:vartheta-intro-def}
\end{equation}
one can check that $\vartheta'(z_j(v);v)=0$ holds identically for $-v_\mathrm{c}<v<v_\mathrm{c}$, and that $\vartheta''(z_1(v);v)<0$ while $\vartheta''(z_2(v);v)>0$ (here prime means differentiation with respect to $z$ for fixed $v$).

\begin{theorem}[Large-$X$ regime]
\label{t:large-X}
Let $\tau:=|b/a|$ and $p:=\frac{1}{2\pi}\ln(1+\tau^2)$, and set $v:=TX^{-\frac{3}{2}}\in\mathbb{R}$.  
Then for each $\delta>0$,
\begin{multline}
\Psi(X, T;\mathbf{G})=\frac{\ee^{-\ii \arg(ab)}}{X^{\frac{3}{4}}}
\left[ \frac{\sqrt{2p}}{\sqrt{-\vartheta''(z_1(v);v)}} \ee^{ -2 \ii  X^{1/2}\vartheta(z_1(v);v)}  \ee^{ - \ii \frac{1}{2} p \ln(X)  } 
\ee^{-\ii( \phi_{z_1}(v) + \phi_{0}(v))}  \right.\\
\left.+ \frac{\sqrt{2p}}{\sqrt{\vartheta''(z_2(v);v)}}\ee^{ -2 \ii  X^{1/2}\vartheta(z_2(v);v)}  \ee^{\ii \frac{1}{2} p \ln(X)  } \ee^{\ii( \phi_{z_2}(v) + \phi_{0}(v))}
\right] + O(X^{-\frac{5}{4}}), \quad X\to+\infty
\label{Psi-large-X-approx-5}
\end{multline}
holds uniformly for $|v|\le v_\mathrm{c}-\delta$ and $\tau=O(1)$, 
where real phases $\phi_0(v)$, $\phi_{z_1}(v)$, and $\phi_{z_2}(v)$ are 
defined by
\begin{equation}
\phi_0(v)
:=
\frac{1}{4}\pi + p \ln( 2 (z_2(v)-z_1(v))^2 ) - \arg(\Gamma(\ii p)),
\label{phi-0-X}
\end{equation}
and
\begin{equation}
\begin{aligned}
\phi_{z_1}(v)&:= p \ln(-\vartheta''(z_1(v);v ) ),\\ 
 \phi_{z_2}(v)&:= p \ln(\vartheta''(z_2(v);v ) ).
 \end{aligned}
  \label{phi-1-phi-2-X}
\end{equation}
%
 \end{theorem}
 
 \begin{corollary}[Large negative $X$]
The asymptotic behavior of $\Psi(X, T;\mathbf{G})$ with $T=|X|^{\frac{3}{2}} v$ in the limit $X\to-\infty$ is given by the same formula as in the right-hand side of  \eqref{Psi-large-X-approx-5} except that $p$ is replaced with $\bar{p}:=\frac{1}{2\pi}\ln\left(1+\tau^{-2} \right)$ throughout,  $X$ is replaced with $|X|$, and the uniformity of the error requires that $\bar{\tau}:=\tau^{-1}=O(1)$.
 \label{cor:large-negative-X}
 \end{corollary}
 \begin{proof}
 Apply Proposition~\ref{prop:X-symmetry}.  Note that $\bar{p}$ can be written explicitly in terms of $(p,\tau)$ by $\bar{p}=p-\ln(\tau)/\pi$.
 \end{proof}
  
These results generalize a theorem \cite[Theorem 4]{BilmanLM2020} of the authors with L.\@ Ling to the general family of solutions parametrized by the $2\times 2$ matrix $\mathbf{G}(a,b)$, and they also sharpen the error estimate (see Remark~\ref{Remark-Sharpen} below). We prove Theorem~\ref{t:large-X} in Section~\ref{s:large-X}. Note that the leading term of \eqref{Psi-large-X-approx-5} vanishes as $p\to 0^+$, which is expected in light of Proposition~\ref{p:degeneration} since this corresponds to $b\to 0$. Similarly, the leading term of the asymptotic formula valid as $X\to-\infty$ vanishes as $\bar{p}\to 0^+$, which corresponds to $a\to 0$.  More generally, aside from the phase factor of $\ee^{-\ii\arg(ab)}$, the dependence on $(a,b)$ enters only via the value of $p$, which appears in the approximate formula \eqref{Psi-large-X-approx-5} as an overall multiplier $\sqrt{2p}$ and in smaller corrections to the $p$-independent dominant phase terms proportional to $X^\frac{1}{2}$.   In Figures~\ref{fig:LargeX-a1-b1}--\ref{fig:LargeX-a0p5EIPiOver4-b1}, the accuracy of the approximation \eqref{Psi-large-X-approx-5} is illustrated by comparing with numerical computations of $\Psi(X,T;\mathbf{G})$ achieved using the software package described in Section~\ref{s:numerics-intro}.
 
\begin{figure}[h]
\begin{center}
\includegraphics[width=0.45\linewidth]{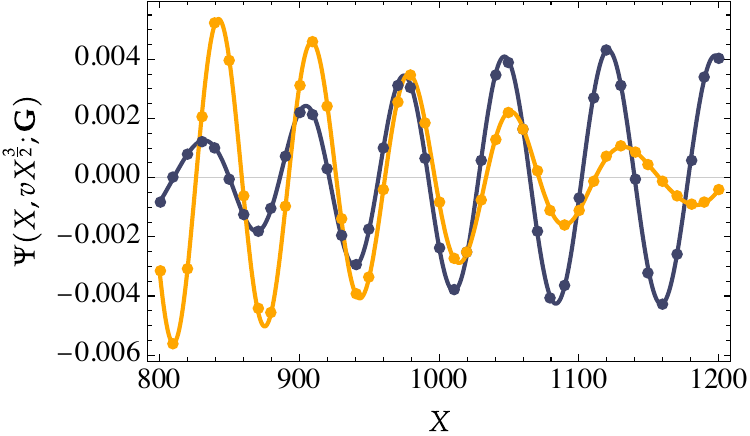}\hfill%
\includegraphics[width=0.45\linewidth]{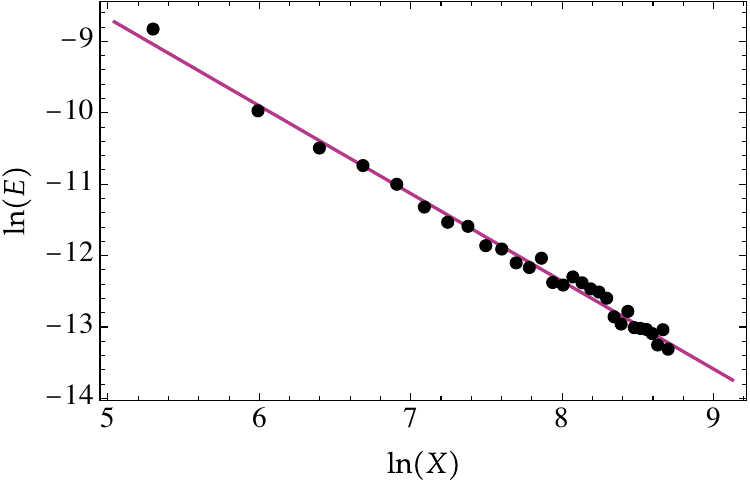}
\end{center}
\caption{Left:  The real part (navy) and imaginary part (yellow) of the explicit terms on the right-hand side of \eqref{Psi-large-X-approx-5} (curves) compared with numerical evaluation of $\Psi(X,T;\mathbf{G})$ (points) for $v=0.5v_\mathrm{c}$ and parameters $a=b=1$.  Right:  the logarithm of the absolute difference $E$ between $\Psi(X,T;\mathbf{G})$ and its explicit approximation in \eqref{Psi-large-X-approx-5} for $v=0.5v_\mathrm{c}$ and parameters $a=b=1$ plotted against $\ln(X)$.  The purple line is a least-squares best fit, and it has a slope of $-1.22883$, matching well with the predicted exponent of $-\frac{5}{4}$.}
\label{fig:LargeX-a1-b1}
\end{figure}

\begin{figure}[h]
\begin{center}
\includegraphics[width=0.45\linewidth]{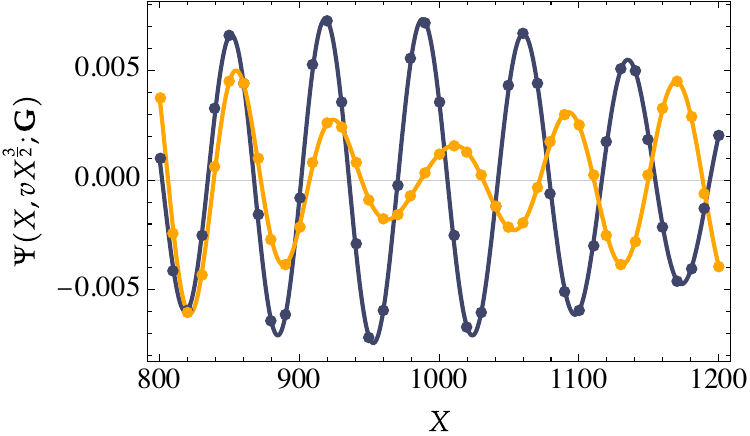}\hfill%
\includegraphics[width=0.45\linewidth]{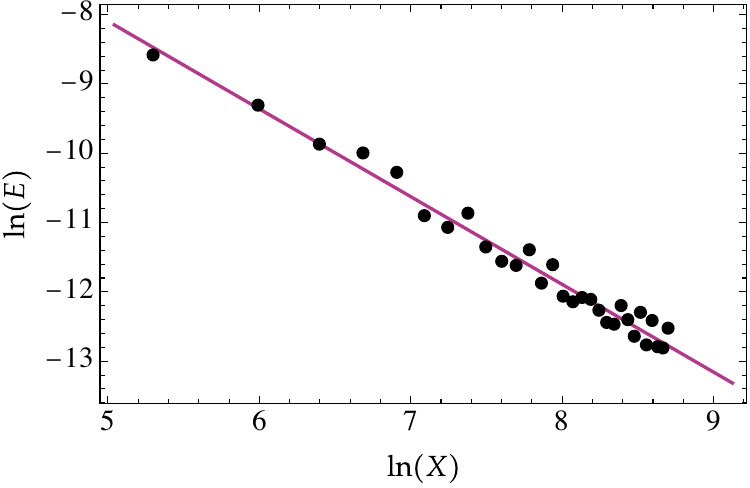}
\end{center}
\caption{Same as in Figure~\ref{fig:LargeX-a1-b1} but for $a=\frac{1}{2}\ee^{\ii\pi/4}$ and $b=1$.  The best fit line in the right-hand plot has slope $-1.26571$, again matching well with the predicted exponent of $-\frac{5}{4}$.}
\label{fig:LargeX-a0p5EIPiOver4-b1}
\end{figure}

 \begin{corollary}[Large-$X$ approximation of the squared modulus] 
  \label{coro:large-X}
 Under the hypotheses of Theorem~\ref{t:large-X}, $ |\Psi(X,T;\mathbf{G}) |^2$ with $T=X^{\frac{3}{2}}v$ has the behavior
\begin{multline}
|\Psi(X,T;\mathbf{G})|^2 = \frac{2p}{X^\frac{3}{2}\sqrt{-\vartheta''(z_1(v);v)\vartheta''(z_2(v);v)}}\left[\sqrt{-\frac{\vartheta''(z_2(v);v)}{\vartheta''(z_1(v);v)}} +\sqrt{-\frac{\vartheta''(z_1(v);v)}{\vartheta''(z_2(v);v)}}\right.\\
\left.{}\vphantom{\sqrt{-\frac{\vartheta''(z_1(v);v)}{\vartheta''(z_2(v);v)}}}+2\cos\left(
\Omega(X,v)\right)\right]
+O(X^{-2}),\quad X\to+\infty,
\label{large-X-mod-squared}
\end{multline} 
where the phase is given by
\begin{equation}
\Omega(X,v):=2\varrho(v)X^{\frac{1}{2}}+p \ln(X) +2\varsigma(v),
\label{eq:large-X-modulus-phase}
\end{equation}
where $\varrho(v):= \vartheta(z_1(v);v) - \vartheta(z_2(v);v)<0$ and $\varsigma(v):=\frac{1}{2}\phi_{z_1}(v)+\frac{1}{2}\phi_{z_2}(v)+\phi_0(v)\in\mathbb{R}$.
\end{corollary}

\begin{remark}
 In particular, Corollary~\ref{coro:large-X} gives the asymptotic behavior of $|\Psi(X,T;\mathbf{G})|^2$ in the limit $X\to +\infty$ with $T\ge 0$ held fixed, in which case the parameter $v$ tends to zero.  After taking the necessary square root, one can see that the function $\mathbb{R}\ni X\mapsto \Psi(X,T;\mathbf{G})$ is not in $L^1(\mathbb{R})$ for any $T\ge 0$.  Therefore, $\Psi(X,T;\mathbf{G})$ is not associated with any sensible scattering data in the inverse-scattering transform solution of the focusing NLS equation \eqref{nls}.  Although general rogue waves of infinite order are in $L^2(\mathbb{R})$ by Theorem~\ref{t:L2-norm} and the focusing NLS equation is globally well-posed on $L^2(\mathbb{R})$ \cite{Tsutsumi87}, the inverse-scattering method nonetheless does not apply to these solutions.  
 \label{rem:no-IST}
 \end{remark}

Aside from an overall factor of $X^{-\frac{3}{2}}$, the only dependence of $|\Psi(X,T;\mathbf{G})|^2$ on $X\gg 1$ in the leading terms for fixed $v$ appears in the phase $\Omega(X,v)$, making the leading contribution to $|\Psi(X,T;\mathbf{G})|^2$ highly oscillatory.  It is therefore reasonable to assert that $|\Psi(X,T;\mathbf{G})|^2$ is 
approximately maximized along curves in the $(X,T)$-plane where $\cos(\Omega(X,TX^{-\frac{3}{2}}))=1$.  Thus, we set $\Omega(X,v)=2\pi n$ for $n\in\mathbb{Z}$ and solve for $X$.
Since $\varrho(v)<0$ is bounded away from zero, $X>0$ large implies that also $n=-N$ is a large negative integer.  The equation $\Omega(X,v)=-2\pi N$ can then be rearranged as
\begin{equation}
\frac{1}{p}\varrho(v)X^\frac{1}{2}\exp\left(\frac{1}{p}\varrho(v)X^\frac{1}{2}\right) = -\ee^{-\eta},\quad\eta:=\frac{\pi N}{p}+\nu(v),\quad \nu(v):=\frac{\varsigma(v)}{p}+\ln\left(-\frac{p}{\varrho(v)}\right),
\end{equation}
which is solved using the Lambert $W$-function \cite[Section 4.13]{DLMF}:
\begin{equation}
X=\frac{p^2}{\varrho(v)^2}W_{\pm 1}(-\ee^{-\eta}\mp \ii 0)^2.
\label{eq:X-channels-curves}
\end{equation}
Using the asymptotic formula $W_{\pm 1}(-\ee^{-\eta}\mp\ii 0)=-\eta -\ln(\eta)-\ln(\eta)/\eta + O(\ln(\eta)^2/\eta^2)$ as $\eta\to+\infty$ (see \cite[Eqn.\@ 4.13.11]{DLMF}), we then obtain for each large positive integer $N$ a solution $X=X_N(v)$ that is accurate up to a small absolute error:
\begin{equation}
\begin{split}
X_N(v)&=\frac{p^2}{\varrho(v)^2}\left[\eta^2 + 2\eta\ln(\eta) + 2\ln(\eta) + \ln(\eta)^2\right] + O\left(\frac{\ln(\eta)^2}{\eta}\right),\quad\eta\to+\infty\\
&=\frac{\pi^2}{\varrho(v)^2}N^2 + \frac{2\pi p}{\varrho(v)^2}N\ln(N) + \frac{2\pi p}{\varrho(v)^2}\left(
\nu(v)+\ln\left(\frac{\pi}{p}\right)
\right)N\\
&\quad\quad {}+\frac{p^2}{\varrho(v)^2}\ln(N)^2+\frac{2p^2}{\varrho(v)^2}\left(1+
\nu(v)+\ln\left(\frac{\pi}{p}\right)
\right)\ln(N)\\
&\quad\quad{}+\frac{p^2}{\varrho(v)^2}\left[
2\nu(v)+\nu(v)^2
+ 2\left(1+
\nu(v)
\right)\ln\left(\frac{\pi}{p}\right) +\log\left(\frac{\pi}{p}\right)^2\right] \\
&\quad\quad{}+O(N^{-1}\ln(N)^2),\quad N\to+\infty.
\end{split}
\end{equation}
%
This expansion is valid uniformly for $|v|\le v_\mathrm{c}-\delta$ for any $\delta>0$ however small, but it fails for $v\approx \pm v_\mathrm{c}$, where $\phi_{z_1}(v)$ and hence also $\varsigma(v)$ blows up.  The curves $X=X_N(TX^{-\frac{3}{2}})$ are superimposed on a density plot of the square modulus of $\Psi(X,T;\mathbf{G})$
in Figure~\ref{f:overlay-peaks}.  This shows that these curves actually approximate the peaks of the modulus in the region $|TX^{-\frac{3}{2}}|<v_\mathrm{c}$ quite accurately even when $X$ is not very large.  

After proving Theorem~\ref{t:large-X} and Corollary~\ref{coro:large-X}, we obtain the proof of Theorem~\ref{t:L2-norm} in Section~\ref{s:large-X} essentially as a corollary.

\subsubsection{Behavior of $\Psi(X,T;\mathbf{G})$ for $T$ large}
\label{s:large-T-results}
Next, we describe the behavior of $\Psi(X,T;\mathbf{G})$ in the limit as $T\to\pm \infty$. Let $w_\mathrm{c}:=54^\frac{1}{3}>0$.  Define for $|w|<w_\mathrm{c}$ the real quantities:
\begin{equation}
\tz_1(w):=\frac{1}{12}\left(-w-\sqrt{w^2+8w_\mathrm{c}^2}\right)<0\quad\text{and}\quad
\tz_2(w):=\frac{1}{12}\left(-w+\sqrt{w^2+8w_\mathrm{c}^2}\right)>0,
\label{eq:Z1Z2-intro}
\end{equation}
and the complex quantity
\begin{equation}
\tz_0(w):=\frac{1}{3}\left(-w +\ii\sqrt{w_\mathrm{c}^2-w^2}\right)\in\mathbb{C}_+,
\label{eq:Z0-intro}
\end{equation}
where in each case the positive square root is taken.  For convenience, we write 
\begin{equation}
V(w):=\mathrm{Im}(\tz_0(w)).
\end{equation}
Related amplitudes are then defined by
\begin{equation}
\begin{split}
m_{\tz_1}^\pm(w)&:=\frac{1}{2}(1\pm\cos(\arg(\tz_1(w)-\tz_0(w))))>0,\\
m_{\tz_2}^\pm(w)&:=\frac{1}{2}(1\pm\cos(\arg(\tz_2(w)-\tz_0(w))))>0,
\end{split}
\label{eq:m-amplitudes-intro}
\end{equation}
which of course satisfy $m_{\tz_j}^+(w)+m_{\tz_j}^-(w)=1$ for $j=1,2$.  Now define some real phases by:
\begin{equation}
\kappa(w):=-\frac{1}{3}(w^2+w_\mathrm{c}^2),
\label{eq:kappa-intro}
\end{equation}
\begin{multline}
\Phi(w):= 
2\bar{p}\ln\left(\frac{\tz_1(w)-\mathrm{Re}(\tz_0(w))+|\tz_1(w)-\tz_0(w)|}{V(w)}\right)\\
{}-2p\ln\left(\frac{\tz_2(w)-\mathrm{Re}(\tz_0(w)) + |\tz_2(w)-\tz_0(w)|}{V(w)}\right) -\frac{\pi}{2},
\label{eq:Phi-intro}
\end{multline}
\begin{multline}
\Phi^0_{\tz_1}(w):=2p\ln\left(\frac{V(w)^2-2(V(w)-|\tz_1(w)-\tz_0(w)|)(V(w)-|\tz_2(w)-\tz_0(w)|)}{V(w)^2-2(V(w)+|\tz_1(w)-\tz_0(w)|)(V(w)-|\tz_2(w)-\tz_0(w)|)}\right)\\
+2p\ln\left(\frac{\mathrm{Re}(\tz_0(w))-\tz_1(w)}{|\tz_1(w)-\tz_0(w)|-V(w)}\right) 
+ 2\bar{p}\ln\left(\frac{-\tz_1(w)V(w)}{4|\tz_1(w)-\tz_0(w)|^\frac{5}{2}(\tz_2(w)-\tz_1(w))^\frac{1}{2}}\right)\\+\frac{\pi}{4}+\arg(\Gamma(\ii\bar{p})),
\label{eq:Phi0-Z1-intro}
\end{multline}
\begin{multline}
\Phi^0_{\tz_2}(w):=2\bar{p}\ln\left(\frac{V(w)^2-2(V(w)-|\tz_2(w)-\tz_0(w)|)(V(w)-|\tz_1(w)-\tz_0(w)|)}{V(w)^2-2(V(w)+|\tz_2(w)-\tz_0(w)|)(V(w)-|\tz_1(w)-\tz_0(w)|)}\right)\\
+2\bar{p}\ln\left(\frac{\tz_2(w)-\mathrm{Re}(\tz_0(w))}{|\tz_2(w)-\tz_0(w)|-V(w)}\right)
+2p\ln\left(\frac{\tz_2(w)V(w)}{4|\tz_2(w)-\tz_0(w)|^\frac{5}{2}(\tz_2(w)-\tz_1(w))^\frac{1}{2}}\right)\\ +\frac{\pi}{4}+\arg(\Gamma(\ii p)),
\label{eq:Phi0-Z2-intro}
\end{multline}
\begin{equation}
\Phi_{\tz_1}(T,w):=2\frac{|\tz_1(w)-\tz_0(w)|^3}{-\tz_1(w)}T^\frac{1}{3} -\frac{1}{3}\bar{p}\ln(T) + \Phi^0_{\tz_1}(w),
\label{eq:Phi-Z1-intro}
\end{equation}
and
\begin{equation}
\Phi_{\tz_2}(T,w):=2\frac{|\tz_2(w)-\tz_0(w)|^3}{\tz_2(w)}T^\frac{1}{3} -\frac{1}{3}p\ln(T) + \Phi_{\tz_2}^0(w).
\label{eq:Phi-Z2-intro}
\end{equation}

\begin{theorem}[Large-$T$ regime]
Let $\tau:=|b/a|$ and $p:=\frac{1}{2\pi}\ln(1+\tau^2)$, $\bar{p}:=\frac{1}{2\pi}\ln(1+\tau^{-2})$, and set $w:=XT^{-\frac{2}{3}}\in\mathbb{R}$.  Then, for each $\delta>0$,
\begin{multline}
\Psi(X,T;\mathbf{G})=\ee^{-\ii\arg(ab)}\ee^{-\ii T^{1/3}\kappa(w)}\ee^{\ii\Phi(w)}\left[\frac{1}{3}\sqrt{w_\mathrm{c}^2-w^2}T^{-\frac{1}{3}}\right.\\
-\frac{1}{\sqrt{\tz_2(w)-\tz_1(w)}}\left\{
\frac{\sqrt{\bar{p}}|\tz_1(w)|\left(m_{\tz_1}^+(w)\ee^{\ii\Phi_{\tz_1}(T,w)}+m_{\tz_1}^-(w)\ee^{-\ii\Phi_{\tz_1}(T,w)}\right)}{|\tz_1(w)-\tz_0(w)|^\frac{1}{2}}\right.\\
\left.\left.+\frac{\sqrt{p}\tz_2(w)\left(m_{\tz_2}^-(w)\ee^{\ii\Phi_{\tz_2}(T,w)}+m_{\tz_2}^+(w)\ee^{-\ii\Phi_{\tz_2}(T,w)}\right)}{|\tz_2(w)-\tz_0(w)|^\frac{1}{2}}\right\}T^{-\frac{1}{2}} + O(T^{-\frac{2}{3}})\right],\quad T\to+\infty
\label{eq:Psi-Large-T-intro}
\end{multline}
holds uniformly for $|w|\le w_\mathrm{c}-\delta$, $\tau=O(1)$, and $\tau^{-1}=O(1)$.  
\label{t:large-T}
\end{theorem}
\begin{corollary}[Large negative $T$]
The asymptotic behavior of $\Psi(X,T;\mathbf{G})$ with $X=w|T|^{\frac{2}{3}}$ in the limit $T\to-\infty$ is given by the right-hand side of \eqref{eq:Psi-Large-T-intro}, except that $T$ is replaced with $|T|$ and the signs of the real phases $\Phi(w)$, $\Phi_{\tz_1}(|T|,w)$, $\Phi_{\tz_2}(|T|,w)$, and $\kappa(w)$ (but not $-\arg(ab)$) are changed.
\label{cor:large-T-negative}
\end{corollary}
\begin{proof}
Apply Proposition~\ref{prop:T-symmetry}.
\end{proof}
These results generalize the long-time asymptotic theorem \cite[Theorem 5]{BilmanLM2020} of the authors with L.\@ Ling to the general family of solutions parametrized by the $2\times 2$ matrix $\mathbf{G}(a,b)$. They also provide an explicit correction term proportional to $T^{-\frac{1}{2}}$ not previously obtained for any parameters.  We  prove Theorem~\ref{t:large-T} in Section~\ref{s:large-T}.
One can see from \eqref{eq:Psi-Large-T-intro} that aside from the factor of $\ee^{-\ii\arg(ab)}$, the asymptotic formula depends on $(a,b)$ only via the values of $p$ and $\bar{p}$, which enter via $\sqrt{p}$ and $\sqrt{p}$ as multipliers of the two correction terms, and which also appear as multipliers in the phase $\Phi(w)$ and in subdominant corrections to the leading terms proportional to $T^\frac{1}{3}$ in the phases $\Phi_{\tz_1}(T,w)$ and $\Phi_{\tz_2}(T,w)$.   The accuracy of the large-$T$ approximation afforded by Theorem~\ref{t:large-T} is illustrated in Figures~\ref{fig:LargeT-a1-b1}--\ref{fig:LargeT-a0p5EIPiOver4-b1}.

\begin{figure}[h]
\begin{center}
\includegraphics[width=0.45\linewidth]{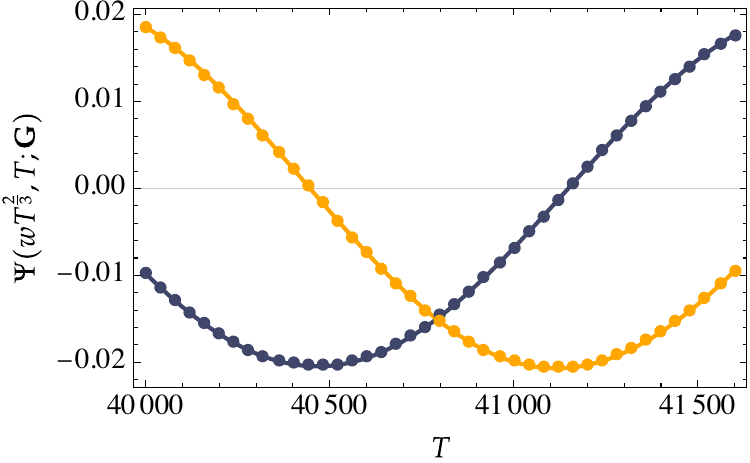}\hfill%
\includegraphics[width=0.45\linewidth]{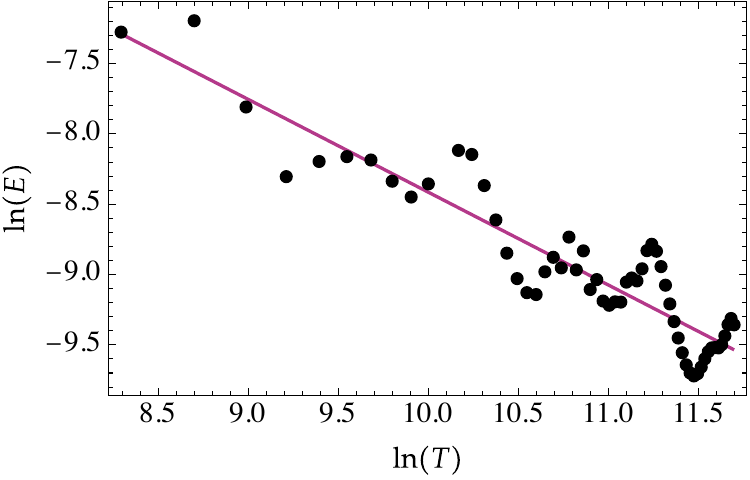}
\end{center}
\caption{Left:  The real part (navy) and imaginary part (yellow) of the explicit terms on the right-hand side of \eqref{eq:Psi-Large-T-intro} (curves) compared with numerical evaluation of $\Psi(X,T;\mathbf{G})$ (points) for $w=0.85w_\mathrm{c}$ and parameters $a=b=1$.  Right:  the logarithm of the absolute difference $E$ between $\Psi(X,T;\mathbf{G})$ and its explicit approximation in \eqref{eq:Psi-Large-T-intro} for $w=0.85w_\mathrm{c}$ and parameters $a=b=1$ plotted against $\ln(T)$.  The purple line is a least-squares best fit, and it has a slope of $-0.66042$, matching well with the predicted exponent of $-\frac{2}{3}$.}
\label{fig:LargeT-a1-b1}
\end{figure}

\begin{figure}[h]
\begin{center}
\includegraphics[width=0.45\linewidth]{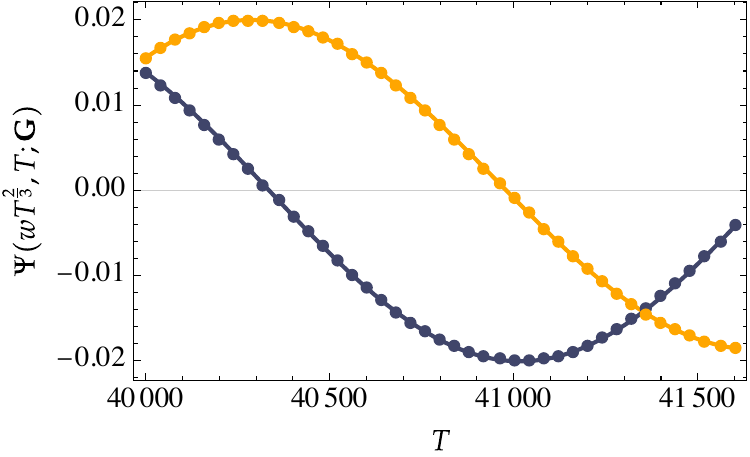}\hfill%
\includegraphics[width=0.45\linewidth]{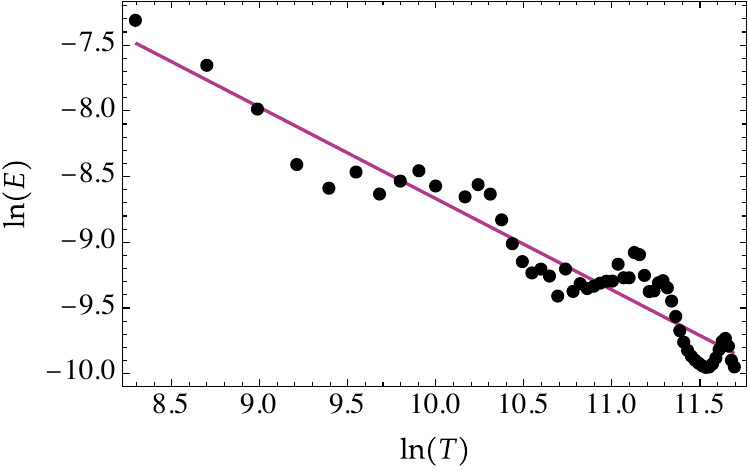}
\end{center}
\caption{Same as in Figure~\ref{fig:LargeT-a1-b1} but for $a=\frac{1}{2}\ee^{\ii\pi/4}$ and $b=1$.  The best fit line in the right-hand plot has slope $-0.69432$, again matching well with the predicted exponent of $-\frac{2}{3}$.}
\label{fig:LargeT-a0p5EIPiOver4-b1}
\end{figure}

 \begin{corollary}[Large-$T$ approximation of the squared modulus] 
  \label{coro:large-T}
 Under the hypotheses of Theorem~\ref{t:large-T}, $ |\Psi(X,T;\mathbf{G})|^2$ with $X=T^\frac{2}{3}w$ has the behavior
\begin{multline}
|\Psi(X,T;\mathbf{G})|^2 = \frac{w_\mathrm{c}^2-w^2}{9}T^{-\frac{2}{3}}\\
-\frac{2}{3}\sqrt{\frac{w_\mathrm{c}^2-w^2}{\tz_2(w)-\tz_1(w)}}\left\{\frac{\sqrt{\bar{p}}|\tz_1(w)|\cos(\Phi_{\tz_1}(T,w))}{|\tz_1(w)-\tz_0(w)|^\frac{1}{2}} +
\frac{\sqrt{p}\tz_2(w)\cos(\Phi_{\tz_2}(T,w))}{|\tz_2(w)-\tz_0(w)|^\frac{1}{2}}\right\}T^{-\frac{5}{6}}\\+O(T^{-1}),\quad T\to+\infty
\label{large-T-mod-squared}
\end{multline}
and for the limit $T\to-\infty$ we simply replace $T$ with $|T|$ on the right-hand side.
 \end{corollary}
 
Comparing with the interpretation of Corollary~\ref{coro:large-X} that the modulus of the solution is maximized along certain curves $X=X_N(TX^{-\frac{3}{2}})$ in the relevant part of the $(X,T)$-plane, here we see that instead the fluctuation proportional to $T^{-\frac{5}{6}}$ is a combination of two sinusoids with different phases, so that one term is approximately maximized along curves satisfying $\Phi_{\tz_1}(T,w)=(2M+1)\pi$ while the other is approximately maximized along curves satisfying $\Phi_{\tz_2}(T,w)=(2N+1)\pi$, for independent integers $(M,N)\in\mathbb{Z}^2$.  These two families of curves are plotted over the region $|XT^{-\frac{2}{3}}|<w_\mathrm{c}$ for different values of the parameters $(a,b)$ in Figures~\ref{f:overlay-peaks} and \ref{f:overlay-peaks-alt}, and indeed one can see that even in the part of this region where $T$ is not very large, still the peaks of the fluctuation appear to be localized near the intersections of one curve from each family, so that both terms of the fluctuation are simultaneously maximized.  

 \begin{figure}
\includegraphics[width=0.75\linewidth]{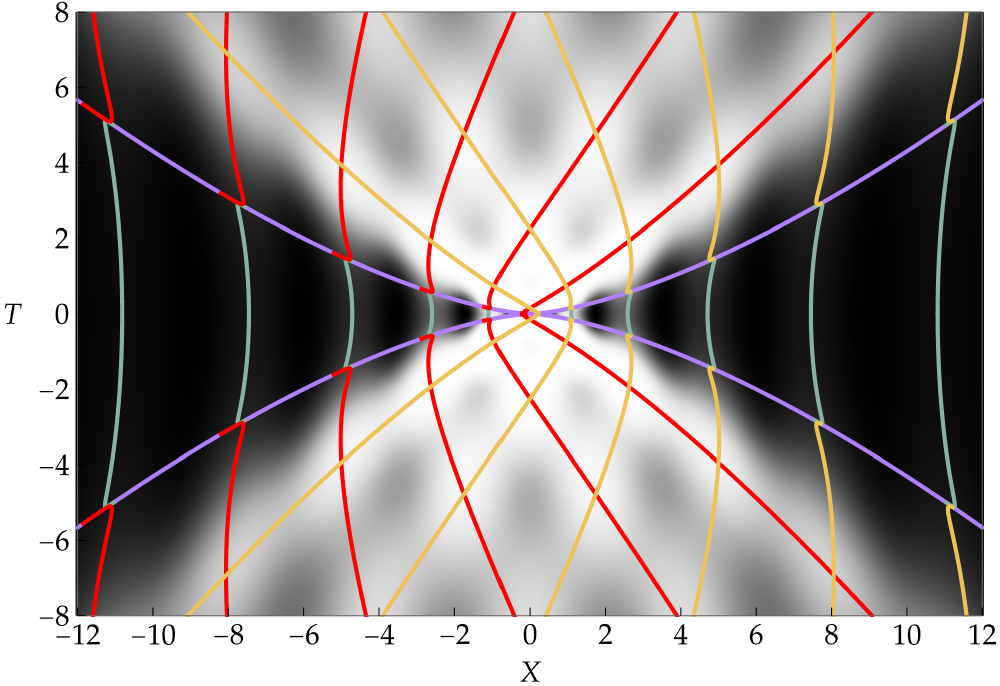}
 \caption{Density plot of $|\Psi(X,T)|^2$ with $a=b=1$ and the boundary curve (purple) $X/T^{\frac{2}{3}}=54^{\frac{1}{3}}$, overlayed with the level curves (green) on which the cosine in \eqref{large-X-mod-squared} in Corollary~\ref{coro:large-X} is maximal 
 and the level curves (red and mustard) on which $\cos(\Phi_{\tz_1}(T,w))$ and $\cos(\Phi_{\tz_2}(T,w))$ in \eqref{large-T-mod-squared} Corollary~\ref{coro:large-T} are minimal (respectively).
}
\label{f:overlay-peaks}
\end{figure}

\begin{figure}
\includegraphics[width=0.75\linewidth]{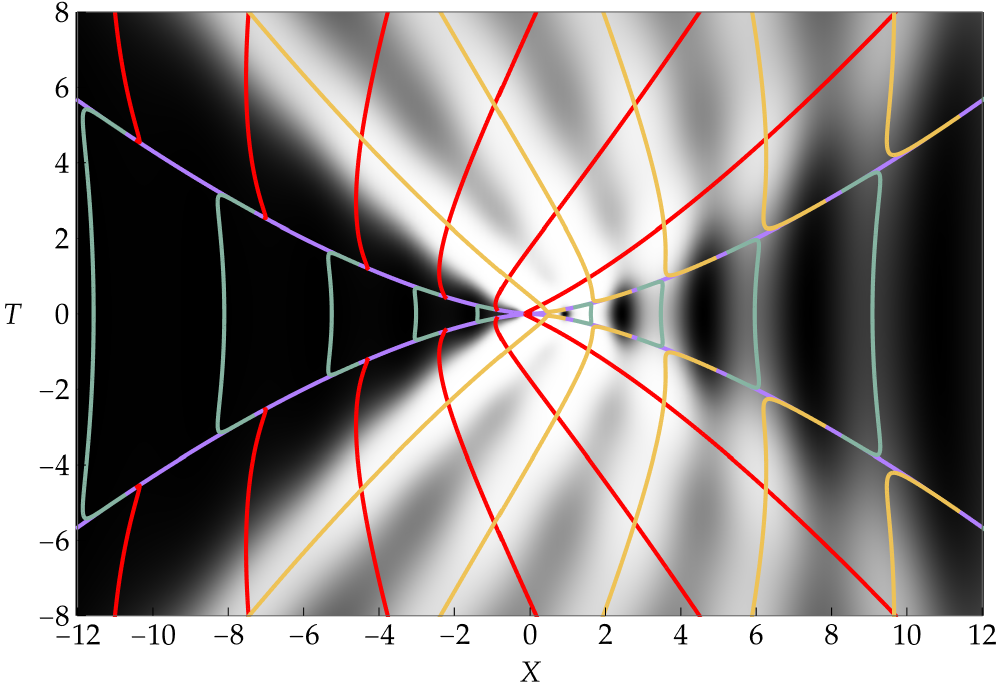}
\caption{As in Figure~\ref{f:overlay-peaks}, but for parameters $a=\frac{1}{4}$ and $b=1$.}
\label{f:overlay-peaks-alt}
\end{figure}

Like the curves $X=X_N(TX^{-\frac{3}{2}})$, the curves in the family $\Phi_{\tz_2}(T,w)=(2N+1)\pi$ shown with mustard color in Figures~\ref{f:overlay-peaks} and \ref{f:overlay-peaks-alt} appear to approach the common boundary in the first quadrant ($X>0$ and $T>0$) of the regions of validity of Corollaries~\ref{coro:large-X} and \ref{coro:large-T} (shown in purple).  It seems from the plot as though these two families of curves could be compared near the boundary curve.  The equation $\Phi_{\tz_2}(T,w)=(2N+1)\pi$ can be rewritten in the form
\begin{multline}
-\frac{2|\tz_2(w)-\tz_0(w)|^3}{p\tz_2(w)}T^\frac{1}{3}\exp\left(-\frac{2|\tz_2(w)-\tz_0(w)|^3}{p\tz_2(w)}T^\frac{1}{3}\right)=-\ee^{-\kappa},\quad\kappa:=\frac{2\pi N}{p} + \varpi(w),\\
\varpi(w):=\frac{\pi}{p}-\frac{1}{p}\Phi^0_{\tz_2}(w)-\ln\left(\frac{2|\tz_2(w)-\tz_0(w)|^3}{p\tz_2(w)}\right),
\end{multline}
which is solved in terms of the Lambert $W$-function similarly to \eqref{eq:X-channels-curves}:
\begin{equation}
T=-\frac{p^3\tz_2(w)^3}{8|\tz_2(w)-\tz_0(w)|^9}W_\pm(-\ee^{-\kappa}\mp\ii 0)^3.
\label{eq:T-shelves-curves}
\end{equation}
Since a cube must be calculated in the limit $\kappa\to +\infty$, we need the more accurate asymptotic formula
\begin{multline}
W_{\pm 1}(-\ee^{-\kappa}\mp\ii 0)=-\kappa-\ln(\kappa)-\kappa^{-1}\ln(\kappa)+\frac{1}{2}\kappa^{-2}\ln(\kappa)^2 -\kappa^{-2}\ln(\kappa) \\+ O(\kappa^{-3}\ln(\kappa)^3),\quad\kappa\to+\infty,
\end{multline}
and therefore for each large positive integer $N$ we obtain a solution $T=T_N(w)$ with the approximation
\begin{equation}
\begin{split}
T_N(w)&=\frac{p^3\tz_2(w)^3}{8|\tz_2(w)-\tz_0(w)|^9}\left[\kappa^3+3\kappa^2\ln(\kappa)+3\kappa\ln(\kappa)^2+3\kappa\ln(\kappa)+\ln(\kappa)^3+\frac{9}{2}\ln(\kappa)^2 + 3\ln(\kappa)\right]\\
&\quad\quad{}+O\left(\frac{\ln(\kappa)^3}{\kappa}\right),\quad\kappa\to+\infty\\
&=\frac{p^3\tz_2(w)^3}{8|\tz_2(w)-\tz_0(w)|^9}\left[\left(\frac{2\pi}{p}\right)^3N^3+3\left(\frac{2\pi}{p}\right)^2N^2\ln(N)+3\left(\frac{2\pi}{p}\right)^2\left(\varpi(w)+\ln\left(\frac{2\pi}{p}\right)\right)N^2\right.\\
&\quad\quad{}+\frac{6\pi}{p}N\ln(N)^2+\frac{6\pi}{p}\left(1+2\varpi(w)+2\ln\left(\frac{2\pi}{p}\right)\right)N\ln(N)\\
&\quad\quad{}+\frac{6\pi}{p}\left(\varpi(w)^2+\varpi(w)+\left(2\varpi(w)+1\right)\ln\left(\frac{2\pi}{p}\right)+\ln\left(\frac{2\pi}{p}\right)^2\right)N + \ln(N)^3\\
&\quad\quad{}+\left(3\varpi(w)+\frac{9}{2}+3\ln\left(\frac{2\pi}{p}\right)\right)\ln(N)^2\\
&\quad\quad{}+\left(3\varpi(w)^2+9\varpi(w) + 3+ 3(2\varpi(w)+3)\ln\left(\frac{2\pi}{p}\right)+3\ln\left(\frac{2\pi}{p}\right)^2 \right)\ln(N)\\
&\quad\quad{}+\varpi(w)^3+\frac{9}{2}\varpi(w)^2+3\varpi(w) +3\left(\varpi(w)^2+3\varpi(w)+1\right)\ln\left(\frac{2\pi}{p}\right) \\
&\quad\quad\quad{}\left.+ 3\left(\varpi(w)+\frac{3}{2}\right)\ln\left(\frac{2\pi}{p}\right)^2 + \ln\left(\frac{2\pi}{p}\right)^3\right]
+O\left(\frac{\ln(N)^3}{N}\right),\quad N\to+\infty.
\end{split}
\end{equation}
One can check that the first two terms of $54T_N(w)^2$, proportional to $N^6$ and $N^5\ln(N)$ respectively, have finite values in the limit $w\uparrow w_\mathrm{c}$.  Exactly the same is true of the two leading terms of $X_N(v)^3$ in the limit $v\uparrow v_\mathrm{c}$.  However, only the leading term of each expansion matches, and a discrepancy appears at the order $N^5\ln(N)$.  Moreover, it is easy to see that the coefficient of $N^5\ln(N)$ cannot be changed in one of the expansions by adding any fixed integer to $N$ (amounting to re-indexing the curves of the relevant family).  Subsequent terms in each expansion actually blow up as $w\uparrow w_\mathrm{c}$ and $v\uparrow v_\mathrm{c}$ respectively.  This computation shows that it is just an illusion that the families of curves appear to match along the boundary curve $54T^2=X^3$ in the first quadrant in Figures~\ref{f:overlay-peaks} and \ref{f:overlay-peaks-alt}.  Moreover, zooming in near the boundary curve shows that the curves actually turn sharply as they approach the boundary, becoming tangent to it as shown in Figure~\ref{fig:overlay-peaks-zoom}.

\begin{figure}
\includegraphics[width=0.49\linewidth]{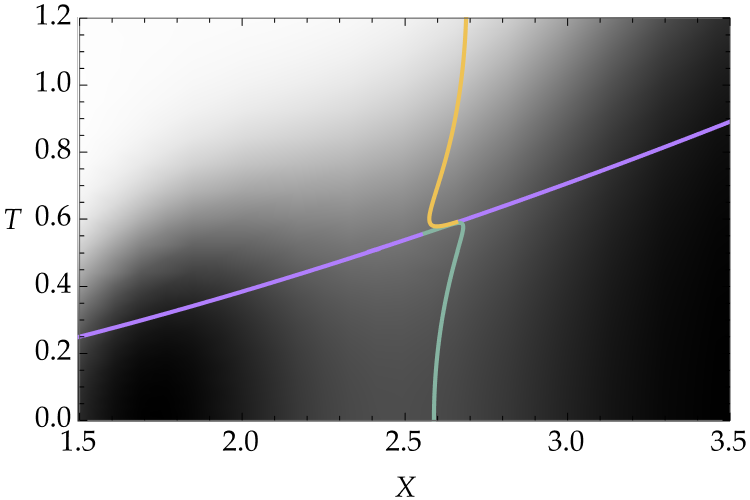}\hfill%
\includegraphics[width=0.49\linewidth]{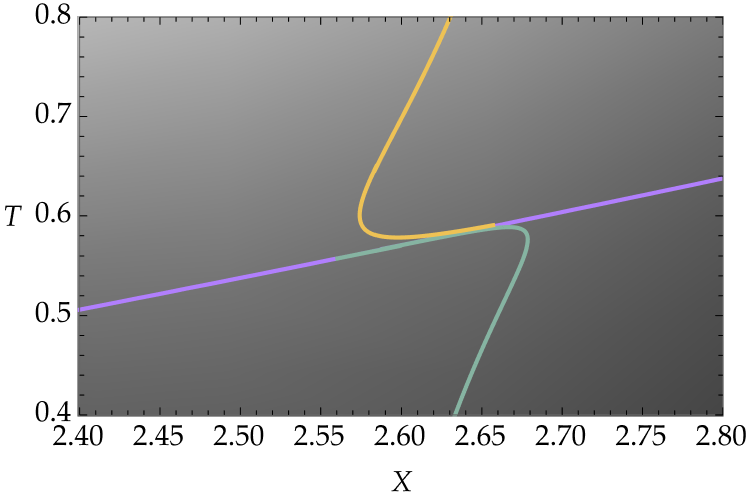}
\caption{Zoom-in plots on two different scales near the boundary curve corresponding to the parameters $a=b=1$ as in Figure~\ref{f:overlay-peaks}, showing the mismatches of the amplitude-maximizing curves near the boundary.}
\label{fig:overlay-peaks-zoom}
\end{figure}


\subsubsection{Transitional behavior of $\Psi(X,T;\mathbf{G})$}
The domains of validity of Theorems~\ref{t:large-X} and \ref{t:large-T} and their corollaries cover all asymptotic directions in the $(X,T)$-plane, except for those near the common boundary curves $|T||X|^{-\frac{3}{2}}=v_\mathrm{c}$ which is equivalent to $|X||T|^\frac{2}{3}=w_\mathrm{c}$.  
Our final asymptotic results concern the behavior of $\Psi(X,T,\mathbf{G})$ in the transitional region for large $(X,T)$ near these curves.  To formulate them, we need to first recall the second Painlev\'e equation with parameter $\alpha\in\mathbb{C}$:
\begin{equation}
\mathcal{U}''(x)=x\mathcal{U}(x)+2\mathcal{U}(x)^3-\alpha,
\label{eq:PII}
\end{equation}
every solution of which is a meromorphic function of $x\in\mathbb{C}$ all of whose poles are simple with residue $\pm 1$.
There is a unique solution $\mathcal{U}(x)$ of \eqref{eq:PII} with asymptotic behavior $\mathcal{U}(x)=-\ii (\frac{1}{2}x)^\frac{1}{2}-\frac{1}{2}\alpha x^{-1}+O(|x|^{-\frac{5}{2}})$ as $x\to\infty$ with $|\arg(x)|<\frac{2}{3}\pi$.  This is one of the so-called \emph{increasing tritronqu\'ee} solutions of \eqref{eq:PII}.  We require this solution in the situation that $\alpha=\frac{1}{2}+\ii p$, where $p=\frac{1}{2\pi}\ln(1+\tau^2)$ with $\tau=|b/a|$.  Given the increasing tritronqu\'ee solution $\mathcal{U}(x)$ for such $\alpha$ determined by $\tau>0$, a related meromorphic function $\mathcal{V}(y;\tau)$ is uniquely defined by the conditions
\begin{equation}
\frac{\mathcal{V}'(y;\tau)}{\mathcal{V}(y;\tau)}=-(\tfrac{2}{3})^\frac{1}{3}\mathcal{U}(-(\tfrac{2}{3})^\frac{1}{3}y),\quad \mathcal{V}(y;\tau)=-\left(\frac{y}{6}\right)^\alpha(1+O(y^{-\frac{3}{4}})),\quad y\to+\infty.
\label{V-P2-def}
\end{equation}
Then it is shown in \cite[Theorem 1.4]{Miller2018} that $\mathcal{V}(y;\tau)$ is analytic and non-vanishing for all real $y$, and it has the complementary asymptotic behavior
\begin{equation}
\mathcal{V}(y;\tau)=\frac{\tau p\Gamma(\ii p)}{2\sqrt{\pi}}\ee^{-3\pi\ii/4}\ee^{-\pi p/2}2^{-\ii p}\ee^{-2\ii (-y/3)^{3/2}}(-3y)^{-\frac{1}{2}(\frac{1}{2}+\ii p)}(1+O(|y|^{-\frac{5}{4}})),\quad y\to -\infty. 
\end{equation}
See \cite[Corollary 1.5]{Miller2018}, which is valid for all $\tau>0$ and $p=\frac{1}{2\pi}\ln(1+\tau^2)$.
\begin{theorem}[Transition regime]
Let $\tau:=|b/a|$ and $p:=\frac{1}{2\pi}\ln(1+\tau^2)$, and set $v:=TX^{-\frac{3}{2}}$ and $v_\mathrm{c}:= 54^{-\frac{1}{2}}$.  In the limit $X\to+\infty$, 
\begin{equation}
\Psi(X,T;\mathbf{G})=2\cdot 3^{\frac{2}{3}}X^{-\frac{2}{3}}\mathcal{V}(2^{\frac{5}{2}}3^{\frac{7}{6}}X^{\frac{1}{3}}(v-v_\mathrm{c});\tau)\ee^{\ii\Omega_\mathrm{c}(X,v)} + 2^\frac{1}{4}3^{-\frac{1}{4}}p^\frac{1}{2}X^{-\frac{3}{4}}
\ee^{\ii\Omega_2(X,v)} +O(X^{-\frac{5}{6}})
\label{eq:Psi-transitional-final}
\end{equation}
holds uniformly for $v-v_\mathrm{c}=O(X^{-\frac{1}{3}})$ and $\tau=O(1)$, where phases $\Omega_\mathrm{c}(X,v)$ and $\Omega_2(X,v)$ are defined by
\begin{equation}
\Omega_\mathrm{c}(X,v):=24^{\frac{1}{2}}X^{\frac{1}{2}}- 12 X^{\frac{1}{2}}(v-v_\mathrm{c})-\frac{1}{3}p\ln(X) +\frac{\pi}{2}-\arg(ab) + p\ln(2)-\frac{5}{3}p\ln(3),
\label{eq:trans-phase-c}
\end{equation}
and
\begin{multline}
\Omega_2(X,v):=-\left(\frac{75}{2}\right)^\frac{1}{2}X^{\frac{1}{2}}-3X^\frac{1}{2}(v-v_\mathrm{c}) +\frac{1}{2}p\ln(X)\\
{}+\frac{\pi}{4}-\arg(\Gamma(\ii p))-\arg(ab)+\frac{1}{2}p\ln(2)+\frac{7}{2}p\ln(3).
\label{eq:trans-phase-2}
\end{multline}
\label{t:transition}
\end{theorem}
\begin{corollary}[$\Psi(X,T)$ near reflected transitional curves]
The following results hold for large coordinates near the reflections of the curve $X^3=54T^2$, $X>0$, $T>0$, in the coordinate axes:
\begin{itemize}
\item
In the limit $X\to-\infty$ with $T|X|^{-\frac{3}{2}}-v_\mathrm{c}=O(|X|^{-\frac{1}{3}})$, $\Psi(X,T;\mathbf{G})$ is given by an analogue of the right-hand side of \eqref{eq:Psi-transitional-final} in which $X$ is replaced by $-X$, $\tau$ is replaced by $\bar{\tau}=\tau^{-1}$, $p$ is replaced by $\bar{p}=\frac{1}{2\pi}\ln(1+\bar{\tau}^2)$, and the error term is uniform for $\bar{\tau}=O(1)$.
\item
In the limit $X\to+\infty$ with $-TX^{-\frac{3}{2}}-v_\mathrm{c}=O(X^{-\frac{1}{3}})$, $\Psi(X,T;\mathbf{G})$ is given by the an analogue of the right-hand side of \eqref{eq:Psi-transitional-final} in which $T$ is replaced by $-T$, $\mathcal{V}(y;\tau)$ is replaced by $\mathcal{V}(y;\tau)^*$, and all signs except that of $-\arg(ab)$ in $\Omega_\mathrm{c}(X,v)$ and $\Omega_2(X,v)$ are changed.
\item 
In the limit $X\to -\infty$ with $-T|X|^{-\frac{3}{2}}-v_\mathrm{c}=O(|X|^{-\frac{1}{3}})$, $\Psi(X,T;\mathbf{G})$ is given by an analogue of the right-hand side of \eqref{eq:Psi-transitional-final} in which $(X,T)$ are replaced by $(-X,-T)$, $\tau$ is replaced by $\bar{\tau}$, $p$ is replaced by $\bar{p}$, $\mathcal{V}(y;\tau)$ is replaced by $\mathcal{V}(y;\bar{\tau})^*$, and all signs except that of $-\arg(ab)$ in $\Omega_\mathrm{c}(-X,v)$ and $\Omega_2(-X,v)$ are changed.  The error term is uniform for $\bar{\tau}=O(1)$.
\end{itemize}
\end{corollary}
\begin{proof}
Apply Propositions~\ref{prop:X-symmetry} and \ref{prop:T-symmetry}.
\end{proof}

Theorem~\ref{t:transition} is proved in Section~\ref{s:transitional}.  The leading term in \eqref{eq:Psi-transitional-final} was obtained\footnote{The error estimate for the leading term was incorrectly reported in \cite[Theorem 6]{BilmanLM2020} as $O(X^{-\frac{5}{6}})$, but according to \eqref{eq:Psi-transitional-final}, this should be corrected to $O(X^{-\frac{3}{4}})$ in equations (251), (253), (254), and (256) of that paper.} in the special case of $a=b$ in \cite[Theorem 6]{BilmanLM2020}.  In the general case, we see that the limiting Painlev\'e-II function $\mathcal{V}(\diamond;\tau)$ corresponds to a variable parameter $\alpha=\frac{1}{2}+\ii p$ depending on  $(a,b)$ via $\tau=|b/a|$ and $p=\frac{1}{2\pi}\ln(1+\tau^2)$.   

The relative size of the correction term in \eqref{eq:Psi-transitional-final} compared to the leading term  is $O(X^{-\frac{1}{12}})$, which decays very slowly as $X\to\infty$.  This observation motivates keeping the correction term although it is asymptotically negligible compared to the leading term.  We observe that the correction term is essentially the same as the contribution to $\Psi(X,T;\mathbf{G})$ of the explicit term on the second line of \eqref{Psi-large-X-approx-5}, approximated in the limit $v\uparrow v_\mathrm{c}$. Theorem~\ref{t:transition} shows that as $v\uparrow v_\mathrm{c}$, the contribution from the critical point $z_2(v)$ persists at the same order while that from $z_1(v)$ becomes larger by a factor proportional to $X^\frac{1}{12}$ and takes on a universal form expressed in terms of the Painlev\'e-II special function $\mathcal{V}(\diamond;\tau)$. 

To illustrate the validity of Theorem~\ref{t:transition}, we took $a=b=1$ and applied the numerical method described in Appendix~\ref{a:P2-numerics} 
to solve Riemann-Hilbert Problem 2.1 from \cite{Miller2018} and hence obtain $\mathcal{V}(y;\tau=1)$ for a dense grid of $y$-values in the interval $[-0.2,0.2]$.  Given such a value of $y$, and a large value of $X>0$, a corresponding value of $T$ is defined by $T=X^\frac{3}{2}v=X^\frac{3}{2}(v_\mathrm{c}+2^{-\frac{5}{2}}3^{-\frac{7}{6}}X^{-\frac{1}{3}}y)$.  Then we numerically computed $\Psi(X,T;\mathbf{G}(1,1))$ using the method described in Section~\ref{s:Numerics}, and used the result to compare with the prediction of Theorem~\ref{t:transition} as shown in Figure~\ref{fig:TransitionalComparison} and also to calculate the pointwise renormalized error
\begin{multline}
E(y):=\left| 2^{-1}3^{-\frac{2}{3}}X^\frac{2}{3}\Psi(X,T;\mathbf{G}(1,1))-\mathcal{V}(y;1)\ee^{\ii\Omega_\mathrm{c}(X,v)}-2^{-\frac{3}{4}}3^{-\frac{11}{12}}X^{-\frac{1}{12}}\ee^{\ii\Omega_2(X,v)}\right|,\\
v=v_\mathrm{c}+2^{-\frac{5}{2}}3^{-\frac{7}{6}}X^{-\frac{1}{3}}y,\quad T=X^\frac{3}{2}v=X^\frac{3}{2}(v_\mathrm{c}+2^{-\frac{5}{2}}3^{-\frac{7}{6}}X^{-\frac{1}{3}}y),
\label{E-transition}
\end{multline}
for $y\in [-1,1]$.  
We next set $E:=\max_{|y|\le 1}E(y)$ and plotted $\ln(E)$ against $\ln(X)$.  See Figure~\ref{fig:TransitionalLogLogError}.
\begin{figure}[h]
\begin{center}
\includegraphics[width=0.3\linewidth]{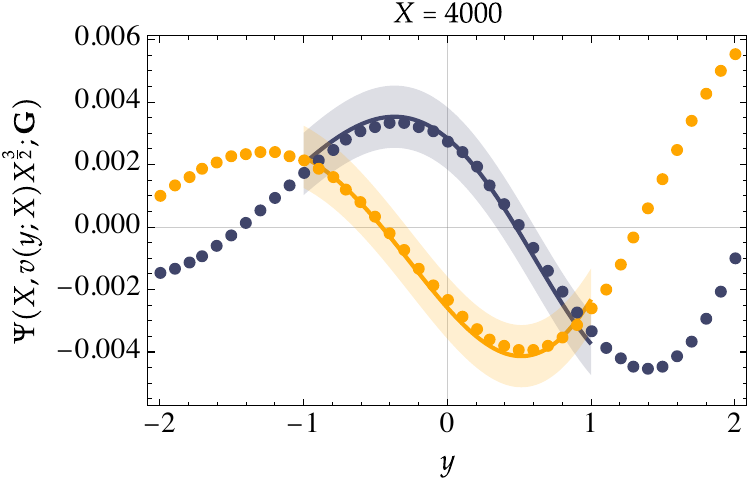}\hfill%
\includegraphics[width=0.3\linewidth]{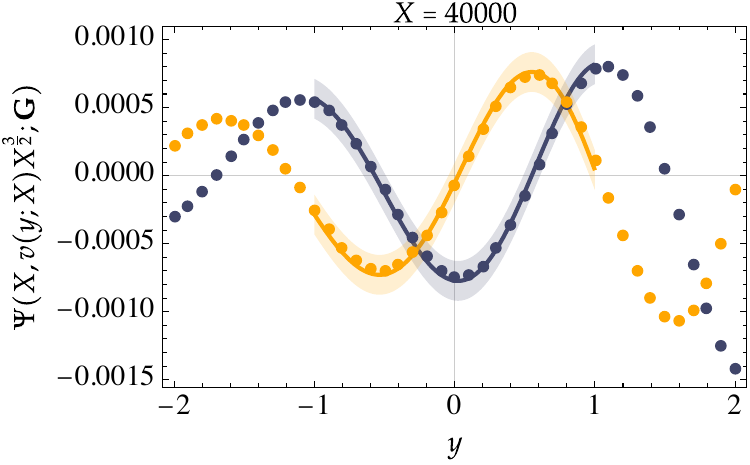}\hfill%
\includegraphics[width=0.3\linewidth]{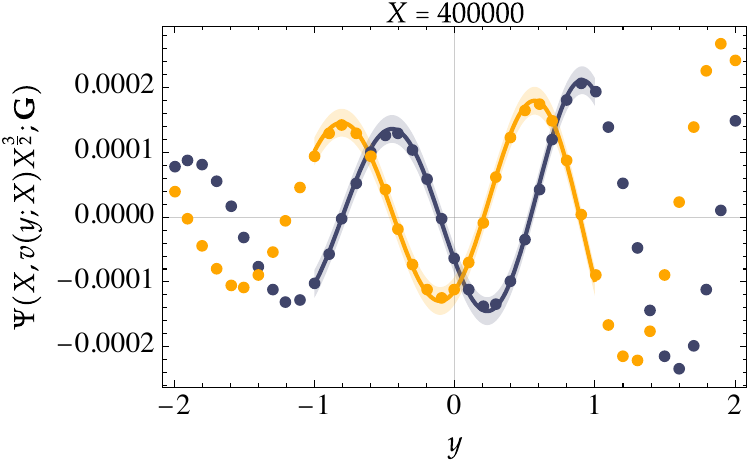}\\
\includegraphics[width=0.3\linewidth]{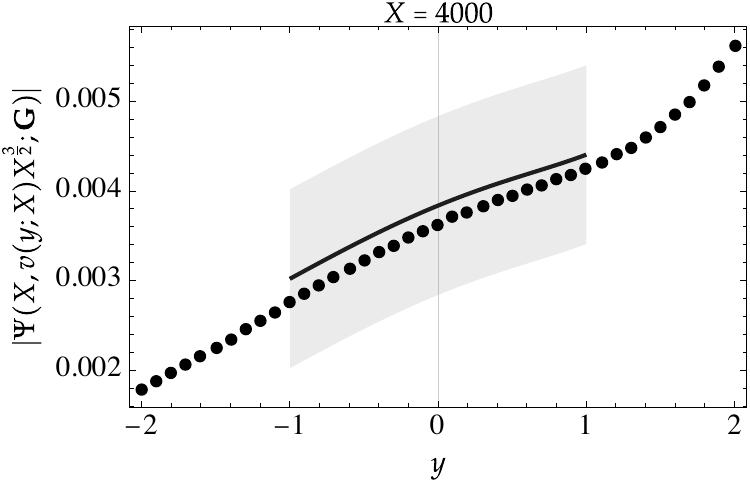}\hfill%
\includegraphics[width=0.3\linewidth]{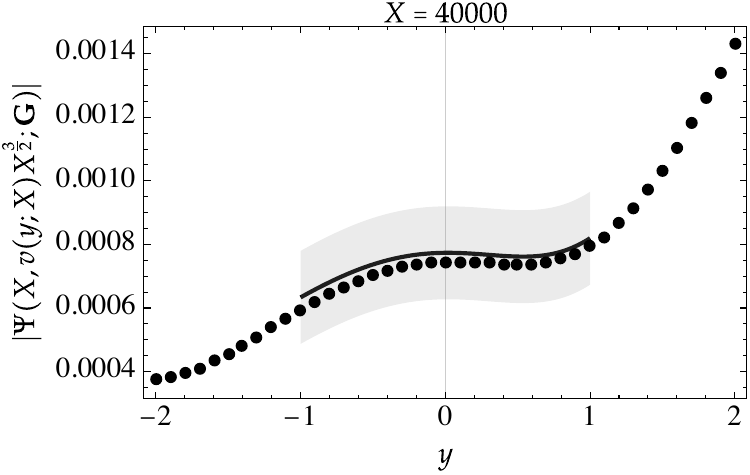}\hfill%
\includegraphics[width=0.3\linewidth]{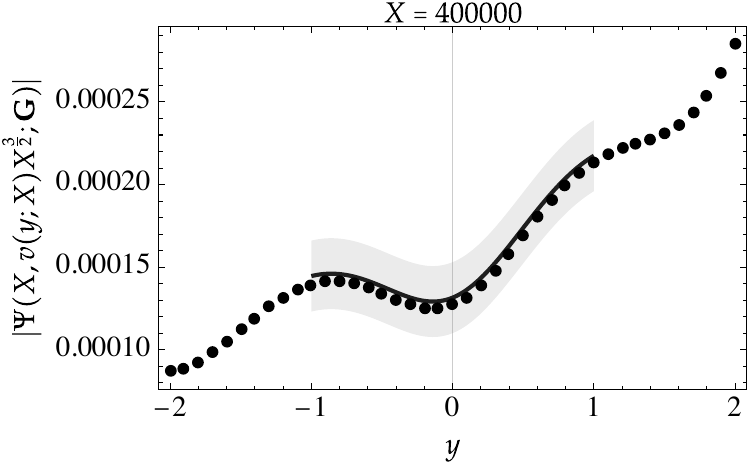}
\end{center}
\caption{Numerical evaluation of $\Psi(X,T;\mathbf{G}(1,1))$ as a function of $y$ for fixed $X$ (points) compared with the approximation of the two explicit terms in Theorem~\ref{t:transition} (solid curves).  First row:  real and imaginary parts; second row:  modulus.  Left-to-right:  $X=4000$, $X=40000$, $X=400000$.  The shaded region in each plot corresponds to an error bar proportional by a fixed constant to $X^{-5/6}$.}
\label{fig:TransitionalComparison}
\end{figure}
\begin{figure}[h]
\begin{center}
\includegraphics[width=0.5\linewidth]{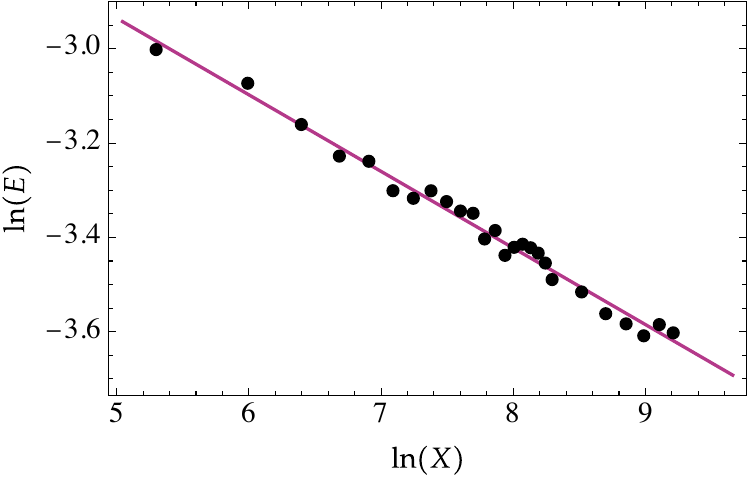}
\end{center}
\caption{The uniform renormalized error $E$ over $y\in [-1,1]$ as a function of $X$.  Black points are numerical computations and the purple line is a least-squares best fit line with slope
$-0.162334$. This matches very well the prediction of Theorem~\ref{t:transition}, namely $E=O(X^{-\frac{1}{6}})$, suggesting that the error term in \eqref{eq:Psi-transitional-final} is sharp.}
\label{fig:TransitionalLogLogError}
\end{figure}

\subsection{The software package \texttt{RogueWaveInfiniteNLS.jl} for \texttt{Julia}}
\label{s:numerics-intro}
As part of this work we introduce a software package titled \texttt{RogueWaveInfiniteNLS.jl} \cite{RogueWaveInfiniteNLS} for the \texttt{Julia} programming language. 
The main utility of \texttt{RogueWaveInfiniteNLS.jl} is that the end user can easily evaluate to high accuracy $\Psi(X,T;\mathbf{G},B)$ at a given $(X,T)\in\mathbb{R}^2$ for arbitrary choice of parameters $(a,b)\in\mathbb{C}^2$ (indexing the family of solutions) and the scalar $B>0$. This is achieved by numerically solving a suitably regularized (via numerical implementation of noncommutative steepest descent techniques) Riemann-Hilbert problem depending on the chosen value of $(X,T)$ and extracting from the solution of that Riemann-Hilbert problem the value of $\Psi(X,T;\mathbf{G},B)$.
The user need not worry about how the parameters affect the computation or about any mechanics underlying the procedure; the computation occurs in a \emph{black-box} manner. Indeed, one can simply call the main routine
\begin{lstlisting}
[julia> psi(2, 9.2, 2im, 4, 1)
\end{lstlisting}
to evaluate $\Psi(X,T;\mathbf{G},B)$ at $(X,T)=(2,9.2)$ with parameters $(a,b)=(2\ii,4)$ and $B=1$. The choice of the appropriate deformed Riemann-Hilbert problem to solve numerically is taken care of automatically.
In this regard \texttt{RogueWaveInfiniteNLS.jl} resembles the \texttt{ISTPackage} (for \texttt{Mathematica}) by T.\@ Trogdon \cite{ISTPackage}. The \texttt{ISTPackage} includes a suite of (again, black-box) routines for computing the solution of the initial-value problem on the full line with rapidly decaying initial data via the numerical inverse-scattering transform for several integrable systems.
See \cite{DeconinckOT2012} for the Korteweg-de Vries equation, \cite{TrogdonO2013} for the focusing and defocusing NLS equations, \cite{BilmanT2017} for the Toda lattice, for example. 
The first step of the numerical inverse-scattering transform procedure involved in these works is of course computation of the scattering data associated with the given initial data.
In contrast, the Riemann-Hilbert problem representation of $\Psi(X,T;\mathbf{G}, B)$ given by \rhref{rhp:near-field} does not arise from an initial-value problem or the inverse-scattering transform associated with the NLS equation \eqref{nls} --- recall Remark~\ref{rem:no-IST}.
%the slow decay rate established in Theorem~\ref{t:large-X}.
Therefore, the aforementioned routines for the NLS equation do not apply to compute the general rogue waves of infinite order studied in this work and the starting point for the framework implemented in \texttt{RogueWaveInfiniteNLS.jl} is directly the Riemann-Hilbert problem representation, which is rather the definition of this special family of solutions of the NLS equation. There is no computation of a forward (or direct) scattering transform.
The theoretical framework behind the numerical solution of Riemann-Hilbert problems is due to S.\@ Olver and T.\@ Trogdon \cite{TrogdonO2015}, see also \cite{Olver2012}. \texttt{RogueWaveInfiniteNLS.jl} relies on the routines available in the software package \texttt{OperatorApproximation.jl} \cite{OperatorApproximation} to solve the relevant Riemann-Hilbert problems numerically.
Full details on the installation, usage, and the implementation for the software package \texttt{RogueWaveInfiniteNLS.jl} are provided in Section~\ref{s:Numerics}.

Sample codes using \texttt{RogueWaveInfiniteNLS.jl} for the computations underlying the comparison and error plots given in Figure~\ref{fig:LargeX-a1-b1}, Figure~\ref{fig:LargeX-a0p5EIPiOver4-b1}, Figure~\ref{fig:LargeT-a1-b1}, Figure~\ref{fig:LargeT-a0p5EIPiOver4-b1}, and Figure~\ref{fig:TransitionalLogLogError} can be found in the notebook titled \texttt{Paper-Code.ipynb} in the public GitHub repository \cite{PaperCode}.

\subsection{Acknowledgements} D. Bilman was supported by the National Science Foundation on grant number DMS-2108029; P. D. Miller was supported by the National Science Foundation on grant numbers DMS-1812625 and DMS-2204896.
The author(s) would like to thank the Isaac Newton Institute for Mathematical Sciences, Cambridge, for support and hospitality during the program ``Emergent phenomena in nonlinear dispersive waves,'' where some of the work on this paper was undertaken. This work was supported by EPSRC grant EP/R014604/1.
%\textcolor{red}{PDM:  Maybe INI gets thanked just in the double-scaling paper?} The authors would like to thank the Isaac Newton Institute for Mathematical Sciences for support and hospitality during the programme Dispersive Hydrodynamics when work on this paper was undertaken (EPSRC Grant Number EP/R014604/1).
The computations in this work were facilitated through the use of the advanced computational, storage, and networking infrastructure provided by the Ohio Supercomputer Center (48-core Pitzer nodes) \cite{OhioSupercomputerCenter1987}.

\section{Asymptotic behavior of $\Psi(X,T;\mathbf{G})$ for large $|X|$}
\label{s:large-X}
This section is devoted to proving Theorem~\ref{t:large-X} and Theorem~\ref{t:L2-norm}. We begin with Theorem~\ref{t:large-X}.
To study $\Psi(X,T;\mathbf{G})$ for $X>0$ large and general $a,b\in\mathbb{C}$ with $ab\neq 0$, it is sufficient in light of 
%Proposition~\ref{prop:X-symmetry} to assume at first that $X>0$, and then using 
Proposition~\ref{prop:a-b-scaling} to write $\Psi(X,T;\mathbf{G}(a,b))=\ee^{-\ii\arg(ab)}\Psi(X,T;\mathbf{G}(\mathfrak{a},\mathfrak{b}))$ where the normalized parameters $(\mathfrak{a},\mathfrak{b})$ are defined in terms of $(a,b)$ as in \eqref{eq:normalized-ab}.

For $X>0$, writing $T=vX^\frac{3}{2}$ and rescaling the spectral parameter $\Lambda$ by $\Lambda=X^{-\frac{1}{2}}z$, the phase conjugating the jump matrix in \eqref{P-jump} for $\bg=1$ takes the form
\begin{equation}
\Lambda X+\Lambda^{2} T+2 \Lambda^{-1}=X^{\frac{1}{2}} \vartheta(z ; v), \quad \vartheta(z ; v):= z + v z^2 + 2 z^{-1}.
\label{phase-X}
\end{equation}
From the solution of \rhref{rhp:near-field} for $\mathbf{G}=\mathbf{G}(\mathfrak{a},\mathfrak{b})$ with $\mathfrak{a},\mathfrak{b}>0$ and $\mathfrak{a}^2+\mathfrak{b}^2=1$, and for brevity omitting $\mathbf{G}$ from the argument lists, we define a related matrix $\mathbf{S}(z;X,v)$ by
\begin{equation}
\mathbf{S}(z ; X, v):=\mathbf{P}(X^{-\frac{1}{2}} z ; X, X^{\frac{3}{2}} v ),\quad X>0,
\label{X:P-to-S}
\end{equation}
and see from \eqref{Psi-def} and Proposition~\ref{prop:a-b-scaling} that
\begin{equation}
\Psi(X, X^{\frac{3}{2}} v )=2 \ii  \ee^{-\ii\arg(ab)}X^{-\frac{1}{2}} \lim _{z \rightarrow \infty} z S_{12}(z ; X, v), \quad X>0.
\label{Psi-from-S-X}
\end{equation}
We consider the limit $X\to+\infty$ with $v\in\mathbb{R}$ held fixed (further conditions on $v$ will be introduced shortly). The matrix function $\mathbf{S}(z ; X, v)$ clearly satisfies $\mathbf{S}(z ; X, v) \to \mathbb{I}$ as $z\to\infty$ and it is analytic in the complement of an arbitrary Jordan curve $\Gamma$ surrounding $z=0$ with clockwise orientation. Across $\Gamma$, $\mathbf{S}(z ; X, v)$ satisfies the jump condition
\begin{equation}
\mathbf{S}_+(z ; X, v) = \mathbf{S}_-(z ; X, v) \ee^{-\ii X^{1/2}\vartheta(z;v)\sigma_3}\mathbf{G}(\mathfrak{a},\mathfrak{b}) \ee^{\ii X^{1/2}\vartheta(z;v)\sigma_3},\quad z\in \Gamma.
\end{equation}

\subsection{Steepest-descent deformation}
The phase $\vartheta(z;v)$ coincides with the one in \cite[Section 4.1]{BilmanLM2020}, therefore we proceed exactly as in \cite{BilmanLM2020}, assuming that $|v|<54^{-\frac{1}{2}}$ so that $\vartheta(z;v)$ has real and simple critical points. Under this assumption, the level curve $\Im(\vartheta(z;w))=0$ has a component that is a Jordan curve surrounding the origin $z=0$ and passing through two distinct (real) critical points of $\vartheta(z;w)$, which we denote by $z_1(v)<z_2(v)$ with $z_1(v)z_2(v)<0$.  We choose this Jordan curve to be the jump contour $\Gamma$ for $\mathbf{S}(z;X,v)$.  A third real critical point is present for the indicated range of $v$ only if $v\neq 0$, and it lies in the unbounded exterior of $\Gamma$.  See Figure~\ref{fig:largeX-signs} for the sign charts of $\Im(\vartheta(z;w))$ as $v$ varies in the range $|v|<54^{-\frac{1}{2}}$. 
\begin{figure}[h]
\includegraphics[width=0.3\textwidth]{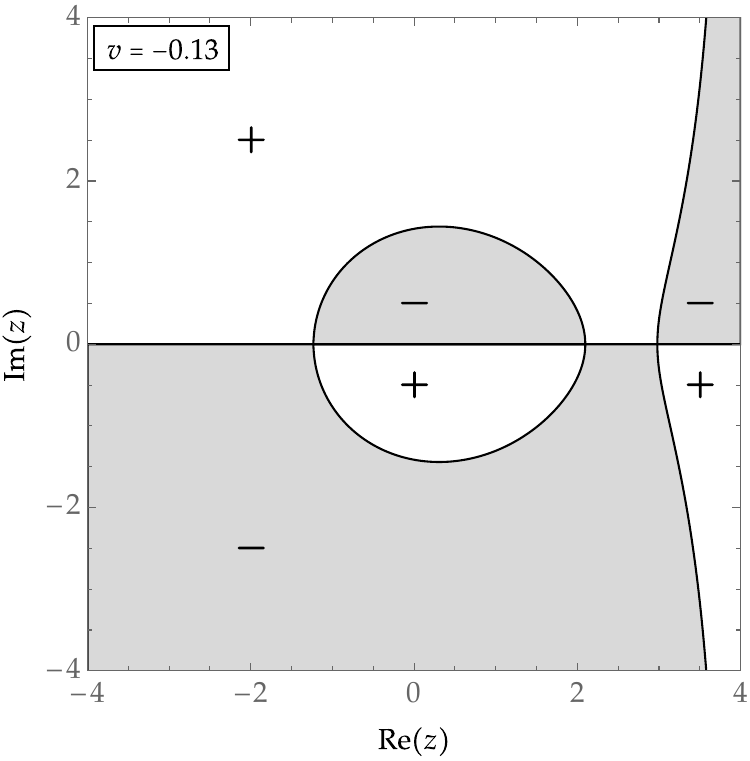}
\includegraphics[width=0.3\textwidth]{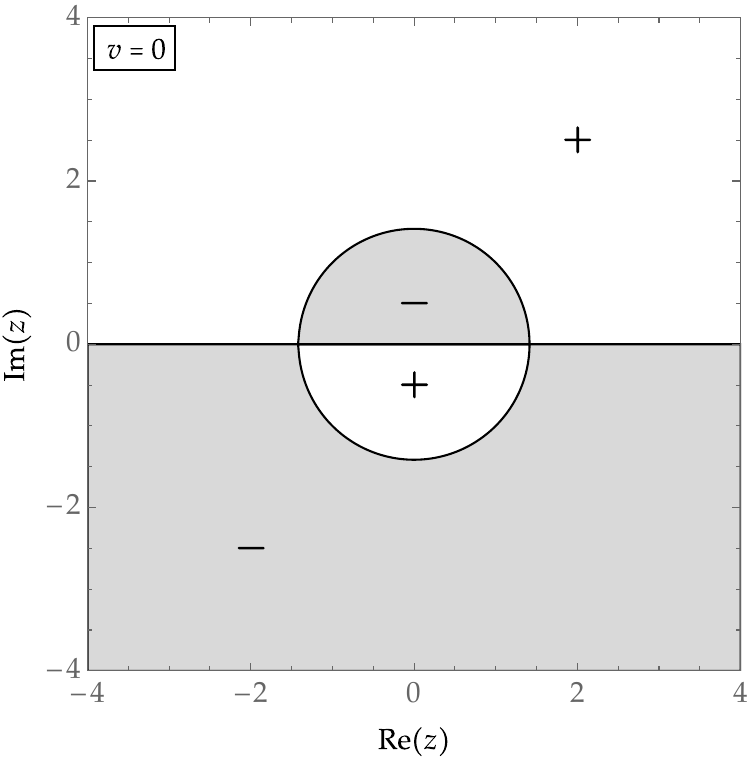}
\includegraphics[width=0.3\textwidth]{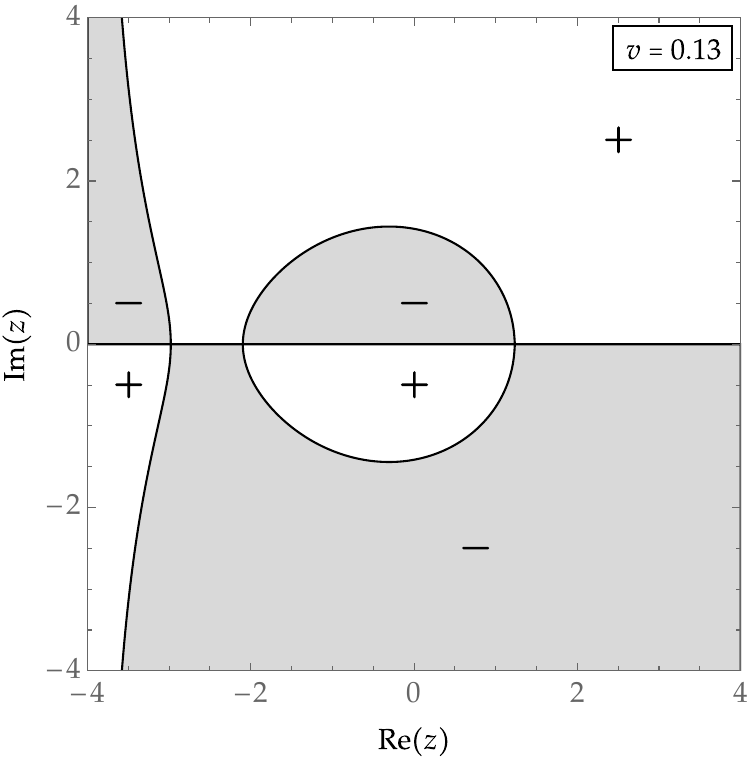}
\caption{The sign charts of $\Im(\vartheta(z;w))$ as $v$ varies in the range $|v|<54^{-\frac{1}{2}}$.}
\label{fig:largeX-signs}
\end{figure}
We define the regions $L^{\pm}$, $R^{\pm}$, and $\Omega^{\pm}$ as shown in the left-hand panel of Figure~\ref{fig:largeX-contour}.
The explicit formul\ae\ \eqref{eq:Cardano-1}--\eqref{eq:Cardano-3} for $z_1(v)$ and $z_2(v)$ are consequences of Cardano's formula, since $z=0$ cannot be a critical point of $\vartheta(\diamond;v)$ and hence $\vartheta'(z;v)=0$ is equivalent to a cubic equation for $z$.

It suffices to employ the factorizations \eqref{Gnorm-LDU} and \eqref{Gnorm-UDL} of the matrix $\mathbf{G}(\mathfrak{a},\mathfrak{b})$ for the steepest descent analysis in this case. We introduce a new unknown matrix function $\mathbf{T}(z;X,v)$ by making the following substitutions based on these factorizations.
\begin{equation}
\mathbf{T}(z;X,v):=\mathbf{S}(z;X,v)\begin{bmatrix} 1 &0 \\  \displaystyle\frac{\mathfrak{b}}{\mathfrak{a}}\ee^{2 \ii X^{1/2} \vartheta(z;v)} & 1\end{bmatrix}
,\quad z\in L^+,
\label{L-plus-sub}
\end{equation}
\begin{equation}
\mathbf{T}(z;X,v):=\mathbf{S}(z;X,v)\mathfrak{a}^{-\sigma_3} \begin{bmatrix} 1 &\mathfrak{ab} \ee^{ - 2 \ii X^{1/2} \vartheta(z;v)} \\  0 & 1\end{bmatrix}
,\quad z\in R^+,
\end{equation}
\begin{equation}
\mathbf{T}(z;X,v):=\mathbf{S}(z;X,v) \mathfrak{a}^{-\sigma_3},\quad z\in \Omega^+,
\end{equation}
\begin{equation}
\mathbf{T}(z;X,v):=\mathbf{S}(z;X,v)\mathfrak{a}^{\sigma_3},\quad z\in \Omega^-,
\end{equation}
\begin{equation}
\mathbf{T}(z;X,v):=\mathbf{S}(z;X,v) \mathfrak{a}^{\sigma_3} \begin{bmatrix} 1 & 0 \\  -\mathfrak{a b}\ee^{  2 \ii X^{1/2} \vartheta(z;v)} & 1\end{bmatrix}
,\quad z\in R^-,
\end{equation}
\begin{equation}
\mathbf{T}(z;X,v):=\mathbf{S}(z;X,v) \begin{bmatrix} 1 & -\displaystyle\frac{\mathfrak{b}}{\mathfrak{a}}\ee^{-2 \ii X^{1/2} \vartheta(z;v)}\\ 0 & 1\end{bmatrix}
,\quad z\in L^-,
\label{L-minus-sub}
\end{equation}
and we simply set $\mathbf{T}(z;X,v):=\mathbf{S}(z;X,v) $ everywhere else. See the left-hand panel of Figure~\ref{fig:largeX-contour} for the definition of the regions $R^\pm$, $L^\pm$, and $\Omega^\pm$. The jump conditions satisfied by $\mathbf{T}(z;X,v)$ are given by
\begin{equation}
\mathbf{T}_+(z;X,v)=\mathbf{T}_-(z;X,v) \begin{bmatrix} 1 &0 \\  -\displaystyle\frac{\mathfrak{b}}{\mathfrak{a}}\ee^{2 \ii X^{1/2} \vartheta(z;v)} & 1\end{bmatrix}
,\quad z\in C_{L}^+,
\label{jump-C-L-plus-X}
\end{equation}
\begin{equation}
\label{jump-C-R-plus-X}
\mathbf{T}_+(z;X,v)=\mathbf{T}_-(z;X,v) \begin{bmatrix} 1 &\mathfrak{a b}\ee^{ - 2 \ii X^{1/2} \vartheta(z;v)} \\  0 & 1\end{bmatrix}
,\quad z\in C_{R}^+,
\end{equation}
\begin{equation}
\mathbf{T}_+(z;X,v)=\mathbf{T}_-(z;X,v) \mathfrak{a}^{-2\sigma_3},
\quad z\in I,
\label{jump-T-I-X}
\end{equation}
\begin{equation}
\label{jump-C-R-minus-X}
\mathbf{T}_+(z;X,v)=\mathbf{T}_-(z;X,v) \begin{bmatrix} 1 & 0 \\  -\mathfrak{a b}\ee^{  2 \ii X^{1/2} \vartheta(z;v)} & 1\end{bmatrix}
,\quad z\in C_{R}^-,
\end{equation}
\begin{equation}
\mathbf{T}_+(z;X,v)=\mathbf{T}_-(z;X,v) \begin{bmatrix} 1 & \displaystyle\frac{\mathfrak{b}}{\mathfrak{a}}\ee^{-2 \ii X^{1/2} \vartheta(z;v)}\\ 0 & 1\end{bmatrix}
,\quad z\in C_{L}^-.
\label{jump-C-L-minus-X}
\end{equation}
See the right-hand panel of Figure~\ref{fig:largeX-contour} for definitions of the jump contours $C_R^\pm$, $C_L^\pm$, and $I$. 
\begin{figure}
\includegraphics[width=0.45\textwidth]{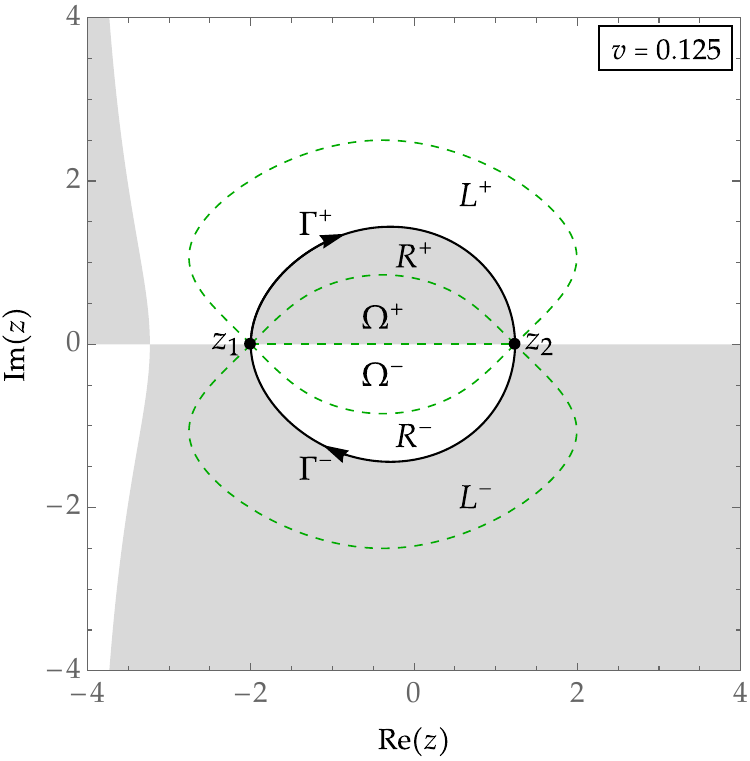}\,\includegraphics[width=0.45\textwidth]{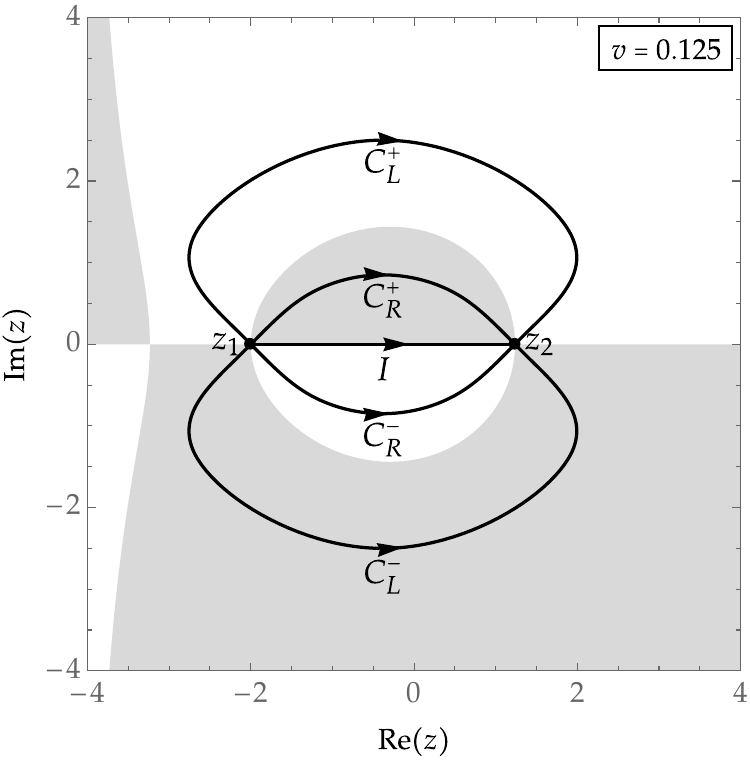}
\caption{Left: the regions $R^\pm$, $L^\pm$, and $\Omega^\pm$ used to define $\mathbf{T}(z;X,v)$. Right: the jump contours of the Riemann-Hilbert problem satisfied by $\mathbf{T}(z;X,v)$.}
\label{fig:largeX-contour}
\end{figure}
Note that we have $\Im(\vartheta(z;v))>0$ on $C_{L}^{+}$and $C_{R}^{-}$ and $\Im(\vartheta(z;v))<0$ on $C_{L}^{-}$and $C_{R}^{+}$. Therefore, as $X\to+\infty$ the jump matrices on these contours are become exponentially small perturbations of $\mathbb{I}$ uniformly except near the critical points $z=z_1(v),z_2(v)$. Since $\mathbf{S}(z;X,v)\equiv\mathbf{T}(z;X,v)$ for $|z|$ sufficiently large, we have from \eqref{Psi-from-S-X}
\begin{equation}
\Psi(X, X^{\frac{3}{2}} v )=2 \ii  \ee^{-\ii\arg(ab)}X^{-\frac{1}{2}} \lim _{z \rightarrow \infty} z T_{12}(z ; X, v), \quad X>0.
\label{Psi-from-T-X}
\end{equation}

\subsection{Parametrix construction} 
The asymptotic analysis as $X\to+\infty$ requires an outer parametrix and two inner parametrices to be used in small disks centered at $z=z_1(v),z_2(v)$. Before we proceed, we define a few quantities that let us rewrite the constant factors in the jump matrices in \eqref{jump-C-L-plus-X}--\eqref{jump-C-L-minus-X} in a more convenient way. First, %recall \eqref{norm-def} and 
note that the identity $\mathfrak{a}^2+\mathfrak{b}^2=1$ implies that
\begin{equation}
\frac{1}{\mathfrak{a}^2} = 1 + \left(\frac{\mathfrak{b}}{\mathfrak{a}}\right)^2 > 1,
\end{equation}
and hence we may define
\begin{equation}
p:= \frac{1}{2\pi} \ln\left( 1 + \left( \frac{\mathfrak{b}}{\mathfrak{a}}\right)^2 \right) > 0,
\label{p-def}
\end{equation}
so that the jump condition \eqref{jump-T-I-X} reads $\mathbf{T}_+(z;X,v)=\mathbf{T}_-(z;X,v)\ee^{2\pi p \sigma_3}$, $z\in I$. Next, introduce
\begin{equation}
\tau:=\frac{\mathfrak{b}}{\mathfrak{a}}.
%\qquad \text{and}\qquad \nu:=\arg\left( \frac{b}{a^*}\right),
\label{tau-def}
\end{equation}
Then, using again $\mathfrak{a}^2+\mathfrak{b}^2=1$,
\begin{equation}
\mathfrak{a b} = \frac{\mathfrak{b}}{\mathfrak{a}} \mathfrak{a}^2 = \tau \ee^{-2\pi p}.
\end{equation}

\subsubsection{Outer parametrix}
\label{s:large-X-outer-parametrix}
We seek an outer parametrix $\dot{\mathbf{T}}^{\mathrm{out}}(z)$ with the 
with the following properties:
\begin{itemize}
\item $\dot{\mathbf{T}}^{\mathrm{out}}(z)$ is analytic in $z$ for $z\in\mathbb{C}\setminus I$.
\item $\dot{\mathbf{T}}^{\mathrm{out}}(z) \to \mathbb{I}$ as $z\to\infty$. 
\item $\dot{\mathbf{T}}^{\mathrm{out}}(z)$ satisfies exactly the jump condition \eqref{jump-T-I-X} on $I$.
\end{itemize}
We define, using the principal branch of the power function,
\begin{equation}
\dot{\mathbf{T}}^{\mathrm{out}}(z)=\dot{\mathbf{T}}^{\mathrm{out}}(z;v):= \left( \frac{z-z_1(v)}{z-z_2(v)}\right)^{\ii p \sigma_3},
\label{W-out-X}
\end{equation}
which satisfies all of the aforementioned properties, where the value of $p$ is given in \eqref{p-def}.

\subsubsection{Inner parametrices}
We now construct inner parametrices to be used within disks $D_{z_1}(\delta)$ and $D_{z_2}(\delta)$ centered at $z=z_1(v)$ and $z=z_2(v)$, respectively, with sufficiently small and fixed radius $\delta>0$ independent of $X$. We recall that $\vartheta'(z_1(v);v)=\vartheta'(z_2(v);v)=0$ for $|v|<v_\mathrm{c}$ and note that $\vartheta''(z_1(v);v)<0$, whereas $\vartheta''(z_2(v);v)>0$. Accordingly, we define the conformal mappings $\varphi_{z_1}(z;v)$ and $\varphi_{z_2}(z;v)$ locally near $z=z_1(v)$ and $z=z_2(v)$, respectively, by the equations
\begin{equation}
\varphi_{z_1}(z ; v)^{2}=2(\vartheta(z_1(v) ; v)-\vartheta(z ; v)) \quad \text{and}\quad
\varphi_{z_2}(z ; v)^{2}=2(\vartheta(z ; v)-\vartheta(z_2(v) ; v)),
\end{equation}
and we choose the analytic solutions satisfying $\varphi_{z_1}'(z_1(v) ; v)<0$ and $\varphi_{z_2}'(z_2(v) ; v)>0$. Next, introducing the rescaled conformal coordinates $\zeta_{z_1}:=X^{\frac{1}{4}}\varphi_{z_1}$ and $\zeta_{z_2}:=X^{\frac{1}{4}}\varphi_{z_2}$, we observe that the jump conditions satisfied by 
\begin{equation}
{\mathbf{U}}^{z_1} := \mathbf{T} \ee^{-\ii X^{\frac{1}{2}}\vartheta(z_1(v);v)\sigma_3}(\ii \sigma_2),\quad \text{for $z$ near $z_1$}
\end{equation}
and by
\begin{equation}
{\mathbf{U}}^{z_2} := \mathbf{T} \ee^{-\ii X^{\frac{1}{2}}\vartheta(z_2(v);v)\sigma_3},\quad \text{for $z$ near $z_2$}
\end{equation}
take exactly the same form when expressed in terms of the respective conformal coordinates $\zeta=\zeta_{z_1}$ and $\zeta=\zeta_{z_2}$ and when the jump contours are locally taken to coincide with the five rays $\arg(\zeta)= \pm \frac{1}{4}\pi$, $\arg(\zeta)= \pm \frac{3}{4}\pi$, and $\arg(-\zeta)= 0$. Moreover, these jump conditions coincide exactly with those in (for example) \cite[Riemann-Hilbert Problem A.1]{Miller2018} for a standard parabolic cylinder parametrix. See Figure~\ref{fig:PC-z2} for the jump contours and matrices for $\mathbf{U}^{z_j}$ expressed in the coordinate $\zeta=\zeta_{z_j}$ for $j=1,2$. 
\begin{figure}[h]
\includegraphics[width=0.5\textwidth]{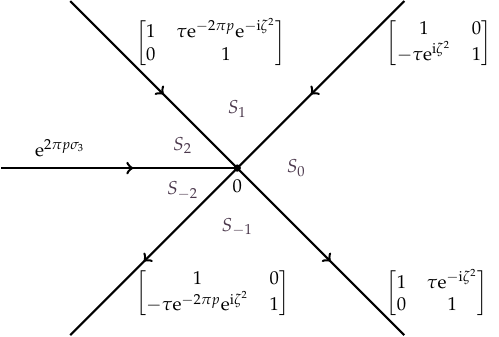}
\caption{The jump contours and jump matrices near $z=z_1(v)$ and $z=z_2(v)$ take the form given in this figure when expressed in the rescaled conformal coordinates $\zeta=\zeta_{z_1}$ and $\zeta=\zeta_{z_2}$, respectively, for which $\zeta=0$ is the image of $z=z_{1,2}(v)$. Compare with \cite[Fig.\@ 9]{Miller2018}.}
\label{fig:PC-z2}
\end{figure}
Note that the consistency condition $\tau^2 = \ee^{2\pi p} - 1$ for the jump matrices at $\zeta=0$ is satisfied by definition of $p$ and $\tau$; see \eqref{p-def} and \eqref{tau-def}.

We now let $\mathbf{U}(\zeta)=\mathbf{U}(\zeta;p,\tau)$ denote the unique solution of \cite[Riemann-Hilbert Problem A.1]{Miller2018}. This solution has the following important properties.
\begin{itemize}
\item
$\mathbf{U}(\zeta)$ is analytic in the five sectors shown in Figure~\ref{fig:PC-z2}, which are $S_0: |\arg(\zeta)|<\frac{1}{4}\pi$, $S_{1}: \frac{1}{4}\pi <\arg (\zeta)<\frac{3}{4}\pi$, $S_{-1}:-\frac{3}{4}\pi <\arg (\zeta)<-\frac{1}{4}\pi$, $S_{2}: \frac{3}{4}\pi <\arg (\zeta)<\pi$, and $S_{-2}:-\pi<\arg (\zeta)<-\frac{3}{4}\pi$. 
\item
$\mathbf{U}(\zeta)$ takes continuous boundary values on the excluded rays and at the origin from each of the five sectors, which are related by the jump condition $\mathbf{U}_+(\zeta) = \mathbf{U}_-(\zeta) \mathbf{V}^{\mathrm{PC}}(\zeta)$, where the jump contours and the jump matrix $\mathbf{V}^{\mathrm{PC}}(\zeta)$ are given in Figure~\ref{fig:PC-z2}.
\item
Importantly, the diagonal (resp., off-diagonal) part of $\mathbf{U}(\zeta)\zeta^{\ii p \sigma_3}$ has a complete asymptotic expansion in descending even (resp., odd) powers of $\zeta$ as $\zeta \to \infty$, with coefficients that are independent of the sector in which $\zeta\to \infty$. In more detail, we have
\begin{equation}
\mathbf{U}(\zeta ; p, \tau) \zeta^{\ii p \sigma_{3}}=\mathbb{I}+\frac{1}{2\ii \zeta}
\begin{bmatrix}
0 & r(p, \tau) \\
-s(p, \tau) & 0
\end{bmatrix}+\begin{bmatrix}
O\left(\zeta^{-2}\right) & O\left(\zeta^{-3}\right) \\
O\left(\zeta^{-3}\right) & O\left(\zeta^{-2}\right)
\end{bmatrix}, \quad \zeta \rightarrow \infty.
\label{U-PC-expansion}
\end{equation}
See \cite[Eqn.\@ (A.9)]{Miller2018}. Here the error terms in \eqref{U-PC-expansion} are uniform for bounded $\tau$ and $p$, 
\begin{equation}
r(p,\tau) := 2 \ee^{\ii \frac{1}{4}\pi}\sqrt{\pi}\frac{\ee^{\frac{1}{2}\pi p} \ee^{\ii p \ln(2)}}{\tau \Gamma(\ii p)},
\label{r-def}
\end{equation}
where $\Gamma(\diamond)$ is the Euler gamma function, and
\begin{equation}
s(p,\tau) := - \frac{2p}{r(p,\tau)} =  \ee^{\ii \frac{3}{4}\pi} \tau p \Gamma(\ii p) \ee^{-\frac{1}{2}\pi p} \ee^{-\ii p \ln(2)}.
\label{X-s-def}
\end{equation}
Again, see \cite[Eqn.\@ (A.7) and Eqn.\@ (A.8)]{Miller2018}. In the special case $p>0$ relevant here, it follows that $s(p,\tau)=-r(p,\tau)^*$. Then \eqref{X-s-def} implies that $|r(p,\tau)| = \sqrt{2 p}$ as well as  $|s(p,\tau)| = \sqrt{2 p}$; see \cite[Remark A.2]{Miller2018}. Therefore,
\begin{equation}
r(p,\tau) = \sqrt{2 p } \ee^{\ii(\frac{1}{4}\pi+ p \ln(2) - \arg(\Gamma(\ii p)))}\qquad\text{and}\qquad
s(p,\tau) = -\sqrt{2 p } \ee^{-\ii(\frac{1}{4}\pi+ p \ln(2) - \arg(\Gamma(\ii p)))}.
\label{r-s-polar}
\end{equation}
\end{itemize}
We define the inner parametrices by
\begin{equation}
\dot{\mathbf{T}}^{z_1}(z;X,v) := \mathbf{Y}^{z_1}(z;X,v) \mathbf{U}(\zeta_{z_1}) (\ii \sigma_2)^{-1} \ee^{\ii  X^{1/2}\vartheta(z_1(v);v)\sigma_3}, \quad z\in D_{z_1}(\delta),
\label{T-in-z1}
\end{equation}
and
\begin{equation}
\dot{\mathbf{T}}^{z_2}(z;X,v) := \mathbf{Y}^{z_2}(z;X,v) \mathbf{U}(\zeta_{z_2}) \ee^{\ii X^{1/2}\vartheta(z_2(v);v)\sigma_3}, \quad z\in D_{z_2}(\delta),
\label{T-in-z2}
\end{equation}
where the holomorphic prefactor matrices $\mathbf{Y}^{z_1}(z;X,v) $ and $\mathbf{Y}^{z_2}(z;X,v) $ will be chosen to ensure that the mismatch with the outer parametrix on the boundaries of the disks  has the behavior $\mathbb{I} + o(1)$ as $X\to +\infty$. To specify these prefactors, we first express the outer parametrix near $z=z_1$ and near $z=z_2$ in terms of the relevant conformal coordinate to see that
\begin{equation}
\begin{aligned}
\dot{\mathbf{T}}^{\mathrm{out}}(z;v)  \ee^{-\ii X^{1/2}\vartheta(z_1(v);v)\sigma_3}(\ii \sigma_2) &= X^{-\ii\frac{1}{4} p\sigma_3} \ee^{-\ii X^{1/2}\vartheta(z_1(v);v)\sigma_3} \mathbf{H}^{z_1}(z;v) \zeta_{z_1}^{-\ii p\sigma_3},\\
\mathbf{H}^{z_1}(z;v)&:= (z_2(v)-z)^{-\ii p \sigma_3}\left( \frac{z_1(v) - z}{\varphi_{z_1}(z;v)}\right)^{\ii p \sigma_3}(\ii \sigma_2),
\end{aligned}
\label{T-out-near-z1}
\end{equation}
and
\begin{equation}
\begin{aligned}
\dot{\mathbf{T}}^{\mathrm{out}}(z;v)  \ee^{-\ii X^{1/2}\vartheta(z_2(v);v)\sigma_3}&= X^{\ii\frac{1}{4} p\sigma_3} \ee^{-\ii X^{1/2}\vartheta(z_2(v);v)\sigma_3} \mathbf{H}^{z_2}(z;v) \zeta_{z_2}^{-\ii p\sigma_3},\\
\mathbf{H}^{z_2}(z;v)&:= (z-z_1(v))^{\ii p \sigma_3}\left( \frac{\varphi_{z_2}(z;v)}{z-z_2(v)}\right)^{\ii p \sigma_3}.
\end{aligned}
\label{T-out-near-z2}
\end{equation}
Here all power functions are defined as principal branches, so it is easy to verify that the matrix functions $\mathbf{H}^{z_1}(z;v)$ and $\mathbf{H}^{z_2}(z;v)$ are holomorphic in neighborhoods of $z=z_1$ and $z=z_2$, respectively. The product of factors to the left of $\zeta_{z_1}^{-\ii p\sigma_3}$ in  \eqref{T-out-near-z1} and to the left of $\zeta_{z_1}^{-\ii p\sigma_3}$ in \eqref{T-out-near-z2} determine the holomorphic prefactors in \eqref{T-in-z1} and \eqref{T-in-z2}, respectively. Thus, we define the inner parametrices by \eqref{T-in-z1} with
\begin{equation}
\mathbf{Y}^{z_1}(z;X,v):=X^{-\ii\frac{1}{4} p\sigma_3} \ee^{-\ii X^{1/2}\vartheta(z_1(v);v)\sigma_3} \mathbf{H}^{z_1}(z;v) ,
 \quad z\in D_{z_1}(\delta),
\label{T-z1-def}
\end{equation}
and by \eqref{T-in-z2} with
\begin{equation}
\mathbf{Y}^{z_2}(z;X,v):=X^{\ii\frac{1}{4} p\sigma_3} \ee^{-\ii X^{1/2}\vartheta(z_2(v);v)\sigma_3} \mathbf{H}^{z_2}(z;v),
\quad z\in D_{z_2}(\delta).
\label{T-z2-def}
\end{equation}
The analyticity properties and the jump conditions satisfied by $\mathbf{U}(\zeta)$ imply that the inner parametrices $\dot{\mathbf{T}}^{z_1}(z;X,v) $ and $\dot{\mathbf{T}}^{z_2}(z;X,v) $ exactly satisfy the jump conditions for $\mathbf{T}(z;X,v)$ within their respective disks after deforming the jump contours in the disk to agree with the preimages under $\lambda\mapsto \zeta_{z_1,z_2}$ of the five rays shown in Figure~\ref{fig:PC-z2}.
\begin{remark} The unique solution $\mathbf{U}(\zeta)=\mathbf{U}(\zeta;p,\tau)$ of \cite[Riemann-Hilbert Problem A.1]{Miller2018} becomes the identity matrix if $p=0$. Note from \eqref{p-def} that as $\mathfrak{b}\to 0$, we have $p\to 0^+$. Therefore, both of the inner parametrices degenerate to the identity matrix as $\mathfrak{b}\to 0$.
\label{rem:p-vanishes}
\end{remark}
We define the global parametrix $\dot{\mathbf{T}}(z;X,v)$ by
\begin{equation}
\dot{\mathbf{T}}(z;X,v):=\begin{cases}
\dot{\mathbf{T}}^{z_1}(z;X,v),&\quad z\in D_{z_1}(\delta),\\
\dot{\mathbf{T}}^{z_2}(z;X,v),&\quad z\in D_{z_2}(\delta),\\
\dot{\mathbf{T}}^{\mathrm{out}}(z;v),&\quad z \in \mathbb{C} \setminus \left(I \cup D_{z_1}(\delta) \cup D_{z_2}(\delta)\right).
\end{cases}
\label{eq:large-X-global-parametrix}
\end{equation}

\subsection{Asymptotics as $X\to+\infty$}
\label{s:asymptotics-X}
We compare the unknown $\mathbf{T}(z;X,v)$ with the global parametrix $\dot{\mathbf{T}}(z;X,v)$ and define the error
\begin{equation}
\mathbf{F}(z;X,v):= \mathbf{T}(z;X,v) \dot{\mathbf{T}}(z;X,v)^{-1}.
\end{equation}
As $\dot{\mathbf{T}}(z;X,v)$ is an exact solution of the jump conditions satisfied by ${\mathbf{T}}(z;X,v)$ on the part of $I$ outside the disks $D_{z_1}(\delta)$ and $D_{z_2}(\delta)$ and on the arcs inside these disks, it follows that $\mathbf{F}(z;X,v)$ extends as an analytic function of $z \in \mathbb{C}\setminus \Sigma_{\mathbf{F}}$, where the contour $\Sigma_{\mathbf{F}}$ consists of the arcs of $C_{L}^{\pm}$and $C_{R}^{\pm}$ lying outside of the disks $D_{z_1, z_2}(\delta)$, and the boundaries $\partial D_{z_1, z_2}(\delta)$. We orient the circular boundaries $\partial D_{z_1, z_2}(\delta)$ clockwise, and consider the jump matrix $\mathbf{V}^{\mathbf{F}}(z;X,v)$ that relates the boundary values of $\mathbf{F}(z;X,v)$ through the jump condition
\begin{equation}
\mathbf{F}_+(z;X,v) = \mathbf{F}_-(z;X,v) \mathbf{V}^{\mathbf{F}}(z;X,v), \qquad z\in\Sigma_{\mathbf{F}}.
\label{E-jump}
\end{equation}
Since $\dot{\mathbf{T}}^{\mathrm{out}}(z;X,v)$ is analytic in $z$ for $z\in (C_{L}^{\pm}\cup C_{R}^{\pm} ) \cap \Sigma_{\mathbf{F}}$, we may express the jump matrix $\mathbf{V}^{\mathbf{F}}(z;X,v)$ for $\mathbf{F}(z;X,v)$ on these arcs as
\begin{multline}
\mathbf{V}^{\mathbf{F}}(z;X,v) = \mathbf{F}_-(z;X,v)^{-1} \mathbf{F}_+(z;X,v) =\dot{\mathbf{T}}^{\mathrm{out}}(z ; v) \mathbf{T}_{-}(z ; X, v)^{-1} \mathbf{T}_{+}(z ; X, v) \dot{\mathbf{T}}^{\mathrm{out}}(z ; v)^{-1},\\
 z \in\left(C_{L}^{\pm} \cup C_{R}^{\pm}\right) \cap \Sigma_{\mathbf{F}}
\end{multline}
As $\delta>0$ is fixed, $\dot{\mathbf{T}}^{\text {out }}(z;v)$ is independent of $X$, and $z$ is restricted to the arcs $C_{L}^{\pm} \cup C_{R}^{\pm}$ that lie outside the disks $D_{z_1,z_2}(\delta)$ on which the jump matrix for $\mathbf{T}(z;X,v)$ becomes an exponentially small perturbation of $\mathbb{I}$ as $X\to+\infty$, 
there exists a constant $K(\varepsilon)>0$ such that 
\begin{equation}
\sup _{z \in\left(C_{L}^{\pm} \cup C_{R}^{\pm}\right) \cap \Sigma_{\mathbf{F}}}\left\|\mathbf{V}^{\mathbf{F}}(z ; X, v)-\mathbb{I}\right\|=O(\ee^{-K(\varepsilon)X^{1/2}}),\quad X\to+\infty,
\label{V-E-outside-decay}
\end{equation}
holds uniformly for $|v|\le 54^{-\frac{1}{2}}-\varepsilon$ and normalized parameters $(\mathfrak{a},\mathfrak{b})$ with $\mathfrak{b}/\mathfrak{a}\le\varepsilon^{-1}$, where $\|\diamond \|$ denotes the matrix norm induced from an arbitrary norm on $\mathbb{C}^{2}$.  

To analyze the jump $\mathbf{V}^{\mathbf{F}}(z;X,v)$ on the circular boundaries $\partial D_{z_1}(\delta)$ and $\partial D_{z_2}(\delta)$, we use the fact that $\mathbf{T}(z;X , v)$ is analytic in $z$ at all but finitely many points on $\partial D_{z_1}(\delta)$ and $\partial D_{z_2}(\delta)$ and hence observe that
\begin{equation}
\mathbf{V}^\mathbf{F}(z;X,v) = 
\begin{cases} 
\dot{\mathbf{T}}^{z_1}(z;X,v) \dot{\mathbf{T}}^{\mathrm{out}}(z;v)^{-1},&\quad \partial D_{z_1}(\delta),\\
\dot{\mathbf{T}}^{z_2}(z;X,v) \dot{\mathbf{T}}^{\mathrm{out}}(z;v)^{-1},&\quad \partial D_{z_2}(\delta).
\end{cases}
\end{equation}
Then, recalling that $\zeta_{z_1} = X^{\frac{1}{4}}\varphi_{z_1}(z;v)$, from \eqref{U-PC-expansion}, \eqref{T-in-z1}, and \eqref{T-z1-def} we have
\begin{multline}
\mathbf{V}^\mathbf{F}(z;X,v) = 
X^{-\ii\frac{1}{4} p\sigma_3} \ee^{-\ii X^{1/2}\vartheta(z_1(v);v)\sigma_3} \mathbf{H}^{z_1}(z;v) \\
{}\cdot \left( \mathbb{I} + \frac{1}{2\ii X^{\frac{1}{4}}\varphi_{z_1}(z;v)} \begin{bmatrix}
0 & r(p, \tau) \\
-s(p, \tau) & 0
\end{bmatrix}+\begin{bmatrix}
O(X^{-\frac{1}{2}}) & O(X^{-\frac{3}{4}}) \\
O(X^{-\frac{3}{4}}) & O(X^{-\frac{1}{2}})
\end{bmatrix}
\right)\\
{}\cdot  \mathbf{H}^{z_1}(z;v)^{-1} \ee^{\ii X^{1/2}\vartheta(z_1(v);v)\sigma_3} X^{\ii\frac{1}{4} p\sigma_3},\quad z\in\partial D_{z_1}(\delta).
\label{V-E-D-z1-decay}
\end{multline}
Similarly, recalling that $\zeta_{z_2} = X^{\frac{1}{4}}\varphi_{z_2}(z;v)$, from \eqref{U-PC-expansion}, \eqref{T-in-z2}, and \eqref{T-z2-def} we have
\begin{multline}
\mathbf{V}^\mathbf{F}(z;X,v) = 
X^{\ii\frac{1}{4} p\sigma_3} \ee^{-\ii X^{1/2}\vartheta(z_2(v);v)\sigma_3} \mathbf{H}^{z_2}(z;v) \\
{}\cdot \left( \mathbb{I} + \frac{1}{2\ii X^{\frac{1}{4}}\varphi_{z_2}(z;v)} \begin{bmatrix}
0 & r(p, \tau) \\
-s(p, \tau) & 0
\end{bmatrix}+\begin{bmatrix}
O(X^{-\frac{1}{2}}) & O(X^{-\frac{3}{4}}) \\
O(X^{-\frac{3}{4}}) & O(X^{-\frac{1}{2}})
\end{bmatrix}
\right)\\
{}\cdot  \mathbf{H}^{z_2}(z;v)^{-1} \ee^{\ii X^{1/2}\vartheta(z_2(v);v)\sigma_3} X^{-\ii\frac{1}{4} p\sigma_3},\quad z\in\partial D_{z_2}(\delta).
\label{V-E-D-z2-decay}
\end{multline}
In both \eqref{V-E-D-z1-decay} and \eqref{V-E-D-z2-decay} the error terms are uniform for normalized parameters $(\mathfrak{a},\mathfrak{b})$ with $\mathfrak{b}/\mathfrak{a}$ bounded because the latter condition implies that $p$ and $\tau$ are bounded.
Combining \eqref{V-E-outside-decay}, \eqref{V-E-D-z1-decay}, and \eqref{V-E-D-z2-decay} gives in particular the uniform estimate 
\begin{equation}
\sup_{z\in \Sigma_{\mathbf{F}}} \| \mathbf{V}^{\mathbf{F}}(z;X,v) - \mathbb{I}\| = O_\varepsilon(X^{-\frac{1}{4}}), \quad X\to+\infty,\quad |v|\le 54^{-\frac{1}{2}}-\varepsilon,\; \mathfrak{b}/\mathfrak{a}\le\varepsilon^{-1}.
\label{error-jump-estimate}
\end{equation}
Here, the notation $O_\varepsilon(\diamond)$ indicates that the implied constant depends on $\varepsilon>0$.
Since $\Sigma_\mathbf{F}$ is a compact contour and $\mathbf{F}(z;X,v)$ is analytic in $z$ for $z\in\mathbb{C}\setminus \Sigma_\mathbf{F} $ with $\mathbf{F}(z;X,v) \to \mathbb{I}$ as $z\to\infty$, the small-norm theory for such Riemann-Hilbert problems implies that the error $\mathbf{F}(z;X,v) $ satisfies
\begin{equation}
\mathbf{F}_-(\diamond; X,v ) - \mathbb{I} = O_\varepsilon(X^{-\frac{1}{4}}),\quad X\to +\infty,\quad |v|\le54^{-\frac{1}{2}}-\varepsilon,\; \mathfrak{b}/\mathfrak{a}\le\varepsilon^{-1}
\label{E-minus-estimate}
\end{equation}
in the $L^2(\Sigma_\mathbf{F})$ sense. See \cite[Section 4.1.3]{BilmanLM2020} for details of the argument implying this estimate. 
%In both \eqref{error-jump-estimate} and \eqref{E-minus-estimate} the estimates are uniform with respect to $(\mathfrak{a},\mathfrak{b})$ for $\mathfrak{b}/\mathfrak{a}$ bounded.  
Reformulating the jump condition \eqref{E-jump} in the form $\mathbf{F}_+ - \mathbf{F}_- = \mathbf{F}_-(\mathbf{V}^{\mathbf{F}} - \mathbb{I})$ and using the fact that both factors on the right-hand side tend to the identity matrix as $z\to\infty$, we obtain from the Plemelj formula
\begin{equation}
\mathbf{F}(z;X,v) = \mathbb{I} + \frac{1}{2\pi \ii} \int_{\Sigma_\mathbf{F}} \frac{ \mathbf{F}_-(\zeta;X,v)(\mathbf{V}^{\mathbf{F}}(\zeta;X,v) - \mathbb{I})}{\zeta-z}\dd \zeta,\quad z\in\mathbb{C}\setminus \Sigma_{\mathbf{F}}.
\label{E-Plemelj-X}
\end{equation}
Now, we have
\begin{equation}
\mathbf{T}(z;X,v) = \mathbf{F}(z;X,v) \dot{\mathbf{T}}^{\mathrm{out}}(z;v)
\label{W-in-F-Wout}
\end{equation}
holding for $|z|$ sufficiently large. Recall from \eqref{W-out-X} that $\dot{\mathbf{T}}^{\mathrm{out}}(z;v)$ is a diagonal matrix that tends to $\mathbb{I}$ as $z\to \infty$. Thus, we see from \eqref{Psi-from-T-X} that
%and the (convergent) Laurent series expansion of \eqref{E-Plemelj-X} as $z\to\infty$ that
\begin{equation}
\Psi(X,X^{\frac{3}{2}}v) = 2\ii \ee^{-\ii \arg(ab)} X^{-\frac{1}{2}}\lim_{z\to\infty} z F_{12}(z;X,v).
\end{equation}
Then, using the Laurent series expansion of $\mathbf{F}(z;X,v)$ obtained from \eqref{E-Plemelj-X} and convergent for $|z|>\sup_{s\in\Sigma_\mathbf{F}}|s|$, we arrive at the exact formula
\begin{multline}
\Psi(X, X^{\frac{3}{2}} v)=-\frac{\ee^{-\ii \arg(ab)}}{\pi X^{\frac{1}{2}}}\left[\int_{\Sigma_{\mathbf{F}}} F_{11-}(z ; X, v) V_{12}^{\mathbf{F}}(z ; X, v) \dd z \right.\\
\left.{}+\int_{\Sigma_{\mathbf{F}}} F_{12-}(z ; X, v)\left(V_{22}^{\mathbf{F}}(z; X, v)-1\right)  \dd z\right].
\label{Psi-large-X-exact}
\end{multline}
Looking at the diagonal elements of \eqref{V-E-D-z1-decay} and \eqref{V-E-D-z2-decay} along with \eqref{V-E-outside-decay}, we see that ${V}_{22}^{\mathbf{F}}(\diamond; X,v) - 1 = O_\varepsilon(X^{-\frac{1}{2}})$ holds uniformly on $\Sigma_\mathbf{F}$. Since $\Sigma_\mathbf{F}$ is compact, we also have ${V}_{22}^{\mathbf{F}}(\diamond; X,v) - 1 = O_\varepsilon(X^{-\frac{1}{2}})$ in $L^2(\Sigma_{\mathbf{F}})$ as $X\to +\infty$. On the other hand, the $L^2$ estimate \eqref{E-minus-estimate} implies that $F_{12-}(\diamond;X,v) = O_\varepsilon(X^{-\frac{1}{4}})$ holds in $L^2(\Sigma_{\mathbf{F}})$. Thus, applying Cauchy-Schwarz inequality to the modulus of the second integral in \eqref{Psi-large-X-exact} yields
\begin{equation}
\Psi(X, X^{\frac{3}{2}} v)=-\frac{\ee^{-\ii \arg(ab)}}{\pi X^{\frac{1}{2}}} \int_{\Sigma_{\mathbf{F}}} F_{11-}(z ; X, v) V_{12}^{\mathbf{F}}(z ; X, v) \dd z + O_\varepsilon(X^{-\frac{5}{4}}),\quad X\to +\infty,
\label{Psi-large-X-approx-1}
\end{equation}
when $|v|\leq 54^{-\frac{1}{2}}-\varepsilon$ and $\mathfrak{b}/\mathfrak{a}\le\varepsilon^{-1}$. We express the integral in \eqref{Psi-large-X-approx-1} as
\begin{equation}
\int_{\Sigma_{\mathbf{F}}} F_{11-}(z ; X, v) V_{12}^{\mathbf{F}}(z ; X, v) \dd z = \int_{\Sigma_{\mathbf{F}}} (F_{11-}(z ; X, v) - 1 ) V_{12}^{\mathbf{F}}(z ; X, v) \dd z + \int_{\Sigma_{\mathbf{F}}}  V_{12}^{\mathbf{F}}(z ; X, v) \dd z.
\label{integrals-X}
\end{equation}
By Cauchy-Schwarz and the fact that because $\Sigma_\mathbf{F}$ is compact the $L^2$-norm is subordinate to the $L^\infty$-norm,
\begin{equation}
\int_{\Sigma_{\mathbf{F}}}(F_{11-}(z;X,v)-1)V_{12}^{\mathbf{F}}(z;X,v)\,\dd z = O\left(\|F_{11-}(\diamond;X,v)-1\|_{L^2(\Sigma_\mathbf{F})}\|V_{12}^{\mathbf{F}}(\diamond;X,v)\|_{L^\infty(\Sigma_\mathbf{F})}\right).
\label{Int-F11mV12}
\end{equation}
Now the conjugating factors in \eqref{V-E-D-z1-decay} and \eqref{V-E-D-z2-decay} are off-diagonal and diagonal respectively, and these factors are also oscillatory.  Therefore,  using the exponential bound \eqref{V-E-outside-decay} gives the estimates ${V}_{12}^{\mathbf{F}}(\diamond; X,v)  = O_\varepsilon(X^{-\frac{1}{4}})$, $V_{21}^{\mathbf{F}}(\diamond;X,v)=O_\varepsilon(X^{-\frac{1}{4}})$, and $V_{11}^{\mathbf{F}}(\diamond;X,v)-1=O_\varepsilon(X^{-\frac{1}{2}})$ all holding in the $L^{\infty}(\Sigma_\mathbf{F})$ sense. For the $L^2$ estimate of $F_{11-}(\diamond;X,v)-1$, we can improve upon \eqref{E-minus-estimate} by using the fact that the ``minus'' boundary value of the Cauchy integral on the right-hand side of \eqref{E-Plemelj-X} is bounded as a linear operator in $L^2(\Sigma_\mathbf{F})$ acting on the matrix-valued numerator to obtain from the $11$-entry:
\begin{multline}
\|F_{11-}(\diamond;X,v)-1\|_{L^2(\Sigma_\mathbf{F})}\le C(\Sigma_\mathbf{F})\left[\|(F_{11-}(\diamond;X,v)-1)(V_{11}^{\mathbf{F}}(\diamond;X,v)-1)\|_{L^2(\Sigma_\mathbf{F})}\right.\\
\left.{}+ \|V_{11}^{\mathbf{F}}(\diamond;X,v)-1\|_{L^2(\Sigma_\mathbf{F})} + \|F_{12-}(\diamond;X,v)V_{21}^{\mathbf{F}}(\diamond;X,v)\|_{L^2(\Sigma_\mathbf{F})}\right].
\end{multline}
The bound $C(\Sigma_\mathbf{F})$ of the operator norm depends only on the contour $\Sigma_\mathbf{F}$, which in turn depends on $v$ but neither on $X$ nor on the normalized parameters $(\mathfrak{a},\mathfrak{b})$; however it may be taken to depend only on $\varepsilon$ if $|v|\le 54^{-\frac{1}{2}}-\varepsilon$.
Using the $L^\infty(\Sigma_\mathbf{F})$ estimate of $V^{\mathbf{F}}_{11}(\diamond;X,v)-1$, this yields (assuming $X>0$ is sufficiently large)
\begin{equation}
\begin{split}
\|F_{11-}(\diamond;X,v)-1\|_{L^2(\Sigma_\mathbf{F})} &= O_\varepsilon\left(\|V_{11}^{\mathbf{F}}(\diamond;X,v)-1\|_{L^2(\Sigma_\mathbf{F})}\right) + O_\varepsilon\left( \|F_{12-}(\diamond;X,v)V_{21}^{\mathbf{F}}(\diamond;X,v)\|_{L^2(\Sigma_\mathbf{F})}\right)\\
&= O_\varepsilon\left(\|V_{11}^{\mathbf{F}}(\diamond;X,v)-1\|_{L^\infty(\Sigma_\mathbf{F})}\right) \\
&\quad\quad{}+ O_\varepsilon\left( \|F_{12-}(\diamond;X,v)\|_{L^2(\Sigma_\mathbf{F})}\|V_{21}^{\mathbf{F}}(\diamond;X,v)\|_{L^\infty(\Sigma_\mathbf{F})}\right),
\end{split}
\end{equation}
again using the fact that $\Sigma_\mathbf{F}$ is compact.  Combining the $L^\infty$ estimates for $V_{11}^\mathbf{F}(\diamond;X,v)-1$ and $V_{21}^\mathbf{F}(\diamond;X,v)$ with the $O_\varepsilon(X^{-\frac{1}{4}})$ estimate of $F_{12-}(\diamond;X,v)$ in $L^2$ implied by \eqref{E-minus-estimate} then shows that $F_{11-}(\diamond;X,v)-1=O_\varepsilon(X^{-\frac{1}{2}})$ holds in the $L^2(\Sigma_\mathbf{F})$ sense, provided $|v|\le 54^{-\frac{1}{2}}-\varepsilon$ and $\mathfrak{b}/\mathfrak{a}\le\varepsilon^{-1}$.  Using this and the $O(X^{-\frac{1}{4}})$ estimate of $V_{12}^\mathbf{F}(\diamond;X,v)$ in $L^\infty(\Sigma_\mathbf{F})$ in \eqref{Int-F11mV12} then shows that the first term on the right-hand side of \eqref{integrals-X} is $O_\varepsilon(X^{-\frac{3}{4}})$ as $X\to+\infty$, and hence its contribution to \eqref{Psi-large-X-approx-1} can be absorbed into the error term already present in that asymptotic formula.  Therefore, \eqref{Psi-large-X-approx-1} gives a formula for $\Psi(X,X^\frac{3}{2}v)$ with an explicit leading term:
\begin{multline}
\Psi(X, X^{\frac{3}{2}} v)=-\frac{\ee^{-\ii \arg(ab)}}{\pi X^{\frac{1}{2}}} \int_{\Sigma_{\mathbf{F}}} V_{12}^{\mathbf{F}}(z ; X, v) \dd z + O_\varepsilon(X^{-\frac{5}{4}}),\\ X\to +\infty,
\quad |v|\le 54^{-\frac{1}{2}}-\varepsilon,\; \mathfrak{b}/\mathfrak{a}\le\varepsilon^{-1}.
\label{Psi-large-X-approx-2}
\end{multline}
Recalling the exponential decay \eqref{V-E-outside-decay}, the estimate \eqref{Psi-large-X-approx-2} holds with a different error, but of the same size, if we replace the contour of integration $\Sigma_{\mathbf{F}}$ by $\partial D_{z_1}(\delta) \cup \partial D_{z_2}(\delta)$. Thus, we arrive at
\begin{multline}
\Psi(X, X^{\frac{3}{2}} v)=-\frac{\ee^{-\ii \arg(ab)}}{\pi X^{\frac{1}{2}}} \int_{\partial D_{z_1}(\delta) \cup \partial D_{z_2}(\delta)} V_{12}^{\mathbf{F}}(z ; X, v) \dd z + O_\varepsilon(X^{-\frac{5}{4}}),\\ X\to +\infty,\quad |v|\le54^{-\frac{1}{2}}-\varepsilon,\;\mathfrak{b}/\mathfrak{a}\le\varepsilon^{-1}.
\label{Psi-large-X-approx-3}
\end{multline}
Since $\mathbf{H}^{z_1}(z;v)$ is an off-diagonal matrix (see \eqref{T-out-near-z1}), 
\eqref{V-E-D-z1-decay} shows that as $X\to+\infty$,
\begin{equation}
V_{12}^{\mathbf{F}}(z;X,v) = \frac{X^{-\ii \frac{1}{2} p} \ee^{-2\ii X^{1/2}\vartheta(z_1(v);v)}}{2\ii X^{\frac{1}{4}} \varphi_{z_1}(z;v)}s(p,\tau) H^{z_1}_{12}(z;v)^2 + O_\varepsilon(X^{-\frac{3}{4}}),\quad z\in \partial D_{z_1}(\delta),
\label{V-12-D-z1}
\end{equation}
and since $\mathbf{H}^{z_2}(z;v)$ is a diagonal matrix (recall \eqref{T-out-near-z2}), \eqref{V-E-D-z2-decay} shows that as $X\to+\infty$,
\begin{equation}
V_{12}^{\mathbf{F}}(z;X,v) = \frac{X^{\ii \frac{1}{2} p} \ee^{-2\ii X^{1/2}\vartheta(z_2(v);v)}}{2\ii X^{\frac{1}{4}} \varphi_{z_2}(z;v)}r(p,\tau) H^{z_2}_{11}(z;v)^2 + O_\varepsilon(X^{-\frac{3}{4}}),\quad z\in \partial D_{z_2}(\delta),
\label{V-12-D-z2}
\end{equation}
where both of the errors are uniform on the indicated boundary contours, and $|v|< 54^{-\frac{1}{2}}-\varepsilon$ and $\mathfrak{b}/\mathfrak{a}\le\varepsilon^{-1}$.
We can compute the integrals of the explicit leading terms in \eqref{V-12-D-z1} and \eqref{V-12-D-z2} on the relevant circular boundaries by residues at $z=z_1, z_2$ since $\varphi_{z_1}(z;v)$ has a simple zero at $z=z_1(v)$ and $\varphi_{z_2}(z;v)$ has a simple zero at $z=z_2(v)$, while $\mathbf{H}^{z_1}(z;v)$ and $\mathbf{H}^{z_2}(z;v)$ are analytic in $D_{z_1}(\delta)$ and $D_{z_2}(\delta)$, respectively. Doing so in \eqref{Psi-large-X-approx-3} noting the clockwise orientation of the circles $\partial D_{z_1,z_2}(\delta)$ yields, as $X\to+\infty$,
\begin{multline}
\Psi(X, X^{\frac{3}{2}} v)=\frac{\ee^{-\ii\arg(ab)}}{X^{\frac{3}{4}}}\left[ X^{-\ii \frac{1}{2} p} \ee^{-2\ii X^{1/2}\vartheta(z_1(v);v)} \frac{s(p,\tau) H^{z_1}_{12}(z_1(v);v)^2}{\varphi_{z_1}^\prime (z_1(v);v)}\right. \\ \left.+  X^{\ii \frac{1}{2} p} \ee^{-2\ii X^{1/2}\vartheta(z_2(v);v)} \frac{r(p,\tau) H^{z_2}_{11}(z_2(v);v)^2}{\varphi_{z_2}^\prime(z_2(v);v)} \right] + O_\varepsilon(X^{-\frac{5}{4}}),
\label{Psi-large-X-approx-4}
\end{multline}
assuming $|v|<54^{-\frac{1}{2}}-\varepsilon$ and $\mathfrak{b}/\mathfrak{a}\le\varepsilon^{-1}$.

\begin{remark}
\label{Remark-Sharpen}
Analogues of \eqref{Psi-large-X-approx-2}, \eqref{V-12-D-z1}, \eqref{V-12-D-z2}, and \eqref{Psi-large-X-approx-3} were obtained for a particular choice of the parameters $(a,b)$ in equations (160), (161), (162), and (163) respectively of our earlier paper \cite{BilmanLM2020}, but with cruder error estimates.  
\end{remark}

It now remains to compute the four quantities $H_{12}^{z_1}(z_1(v) ; v)$, $\varphi_{z_1}^{\prime}(z_1(v) ; v)$, $H_{11}^{z_2}(z_2(v) ; v)$, and $\varphi_{z_2}^{\prime}(z_2(v) ; v)$. First, by definition
\begin{equation}
\varphi_{z_1}^{\prime}(z_1(v) ; v) = -\sqrt{- \vartheta''(z_1(v);v)}\qquad\text{and}\qquad \varphi_{z_2}^{\prime}(z_2(v) ; v) = \sqrt{ \vartheta''(z_2(v);v)}.
\end{equation}
Next, using \eqref{T-out-near-z1} and \eqref{T-out-near-z2} and L'H\^opital's rule,
\begin{equation}
\mathbf{H}^{z_1}(z_1(v);v) = (z_2(v) - z_1(v))^{-\ii p \sigma_3} \left(\frac{-1}{\varphi_{z_1}^\prime(z_1(v); v)}\right)^{\ii p \sigma_3}(\ii\sigma_2)
\end{equation}
and
\begin{equation}
\mathbf{H}^{z_2}(z_2(v);v) = (z_2(v) - z_1(v))^{\ii p \sigma_3} \varphi_{z_2}^\prime(z_2(v); v)^{\ii p \sigma_3}.
\end{equation}
Therefore, using \eqref{r-s-polar} we obtain
\begin{equation}
\begin{aligned}
\frac{s(p,\tau) H^{z_1}_{12}(z_1(v);v)^2}{\varphi_{z_1}^\prime (z_1(v);v)} &= - (z_2(v) - z_1(v))^{- 2\ii p} (-\vartheta''(z_1(v);v))^{-\ii p}
\frac{s(p,\tau)}{\sqrt{-\vartheta''(z_1(v);v)}}\\
&= \frac{\sqrt{2p}\ee^{-\ii(\frac{1}{4}\pi+ p \ln(2) - \arg(\Gamma(\ii p)))}}{\sqrt{-\vartheta''(z_1(v);v)}} (z_2(v) - z_1(v))^{- 2\ii p} (-\vartheta''(z_1(v);v))^{-\ii p}  
\end{aligned}
\label{H-z1-expand}
\end{equation}
and
\begin{equation}
\begin{aligned}
\frac{r(p,\tau) H^{z_2}_{11}(z_2(v);v)^2}{ \varphi_{z_2}^\prime(z_2(v);v)} &=  (z_2(v) - z_1(v))^{2\ii p} \vartheta''(z_2(v);v)^{\ii p} \frac{r(p,\tau)}{\sqrt{\vartheta''(z_2(v);v)}}\\
&=  \frac{ \sqrt{2 p } \ee^{\ii(\frac{1}{4}\pi+ p \ln(2) - \arg(\Gamma(\ii p)))}}{\sqrt{\vartheta''(z_2(v);v)}}  (z_2(v) - z_1(v))^{2\ii p} \vartheta''(z_2(v);v)^{\ii p}.
\end{aligned}
\label{H-z2-expand}
\end{equation}
Recalling that $z_1(v)<z_2(v)$ and using $\mathfrak{b}/\mathfrak{a}=|b/a|$ in the definition \eqref{p-def}, we let phases $\phi_0(v)$, $\phi_{z_1}(v)$, and $\phi_{z_2}(v)$ independent of the large parameter $X$ be defined by \eqref{phi-0-X} and \eqref{phi-1-phi-2-X}.
%where we have recalled \eqref{p-def}.
%We substitute \eqref{H-z1-expand}--\eqref{H-z2-expand} in \eqref{Psi-large-X-approx-3}, and use \eqref{phi-0-X}--\eqref{phi-1-phi-2-X} in the resulting expressions to arrive at the following result.
Substituting \eqref{H-z1-expand}--\eqref{H-z2-expand} in \eqref{Psi-large-X-approx-3}, and using \eqref{phi-0-X}--\eqref{phi-1-phi-2-X} in the resulting expressions establishes \eqref{Psi-large-X-approx-5} in Theorem~\ref{t:large-X}.

\subsection{$L^2(\mathbb{R})$ norm of $\Psi(X,T;\mathbf{G})$}
We now prove Theorem~\ref{t:L2-norm}.
Since the $L^2(\mathbb{R})$ norm of a solution of \eqref{nls} is a conserved quantity, it suffices to compute $\| \Psi(\diamond,0;\mathbf{G})\|_{L^2(\mathbb{R})}$. We let $\mathbf{P}^{[1]}(X,T;\mathbf{G})$ and $\mathbf{P}^{[2]}(X,T;\mathbf{G})$ denote the coefficients in the asymptotic expansion
\begin{equation}
\mathbf{P}(\Lambda;X,T,\mathbf{G}) = \mathbb{I} + \mathbf{P}^{[1]}(X,T;\mathbf{G}) \Lambda^{-1} + \mathbf{P}^{[2]}(X,T;\mathbf{G}) \Lambda^{-2} + O(\Lambda^{-3}),\quad \Lambda\to \infty,
\end{equation}
of the unique solution $\mathbf{P}(\Lambda)=\mathbf{P}(\Lambda;X,T,\mathbf{G}) $ of \rhref{rhp:near-field}. A standard dressing calculation using the symmetry $\sigma_2 \mathbf{P}(\Lambda^*)^*\sigma_2 =\mathbf{P}(\Lambda)$ shows that
\begin{equation}
 \mathbf{P}^{[1]}(X,T;\mathbf{G}) = 
\frac{1}{2\ii} 
\begin{bmatrix}
\Phi(X,T;\mathbf{G}) & \Psi(X,T;\mathbf{G}) \\
\Psi(X,T;\mathbf{G})^* & - \Phi(X,T;\mathbf{G})
\end{bmatrix},
\end{equation}
where $\dfrac{\partial}{\partial X}\Phi(X,T;\mathbf{G}) = | \Psi(X,T;\mathbf{G})|^2$. Thus,
\begin{equation}
\begin{split}
\| \Psi(\diamond,0;\mathbf{G})\|_{L^2(\mathbb{R})}^2 &=-2\ii \int_{-\infty}^{+\infty} \frac{\partial}{\partial X} P_{22}^{[1]}(Y,0;\mathbf{G}) \dd Y\\
&=-2\ii \lim_{X\to +\infty} \left( P^{[1]}_{22}(X,0;\mathbf{G}) - P^{[1]}_{22}( - X,0;\mathbf{G})  \right).
\label{L2-norm-limit}
\end{split}
\end{equation}
Using Proposition~\ref{prop:a-b-scaling} we can assume that $\mathbf{G}=\mathbf{G}(\mathfrak{a},\mathfrak{b})$ depends on the normalized parameters defined from $(a,b)$ in \eqref{eq:normalized-ab}.
For $|\Lambda|>1$, Proposition~\ref{P-symmetry-X} implies that
\begin{equation}
\mathbf{P}(\Lambda;-X,0,\mathbf{G}(\mathfrak{a},\mathfrak{b})) = \sigma_3 \mathbf{P}(-\Lambda;X,0,\mathbf{G}(\mathfrak{b},\mathfrak{a})) \ee^{4\ii\Lambda^{-1}\sigma_3} \sigma_3, \quad |\Lambda|>1.
\end{equation}
Expanding the right-hand side as $\Lambda\to \infty$ yields the identity
\begin{equation}
P_{22}^{[1]}(-X,0;\mathbf{G}(\mathfrak{a},\mathfrak{b})) = - P_{22}^{[1]}(X,0;\mathbf{G}(\mathfrak{b},\mathfrak{a})) - 4\ii .
\end{equation}
Using this in \eqref{L2-norm-limit} gives
\begin{equation}
\| \Psi(\diamond,0;\mathbf{G})\|_{L^2(\mathbb{R})}^2 = 8 - 2\ii  \lim_{X\to +\infty} \left( P^{[1]}_{22}(X,0;\mathbf{G}(\mathfrak{a},\mathfrak{b})) + P^{[1]}_{22}( X,0;\mathbf{G}(\mathfrak{b},\mathfrak{a})) \right).
\label{L2-norm-limit-positive}
\end{equation}
Recalling that $\mathbf{S}(z;X,0,\mathbf{G}(\mathfrak{a},\mathfrak{b}))=\mathbf{T}(z;X,0,\mathbf{G}(\mathfrak{a},\mathfrak{b}))$ for $|z|$ sufficiently large together with \eqref{X:P-to-S} 
%and \eqref{T-to-W} 
we have
\begin{equation}
\mathbf{P}(X^{-\frac{1}{2}} z; X, 0, \mathbf{G}(\mathfrak{a},\mathfrak{b})) =   \mathbf{T}(z;X,0,\mathbf{G}(\mathfrak{a},\mathfrak{b})), \quad |z|\gg 1,
\end{equation}
which implies the identity
\begin{equation}
P^{[1]}_{22}(X,0;\mathbf{G}(\mathfrak{a},\mathfrak{b})) = X^{-\frac{1}{2}} T^{[1]}_{22}(X,0;\mathbf{G}(\mathfrak{a},\mathfrak{b})),
\label{P-to-W-identity}
\end{equation}
where $\mathbf{T}^{[1]}$ is the sub-leading coefficient matrix in the large-$z$ expansion of the matrix function $\mathbf{T}(z;X,0,\mathbf{G}(\mathfrak{a},\mathfrak{b}))$: $\mathbf{T}(z;X,0,\mathbf{G}(\mathfrak{a},\mathfrak{b})) = \mathbb{I} + \mathbf{T}^{[1]}(X,0,\mathbf{G}(\mathfrak{a},\mathfrak{b}))z^{-1} + O(z^{-2})$ as $z\to\infty$. The identity \eqref{P-to-W-identity} clearly holds regardless of the values of $\mathfrak{a},\mathfrak{b}$, and in particular when $\mathbf{G}(\mathfrak{a},\mathfrak{b})$ is replaced with $\mathbf{G}(\mathfrak{b},\mathfrak{a})$. On the other hand, from \eqref{W-in-F-Wout} we have $\mathbf{T}(z;X,0,\mathbf{G}) = \mathbf{F}(z;X,0,\mathbf{G}) \dot{\mathbf{T}}(z;X,0,\mathbf{G})$ and for $|z|$ large enough we have $\dot{\mathbf{T}}(z;X,0,\mathbf{G}) = \dot{\mathbf{T}}^{\mathrm{out}}(z;0,\mathbf{G})$, which is independent of $X$. Then, expanding \eqref{E-Plemelj-X} for $|z|$ large and using the definition \eqref{W-out-X} shows that $\mathbf{T}(z;X,0,\mathbf{G})$ has the expansion
\begin{multline}
\mathbf{T}(z;X,0,\mathbf{G}) = \mathbb{I} \\
{}+ \left[ \ii(z_2(0) - z_1(0) )p \sigma_3 
-\frac{1}{2\pi \ii}\int_{\Sigma_\mathbf{F}} \mathbf{F}_-(\zeta;X,0,\mathbf{G})\left( \mathbf{V}^{\mathbf{F}}(\zeta;X,0,\mathbf{G}) - \mathbb{I} \right) \dd \zeta \right]z^{-1}\\
{} + O(z^{-2}),\quad z\to \infty.
\end{multline}
Then we obtain
\begin{equation}
\begin{split}
P_{22}^{[1]}(X,0;\mathbf{G}) = X^{-\frac{1}{2}} &\left[-\ii(z_2(0) - z_1(0))p -\frac{1}{2\pi \ii}\int_{\Sigma_\mathbf{F}} F_{21-}(\zeta;X,0,\mathbf{G}) V^{\mathbf{F}}_{12}(\zeta;X,0,\mathbf{G}) \dd \zeta \right.\\
&\quad\left.  -\frac{1}{2\pi \ii}\int_{\Sigma_\mathbf{F}} F_{22-}(\zeta;X,0,\mathbf{G}) \left( V^{\mathbf{F}}_{22}(\zeta;X,0,\mathbf{G})  - 1 \right) \dd \zeta \right],
\end{split}
\end{equation}
which is exact. Now combining the estimates \eqref{error-jump-estimate} and \eqref{E-minus-estimate} shows that $P_{22}^{[1]}(X,0;\mathbf{G})=O(X^{-\frac{1}{2}})$ as $X\to+\infty$, and this fact holds regardless of the values of the normalized parameters $\mathfrak{a},\mathfrak{b}$ provided $\mathfrak{a}\neq 0$ (if $\mathfrak{a}=0$ then $p=\infty$); in particular it is true for $\mathbf{G}=\mathbf{G}(\mathfrak{a},\mathfrak{b})$ and $\mathbf{G}=\mathbf{G}(\mathfrak{b},\mathfrak{a})$ if $\mathfrak{ab}\neq 0$. Therefore, we have from \eqref{L2-norm-limit-positive} that
\begin{equation}
\| \Psi(\diamond,0;\mathbf{G})\|_{L^2(\mathbb{R})} = \sqrt{8}
\end{equation}
as long as $ab\neq 0$.

\section{Asymptotic behavior of $\Psi(X,T;\mathbf{G})$ for large $|T|$}
\label{s:large-T}
This section is devoted to proving Theorem~\ref{t:large-T}. To analyze $\Psi(X,T;\mathbf{G})$ for large $T>0$ and general $a,b\in\mathbb{C}$ with $ab\neq 0$, we can appeal to %Proposition~\ref{prop:T-symmetry} and assume that $T>0$, and we also use 
Proposition~\ref{prop:a-b-scaling}, which allows us to work with normalized parameters, replacing $\mathbf{G}=\mathbf{G}(a,b)$ with $\mathbf{G}(\mathfrak{a},\mathfrak{b})$ for which $\mathfrak{a},\mathfrak{b}>0$ with $\mathfrak{a}^2+\mathfrak{b}^2=1$, at the cost of including a phase factor $\ee^{-\ii\arg(ab)}$.
Hence, our aim is to generalize the analysis of \cite[Section 4.2]{BilmanLM2020} from the specific case $\mathfrak{a}=\mathfrak{b}=1/\sqrt{2}$ to allow for general normalized parameters; we will also compute a correction term not obtained in \cite{BilmanLM2020}.

We introduce a real parameter $w$ and set $X=w T^\frac{2}{3}$, and then rescale the spectral parameter $\Lambda$ by setting $\tz:=T^{\frac{1}{3}}\Lambda$. Then the phase conjugating the jump matrix in \eqref{P-jump} takes the form
\begin{equation}
\Lambda X+\Lambda^{2} T+2 \Lambda^{-1}=T^{\frac{1}{3}} \theta(\tz ; w), \quad \theta(\tz ; w):=w \tz+\tz^{2}+2 \tz^{-1}.
\label{large-T-phase}
\end{equation}
In analogy with Section~\ref{s:large-X}, taking the solution of \rhref{rhp:near-field} for $\mathbf{G}=\mathbf{G}(\mathfrak{a},\mathfrak{b})$ with $\mathfrak{a},\mathfrak{b}>0$ and $\mathfrak{a}^2+\mathfrak{b}^2=1$, for brevity we omit $\mathbf{G}$ from the argument lists since $\mathfrak{a}$ and $\mathfrak{b}$ are fixed in this section.  
Thus, setting 
\begin{equation}
\mathbf{S}(\tz;T,w):=\mathbf{P}(T^{-\frac{1}{3}}\tz; T^{\frac{2}{3}}w,T),\quad T>0,
\label{T:P-to-S}
\end{equation}
from \eqref{Psi-def} and Proposition~\ref{prop:a-b-scaling} we get
\begin{equation}
\Psi(T^\frac{2}{3} w , T)=2 \ii  \ee^{-\ii\arg(ab)}T^{-\frac{1}{3}}\lim _{\tz \rightarrow \infty} \tz S_{12}(\tz ; T, w).
\label{eq:T-Psi-from-S}
\end{equation}
The matrix $\mathbf{S}(\tz ; T, w)$ is normalized to satisfy $\mathbf{S}(\tz ; T, w) \rightarrow \mathbb{I}$ as $\tz \rightarrow \infty$ for each $T>0$ and $\mathbf{S}(\tz ; T, w)$ is analytic in the complement of an arbitrary Jordan curve $\Gamma$ surrounding $\tz=0$ in the clockwise sense. The jump condition satisfied by $\mathbf{S}(\tz ; T, w)$ across $\Gamma$ is 
\begin{equation}
\mathbf{S}_{+}(\tz ; T, w)=\mathbf{S}_{-}(\tz ; T, w) \ee^{-\ii  T^{1/3} \theta(\tz ; w) \sigma_{3}} \mathbf{G}(\mathfrak{a},\mathfrak{b}) \ee^{\ii  T^{1/3} \theta(\tz ; w) \sigma_{3}}, \quad \tz \in \Gamma.
\end{equation}

\subsection{Spectral curve and $g$-function}
Since the phase conjugating the jump matrix for $\mathbf{P}(\Lambda;X,T,\mathbf{G})$ in \rhref{rhp:near-field} does not involve $\mathbf{G}$ in any way, 
all of the analysis in \cite[Section 4.2]{BilmanLM2020} regarding the spectral curve and construction of the relevant $g$-function\footnote{In \cite[Section 4.2]{BilmanLM2020} this analysis is presented under the assumption that $w\ge 0$, but all of the results also hold for $w<0$ provided that $|w|<54^\frac{1}{3}$.} goes through without any modification. We summarize that analysis here.

We assume that $|w|<54^{\frac{1}{3}}$, and  we 
recall the $g$-function $g(\tz;w)$ defined in \cite[Section 4.2]{BilmanLM2020}, which is bounded and analytic in $\tz$ for $\tz\in \mathbb{C} \setminus \Sigma$, where $\Sigma$ is a Schwarz-symmetrical arc determined below. It also satisfies $g(\tz;w)\to 0 $ as $\tz\to \infty$, and the boundary values taken by $g(\tz;w)$ on $\Sigma$ satisfy the jump condition
\begin{equation}
g_+(\tz;w) + g_-(\tz;w) + 2 \theta(\tz;w) = \kappa(w),\quad \tz\in \Sigma,
\label{T-g-theta-kappa}
\end{equation}
where $\kappa(w)$ is a constant whose explicit value was found in \cite[Eqn.\@ (218)]{BilmanLM2020} to be
\begin{equation}
\kappa(w)=-108^{\frac{1}{3}}-\frac{1}{3} w^{2},
\label{kappa}
\end{equation}
which can also be written in the form \eqref{eq:kappa-intro}.
The derivative $g'(\tz;w)$ of $g(\tz;w)$ with respect to $\tz$ satisfies the relation 
\begin{equation}
\begin{split}
\left(g'(\tz;w) + \theta'(\tz;w)\right)^2 &= 4\tz^{-4}(\tz-\tz_1(w))^2(\tz-\tz_2(w))^2 (\tz-\tz_0(w))(\tz-\tz_0(w)^*)\\ &=:\tz^{-4} P(\tz;w),
\end{split}
\label{h-prime-squared}
\end{equation}
which is the \emph{spectral curve} for the problem at hand, see \cite[Eqn. (183)]{BilmanLM2020}. The double roots of the sextic polynomial $P(\diamond;w)$ are the real values $\tz_1(w)<0$ and $\tz_2(w)>0$ defined in \eqref{eq:Z1Z2-intro},
%\begin{equation}
%\begin{aligned}
%&\tz_1(w):=\frac{1}{2}\left(-\frac{1}{6} w-\sqrt{\frac{w^{2}}{36}+2^{\frac{5}{3}}}\right)<0 \\
%&\tz_2(w):=\frac{1}{2}\left(-\frac{1}{6} w+\sqrt{\frac{w^{2}}{36}+2^{\frac{5}{3}}}\right)>0 .
%\end{aligned}
%\end{equation}
and there is also a complex-conjugate pair of simple roots $\tz_0(w)$, $\tz_0(w)^*$ given explicitly by \eqref{eq:Z0-intro} assuming that $|w|<w_\mathrm{c}=54^\frac{1}{3}$.
%\begin{equation}
%\tz_{0}(w):=\frac{1}{3}\left(-w+\ii  \sqrt{54^{\frac{2}{3}}-w^{2}}\right), \quad |w|<54^\frac{1}{3},\label{z-0}
%\end{equation}
See \cite[Eqn.\@ (190) and Eqn.\@ (191)]{BilmanLM2020}. The cut $\Sigma$ for the $g$-function connects the conjugate pair of simple roots $\tz_0(w)$ and $\tz_0(w)^*$ of $P(\diamond;w)$, and is chosen to cross the real axis at the negative value $\tz=\tz_1(w)$. The function $g(\tz;w)$ is given explicitly in \cite[Eqn.\@ (195)]{BilmanLM2020} by
\begin{equation}
g(\tz;w) = \frac{R(\tz ; w)^{3}}{\tz}-\theta(\tz ; w)-3 \cdot 2^{-\frac{1}{3}}-\frac{1}{6} w^{2},
\label{g-def-large-T}
\end{equation}
where $R(\tz;w)$ is the function analytic for $\tz\in \mathbb{C}\setminus \Sigma$ uniquely determined by the conditions 
\begin{equation}
R(\tz;w)^2 = (\tz-\tz_0(w))(\tz-\tz_0(w)^*)\quad\text{and}\quad R(\tz;w)=\tz+O(1),\; \tz\to\infty.
\label{eq:Rdef}
\end{equation}
We define 
\begin{equation}
h(\tz;w):= g(\tz;w) + \theta(\tz;w)
\label{T-h-g-theta}
\end{equation}
as in \cite[Eqn.\@ (196)]{BilmanLM2020}, and take the jump contour $\Gamma$ for $\mathbf{S}(\tz;T,w)$ so that $\Im(h(\tz;w))=0$ holds for $\tz\in\Gamma$.
\begin{figure}
\includegraphics[width=0.24\textwidth]{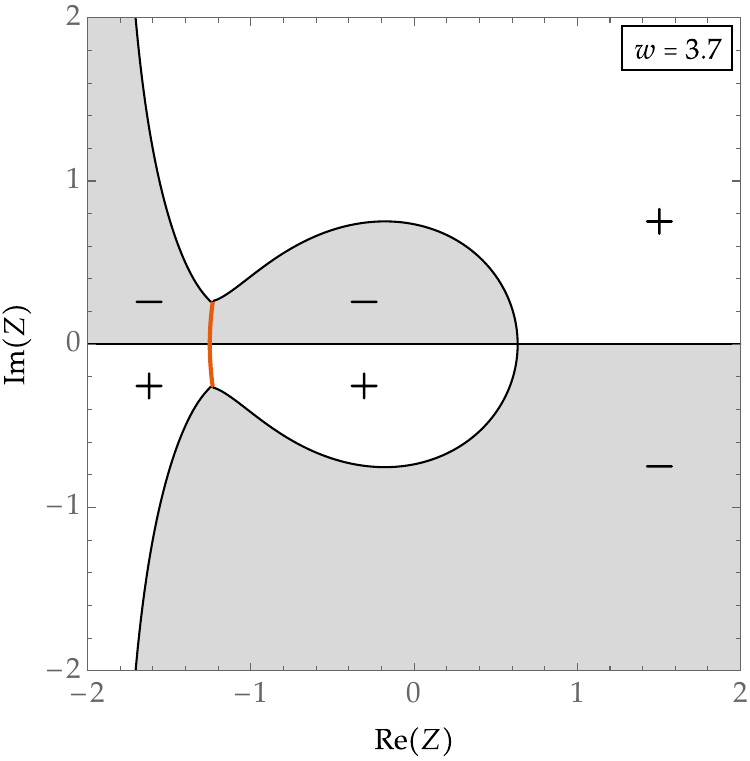}
\includegraphics[width=0.24\textwidth]{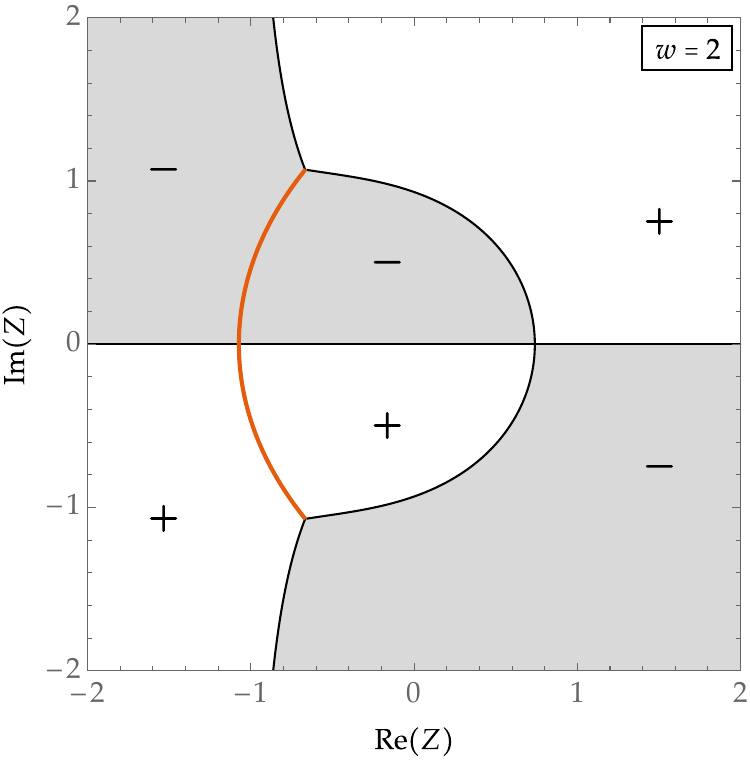}
\includegraphics[width=0.24\textwidth]{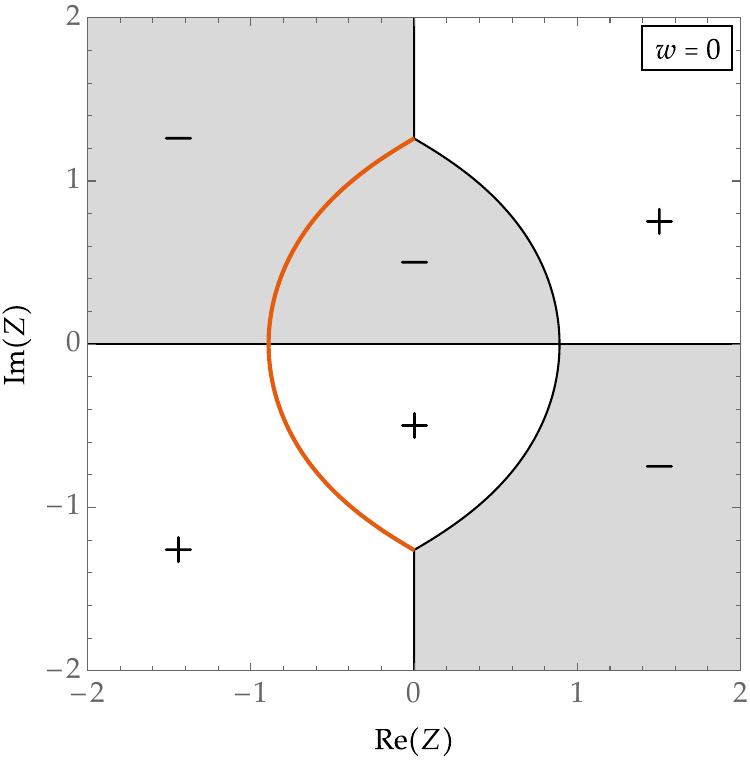}
\includegraphics[width=0.24\textwidth]{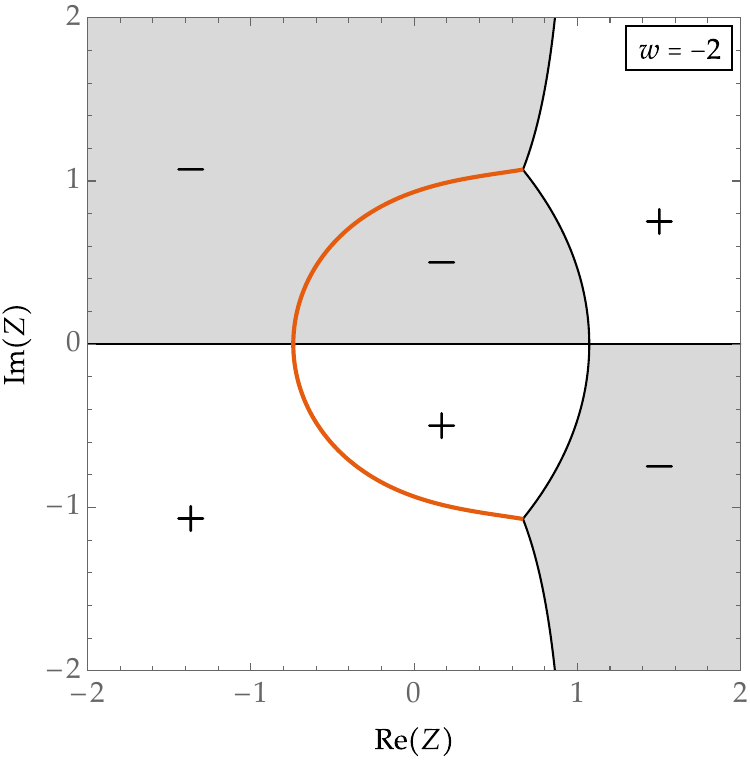}
\caption{The sign charts of $\Im(h(\tz;w))$ for $w$ in the range $|w|<w_{\mathrm{c}}$.}
\label{fig:largeT-signs}
\end{figure}
It consists of four oriented arcs $\Gamma^{\pm}$ and $\Sigma^{\pm}$ as shown in the left-hand panel of Figure~\ref{fig:largeT-contour},
and $\Im(h(\tz;w))$ is continuous across $\Sigma=\Sigma^+ \cup \Sigma^-$, vanishing there but not changing sign. See \cite[Section 4.2.1]{BilmanLM2020} for the construction of $g$ and determination of $\Sigma$ and $\Gamma$ in full detail.
See Figure~\ref{fig:largeT-signs} for the sign chart of $\Im(h(\tz;w))$ and how it varies with $w$.
We define the domains $L_\Gamma^{\pm}, R_\Gamma^{\pm}$,$L_\Sigma^{\pm}, R_\Sigma^{\pm}$, $\Omega^{\pm}$ exactly as in 
the left-hand panel of Figure~\ref{fig:largeT-contour}.
\begin{figure}
\includegraphics[width=0.45\textwidth]{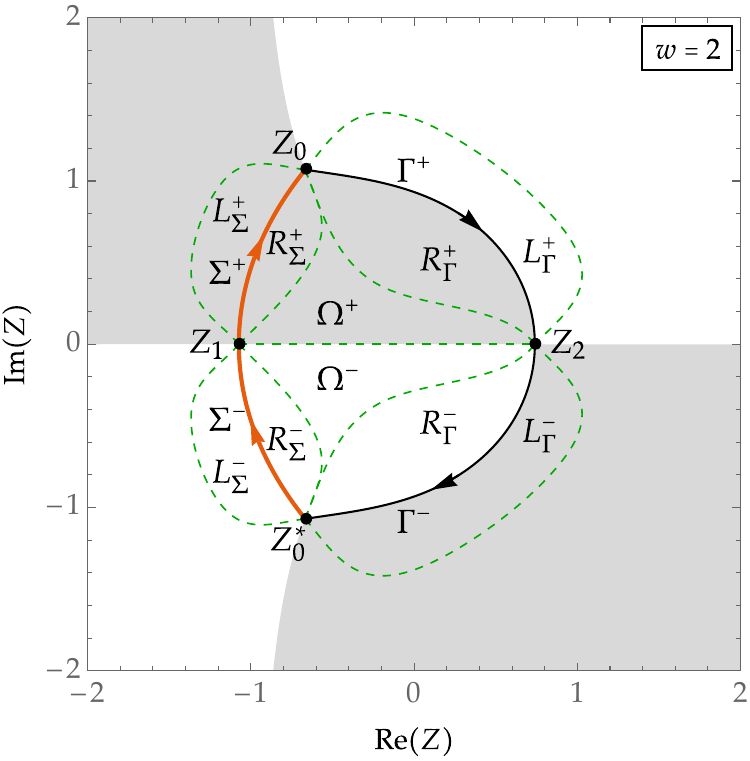}\,
\includegraphics[width=0.45\textwidth]{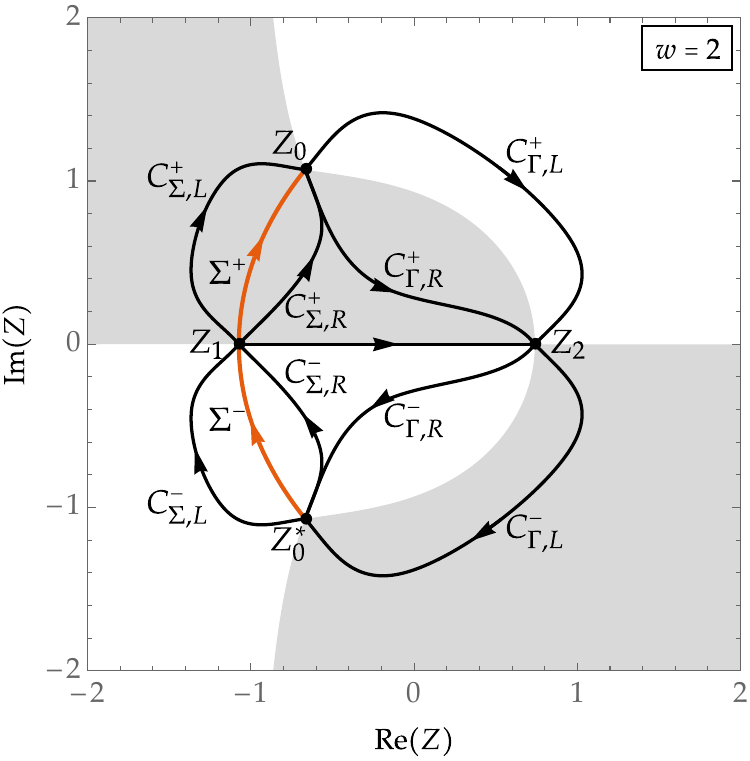}
\caption{Left: the jump contour $\Gamma=\Gamma^{+} \cup \Gamma^{-} \cup \Sigma^{+} \cup \Sigma^{-}$ for $\mathbf{S}$ and the regions $L_{\Gamma}^{ \pm}, L_{\Sigma}^{ \pm}, R_{\Gamma}^{ \pm}, R_{\Sigma}^{ \pm}$, and $\Omega^{ \pm}$. Right: the jump contour for $\mathbf{T}$.}
\label{fig:largeT-contour}
\end{figure}

\subsection{Steepest-descent deformation}
We make use of all of the factorizations in \eqref{Gnorm-LDU}--\eqref{Gnorm-UTU} and introduce the $g$-function by making the following substitutions.
\begin{equation}
\mathbf{T}(\tz;T,w):=\mathbf{S}(\tz;T,w) 
\begin{bmatrix}1 & 0 \\ \dfrac{\mathfrak{b}}{\mathfrak{a}}\ee^{2\ii T^{1/3}\theta(\tz;w)} & 1 \end{bmatrix} \ee^{\ii T^{1/3} g(\tz;w)\sigma_3},\quad \tz\in L^+_\Gamma,
\label{T-from-S-L-plus-Gamma}
\end{equation}
\begin{equation}
\mathbf{T}(\tz;T,w):=\mathbf{S}(\tz;T,w) 
\mathfrak{a}^{-\sigma_3} \begin{bmatrix}1 &\mathfrak{ab} \ee^{-2\ii T^{1/3}\theta(\tz;w)} \\ 0 & 1 \end{bmatrix}\ee^{\ii T^{1/3} g(\tz;w)\sigma_3},\quad \tz\in R^+_\Gamma,
\end{equation}
\begin{equation}
\mathbf{T}(\tz;T,w):=\mathbf{S}(\tz;T,w) 
\mathfrak{a}^{-\sigma_3} \ee^{\ii T^{1/3} g(\tz;w)\sigma_3},\quad \tz\in \Omega^+,
\end{equation}
\begin{equation}
\mathbf{T}(\tz;T,w):=\mathbf{S}(\tz;T,w) 
\mathfrak{a}^{\sigma_3} \ee^{\ii T^{1/3} g(\tz;w)\sigma_3},\quad \tz\in \Omega^-,
\end{equation}
\begin{equation}
\mathbf{T}(\tz;T,w):=\mathbf{S}(\tz;T,w) 
\mathfrak{a}^{\sigma_3}\begin{bmatrix}1 & 0 \\ -\mathfrak{a b} \ee^{2\ii T^{1/3}\theta(\tz;w)} & 1 \end{bmatrix} \ee^{\ii T^{1/3} g(\tz;w)\sigma_3},\quad \tz\in R^-_\Gamma,
\end{equation}
\begin{equation}
\mathbf{T}(\tz;T,w):=\mathbf{S}(\tz;T,w) 
\begin{bmatrix}1 &  -\dfrac{\mathfrak{b}}{\mathfrak{a}} \ee^{-2\ii T^{1/3}\theta(\tz;w)} \\ 0 & 1 \end{bmatrix} \ee^{\ii T^{1/3} g(\tz;w)\sigma_3},\quad \tz\in L^-_\Gamma,
\end{equation}
\begin{equation}
\mathbf{T}(\tz;T,w):=\mathbf{S}(\tz;T,w) 
\mathfrak{a}^{-\sigma_3}\begin{bmatrix}1 & -\dfrac{\mathfrak{a}^3}{\mathfrak{b}} \ee^{-2\ii T^{1/3}\theta(\tz;w)} \\ 0 & 1 \end{bmatrix}\ee^{\ii T^{1/3} g(\tz;w)\sigma_3},\quad \tz\in R^+_\Sigma,
\end{equation}
\begin{equation}
\mathbf{T}(\tz;T,w):=\mathbf{S}(\tz;T,w) 
\begin{bmatrix}1 & \dfrac{\mathfrak{a}}{\mathfrak{b}}\ee^{-2\ii T^{1/3}\theta(\tz;w)} \\ 0 & 1 \end{bmatrix}\ee^{\ii T^{1/3} g(\tz;w)\sigma_3},\quad \tz\in L^+_\Sigma,
\end{equation}
\begin{equation}
\mathbf{T}(\tz;T,w):=\mathbf{S}(\tz;T,w) 
\mathfrak{a}^{\sigma_3} \begin{bmatrix}1 & 0 \\  \dfrac{\mathfrak{a}^3}{\mathfrak{b}} \ee^{2\ii T^{1/3}\theta(\tz;w)} & 1 \end{bmatrix}  \ee^{\ii T^{1/3} g(\tz;w)\sigma_3},\quad \tz\in R^-_\Sigma,
\end{equation}
\begin{equation}
\mathbf{T}(\tz;T,w):=\mathbf{S}(\tz;T,w) 
\begin{bmatrix}1 & 0 \\ -\dfrac{\mathfrak{a}}{\mathfrak{b}}\ee^{2\ii T^{1/3}\theta(\tz;w)} & 1 \end{bmatrix}  \ee^{\ii T^{1/3} g(\tz;w)\sigma_3},\quad \tz\in L^-_\Sigma,
\label{T-from-S-L-minus-Sigma}
\end{equation}
and we set $\mathbf{T}(\tz;T,w):=\mathbf{S}(\tz;T,w) \ee^{\ii T^{1/3}g(\tz;w)\sigma_3}$ elsewhere. The jump contours for the jump conditions satisfied by $\mathbf{T}(\tz;T,w)$ are illustrated in 
%\textcolor{red}{\cite[Figure 11 (right-hand panel)]{BilmanLM2020}}, 
the right-hand panel of Figure~\ref{fig:largeT-contour},
and the jump conditions are the following.
\begin{equation}
\mathbf{T}_+(\tz;T,w) = \mathbf{T}_-(\tz;T,w)
\begin{bmatrix}1 & 0 \\-\dfrac{\mathfrak{b}}{\mathfrak{a}}\ee^{2\ii T^{1/3}h(\tz;w)} & 1 \end{bmatrix},\quad \tz\in C^+_{\Gamma,L},
\end{equation}
\begin{equation}
\mathbf{T}_+(\tz;T,w) = \mathbf{T}_-(\tz;T,w)
\begin{bmatrix}1 &\mathfrak{a b} \ee^{-2\ii T^{1/3}h(\tz;w)} \\ 0 & 1 \end{bmatrix},\quad \tz\in  C^+_{\Gamma,R},
\end{equation}
\begin{equation}
\mathbf{T}_+(\tz;T,w) = \mathbf{T}_-(\tz;T,w)\mathfrak{a}^{-2\sigma_3},\quad \tz\in  I,
\label{jump-T-I}
\end{equation}
\begin{equation}
\mathbf{T}_+(\tz;T,w) = \mathbf{T}_-(\tz;T,w)
\begin{bmatrix}1 & 0 \\ -\mathfrak{a b} \ee^{2\ii T^{1/3}h(\tz;w)} & 1 \end{bmatrix},\quad \tz\in  C^-_{\Gamma,R},
\end{equation}
\begin{equation}
\mathbf{T}_+(\tz;T,w) = \mathbf{T}_-(\tz;T,w)
\begin{bmatrix}1 &  \dfrac{\mathfrak{b}}{\mathfrak{a}} \ee^{-2\ii T^{1/3}h(\tz;w)} \\ 0 & 1 \end{bmatrix},\quad \tz\in  C^-_{\Gamma,L},
\end{equation}
\begin{equation}
\mathbf{T}_+(\tz;T,w) = \mathbf{T}_-(\tz;T,w)
\begin{bmatrix}1 & -\dfrac{\mathfrak{a}}{\mathfrak{b}}\ee^{-2\ii T^{1/3}h(\tz;w)} \\ 0 & 1 \end{bmatrix},\quad \tz\in  C^+_{\Sigma,L},
\end{equation}
\begin{equation}
\mathbf{T}_+(\tz;T,w) = \mathbf{T}_-(\tz;T,w)
\begin{bmatrix}1 & -\dfrac{\mathfrak{a}^3}{\mathfrak{b}}  \ee^{-2\ii T^{1/3}h(\tz;w)} \\ 0 & 1 \end{bmatrix},\quad \tz\in  C^+_{\Sigma,R},
\end{equation}
\begin{equation}
\mathbf{T}_+(\tz;T,w) = \mathbf{T}_-(\tz;T,w)
\begin{bmatrix}1 & 0 \\  \dfrac{\mathfrak{a}^3}{\mathfrak{b}} \ee^{2\ii T^{1/3}h(\tz;w)} & 1 \end{bmatrix},\quad \tz\in  C^-_{\Sigma,R},
\end{equation}
\begin{equation}
\mathbf{T}_+(\tz;T,w) = \mathbf{T}_-(\tz;T,w)
\begin{bmatrix}1 & 0 \\ \dfrac{\mathfrak{a}}{\mathfrak{b}}\ee^{2\ii T^{1/3}h(\tz;w)} & 1 \end{bmatrix},\quad \tz\in  C^-_{\Sigma,L}.
\end{equation}
Finally, on $\Sigma=\Sigma^+\cup\Sigma^-$ we have
\begin{alignat}{2}
\mathbf{T}_+(\tz;T,w) &= \mathbf{T}_-(\tz;T,w)
\begin{bmatrix} 0 & \dfrac{\mathfrak{a}}{\mathfrak{b}} \ee^{-\ii T^{1/3}\kappa(w)} \\ - \dfrac{\mathfrak{b}}{\mathfrak{a}} \ee^{\ii T^{1/3}\kappa(w)}& 0\end{bmatrix},&&\quad \tz\in\Sigma^+,\label{jump-T-Sigma-p}\\
\mathbf{T}_+(\tz;T,w) &= \mathbf{T}_-(\tz;T,w)
\begin{bmatrix} 0 &\dfrac{\mathfrak{b}}{\mathfrak{a}} \ee^{-\ii T^{1/3}\kappa(w)} \\ - \dfrac{\mathfrak{a}}{\mathfrak{b}}  \ee^{\ii T^{1/3}\kappa(w)} & 0\end{bmatrix},&&\quad \tz\in\Sigma^-.\label{jump-T-Sigma-m}
\end{alignat}
Since $\mathbf{T}(\tz;T,w)=\mathbf{S}(\tz;T,w)\ee^{\ii T^{1/3}g(\tz;w)\sigma_3}$ for large $\tz$, from \eqref{eq:T-Psi-from-S} we get
\begin{equation}
\begin{split}
\Psi(T^\frac{2}{3} w , T)&=2 \ii  \ee^{-\ii\arg(ab)}T^{-\frac{1}{3}}\lim _{\tz \rightarrow \infty} \tz T_{12}(\tz ; T, w)\ee^{\ii T^{1/3}g(\tz;w)}\\ &=2 \ii  \ee^{-\ii\arg(ab)}T^{-\frac{1}{3}}\lim _{\tz \rightarrow \infty} \tz T_{12}(\tz ; T, w),
\end{split}
\label{eq:T-Psi-from-T}
\end{equation}
where the second equality comes from the fact that $g(\tz;w)\to 0$ as $\tz\to\infty$.

\subsection{Parametrix construction}
\subsubsection{Outer parametrix} 
We construct a parametrix $\dot{\mathbf{T}}^{\mathrm{out}}(\tz;T,w)$ with the following properties:
\begin{itemize}
\item $\dot{\mathbf{T}}^{\mathrm{out}}(\tz;T,w)$ is analytic in $\tz$ for $\tz\in\mathbb{C}\setminus(\Sigma \cup I)$.
\item $\dot{\mathbf{T}}^{\mathrm{out}}(\tz;T,w) \to \mathbb{I}$ as $\tz\to\infty$. 
\item $\dot{\mathbf{T}}^{\mathrm{out}}(\tz;T,w)$ satisfies exactly the jump conditions \eqref{jump-T-I} on $I$ and \eqref{jump-T-Sigma-p}--\eqref{jump-T-Sigma-m} on $\Sigma^+$ and $\Sigma^-$.
\end{itemize}
These conditions are not sufficient to determine $\dot{\mathbf{T}}^{\mathrm{out}}$ uniquely. Nevertheless, we will just construct a particular function satisfying the properties listed above. 
We first simplify the jump matrices defined on $\Sigma^{+}$ and $\Sigma^{-}$ by introducing a Szeg\H{o} function $\newf(\tz;w)$ analytic in $\tz$ for $\tz\in \mathbb{C}\setminus \Sigma$ given by 
\begin{equation}
\newf(\tz;w) :=  \frac{1}{2}R(\tz;w) \left( 
\int_{\Sigma^{-}} \frac{\dd \zeta}{R_{+}(\zeta;w)(\zeta-\tz)}-\int_{\Sigma^{+}} \frac{\dd \zeta}{R_{+}(\zeta;w)(\zeta-\tz)}
\right).
\label{f-def-large-T}
\end{equation}
We can find an antiderivative of $1/(R(\zeta;w)(\zeta-Z))$ as follows.
We rationally parametrize the genus-zero curve $R^2=(\zeta-Z_0(w))(\zeta-Z_0(w)^*)$ by a coordinate $t\in\mathbb{C}$ by setting
\begin{equation}
\zeta=\zeta(t):=U -\frac{2V^2t}{t^2-V^2},\quad R(\zeta;w)=r(t):=V\frac{t^2+V^2}{t^2-V^2},\quad Z_0(w)=U+\ii V.
\end{equation}
Note that $U,V$ depend on $w$, but we suppress this dependence below for brevity.
Then, 
\begin{equation}
\frac{\dd\zeta}{R(\zeta;w)(\zeta-Z)}=\frac{1}{r(t)(\zeta(t)-Z)}\frac{\dd \zeta}{\dd t}(t)\,\dd t=-\frac{2V}{Z-U}\frac{\dd t}{t^2+(2V^2/(Z-U))t-V^2}.
\end{equation}
The roots of the denominator are $t=(-V^2\pm V\sqrt{(Z-U)^2+V^2})/(Z-U)$, which can be written in the form $t=(-V^2\pm VR(Z;w))/(Z-U)$.  Therefore, by partial fraction expansion,
\begin{equation}
\frac{\dd\zeta}{R(\zeta;w)(\zeta-Z)}=
\frac{\dd t}{R(Z;w)}\left[\frac{1}{t-(-V^2-VR(Z;w))/(Z-U)}-\frac{1}{t-(-V^2+VR(Z;w))/(Z-U)}\right].
\end{equation}
It follows that an antiderivative is
\begin{equation}
\int\frac{\dd\zeta}{R(\zeta;w)(\zeta-Z)}=\frac{1}{R(Z;w)}\log\left(\frac{t-(-V^2-VR(Z;w))/(Z-U)}{t-(-V^2+VR(Z;w))/(Z-U)}\right).
\end{equation}
Then, we can explicitly express $t$ in terms of $\zeta$ and $R(\zeta;w)$ as
\begin{equation}
t=-V\frac{\zeta-U}{R(\zeta;w)-V}, 
\end{equation}
so
\begin{equation}
\int\frac{\dd\zeta}{R(\zeta;w)(\zeta-Z)}=\frac{1}{R(Z;w)}\log\left(\frac{(Z-U)(\zeta-U)+(V+R(Z;w))(V-R(\zeta;w))}{(Z-U)(\zeta-U)+(V-R(Z;w))(V-R(\zeta;w))}\right).
\label{eq:antiderivative-Cauchy}
\end{equation}
One can check that if $\tz$ is real and less than $\tz_1(w)$, then taking the principal branch of the logarithm yields a function of $\zeta$ that is analytic except on $\Sigma$ and the real interval $\tz\le\zeta\le \tz_1(w)$, and $\zeta=\tz$ is a simple root of the numerator of the argument of the logarithm only.  Therefore, with this choice of branch, the antiderivative is suitable for evaluation of $\newf(\tz;w)$ with $\tz<\tz_1(w)$ provided one takes the imaginary part of the logarithm to be $\pm\pi$ at the endpoint $\zeta=\tz_1(w)$ of $\Sigma^{\mp}$.  It follows that 
\begin{multline}
\newf(\tz;w)=\frac{1}{2}\log\left(\frac{(V-R(\tz;w))^2+(\tz-U)^2}{(V+R(\tz;w))^2+(\tz-U)^2}\right) \\
{}+\log\left(-\frac{(\tz-U)(\tz_1(w)-U)+(V+R(\tz;w))(V-R_+(\tz_1(w);w))}{(\tz-U)(\tz_1(w)-U)+(V-R(\tz;w))(V-R_+(\tz_1(w);w))}\right).
\end{multline}
The term on the first line comes from the contributions at $\tz=\tz_0(w)=U+\ii V$ and $\tz=\tz_0(w)^*=U-\ii V$, and it is an analytic function of $\tz<\tz_1(w)$ that admits continuation to a neighborhood of $\tz_1(w)$; to evaluate it at $\tz=\tz_1(w)$ we need only replace $\tz$ with $\tz_1(w)$ and $R(\tz;w)$ with $R_+(\tz_1(w);w)$.  The term on the second line comes from the contributions at $\tz=\tz_1(w)$, and the numerator of the argument of the logarithm has a simple root at $\tz=\tz_1(w)$ and produces a branch cut in the $\tz$-plane emanating from $\tz=\tz_1(w)$ to the right.  To reveal the simple root, we may use the identity $(\tz_1(w)-U)^2+V^2-R_+(\tz_1(w);w)^2=0$ to write
\begin{equation}
-\frac{(\tz-U)(\tz_1(w)-U)+(V+R(\tz;w))(V-R_+(\tz_1(w);w))}{(\tz-U)(\tz_1(w)-U)+(V-R(\tz;w))(V-R_+(\tz_1(w);w))}= -\omega^\newf(\tz;w)(\tz-\tz_1(w)),
\end{equation}
where
\begin{equation}
\omega^\newf(\tz;w):=\frac{\displaystyle \tz_1(w)-U+\frac{R(\tz;w)-R_+(\tz_1(w);w)}{\tz-\tz_1(w)}(V-R_+(\tz_1(w);w))}{(\tz-U)(\tz_1(w)-U)+(V-R(\tz;w))(V-R_+(\tz_1(w);w))}
\end{equation}
is a function analytic and non-vanishing at $\tz=\tz_1(w)$ and $R(\tz;w)=R_+(\tz_1(w);w)$ (i.e., admitting analytic continuation through $\Sigma$ at $\tz=\tz_1(w)$ from the left) with 
\begin{equation}
\omega^\newf_+(\tz_1(w);w)=\frac{(\tz_1(w)-U)V}{2((\tz_1(w)-U)^2+V^2)(R_+(\tz_1(w);w)-V)}
\label{eq:omega-at-z1}
\end{equation}
which one can verify is a positive number by the definitions of $\tz_1(w)$ and $\tz_0(w)=U+\ii V$.
We may therefore write $\newf(\tz;w)$ in the form 
\begin{equation}
\newf(\tz;w)=\frac{1}{2}\log\left(\frac{(V-R(\tz;w))^2+(\tz-U)^2}{(V+R(\tz;w))^2+(\tz-U)^2}\cdot\omega^\newf(\tz;w)^2\right)+\log(\tz_1(w)-\tz),
\label{eq:f-explicit}
\end{equation}
with only the second term not being analytic at $\tz=\tz_1(w)$.

Similarly, if $\tz$ is a real number greater than $\tz_1(w)$, one can check that the antiderivative in \eqref{eq:antiderivative-Cauchy} is analytic except for $\zeta\in\Sigma$ and on the half-lines $\zeta<\tz_1(w)$ and $\zeta>\tz$, and that $\zeta=\tz$ is a simple root of the numerator only in the argument of the (principal branch) logarithm. In particular, the antiderivative is continuous along the minus side of $\Sigma$, and therefore for such $\tz$, 
\begin{equation}
\begin{split}
\newf(\tz;w)&=\frac{1}{2}R(\tz;w)\left(\int_{\Sigma^+}\frac{\dd\zeta}{R_-(\zeta;w)(\zeta-\tz)}-\int_{\Sigma^-}\frac{\dd\zeta}{R_-(\zeta;w)(\zeta-\tz)}\right) \\
&=\frac{1}{2}\log\left(\frac{(V+R(\tz;w))^2+(\tz-U)^2}{(V-R(\tz;w))^2+(\tz-U)^2}\right)\\
&\quad\quad{}+\log\left(\frac{(\tz-U)(\tz_1(w)-U)+(V-R(\tz;w))(V-R_-(\tz_1(w);w))}{(\tz-U)(\tz_1(w)-U)+(V+R(\tz;w))(V-R_-(\tz_1(w);w))}\right).
\end{split}
\end{equation}
This is analytic at $\tz>\tz_1(w)$ and hence can be evaluated in particular at $\tz=\tz_2(w)$ simply by replacing $\tz$ with $\tz_2(w)$ and $R(\tz;w)$ with $R(\tz_2(w);w)$.  All arguments of logarithms are then positive and we interpret $\log(\diamond)$ as $\ln(\diamond)$.

It is easy to verify that 
\begin{equation}
\newf_+(\tz;w) + \newf_-(\tz;w) = \mp \ii\pi,\quad \tz\in\Sigma^\pm.
\label{T-f-jump}
\end{equation}
Although the sum of the boundary values is constant along each arc $\Sigma^\pm$, both boundary values exhibit a logarithmic singularity as $\tz\to \tz_1(w)$ and are otherwise continuous.
We note that $\newf(\tz;w)$ does not vanish in the limit $\tz\to \infty$; indeed,
expanding \eqref{f-def-large-T} shows that
\begin{equation}
\newf(\tz;w) =  \gamma(w) + O(\tz^{-1}),\quad \tz\to \infty,
\label{f-normalization}
\end{equation}
where
\begin{equation}
\gamma(w):=\frac{1}{2} 
\left( 
\int_{\Sigma^+} \frac{\dd \zeta}{R_+(\zeta;w)} - \int_{\Sigma^-} \frac{\dd \zeta}{R_+(\zeta;w)}
\right)\in \mathbb{R}.
\label{gamma}
\end{equation}
As pointed out in \cite{BilmanLM2020}, it is straightforward to verify that an antiderivative of $1/R(\zeta;w)$ is 
\begin{equation}
\int\frac{\dd\zeta}{R(\zeta;w)}=%\log\left(\frac{\zeta-u+R(\zeta;w)-v}{\zeta-u-R(\zeta;w)+v}\right).
\log\left(\zeta-U+R(\zeta;w)\right).
\label{eq:antiderivative-root}
\end{equation}
Taking the principal branch of the logarithm gives a function of $\zeta$ that is analytic on the complement of $\Sigma$ and the ray $(-\infty,\tz_1(w)]$.  This is continuous along the $-$ side of $\Sigma$ and hence can be used to evaluate $\gamma(w)$:  
\begin{equation}
\gamma(w)=\frac{1}{2}\left(\int_{\Sigma^-}\frac{\dd\zeta}{R_-(\zeta;w)}-\int_{\Sigma^+}\frac{\dd\zeta}{R_-(\zeta;w)}\right)=
%-\frac{1}{\pi}\ln\left(-\frac{\tz_1(w)-u+R_+(\tz_1(w);w)-v}{\tz_1(w)-u-R_+(\tz_1(w);w)+v}\right).
\ln\left(\frac{1}{V}\left(\tz_1(w)-U+R_-(\tz_1(w);w)\right)\right).
\end{equation}
One can check that the argument of the logarithm is positive over the whole range $|w|<w_\mathrm{c}$, which is why we use the notation $\ln(\diamond)$ instead of the complex logarithm $\log(\diamond)$.
We use $\newf(\tz;w)$ to define a new unknown $\mathbf{J}(\tz;T,w)$ by
\begin{equation}
%\dot{\mathbf{T}}^{\mathrm{out}}(\tz; T, w)= \ee^{\ii q \gamma(w)\sigma_3} \mathbf{J}(\tz;T,w) \ee^{-qf(\tz;w)\sigma_3},
\dot{\mathbf{T}}^{\mathrm{out}}(\tz; T, w)= \ee^{\ii q \gamma(w)\sigma_3} \mathbf{J}(\tz;T,w) \ee^{-\ii q\newf(\tz;w)\sigma_3},
\label{eq:T-J}
\end{equation}
where 
\begin{equation}
q:=\frac{1}{\pi}\ln\left(\frac{\mathfrak{a}}{\mathfrak{b}}\right)\in\mathbb{R}.
\label{T-q}
\end{equation}
$\mathbf{J}(\tz;T,w)$ has all the properties of $\dot{\mathbf{T}}^{\mathrm{out}}(\tz; T, w)$ except that it satisfies a simpler jump condition across $\Sigma=\Sigma^+\cup \Sigma^-$, given by
\begin{equation}
\mathbf{J}_+(\tz;T,w) = \mathbf{J}_-(\tz;T,w)\begin{bmatrix} 0 &  \ee^{-\ii T^{1/3}\kappa(w)} \\ -  \ee^{\ii T^{1/3}\kappa(w)} & 0 \end{bmatrix},\quad \tz\in \Sigma.
\label{J-jump-simple}
\end{equation}
To satisfy the jump condition \eqref{jump-T-I} on $I$, we write
\begin{equation}
\mathbf{J}(\tz;T,w) = \mathbf{K}(\tz;T,w) \left(\frac{\tz-\tz_1(w)}{\tz-\tz_2(w)}\right)^{\ii  p \sigma_{3}}
\label{eq:J-K}
\end{equation}
where the power function is defined to be the principal branch and $p>0$ was given in terms of $\mathfrak{a},\mathfrak{b}$ in \eqref{p-def}. The new unknown $\mathbf{K}(\tz;T,w)$ extends analytically to $I$, and we assume that it is bounded near $\tz=\tz_j(w)$, $j=1,2$, which in particular makes it analytic at $\tz=\tz_2(w)$. Thus, $\mathbf{K}(\tz;T,w)$ is analytic in $\mathbb{C}\setminus \Sigma$ and tends to $\mathbb{I}$ as $\tz\to\infty$. The constant and simple jump condition \eqref{J-jump-simple} satisfied by $\mathbf{J}(\tz;T,w)$ across $\Sigma$ is modified for $\mathbf{K}(\tz;T,w)$:
\begin{equation}
\mathbf{K}_+(\tz;T,w) = \mathbf{K}_-(\tz;T,w) \left(\frac{\tz-\tz_1(w)}{\tz-\tz_2(w)}\right)^{\ii  p \sigma_{3}}
 \begin{bmatrix} 0 &  \ee^{-\ii T^{1/3}\kappa(w)} \\ -  \ee^{\ii T^{1/3}\kappa(w)} & 0 \end{bmatrix}
 \left(\frac{\tz-\tz_1(w)}{\tz-\tz_2(w)}\right)^{-\ii  p \sigma_{3}},\quad \tz\in\Sigma.
\end{equation}
To convert this back into a constant jump condition on $\Sigma$ alone, we follow \cite[Eqn.\@ (222)]{BilmanLM2020} up to a scaling, and introduce another Szeg\H{o} function\footnote{The scaling involves an imaginary factor for convenience.  Thus the Szeg\H{o} function $k$ defined in \cite[Eqn.\@ (222)]{BilmanLM2020} is given by $\ii p\newk$ and the constants $\mu(w)$ defined in \eqref{eq:mu-def-large-T} and \cite[Eqn.\@ (223)]{BilmanLM2020} differ by a factor of $p$.}
\begin{equation}
%k(\tz;w):= \ii  \log \left(\frac{\tz-\tz_1(w)}{\tz-\tz_2(w)}\right)+ \ii  R(\tz ;w) \int_{\tz_1(w)}^{\tz_2(w)} \frac{\dd \zeta}{R(\zeta ; w)(\zeta-\tz)} + \frac{1}{2}\ii \mu(w),
\newk(\tz;w):=   \log \left(\frac{\tz-\tz_1(w)}{\tz-\tz_2(w)}\right)+  R(\tz ;w) \int_{\tz_1(w)}^{\tz_2(w)} \frac{\dd \zeta}{R(\zeta ; w)(\zeta-\tz)} + \frac{1}{2} \mu(w),
\label{k-def-large-T}
\end{equation}
where the logarithm is taken to be the principal branch, $-\pi<\operatorname{Im}(\log (\diamond))<\pi$, and where the constant $\mu(w)$ is given by 
\begin{equation}
\mu(w):=2  \int_{\tz_1(w)}^{\tz_2(w)} \frac{\dd \zeta}{R(\zeta ; w)}>0.
\label{eq:mu-def-large-T}
\end{equation}
Using the same antiderivative \eqref{eq:antiderivative-root} used to integrate $\gamma(w)$, which is analytic for $\tz_1(w)<\zeta<\tz_2(w)$, we take care to evaluate $R(\tz;w)$ at the lower limit of integration by the boundary value $R_-(\tz_1(w);w)$ and obtain (an analogue of \cite[Eqn.\@ (225)]{BilmanLM2020})
\begin{equation}
\mu(w)=2\ln\left(\frac{\tz_2(w)-U+R(\tz_2(w);w)}{\tz_1(w)-U+R_-(\tz_1(w);w)}\right).
%2\log\left(\frac{\tz_2(w)-u+R(\tz_2(w);w)-v}{\tz_2(w)-u-R(\tz_2(w);w)+v}\right)-2\log\left(\frac{\tz_1(w)-u+R_-(\tz_1(w);w)-v}{\tz_1(w)-u-R_-(\tz_1(w);w)+v}\right).
\label{mu}
\end{equation}
The function $\newk(\tz;w)$ has the following properties. Firstly, $\newk(\tz;w) = O(\tz^{-1})$ as $\tz\to\infty$ by definition of $\mu(w)$. Next, it can be easily confirmed that $\newk(\tz;w)$ has no jump across $I$ by using the Plemelj formula and comparing the boundary values of the logarithm.
The function $\tz\mapsto \newk(\tz;w)$ has a removable singularity at $\tz=\tz_2(w)$, and the boundary value $\newk_-(\tz;w)$ is continuous along $\Sigma$, including at $\tz=\tz_1(w)$. However, $\newk_{+}(\tz;w)$ has a logarithmic singularity at $\tz=\tz_1(w)$.  Thus, the domain of analyticity for $\newk(\tz;w)$ is $\tz\in \mathbb{C}\setminus \Sigma$ and the jump condition satisfied by the (continuous, except at $\tz=\tz_1(w)$ from the left) boundary values of $\newk(\tz;w)$ is 
\begin{equation}
%k_{+}(\tz ; w)+k_{-}(\tz ; w)=2 \ii  \log \left(\frac{\tz-\tz_1(w)}{\tz-\tz_2(w)}\right)+ \ii  \mu(w),\quad \tz\in\Sigma.
\newk_{+}(\tz ; w)+\newk_{-}(\tz ; w)=2  \log \left(\frac{\tz-\tz_1(w)}{\tz-\tz_2(w)}\right)+  \mu(w),\quad \tz\in\Sigma.
\label{T-k-jump}
\end{equation}
We next evaluate $\newk(\tz;w)$ for $\tz\in\mathbb{R}$ with $\tz<\tz_1(w)$ using the antiderivative \eqref{eq:antiderivative-Cauchy}, which is analytic for $\zeta\in (\tz_1(w),\tz_2(w))$, and hence 
\begin{multline}
\newk(\tz;w)=\log\left(\frac{\tz-\tz_1(w)}{\tz-\tz_2(w)}\right)+\frac{1}{2}\mu(w) \\
{}+\log\left(\frac{(\tz-U)(\tz_2(w)-U)+(V+R(\tz;w))(V-R(\tz_2(w);w))}{(\tz-U)(\tz_2(w)-U)+(V-R(\tz;w))(V-R(\tz_2(w);w))}\right)\\
{}-\log\left(\frac{(\tz-U)(\tz_1(w)-U)+(V+R(\tz;w))(V-R_-(\tz_1(w);w))}{(\tz-U)(\tz_1(w)-U)+(V-R(\tz;w))(V-R_-(\tz_1(w);w))}\right).
\label{eq:k-Z-less-than-Z1}
\end{multline}
The term on the middle line admits analytic continuation through $\Sigma$ from the left; in particular it the argument of the logarithm is positive when evaluated at $\tz=\tz_1(w)$ and $R(\tz;w)=R_+(\tz_1(w);w)$ whenever $|w|<w_\mathrm{c}=54^\frac{1}{3}$.   The term on the last line has a logarithmic singularity at $\tz=\tz_1(w)$ however, coming from the denominator of the argument of the logarithm.  We can write
\begin{equation}
\frac{(\tz-U)(\tz_1(w)-U)+(V+R(\tz;w))(V-R_-(\tz_1(w);w))}{(\tz-U)(\tz_1(w)-U)+(V-R(\tz;w))(V-R_-(\tz_1(w);w))}=-\frac{1}{\omega^{\newk,1}(\tz;w)(\tz-\tz_1(w))},
\end{equation}
where
\begin{equation}
\omega^{\newk,1}(\tz;w):=\frac{\displaystyle U-\tz_1(w)+\frac{R(\tz;w)-R_+(\tz_1(w);w)}{\tz-\tz_1(w)}(V-R_-(\tz_1(w);w))}{(\tz-U)(\tz_1(w)-U)+(V+R(\tz;w))(V-R_-(\tz_1(w);w))}
\end{equation}
is a function analytic and non-vanishing at $\tz=\tz_1(w)$ and $R(\tz;w)=R_+(\tz_1(w);w)$ with value
\begin{equation}
\omega^{\newk,1}_+(\tz_1(w);w)=\frac{(\tz_1(w)-U)V}{2((\tz_1(w)-U)^2+V^2)(R_+(\tz_1(w);w)+V)}
\label{eq:tilde-omega-at-z1}
\end{equation}
which one can confirm is a positive number whenever $|w|<w_\mathrm{c}=54^\frac{1}{3}$.  Therefore, for $\tz$ near $\tz_1(w)$ on the left of $\Sigma$, $\newk(\tz;w)$ can be written in the form
\begin{multline}
\newk(\tz;w)=\frac{1}{2}\mu(w) 
+\log\left(\frac{(\tz-U)(\tz_2(w)-U)+(V+R(\tz;w))(V-R(\tz_2(w);w))}{(\tz-U)(\tz_2(w)-U)+(V-R(\tz;w))(V-R(\tz_2(w);w))}\cdot\frac{\omega^{\newk,1}(\tz;w)}{\tz_2(w)-\tz}\right)\\
{}+2\log(\tz_1(w)-\tz),
\label{eq:k-explicit}
\end{multline}
where only the last term fails to be analytic at $\tz=\tz_1(w)$.

Similarly, if $\tz>\tz_2(w)$, then once again the antiderivative in \eqref{eq:antiderivative-Cauchy} is an analytic function of $\zeta\in (\tz_1(w),\tz_2(w))$, so the formula \eqref{eq:k-Z-less-than-Z1} is also valid in this situation.  If $\tz$ approaches $\tz_2(w)$ from above, then the terms on the final line in \eqref{eq:k-Z-less-than-Z1} are now analytic at $\tz=\tz_2(w)$, while the terms on the middle line produce a logarithmic singularity and a branch cut emanating from $\tz=\tz_2(w)$ to the left.  The latter cancels with the explicit logarithmic singularity on the first line of \eqref{eq:k-Z-less-than-Z1}, and the resulting formula is analytic at $\tz=\tz_2(w)$:
\begin{multline}
\newk(\tz;w)=
\frac{1}{2} \mu(w) \\
{}+\log\left(\frac{(\tz-U)(\tz_1(w)-U)+(V-R(\tz;w))(V-R_-(\tz_1(w);w))}{(\tz-U)(\tz_1(w)-U)+(V+R(\tz;w))(V-R_-(\tz_1(w);w))}(\tz-\tz_1(w))\omega^{\newk,2}(\tz;w)\right),
\end{multline}
wherein $\omega^{\newk,2}(\tz;w)$ is a function analytic and positive at $\tz=\tz_2(w)$ given by
\begin{equation}
\omega^{\newk,2}(\tz;w):=\frac{\displaystyle\tz_2(w)-U+\frac{R(\tz;w)-R(\tz_2(w);w)}{\tz-\tz_2(w)}(V-R(\tz_2(w);w))}{(\tz-U)(\tz_2(w)-U)+(V-R(\tz;w))(V-R(\tz_2(w);w))}.
\end{equation}
In particular,
\begin{equation}
\omega^{\newk,2}(\tz_2(w);w)=\frac{(\tz_2(w)-U)V}{2((\tz_2(w)-U)^2+V^2)(R(\tz_2(w);w)-V)}.
\end{equation}

We introduce the function $\newk(\tz;w)$ in the analysis by writing
\begin{equation}
%\mathbf{K}(\tz;T,w) = \mathbf{L}(\tz;T,w)\ee^{-pk(\tz;w)\sigma_3}.
\mathbf{K}(\tz;T,w) = \mathbf{L}(\tz;T,w)\ee^{-\ii p\newk(\tz;w)\sigma_3}.
\label{eq:K-L}
\end{equation}
It follows that $\mathbf{L}(\tz;T,w)$ is a matrix function analytic for $\tz\in\mathbb{C}\setminus\Sigma$, which also tends to $\mathbb{I}$ as $\tz\to\infty$ and satisfies the jump condition
\begin{equation}
\mathbf{L}_+(\tz;T,w) = \mathbf{L}_-(\tz;T,w) \begin{bmatrix} 0 &  \ee^{-\ii (T^{1/3}\kappa(w) +p\mu(w))} \\ -  \ee^{\ii ( T^{1/3}\kappa(w) + p\mu(w))} & 0 \end{bmatrix},\quad \tz\in\Sigma.
\label{T-L-jump}
\end{equation}
We can directly solve for $\mathbf{L}(\tz;T,w)$ by diagonalizing the constant jump matrix 
and choosing the unique solution that exhibits $-\frac{1}{4}$-power growth at the endpoints $\tz=\tz_0(w),\tz_0(w)^*$ of $\Sigma$:
\begin{equation}
\mathbf{L}(\tz;T,w)=  \ee^{-\frac{1}{2}\ii (T^{1/3}\kappa(w) +p\mu(w))\sigma_3} \mathbf{Z} y(\tz;w)^{\sigma_3}\mathbf{Z}^{-1}  \ee^{\frac{1}{2}\ii (T^{1/3}\kappa(w) +p\mu(w))\sigma_3},\quad 
\mathbf{Z}:=
\frac{1}{\sqrt{2}}\begin{bmatrix}
1 & \ii \\
\ii & 1
\end{bmatrix},
\label{L-def}
\end{equation}
where $y(\tz;w)$ is the function analytic for $\tz\in \mathbb{C}\setminus \Sigma$, determined by the properties
\begin{equation}
y(\tz;w)^4 =  \frac{\tz-\tz_0(w)}{\tz-\tz_0(w)^*}\quad\text{and}\quad \lim_{\tz\to\infty}y(\tz;w)= 1.
\end{equation}
Combining \eqref{eq:T-J}, \eqref{eq:J-K}, and \eqref{eq:K-L} finishes the construction of the outer parametrix, yielding
\begin{equation}
%\dot{\mathbf{T}}^{\rm out}(\tz;T,w) := \ee^{\ii q \gamma(w)\sigma_3} \mathbf{L}(\tz;T,w) \ee^{-(pk(\tz;w)+qf(\tz;w))\sigma_3}\left(\frac{\tz-\tz_1(w)}{\tz-\tz_2(w)}\right)^{\ii p \sigma_{3}}.
\dot{\mathbf{T}}^{\rm out}(\tz;T,w) := \ee^{\ii q \gamma(w)\sigma_3} \mathbf{L}(\tz;T,w) \ee^{-\ii(p\newk(\tz;w)+q\newf(\tz;w))\sigma_3}\left(\frac{\tz-\tz_1(w)}{\tz-\tz_2(w)}\right)^{\ii p \sigma_{3}}.
\label{W-out-T}
\end{equation}
where $\mathbf{L}(\tz;T,w)$ is given by \eqref{L-def}.
Note that unlike the outer parametrix constructed for the analysis in the regime $X\to+\infty$, the outer parametrix $\dot{\mathbf{T}}^{\rm out}(\tz;T,w)$ depends on $T$.  However, this dependence is purely oscillatory, coming solely from the conjugating exponential factors in \eqref{L-def}.

\subsubsection{Inner parametrices near $\tz=\tz_1(w)$ and $\tz=\tz_2(w)$}
\label{sec:Airy-parametrices}
Let $D_{\tz_j}(\delta)$, $j=1,2$, denote the disk of radius $\delta>0$ centered at $\tz_j(w)$.
To construct an inner parametrix $\dot{\mathbf{T}}^{\tz_2}$ in $D_{\tz_2}(\delta)$, first note that $h(\tz;w) - h(\tz_2(w);w)$ vanishes precisely to second order as $\tz\to \tz_2(w)$. We define the $T$-independent conformal coordinate $\varphi_{\tz_2}$ by choosing the solution of
\begin{equation}
\varphi_{\tz_2}(\tz;w)^2 = 2( h(\tz;w) - h(\tz_2(w);w) ),\quad \tz\in D_{\tz_2}(\delta),
\label{T-varphi-z2}
\end{equation}
that is analytic at $\tz=\tz_2(w)$ and that satisfies $\varphi_{\tz_2}'(\tz_2(w),w)>0$. This choice ensures that the arc $I\cap D_{\tz_2}(\delta)$ is mapped by $\varphi_{\tz_2}(\diamond; w)$ locally to the negative real axis. We define the rescaled conformal coordinate $\zeta_{\tz_2}:= T^{\frac{1}{6}}\varphi_{\tz_2}(\tz;w)$ and observe that the jump conditions satisfied by the matrix
\begin{equation}
\mathbf{U}^{\tz_2}:= \mathbf{T}(\tz;T,w) \ee^{-\ii T^{1/3}h(\tz_2(w);w)\sigma_3}.
\end{equation}
match exactly those shown in Figure~\ref{fig:PC-z2}
when expressed in terms of the conformal coordinate $\zeta=\zeta_{\tz_{2}}$ and when the jump contours are locally taken to coincide with the five rays $\arg (\zeta)=\pm \frac{1}{4} \pi$, $\arg (\zeta)=\pm \frac{3}{4} \pi$, and $\arg (-\zeta)=0$, with the same values of $p$ and $\tau$ given in terms of $\mathfrak{a},\mathfrak{b}$ by \eqref{p-def}--\eqref{tau-def}.
These jump conditions coincide exactly with those in \cite[Riemann-Hilbert Problem A.1]{Miller2018} for a well-defined and explicit standard parabolic cylinder parametrix $\mathbf{U}(\zeta)=\mathbf{U}(\zeta;p,\tau)$. Thus, an inner parametrix $\dot{\mathbf{T}}^{\tz_2}(\tz;T,w)$ that satisfies exactly the jump conditions inside $D_{\tz_2}(\delta)$ can be taken in the form
\begin{equation}
\dot{\mathbf{T}}^{\tz_2}(\tz;T,w):=\mathbf{Y}^{\tz_2}(\tz;T,w)\mathbf{U}(\zeta_{\tz_2};p,\tau)\ee^{\ii T^{1/3}h(\tz_2(w);w)\sigma_3}
\label{eq:T-z2-parametrix-Y}
\end{equation}
where $\mathbf{Y}^{\tz_2}(\tz;T,w)$ is any matrix analytic in $D_{\tz_2}(\delta)$.  To specify $\mathbf{Y}^{\tz_2}(\tz;T,w)$, we write the outer parametrix in terms of the conformal map $\tz\mapsto\varphi_{\tz_2}(\tz;w)$ by noting that $\varphi_{\tz_2}(\tz;w)^{-\ii p\sigma_3}$ is an exact solution of the jump conditions for $\dot{\mathbf{T}}^\mathrm{out}(\tz;T,w)$ in $D_{\tz_2}(\delta)$.  Hence we may write
\begin{equation}
\dot{\mathbf{T}}^\mathrm{out}(\tz;T,w)\ee^{-\ii T^{1/3}h(\tz_2(w);w)\sigma_3} = \mathbf{H}^{\tz_2}(\tz;T,w)\varphi_{\tz_2}(\tz;w)^{-\ii p\sigma_3},
\end{equation}
where
 $\mathbf{H}^{\tz_2}(\tz;T,w)$ is analytic for $\tz\in D_{\tz_2}(\delta)$ and is uniformly bounded on this disk as $T\to+\infty$.  We then define the parametrix near $\tz=\tz_2(w)$ by the formula \eqref{eq:T-z2-parametrix-Y} in which we take
\begin{equation}
\mathbf{Y}^{\tz_2}(\tz;T,w):=\mathbf{H}^{\tz_2}(\tz;T,w)T^{\frac{1}{6}\ii p\sigma_3}.
\label{T-Y-z2}
\end{equation}
%constructed following exactly the same steps in Section~\ref{s:large-X}. 
For $\tz\in D_{\tz_2}(\delta)$, the parametrix $\dot{\mathbf{T}}^{\tz_2}(\tz;T,w)$ satisfies
\begin{equation}
\dot{\mathbf{T}}^{\tz_2}(\tz;T,w)\dot{\mathbf{T}}^\mathrm{out}(\tz;T,w)^{-1} = 
\mathbf{H}^{\tz_2}(\tz;T,w)T^{\frac{1}{6}\ii p\sigma_3}\mathbf{U}(\zeta_{\tz_2};p,\tau)\zeta_{\tz_2}^{\ii p\sigma_3}T^{-\frac{1}{6}\ii p\sigma_3}\mathbf{H}^{\tz_2}(\tz;T,w)^{-1},
\label{T-z2-mismatch}
\end{equation}
in which we note that the product $\mathbf{U}(\zeta_{\tz_2};p,\tau)\zeta_{\tz_2}^{\ii p\sigma_3}$ has an asymptotic expansion in descending powers of $\zeta_{\tz_2}$ according to \eqref{U-PC-expansion}, and  $\zeta_{\tz_2}$ is large of size $T^{\frac{1}{6}}$ when $\tz\in \partial D_{\tz_2}(\delta)$.

An explicit formula for $\mathbf{H}^{\tz_2}(\tz;T,w)$ is:
\begin{multline}
%\mathbf{H}^{\tz_2}(\tz;T,w):=\ee^{\ii q\gamma(w)\sigma_3}\mathbf{L}(\tz;T,w)\ee^{-(pk(\tz;w)+qf(\tz;w))\sigma_3}\ee^{-\ii T^{1/3}h(\tz_2(w);w)\sigma_3}\\
%{}\cdot (\tz-\tz_1(w))^{\ii p\sigma_3}\left(\frac{\varphi_{\tz_2}(\tz;w)}{\tz-\tz_2(w)}\right)^{\ii p\sigma_3}.
\mathbf{H}^{\tz_2}(\tz;T,w):=\ee^{\ii q\gamma(w)\sigma_3}\mathbf{L}(\tz;T,w)\ee^{-\ii(p\newk(\tz;w)+q\newf(\tz;w))\sigma_3}\ee^{-\ii T^{1/3}h(\tz_2(w);w)\sigma_3}\\
{}\cdot (\tz-\tz_1(w))^{\ii p\sigma_3}\left(\frac{\varphi_{\tz_2}(\tz;w)}{\tz-\tz_2(w)}\right)^{\ii p\sigma_3}.
\label{T-Hz2-formula}
\end{multline}

Constructing an inner parametrix in $D_{\tz_1}(\delta)$ is slightly more involved due to the presence of jump conditions satisfied by $\mathbf{T}(\tz;T,w)$ across $\Sigma^+\cup\Sigma^-$. 
Recall that $h(\tz;w)$ is analytic in $\tz$ for $\tz\in D_{\tz_1}(\delta) \setminus \Sigma$, and the values of $h(\tz;w)$ in the left ($h_+(\tz;w)$) and right ($h_-(\tz;w)$) half-disks both admit analytic continuation to the full disk $D_{\tz_1}(\delta)$ where, according to \eqref{T-g-theta-kappa} and \eqref{T-h-g-theta} we have 
\begin{equation}
h_+(\tz;w) + h_-(\tz;w) = \kappa(w),\quad \tz\in D_{\tz_1}(\delta).
\label{T-h-jump}
\end{equation}
We will base the construction of the inner parametrix in $D_{\tz_1}(\delta)$ on a $T$-independent conformal mapping $\varphi_{\tz_1}$ constructed from $h_{-}(\tz;w)$ by choosing the solution of
\begin{equation}
\varphi_{\tz_1}(\tz;w)^2 = 2(   h_-(\tz_1(w);w) - h_-(\tz;w) ),\quad \tz\in D_{\tz_1}(\delta),
\label{T-varphi-z1}
\end{equation}
that is analytic in $D_{\tz_1}(\delta)$ and satisfies $\varphi_{\tz_1}'(\tz_1(w),w)<0$, and we introduce the rescaled conformal coordinate $\zeta_{\tz_1}:= T^{\frac{1}{6}}\varphi_{\tz_1}(\tz;w)$. Letting $\Omega_\infty$ denote the region near $\tz_1(w)$ complementary to $R_\Sigma^+\cup\Omega_+\cup R_\Sigma^-\cup\Omega_-\cup L_\Sigma^+\cup L_\Sigma^-$ (i.e., to the left of $\Sigma=\Sigma^+\cup\Sigma^-$), we then consider the matrix
\begin{equation}
\mathbf{U}^{\tz_1}:=
\begin{cases}
\mathbf{T}(\tz;T,w) \left(\dfrac{\mathfrak{a}}{\mathfrak{b}}\right)^{\sigma_3}(\ii \sigma_2) \ee^{\ii T^{1/3} h_{-}(\tz_1(w);w)\sigma_3}\ii^{\sigma_3},&\quad \tz\in (R_\Sigma^+ \cup \Omega_+)\cap D_{\tz_1}(\delta),\\
\mathbf{T}(\tz;T,w) \left(\dfrac{\mathfrak{b}}{\mathfrak{a}}\right)^{\sigma_3}(\ii \sigma_2) \ee^{\ii T^{1/3} h_{-}(\tz_1(w);w)\sigma_3}\ii^{\sigma_3},&\quad \tz\in (R_\Sigma^- \cup \Omega_-) \cap D_{\tz_1}(\delta),\\
\mathbf{T}(\tz;T,w) \ee^{\ii T^{1/3} h_{-}(\tz_1(w);w)\sigma_3}\ee^{-\ii T^{1/3}\kappa(w)\sigma_3}\ii^{\sigma_3} ,&\quad \tz\in (L_\Sigma^+\cup L_\Sigma^- \cup \Omega_\infty)\cap D_{\tz_1}(\delta).\\
\end{cases}
\label{T-U-z1}
\end{equation}
Using \eqref{T-h-jump}, one checks that this transformation results in a trivial identity jump for $\mathbf{U}^{\tz_1}$ across $\Sigma^+\cup\Sigma^-$. 
Now, recall the values $\bar{p}$ and $\bar{\tau}$ defined in Corollary~\ref{cor:large-negative-X}. 
Taking into account that the conformal coordinate $\zeta_{\tz_1}:= T^{\frac{1}{6}}\varphi_{\tz_1}(\tz;w)$ satisfies $\varphi_{\tz_1}'(\tz_1(w);w)<0$, 
we see that the jump conditions satisfied by $\mathbf{U}^{\tz_1}$ 
%(see Figure~\ref{fig:z1-jumps-T-U-z1}) 
%take exactly the same form as 
match those shown in Figure~\ref{fig:PC-z2} with $(p,\tau)$ replaced by $(\bar{p},\bar{\tau})$ (equivalent to swapping $\mathfrak{a}$ and $\mathfrak{b}$).
These jump conditions therefore coincide exactly with those of the standard parabolic cylinder parametrix $\mathbf{U}(\zeta;\bar{p},\bar{\tau})$ solving \cite[Riemann-Hilbert Problem A.1]{Miller2018}.

Thus, an inner parametrix $\dot{\mathbf{T}}^{\tz_1}(\tz;T,w)$ that satisfies exactly the jump conditions of $\mathbf{T}(\tz;T,w)$ within $D_{\tz_1}(\delta)$ can be taken in the form
\begin{multline}
\dot{\mathbf{T}}^{\tz_1}(\tz;T,w):=\mathbf{Y}^{\tz_1}(\tz;T,w)\mathbf{U}(\zeta_{\tz_1};\bar{p},\bar{\tau})\ii^{-\sigma_3}\ee^{-\ii T^{1/3}h_-(\tz_1(w);w)\sigma_3}\\
{}\cdot\begin{cases}
(\ii\sigma_2)^{-1}\left(\dfrac{\mathfrak{b}}{\mathfrak{a}}\right)^{\sigma_3},&\tz\in (R_\Sigma^+\cup\Omega_+)\cap D_{\tz_1}(\delta),\\
(\ii\sigma_2)^{-1}\left(\dfrac{\mathfrak{a}}{\mathfrak{b}}\right)^{\sigma_3},&\tz\in (R_\Sigma^-\cup\Omega_-)\cap D_{\tz_1}(\delta),\\
\ee^{\ii T^{1/3}\kappa(w)\sigma_3},&\tz\in (L_\Sigma^+\cup L_\Sigma^-\cup\Omega_\infty)\cap D_{\tz_1}(\delta),
\end{cases}
\label{T-inner-z1}
\end{multline}
where $\mathbf{Y}^{\tz_1}(\tz;T,w)$ is a matrix factor analytic in $D_{\tz_1}(\delta)$.  To determine $\mathbf{Y}^{\tz_1}(\tz;T,w)$, we first define a matrix $\dot{\mathbf{U}}^{\tz_1}$ within $D_{\tz_1}(\delta)$ exactly as in \eqref{T-U-z1} replacing $\mathbf{T}(\tz;T,w)$ with $\dot{\mathbf{T}}^\mathrm{out}(\tz;T,w)$.  Then one checks that $\dot{\mathbf{U}}^{\tz_1}$ is analytic in $D_{\tz_1}(\delta)$ except on $I\cap D_{\tz_1}(\delta)$ where it satisfies $\dot{\mathbf{U}}^{\tz_1}_+=\dot{\mathbf{U}}^{\tz_1}_-\mathfrak{b}^{2\sigma_3}$.  Since $I$ corresponds to the ray $\arg(-\varphi_{\tz_1})=0$ oriented away from the origin in the $\varphi_{\tz_1}$-plane and $\mathfrak{b}^{2\sigma_3}=\ee^{-2\pi\bar{p}\sigma_3}$, it follows that $\dot{\mathbf{U}}^{\tz_1}=\mathbf{H}^{\tz_1}(\tz;T,w)\varphi_{\tz_1}(\tz;w)^{-\ii\bar{p}\sigma_3}$ where $\mathbf{H}^{\tz_1}(\tz;T,w)$ is holomorphic in $D_{\tz_1}(\delta)$ and bounded as $T\to+\infty$.  By analogy with \eqref{T-Y-z2} we define the inner parametrix $\dot{\mathbf{T}}^{\tz_1}(\tz;T,w)$ by \eqref{T-inner-z1} in which the holomorphic factor $\mathbf{Y}^{\tz_1}(\tz;T,w)$ is given by
\begin{equation}
\mathbf{Y}^{\tz_1}(\tz;T,w):=\mathbf{H}^{\tz_1}(\tz;T,w)T^{\frac{1}{6}\ii\bar{p}\sigma_3}.
\label{T-Y-z1}
\end{equation}
%constructed following exactly the same steps in Section~\ref{s:large-X}. 
By analogy with \eqref{T-z2-mismatch}, for $\tz\in D_{\tz_1}(\delta)$, the parametrix $\dot{\mathbf{T}}^{\tz_1}(\tz;T,w)$ satisfies
\begin{equation}
\dot{\mathbf{T}}^{\tz_1}(\tz;T,w)\dot{\mathbf{T}}^\mathrm{out}(\tz;T,w)^{-1}=
\mathbf{H}^{\tz_1}(\tz;T,w)T^{\frac{1}{6}\ii\bar{p}\sigma_3}\mathbf{U}(\zeta_{\tz_1};\bar{p},\bar{\tau})\zeta_{\tz_1}^{\ii \bar{p}\sigma_3}T^{-\frac{1}{6}\ii\bar{p}\sigma_3}\mathbf{H}^{\tz_1}(\tz;T,w)^{-1}.
\label{T-z1-mismatch}
\end{equation}

To find an analogue of \eqref{T-Hz2-formula} for $\mathbf{H}^{\tz_1}(\tz;T,w)$ valid for $\tz\in D_{\tz_1}(\delta)$, it is enough to first assume that $\tz\in\Omega_\infty$ and then extend the result to $D_{\tz_1}(\delta)$ by analytic continuation.  By definition, 
\begin{multline}
%\mathbf{H}^{\tz_1}(\tz;T,w):=\ee^{\ii q\gamma(w)\sigma_3}\mathbf{L}(\tz;T,w)\ee^{-(pk(\tz;w)+qf(\tz;w))\sigma_3}\left(\frac{\tz_1(w)-\tz}{\tz_2(w)-\tz}\right)^{\ii p\sigma_3}\\
%{}\cdot\ee^{\ii T^{1/3}h(\tz_1(w);w)\sigma_3}\ee^{-\ii T^{1/3}\kappa(w)\sigma_3}\ii^{\sigma_3}\varphi_{\tz_1}(\tz;w)^{\ii\overline{p}\sigma_3},\quad \tz\in\Omega_\infty\cap D_{\tz_1}(\delta).
\mathbf{H}^{\tz_1}(\tz;T,w):=\ee^{\ii q\gamma(w)\sigma_3}\mathbf{L}(\tz;T,w)\ee^{-\ii(p\newk(\tz;w)+q\newf(\tz;w))\sigma_3}\left(\frac{\tz_1(w)-\tz}{\tz_2(w)-\tz}\right)^{\ii p\sigma_3}\\
{}\cdot\ee^{\ii T^{1/3}h(\tz_1(w);w)\sigma_3}\ee^{-\ii T^{1/3}\kappa(w)\sigma_3}\ii^{\sigma_3}\varphi_{\tz_1}(\tz;w)^{\ii\bar{p}\sigma_3},\quad \tz\in\Omega_\infty\cap D_{\tz_1}(\delta).
\end{multline}
Since the jump matrix for $\mathbf{L}(\tz;T,w)$ across $\Sigma$ is constant (see \eqref{T-L-jump}), upon identifying $\mathbf{L}(\tz;T,w)$ for $\tz\in\Omega_\infty$ with $\mathbf{L}_+(\tz;T,w)$, we see that $\mathbf{L}(\tz;T,w)=\mathbf{L}_+(\tz;T,w)$ has an analytic continuation to all of $D_{\tz_1}(\delta)$ that we will denote by the same symbol $\mathbf{L}_+(\tz;T,w)$ and that is given by \eqref{L-def} in which the branch cut of $y(\tz;w)$ is deformed toward the right near $\tz=\tz_1(w)$ to coincide with the corresponding arc of $\partial D_{\tz_1}(\delta)$.  On the other hand, neither $\newf(\tz;w)$ nor $\newk(\tz;w)$ can be analytically continued from $\Omega_\infty$ to all of $D_{\tz_1}(\delta)$ without a branch cut appearing in each case that we take to agree with $I\cap D_{\tz_1}(\delta)$.  However, using \eqref{T-f-jump} to continue $\newf(\tz;w)=\newf_+(\tz;w)$ from $\Omega_\infty\cap D_{\tz_1}(\delta)$ through $\Sigma^\pm$ to $D_{\tz_1}(\delta)\setminus I$ shows that the continuation has the form 
\begin{equation}
%f_+(\tz;w)=f^0(\tz;w)-\frac{1}{\pi\ii}\log(\varphi_{\tz_1}(\tz;w)),\quad \tz\in D_{\tz_1}(\delta)\setminus I,
\newf_+(\tz;w)=\newf^0(\tz;w)+\log(\varphi_{\tz_1}(\tz;w)),\quad \tz\in D_{\tz_1}(\delta)\setminus I,
\label{f0-def}
\end{equation}
where $\newf^0(\tz;w)$ is holomorphic in $D_{\tz_1}(\delta)$ and the logarithm is the principal branch.  
In fact, we can use \eqref{eq:f-explicit} to express $\newf^0(\tz;w)$ explicitly as 
\begin{equation}
%f^0(\tz;w)=\frac{1}{2\pi\ii}\log\left(\frac{(v+R(\tz;w))^2+(\tz-u)^2}{(v-R(\tz;w))^2+(\tz-u)^2}\cdot\frac{1}{\omega(\tz;w)^2}\cdot\frac{\varphi_{\tz_1}(\tz;w)^2}{(\tz_1(w)-\tz)^2}\right).
\newf^0(\tz;w)=\frac{1}{2}\log\left(\frac{(V-R(\tz;w))^2+(\tz-U)^2}{(V+R(\tz;w))^2+(\tz-U)^2}\cdot\omega^\newf(\tz;w)^2\cdot\frac{(\tz_1(w)-\tz)^2}{\varphi_{\tz_1}(\tz;w)^2}\right).
\end{equation}
This formula holds for $\tz\in \Omega_\infty\cap D_{\tz_1}(\delta)$, but to continue analytically to $D_{\tz_1}(\delta)\setminus\Omega_\infty$ one just replaces $R(\tz;w)$ with $-R(\tz;w)$.  In particular, to compute the value at the center of the disk, one simply takes $R(\tz;w)$ as $R_+(\tz_1(w);w)<0$ and obtains
\begin{equation}
%f^0(\tz_1(w);w)=\frac{1}{2\pi\ii}\ln\left(\frac{(v+R_+(\tz_1(w);w))^2+(\tz_1(w)-u)^2}{(v-R_+(\tz_1(w);w))^2+(\tz_1(w)-u)^2}\cdot\frac{\varphi'_{\tz_1}(\tz_1(w);w)^2}{\omega_+(\tz_1(w);w)^2}\right),
\newf^0(\tz_1(w);w)=\frac{1}{2}\ln\left(\frac{(V-R_+(\tz_1(w);w))^2+(\tz_1(w)-U)^2}{(V+R_+(\tz_1(w);w))^2+(\tz_1(w)-U)^2}\cdot\frac{\omega^\newf_+(\tz_1(w);w)^2}{\varphi'_{\tz_1}(\tz_1(w);w)^2}\right),
\end{equation}
wherein $\omega^\newf_+(\tz_1(w);w)>0$ is defined by \eqref{eq:omega-at-z1}.  
Similarly using \eqref{T-k-jump} to continue $\newk(\tz;w)=\newk_+(\tz;w)$ from $\Omega_\infty\cap D_{\tz_1}(\delta)$ through $\Sigma^\pm$ to $D_{\tz_1}(\delta)\setminus I$ shows that
\begin{equation}
%k_+(\tz;w)=k^0(\tz;w) + 2\ii\log(\varphi_{\tz_1}(\tz;w)),\quad \tz\in D_{\tz_1}(\delta)\setminus I,
\newk_+(\tz;w)=\newk^0(\tz;w) + 2\log(\varphi_{\tz_1}(\tz;w)),\quad \tz\in D_{\tz_1}(\delta)\setminus I,
\label{k0-def}
\end{equation}
where $\newk^0(\tz;w)$ is holomorphic in $D_{\tz_1}(\delta)$.  Using \eqref{eq:k-explicit}, we can explicitly write $\newk^0(\tz;w)$ in the form
\begin{multline}
%k^0(\tz;w)=\frac{1}{2}\ii\mu(w)\\
%{}+\ii\log\left(\frac{(\tz-u)(\tz_2(w)-u)+(v+R(\tz;w))(v-R(\tz_2(w);w))}{(\tz-u)(\tz_2(w)-u)+(v-R(\tz;w))(v-R(\tz_2(w);w))}\cdot\frac{\widetilde{\omega}(\tz;w)}{\tz_2(w)-\tz}\cdot \frac{(\tz_1(w)-\tz)^2}{\varphi_{\tz_1}(\tz;w)^2}\right).
\newk^0(\tz;w)=\frac{1}{2}\mu(w)\\
{}+\log\left(\frac{(\tz-U)(\tz_2(w)-U)+(V+R(\tz;w))(V-R(\tz_2(w);w))}{(\tz-U)(\tz_2(w)-U)+(V-R(\tz;w))(V-R(\tz_2(w);w))}\cdot\frac{\omega^{\newk,1}(\tz;w)}{\tz_2(w)-\tz}\cdot \frac{(\tz_1(w)-\tz)^2}{\varphi_{\tz_1}(\tz;w)^2}\right).
\end{multline}
Again, this holds as written for $\tz\in D_{\tz_1}(\delta)\cap\Omega_\infty$ but to analytically continue to $D_{\tz_1}(\delta)\setminus\Omega_\infty$ one just replaces $R(\tz;w)$ with $-R(\tz;w)$.
Evaluating at $\tz=\tz_1(w)$ means replacing $R(\tz;w)$ with $R_+(\tz_1(w);w)<0$ and computing a limit of a difference quotient:
\begin{multline}
%k^0(\tz_1(w);w)=\frac{1}{2}\ii\mu(w)\\
%{}+\ii\ln\left(\frac{(\tz_1(w)-u)(\tz_2(w)-u)+(v+R_+(\tz_1(w);w))(v-R(\tz_2(w);w))}{(\tz_1(w)-u)(\tz_2(w)-u)+(v-R_+(\tz_1(w);w))(v-R(\tz_2(w);w))}\right.\\
%\left.{}\cdot\frac{\widetilde{\omega}_+(\tz_1(w);w)}{(\tz_2(w)-\tz_1(w))\varphi_{\tz_1}'(\tz_1(w);w)^2}\right),
\newk^0(\tz_1(w);w)=\frac{1}{2}\mu(w)\\
{}+\ln\left(\frac{(\tz_1(w)-U)(\tz_2(w)-U)+(V+R_+(\tz_1(w);w))(V-R(\tz_2(w);w))}{(\tz_1(w)-U)(\tz_2(w)-U)+(V-R_+(\tz_1(w);w))(V-R(\tz_2(w);w))}\right.\\
\left.{}\cdot\frac{\omega^{\newk,1}_+(\tz_1(w);w)}{(\tz_2(w)-\tz_1(w))\varphi_{\tz_1}'(\tz_1(w);w)^2}\right),
\end{multline}
wherein $\omega^{\newk,1}_+(\tz_1(w);w)>0$ is given by \eqref{eq:tilde-omega-at-z1}.  %This agrees with the value computed by numerical integration as \texttt{k0atz1[w]} in the Mathematica notebook \texttt{RWIO-Comparison-Plots-NEW.nb}.
Therefore, using the identity $q=\bar{p}-p$, recalling that $\mathbf{L}_+(\tz;T,w)$ is interpreted as a holomorphic function in $D_{\tz_1}(\delta)$, and using \eqref{T-h-jump},
\begin{multline}
%\mathbf{H}^{\tz_1}(\tz;T,w)=\ee^{\ii q\gamma(w)\sigma_3}\mathbf{L}_+(\tz;T,w)\ee^{-(pk^0(\tz;w)+qf^0(\tz;w))\sigma_3}\ee^{-\ii T^{1/3}h_+(\tz_1(w);w)\sigma_3}\ii^{\sigma_3}\\
%{}\cdot (\tz_2(w)-\tz)^{-\ii p\sigma_3}\left(\frac{\tz_1(w)-\tz}{\varphi_{\tz_1}(\tz;w)}\right)^{\ii p\sigma_3},\quad \tz\in D_{\tz_1}(\delta).
\mathbf{H}^{\tz_1}(\tz;T,w)=\ee^{\ii q\gamma(w)\sigma_3}\mathbf{L}_+(\tz;T,w)\ee^{-\ii(p\newk^0(\tz;w)+q\newf^0(\tz;w))\sigma_3}\ee^{-\ii T^{1/3}h_+(\tz_1(w);w)\sigma_3}\ii^{\sigma_3}\\
{}\cdot (\tz_2(w)-\tz)^{-\ii p\sigma_3}\left(\frac{\tz_1(w)-\tz}{\varphi_{\tz_1}(\tz;w)}\right)^{\ii p\sigma_3},\quad \tz\in D_{\tz_1}(\delta).
\label{T-Hz1-formula}
\end{multline}

\subsubsection{Inner parametrices near $\tz=\tz_0(w)$ and $\tz=\tz_0(w)^*$}
It suffices to construct a parametrix for $\mathbf{T}(\tz;T,w)$ for $\tz$ near $\tz_0(w)$ and obtain a corresponding parametrix for $\tz$ near $\tz_0(w)^*$ using Schwarz reflection.  Let $D_{\tz_0}(\delta)$ denote a disk of radius $\delta>0$ centered at $\tz=\tz_0(w)$.  We define a conformal mapping on $D_{\tz_0}(\delta)$ by setting $\varphi_{\tz_0}(\tz;w):=(2\ii (h(\tz_0(w);w)-h(\tz;w)))^\frac{2}{3}$, analytically continued to $\tz\in D_{\tz_0}(\delta)$ from the arc along which $h(\diamond;w)-h(\tz_0(w);w)$ is positive imaginary, and we choose $C_{\Gamma,L}^+$ to coincide with this arc within $D_{\tz_0}(\delta)$.  Then within $D_{\tz_0}(\delta)$ we choose $\Sigma^+$ to be mapped by $\varphi=\varphi_{\tz_0}(\tz;w)$ to $\varphi\le 0$, we choose $C_{\Sigma,L}^+$ to be mapped to $\arg(\varphi)=\frac{2}{3}\pi$, and fusing $C_{\Sigma,R}^+$ and $C_{\Gamma,R}^+$ locally we choose both to be mapped to $\arg(\varphi)=-\frac{2}{3}\pi$.  We define a rescaling of the conformal coordinate by $\zeta_{\tz_0}:=T^\frac{2}{9}\varphi_{\tz_0}(\tz;w)$.

Using the facts that $2h(\tz_0(w);w)=h_+(\tz_0(w);w)+h_-(\tz_0(w);w)=\kappa(w)$ (as $h(\diamond;w)$ is continuous at $\tz_0(w)$) and  that $-\mathfrak{a}^3/\mathfrak{b}-\mathfrak{ab}=-\mathfrak{a}/\mathfrak{b}$ (as $\mathfrak{a}^2+\mathfrak{b}^2=1$), one can then check that the matrix $\mathbf{P}(\tz;T,w)$ defined by
\begin{equation}
\mathbf{P}(\tz;T,w):=\mathbf{T}(\tz;T,w)(\ii\sigma_2)\left(\frac{\mathfrak{b}}{\mathfrak{a}}\right)^{\frac{1}{2}\sigma_3}\ee^{\frac{1}{2}\ii T^{1/3}\kappa(w)\sigma_3}
\label{T-P-from-T}
\end{equation}
satisfies the following jump conditions within $D_{\tz_0}(\delta)$:
\begin{equation}
\mathbf{P}_+(\tz;T,w)=\mathbf{P}_-(\tz;T,w)\begin{bmatrix}1 & \ee^{-\zeta_{\tz_0}^{3/2}}\\0 & 1\end{bmatrix},\quad \arg(\zeta_{\tz_0})=0,
\label{T-P-pos}
\end{equation}
\begin{equation}
\mathbf{P}_+(\tz;T,w)=\mathbf{P}_-(\tz;T,w)\begin{bmatrix}1&0\\\ee^{\zeta_{\tz_0}^{3/2}}&1\end{bmatrix},\quad\arg(\zeta_{\tz_0})=\pm\frac{2\pi}{3},
\end{equation}
and
\begin{equation}
\mathbf{P}_+(\tz;T,w)=\mathbf{P}_-(\tz;T,w)\begin{bmatrix}0 & 1\\-1 & 0\end{bmatrix},\quad\arg(-\zeta_{\tz_0})=0,
\label{T-P-twist}
\end{equation}
where we are orienting all four rays in the direction of increasing real part of $\zeta_{\tz_0}$.  Defining a matrix $\dot{\mathbf{P}}^{\mathrm{out}}(\tz;T,w)$ by the right-hand side of \eqref{T-P-from-T} replacing $\mathbf{T}(\tz;T,w)$ with $\dot{\mathbf{T}}^\mathrm{out}(\tz;T,w)$ we see that $\dot{\mathbf{P}}^\mathrm{out}(\tz;T,w)$ is analytic within $D_{\tz_0}(\delta)$ except for the arc where $\varphi_{\tz_0}(\tz;w)\le 0$, along which the same jump condition as in \eqref{T-P-twist} is satisfied, and $\dot{\mathbf{P}}^\mathrm{out}(\tz;T,w)$ exhibits negative one-fourth power singularities near $\tz_0$.  Therefore, the matrix function defined on $D_{\tz_0}(\delta)$ by
\begin{equation}
\mathbf{H}^{\tz_0}(\tz;T,w):=\dot{\mathbf{P}}^\mathrm{out}(\tz;T,w)\mathbf{V}^{-1}\varphi_{\tz_0}(\tz;w)^{-\frac{1}{4}\sigma_3},\quad \tz\in D_{\tz_0}(\delta),\quad \mathbf{V}:=\frac{1}{\sqrt{2}}\begin{bmatrix}1&-\ii\\-\ii & 1\end{bmatrix}
\end{equation}
is actually analytic on the whole disk, and it is easy to check that it is uniformly bounded on the disk in the limit $T\to+\infty$.  Now let $\mathbf{A}(\zeta)$ denote the standard Airy parametrix analytic for $\mathrm{Im}(\zeta)\neq 0$ except across the rays $\arg(\zeta)=\pm\frac{2}{3}\pi$, satisfying the exact jump conditions \eqref{T-P-pos}--\eqref{T-P-twist}, and satisfying the asymptotic condition
\begin{equation}
\mathbf{A}(\zeta)\mathbf{V}^{-1}\zeta^{-\frac{1}{4}\sigma_3}=\mathbb{I} + \begin{bmatrix}O(\zeta^{-3}) & O(\zeta^{-1})\\O(\zeta^{-2}) & O(\zeta^{-3})\end{bmatrix},\quad\zeta\to\infty,
\label{T-Airy-norm}
\end{equation}
(i.e., $\mathbf{A}(\zeta)$ is the unique solution of Riemann-Hilbert Problem 4 of \cite{BothnerM19}, for instance --- see \cite[Appendix B]{BothnerM19} for full details), we then define the parametrix for $\mathbf{T}(\tz;T,w)$ in $D_{\tz_0}(\delta)$ by
\begin{equation}
\dot{\mathbf{T}}^{\tz_0}(\tz;T,w):=\mathbf{H}^{\tz_0}(\tz;T,w)T^{-\frac{1}{18}\sigma_3}\mathbf{A}(T^{\frac{2}{9}}\varphi_{\tz_0}(\tz;w))\ee^{-\frac{1}{2}\ii T^{1/3}\kappa(w)\sigma_3}\left(\frac{\mathfrak{b}}{\mathfrak{a}}\right)^{-\frac{1}{2}\sigma_3}(\ii\sigma_2)^{-1},\quad \tz\in D_{\tz_0}(\delta).
\end{equation}
Then, comparing with the outer parametrix, we have for $\tz\in D_{\tz_0}(\delta)$,
\begin{equation}
\dot{\mathbf{T}}^{\tz_0}(\tz;T,w)\dot{\mathbf{T}}^{\mathrm{out}}(\tz;T,w)^{-1}=\mathbf{H}^{\tz_0}(\tz;T,w)T^{-\frac{1}{18}\sigma_3}\mathbf{A}(\zeta_{\tz_0})\mathbf{V}^{-1}\zeta_{\tz_0}^{-\frac{1}{4}\sigma_3}T^{\frac{1}{18}\sigma_3}\mathbf{H}^{\tz_0}(\tz;T,w)^{-1},%\quad \tz\in D_{\tz_0}(\delta),
\end{equation}
where we recall $\zeta_{\tz_0}=T^\frac{2}{9}\varphi_{\tz_0}(\tz;w)$.
Using \eqref{T-Airy-norm}, the fact that $\varphi_{\tz_0}(\diamond;w)$ is bounded away from zero on $\partial D_{\tz_0}(\delta)$, and the fact that $\mathbf{H}^{\tz_0}(\diamond;T,w)$ is bounded as $T\to+\infty$ and has unit determinant yields
\begin{equation}
\sup_{\tz\in\partial D_{\tz_0}(\delta)}\left\|\dot{\mathbf{T}}^{\tz_0}(\tz;T,w)\dot{\mathbf{T}}^{\mathrm{out}}(\tz;T,w)^{-1}-\mathbb{I}\right\|=O(T^{-\frac{1}{3}}),\quad T\to+\infty.
\label{D-z0-bound}
\end{equation}
Since $\mathbf{T}(\tz;T,w)$ satisfies the exact Schwarz symmetry $\mathbf{T}(\tz^*;T,w)=\sigma_2\mathbf{T}(\tz;T,w)^*\sigma_2$, we obtain an inner parametrix near $\tz_0^*$ by applying the same reflection to $\dot{\mathbf{T}}^{\tz_0}(\tz;T,w)$.

\subsubsection{Global parametrix}
We define the global parametrix $\dot{\mathbf{T}}(\tz;T,w)$ by
\begin{equation}
\dot{\mathbf{T}}(\tz;T,w):=\begin{cases}
\dot{\mathbf{T}}^{\tz_1}(\tz;T,w),&\quad \tz\in D_{\tz_1}(\delta),\\
\dot{\mathbf{T}}^{\tz_2}(\tz;T,w),&\quad \tz\in D_{\tz_2}(\delta),\\
\dot{\mathbf{T}}^{\tz_0}(\tz;T,w),&\quad \tz\in D_{\tz_0}(\delta),\\
\sigma_2\dot{\mathbf{T}}^{\tz_0}(\tz^*;T,w)^*\sigma_2,&\quad \tz\in D_{\tz_0^*}(\delta),\\
\dot{\mathbf{T}}^{\mathrm{out}}(\tz;T,w),&\quad \tz \in \mathbb{C} \setminus \left(I \cup D_{\tz_1}(\delta) \cup D_{\tz_2}(\delta)\cup D_{\tz_0}(\delta) \cup D_{\tz_0^*}(\delta) \right).
\end{cases}
\end{equation}

\subsection{Asymptotics as $T\to+\infty$} 
As in Section~\ref{s:asymptotics-X} we compare the matrix $\mathbf{T}(\tz;T,w)$ with the global parametrix $\dot{\mathbf{T}}(\tz;T,w)$ by defining the error matrix $\mathbf{F}(\tz;T,w):=\mathbf{T}(\tz;T,w)\dot{\mathbf{T}}(\tz;T,w)^{-1}$ whenever both of the factors are defined. It is straightforward to verify that $\mathbf{F}(\tz;T,w)$ satisfies a small-norm Riemann-Hilbert problem with the jump contour $\Sigma_{\mathbf{F}}$ consisting of the disk boundaries $\partial D_{\tz_1}(\delta)$, $\partial D_{\tz_2}(\delta)$, $\partial D_{\tz_0}(\delta)$, and $\partial D_{\tz_0^*}(\delta)$, together with the restrictions of the arcs $C^{\pm}_{\Gamma,L}$, $C^{\pm}_{\Gamma,R}$, $C^{\pm}_{\Sigma,L}$, and $C^{\pm}_{\Sigma,R}$ to the exterior of the four disks. 

The jump matrix $\mathbf{V}^{\mathbf{F}}$ for
$\mathbf{F}(\tz;T,w)$ is expressed on the latter arcs as
\begin{equation}
\begin{aligned}
\mathbf{V}^{\mathbf{F}}(\tz;T,w) &= \mathbf{F}_-(\tz;T,w)^{-1}  \mathbf{F}_+(\tz;T,w)\\
&= \dot{\mathbf{T}}_-(\tz;T,w){\mathbf{T}}_-(\tz;T,w)^{-1} {\mathbf{T}}_+(\tz;T,w)\dot{\mathbf{T}}_+(\tz;T,w)^{-1}\\
&=\dot{\mathbf{T}}^\mathrm{out}_-(\tz;T,w)\mathbf{T}_-(\tz;T,w)^{-1}\mathbf{T}_+(\tz;T,w)\dot{\mathbf{T}}^\mathrm{out}_-(\tz;T,w)^{-1},
\end{aligned}
\end{equation} 
because on these arcs $\dot{\mathbf{T}}(\tz;T,w)=\dot{\mathbf{T}}^\mathrm{out}(\tz;T,w)$, which has no jump discontinuity.  The restriction to the exterior of the disks makes the conjugating factors bounded independently of $T\to+\infty$, while the jump matrix $\mathbf{T}_-(\tz;T,w)^{-1}\mathbf{T}_+(\tz;T,w)$ for $\mathbf{T}(\tz;T,w)$ is a uniformly exponentially small perturbation of the identity.  Therefore there is a positive constant $K(\varepsilon)>0$ such that
\begin{equation}
\sup_{\tz\in\left(C^{\pm}_{\Gamma,L}\cup C^{\pm}_{\Gamma,R} \cup C^{\pm}_{\Sigma,L}\cup C^{\pm}_{\Sigma,R}\right)\cap \Sigma_\mathbf{F}} \| \mathbf{V}^\mathbf{F}(\tz;T,w) - \mathbb{I} \| = O(\ee^{-K(\varepsilon) T^{1/3}}),\quad T\to +\infty,
\end{equation}
holds uniformly for $|w| \le w_\mathrm{c}-\varepsilon$ and normalized parameters $(\mathfrak{a},\mathfrak{b})$ satisfying the double-sided inequality $\varepsilon\le \mathfrak{b}/\mathfrak{a}\le\varepsilon^{-1}$.  For the rest of the section we assume that these inequalities on $w$ and $\mathfrak{b}/\mathfrak{a}$ hold for some $\varepsilon>0$ and use the notation $O_\varepsilon(\diamond)$ introduced after \eqref{error-jump-estimate} to indicate the dependence of implied constants on $\varepsilon$.

Taking the disk boundary $\partial D_{\tz_0}(\delta)\subset\Sigma_\mathbf{F}$ to have clockwise orientation, the jump matrix on this circle takes the form $\mathbf{V}^\mathbf{F}(\tz;T,w)=\dot{\mathbf{T}}^{\tz_0}(\tz;T,w)\dot{\mathbf{T}}^\mathrm{out}(\tz;T,w)^{-1}$, and using \eqref{D-z0-bound} shows that $\mathbf{V}^\mathbf{F}-\mathbb{I}$ is uniformly $O_\varepsilon(T^{-\frac{1}{3}})$ on this circle.   By Schwarz reflection a similar estimate holds for $\mathbf{V}^\mathbf{F}-\mathbb{I}$ on $\partial D_{\tz_0^*}(\delta)\subset\Sigma_\mathbf{F}$.  

The discrepancy $\mathbf{V}^\mathbf{F}-\mathbb{I}$ is dominated by its behavior on the boundaries of the disks $D_{\tz_j}(\delta)$, $j=1,2$, which we also take to be clockwise-oriented.  On these two circles we have $\mathbf{V}^\mathbf{F}(\tz;T,w)=\dot{\mathbf{T}}^{\tz_j}(\tz;T,w)\dot{\mathbf{T}}^\mathrm{out}(\tz;T,w)^{-1}$.  
Therefore, the formul\ae\ \eqref{T-z2-mismatch} and \eqref{T-z1-mismatch} together with the basic estimate $\mathbf{U}(\zeta;p,\tau)\zeta^{\ii p\sigma_3}=\mathbb{I}+O_\varepsilon(\zeta^{-1})$ (see \eqref{U-PC-expansion}, here valid also with $(p,\tau)$ replaced by $(\bar{p},\bar{\tau})$ due to the double-sided inequality $\varepsilon\le\mathfrak{b}/\mathfrak{a}\le\varepsilon^{-1}$) where $\zeta$ is large of size $T^\frac{1}{6}$ when $\tz\in\partial D_{\tz_j}(\delta)$ immediately gives
\begin{equation}
\sup_{\tz\in\partial D_{\tz_1}(\delta) \cup   \partial D_{\tz_2}(\delta)} \| \mathbf{V}^\mathbf{F}(\tz;T,w) - \mathbb{I} \| = O_\varepsilon(T^{-\frac{1}{6}}),\quad T\to +\infty.
\end{equation}
This estimate is sharp, and we will extract a leading term proportional to $T^{-\frac{1}{6}}$ below.

 Just as in Section~\ref{s:asymptotics-X}, from the $L^2(\Sigma_\mathbf{F})$ theory of small-norm Riemann-Hilbert problems it follows that 
\begin{equation}
\mathbf{F}_{-}(\diamond ; T, w)-\mathbb{I}=O_\varepsilon(T^{-\frac{1}{6}}),\quad T \rightarrow+\infty
\label{T-F-minus-I-in-L2}
\end{equation} 
holds in the $L^2(\Sigma_\mathbf{F})$ sense. 
Therefore, again every coefficient 
\begin{equation}
\mathbf{F}^{[m]}(T, w):=-\frac{1}{2 \pi \ii} \int_{\Sigma_{\mathbf{F}}} \mathbf{F}_{-}(\tz; T, w)\left(\mathbf{V}^{\mathbf{F}}(\tz ; T, w)-\mathbb{I}\right) \tz^{m-1} \dd \tz
\label{T-F-coeffs}
\end{equation}
in the Laurent series for $\mathbf{F}(\tz ; T, w)$ 
\begin{equation}
\mathbf{F}(\tz ; T, w)=\mathbb{I}+\sum_{m=1}^{\infty} \tz^{-m} \mathbf{F}^{[m]}(T, w),
\end{equation}
which is convergent for $|\tz|$ sufficiently large, satisfies $\|\mathbf{F}^{[m]}(T, w)\|=O_\varepsilon(T^{-\frac{1}{6} })$ as $T \rightarrow+\infty$.

Now, note that $\mathbf{T}(\tz;T,w) = \mathbf{F}(\tz;T,w) \dot{\mathbf{T}}^{\mathrm{out}}(\tz;T,w)$, and we have $\dot{\mathbf{T}}(\tz;T,w)=\dot{\mathbf{T}}^{\mathrm{out}}(\tz;T,w)$ for $|\tz|$ large enough. Thus, from \eqref{eq:T-Psi-from-T} we obtain
\begin{multline}
\Psi(T^\frac{2}{3} w , T) \\
\begin{aligned}
&= 2 \ii \ee^{-\ii\arg(ab)} T^{-\frac{1}{3}}\lim _{\tz \rightarrow \infty} \tz T_{12}(\tz ; T, w)\ee^{\ii  T^{1/3} g(\tz ; w)}\\
&= 2 \ii \ee^{-\ii\arg(ab)} T^{-\frac{1}{3}}\lim _{\tz \rightarrow \infty} \tz \left( F_{11}(\tz;T,w) \dot{T}^{\mathrm{out}}_{12}(\tz;T,w) + F_{12}(\tz;T,w) \dot{T}^{\mathrm{out}}_{22}(\tz;T,w) \right)\ee^{\ii  T^{1/3} g(\tz ; w)}\\
&= 2 \ii \ee^{-\ii\arg(ab)} T^{-\frac{1}{3}}\lim _{\tz \rightarrow \infty} \tz \left( \dot{T}^{\mathrm{out}}_{12}(\tz;T,w) + F_{12}(\tz;T,w) \right)\\
&= 2 \ii \ee^{-\ii\arg(ab)} T^{-\frac{1}{3}}\left(\lim _{\tz \rightarrow \infty} \tz  \dot{T}^{\mathrm{out}}_{12}(\tz;T,w)  + F^{[1]}_{12}(T,w)\right),
\end{aligned}
\label{Psi-from-W-aspymptotics}
\end{multline}
where we have also used the properties $g(\tz;w)=O(\tz^{-1})$, $\dot{\mathbf{T}}^{\mathrm{out}}(\tz;T,w)-\mathbb{I} = O(\tz^{-1})$, and $\mathbf{F}(\tz;T,w)-\mathbb{I} = O(\tz^{-1})$ as $\tz\to\infty$. 

This has the form of a leading term plus a correction proportional to $T^{-\frac{1}{3}}F^{[1]}_{12}(T,w)$, which since $\|F^{[m]}(T,w)\|=O_\varepsilon(T^{-\frac{1}{6}})$ for all $m\ge 1$ is of size $O_\varepsilon(T^{-\frac{1}{2}})$.  We will now compute the leading term explicitly, and also expand the error term to obtain a sub-leading term.  For the leading term, we observe that
\begin{equation}
L_{12}(\tz;T,w) = -\frac{\Im(\tz_0(w))}{2\tz} \ee^{-\ii (T^{1/3}\kappa(w) +p \mu(w))} + O(\tz^{-2}),\quad \tz\to \infty,
\end{equation}
and that $\Im(\tz_0(w))= \tfrac{1}{3}\sqrt{w_\mathrm{c}^2-w^{2}}$ from \eqref{eq:Z0-intro}. Using this in \eqref{Psi-from-W-aspymptotics} while recalling \eqref{f-normalization} and the form of the outer parametrix given in \eqref{W-out-T} shows that the leading term is exactly
\begin{equation}
2\ii\ee^{-\ii\arg(ab)}T^{-\frac{1}{3}}\lim_{\tz\to\infty}\tz\dot{T}^\mathrm{out}_{12}(\tz;T,w)=-\ii\ee^{-\ii\arg(ab)}\ee^{2\ii q\gamma(w)}\ee^{-\ii (T^{1/3}\kappa(w)+p\mu(w))}T^{-\frac{1}{3}}\frac{1}{3}\sqrt{w_\mathrm{c}^2-w^2},
\label{eq:LeadingTerm-Large-T}
\end{equation}
in which $q=\ln(\mathfrak{a}/\mathfrak{b})=\ln(|a/b|)$ and $2\pi p=\ln(1+\mathfrak{b}^2/\mathfrak{a}^2)=\ln(1+|b/a|^2)$.

For the sub-leading term, we use \eqref{T-F-coeffs} to obtain
\begin{multline}
F^{[1]}_{12}(T,w)=-\frac{1}{2\pi\ii}\int_{\Sigma_\mathbf{F}}V^\mathbf{F}_{12}(\tz;T,w)\,\dd \tz\\
{}-\frac{1}{2\pi\ii}\int_{\Sigma_\mathbf{F}}((F_{11-}(\tz;T,w)-1)V^\mathbf{F}_{12}(\tz;T,w) + F_{12-}(\tz;T,w)(V^\mathbf{F}_{22}(\tz;T,w)-1))\,\dd \tz.
\end{multline}
Using \eqref{T-F-minus-I-in-L2} and $\mathbf{V}^\mathbf{F}(\diamond;T,w)-\mathbb{I}=O_\varepsilon(T^{-\frac{1}{6}})$ in $L^\infty(\Sigma_\mathbf{F})$ and hence also in $L^2(\Sigma_\mathbf{F})$ for $\Sigma_\mathbf{F}$ compact, Cauchy-Schwarz implies that
\begin{equation}
\begin{split}
F^{[1]}_{12}(T,w)&=-\frac{1}{2\pi\ii}\int_{\Sigma_\mathbf{F}}V^\mathbf{F}_{12}(\tz;T,w)\,\dd \tz + O_\varepsilon(T^{-\frac{1}{3}})\\
&=-\frac{1}{2\pi\ii}\int_{\partial D_{\tz_1}(\delta)\cup\partial D_{\tz_2}(\delta)}V^\mathbf{F}_{12}(\tz;T,w)\,\dd \tz + O_\varepsilon(T^{-\frac{1}{3}}),\quad T\to+\infty,
\end{split}
\label{T-F1-12}
\end{equation}
where on the second line we used the fact that $V^\mathbf{F}_{12}(\diamond;T,w)=O_\varepsilon(T^{-\frac{1}{3}})$ in $L^1(\Sigma^\mathbf{F}\setminus(\partial D_{\tz_1}(\delta)\cup\partial D_{\tz_2}(\delta)))$.
Note that in the situation described in Section~\ref{s:asymptotics-X}, $\mathbf{V}^\mathbf{F}-\mathbb{I}$ had additional structure allowing for a refinement of the analogous estimate; however that structure is  is not present here.  Now, $V^\mathbf{F}_{12}(\tz;T,w)$ is given by the $12$-element of \eqref{T-z2-mismatch} and \eqref{T-z1-mismatch} on $\partial D_{\tz_2}(\delta)$ and $\partial D_{\tz_1}(\delta)$, respectively.  Using \eqref{U-PC-expansion} and the fact that $\zeta_{\tz_j}$ is large of size $T^{\frac{1}{6}}$ on the disk boundaries, along with $\det(\mathbf{H}^{\tz_j}(\tz;T,w))=1$ for $j=1,2$,
\begin{equation}
V_{12}^\mathbf{F}(\tz;T,w)=\frac{T^{-\frac{1}{3}\ii \bar{p}}s(\bar{p},\bar{\tau})H_{12}^{\tz_1}(\tz;T;w)^2+T^{\frac{1}{3}\ii \bar{p}}r(\bar{p},\bar{\tau})H_{11}^{\tz_1}(\tz;T,w)^2}{2\ii T^\frac{1}{6}\varphi_{\tz_1}(\tz;w)}+O_\varepsilon(T^{-\frac{1}{3}}),\quad \tz\in\partial D_{\tz_1}(\delta),
\end{equation}
and
\begin{equation}
V_{12}^\mathbf{F}(\tz;T,w)=\frac{T^{-\frac{1}{3}\ii p}s(p,\tau)H_{12}^{\tz_2}(\tz;T;w)^2+T^{\frac{1}{3}\ii p}r(p,\tau)H_{11}^{\tz_2}(\tz;T,w)^2}{2\ii T^\frac{1}{6}\varphi_{\tz_2}(\tz;w)}+O_\varepsilon(T^{-\frac{1}{3}}),\quad \tz\in\partial D_{\tz_2}(\delta),
\end{equation}
with both estimates holding in the $L^\infty$ sense and hence also the $L^1$ sense.  Since $\mathbf{H}^{\tz_j}(\tz;T,w)$ is holomorphic in $D_{\tz_j}(\delta)$, and $Z\mapsto\varphi_{\tz_j}(\tz;w)$ is conformal at $\tz_j(w)$ with $\varphi_{\tz_j}(\tz_j(w);w)=0$, if $\delta>0$ is sufficiently small, we substitute into \eqref{T-F1-12} and compute by residues to obtain
\begin{multline}
F_{12}^{[1]}(T,w)=\frac{T^{-\frac{1}{3}\ii\bar{p}}s(\bar{p},\bar{\tau})H_{12}^{\tz_1}(\tz_1(w);T,w)^2+T^{\frac{1}{3}\ii \bar{p}}r(\bar{p},\bar{\tau})H_{11}^{\tz_1}(\tz_1(w);T,w)^2}{2\ii T^\frac{1}{6}\varphi'_{\tz_1}(\tz_1(w);w)}\\
{}+\frac{T^{-\frac{1}{3}\ii p}s(p,\tau)H_{12}^{\tz_2}(\tz_2(w);T,w)^2 + T^{\frac{1}{3}\ii p}r(p,\tau)H_{11}^{\tz_2}(\tz_2(w);T,w)^2}{2\ii T^{\frac{1}{6}}\varphi_{\tz_2}'(\tz_2(w);w)} + O_\varepsilon(T^{-\frac{1}{3}}).
\label{eq:T-F12-1-approx}
\end{multline}
Here, $r(p,\tau)$ and $s(p,\tau)$ are given by \eqref{r-s-polar}.  
By implicit differentiation of the equations \eqref{T-varphi-z2} and \eqref{T-varphi-z1} and using the facts that $\varphi_{\tz_j}(\tz_j(w);w)=0$ and $\varphi'_{\tz_2}(\tz_2(w);w)>0$ while $\varphi'_{\tz_1}(\tz_1(w);w)<0$ we obtain
\begin{equation}
\varphi'_{\tz_2}(\tz_2(w);w)=\sqrt{h''(\tz_2(w);w)}\quad\text{and}\quad\varphi_{\tz_1}'(\tz_1(w);w)=-\sqrt{-h_-''(\tz_1(w);w)}.
\end{equation}
Since $g(\tz;w)=O(\tz^{-1})$ and $R(\tz)=\tz+O(1)$ as $\tz\to\infty$, it follows from \eqref{large-T-phase}, \eqref{h-prime-squared}, and \eqref{T-h-g-theta} that
\begin{equation}
h'(\tz;w) =2 \frac{(\tz-\tz_1(w))(\tz-\tz_2(w))}{\tz^2}R(\tz;w),
\end{equation}
implying
\begin{equation}
h''(\tz;w) =2 \frac{(\tz_1(w)+\tz_2(w))\tz - 2 \tz_1(w) \tz_2(w)}{\tz^3}R(\tz;w) + 2\frac{(\tz-\tz_1(w))(\tz-\tz_2(w))}{\tz^2}R'(\tz;w).
\end{equation}
Thus,
\begin{align}
h''(\tz_2(w);w) &= \frac{2(\tz_2(w)-\tz_1(w))}{\tz_2(w)^2}|\tz_2(w)-\tz_0(w)|>0,\\
%\implies \varphi'_{z_2}(z_2(w);w)= \frac{\sqrt{z_2(w)-z_1(w)}}{z_2(w)}|z_2(w)-z_0(w)|^{1/2}>0
h''_-(\tz_1(w);w) &= \frac{2(\tz_1(w)-\tz_2(w))}{\tz_1(w)^2}|\tz_1(w)-\tz_0(w)|<0,
%\implies \varphi'_{z_1}(z_1(w);w)= \frac{\sqrt{z_2(w)-z_1(w)}}{z_1(w)}|z_1(w)-z_0(w)|^{1/2}>0
\end{align}
and accordingly
\begin{align}
\varphi'_{\tz_2}(\tz_2(w);w) &= \frac{\sqrt{2(\tz_2(w)-\tz_1(w))}}{\tz_2(w)}|\tz_2(w)-\tz_0(w)|^{\frac{1}{2}}>0,\label{phiz1-prime-atz1}\\
\varphi'_{\tz_1}(\tz_1(w);w) &= \frac{\sqrt{2(\tz_2(w)-\tz_1(w))}}{\tz_1(w)}|\tz_1(w)-\tz_0(w)|^{\frac{1}{2}}<0.\label{phiz2-prime-atz2}
\end{align}
Using $\varphi_{\tz_2}(\tz_2(w);w)=0$, we find from \eqref{T-Hz2-formula} that
\begin{multline}
%H_{11}^{\tz_2}(\tz_2(w);T,w)^2 = \ee^{2\ii q \gamma(w)} L_{11}(\tz_2(w);T,w)^2 \ee^{-2(p k(\tz_2(w);w) + q f(\tz_2(w);w))} \ee^{-2\ii T^{1/3}h(\tz_2(w);w)}\\
%\cdot (\tz_2(w)-\tz_1(w))^{2\ii p} \left[ \varphi'_{\tz_2}(\tz_2(w);w)^2\right]^{\ii p }
H_{11}^{\tz_2}(\tz_2(w);T,w)^2 = \ee^{2\ii q \gamma(w)} L_{11}(\tz_2(w);T,w)^2 \ee^{-2\ii(p \newk(\tz_2(w);w) + q \newf(\tz_2(w);w))} \ee^{-2\ii T^{1/3}h(\tz_2(w);w)}\\
\cdot (\tz_2(w)-\tz_1(w))^{2\ii p} \left[ \varphi'_{\tz_2}(\tz_2(w);w)^2\right]^{\ii p }
\label{T-Hz2-11-formula}
\end{multline}
and
\begin{multline}
%H_{12}^{\tz_2}(\tz_2(w);T,w)^2 = \ee^{2\ii q \gamma(w)} L_{12}(\tz_2(w);T,w)^2 \ee^{2(p k(\tz_2(w);w) + q f(\tz_2(w);w))} \ee^{2\ii T^{1/3}h(\tz_2(w);w)}\\
%\cdot (\tz_2(w)-\tz_1(w))^{-2\ii p} \left[ \varphi'_{\tz_2}(\tz_2(w);w)^2\right]^{-\ii p }.
H_{12}^{\tz_2}(\tz_2(w);T,w)^2 = \ee^{2\ii q \gamma(w)} L_{12}(\tz_2(w);T,w)^2 \ee^{2\ii(p \newk(\tz_2(w);w) + q \newf(\tz_2(w);w))} \ee^{2\ii T^{1/3}h(\tz_2(w);w)}\\
\cdot (\tz_2(w)-\tz_1(w))^{-2\ii p} \left[ \varphi'_{\tz_2}(\tz_2(w);w)^2\right]^{-\ii p }.
\label{T-Hz2-12-formula}
\end{multline}
Since $\tz_2(w)-\tz_1(w)>0$, we use \eqref{phiz2-prime-atz2} to write
\begin{align}
%(\tz_2(w)-\tz_1(w))^{\pm2 \ii p}&=\ee^{\pm 2 \ii p \ln(\tz_2(w)-\tz_1(w))},\\
%\left[ \varphi'_{\tz_2}(\tz_2(w);w)^2\right]^{\pm \ii p } &=\ee^{\pm \ii p \ln( {\color{red}2}(\tz_2(w)-\tz_1(w)) \tz_2(w)^{-2} |\tz_2(w)-\tz_0(w)|)},
(\tz_2(w)-\tz_1(w))^{\pm2 \ii p}&=\ee^{\pm 2 \ii p \ln(\tz_2(w)-\tz_1(w))},\\
\left[ \varphi'_{\tz_2}(\tz_2(w);w)^2\right]^{\pm \ii p } &=\ee^{\pm \ii p \ln( 2(\tz_2(w)-\tz_1(w)) \tz_2(w)^{-2} |\tz_2(w)-\tz_0(w)|)},
\label{eq:T-varphi-2-prime-pmip}
\end{align} 
and from the definition \eqref{L-def} with $y(\tz_2(w);w)=\ee^{\ii\arg(\tz_2(w)-\tz_0(w))/2}$ with the principal branch of the argument, we obtain using the amplitude notation from \eqref{eq:m-amplitudes-intro} in the introduction
\begin{equation}
\begin{split}
L_{11}(\tz_2(w);T,w)^2 &= \frac{1}{2} + \frac{1}{4}\left( y(\tz_2(w);w)^2 + \frac{1}{y(\tz_2(w);w)^2} \right)\\
& =
%\frac{1}{2}\left(1+\cos(\arg(	\tz_2(w)-\tz_0(w)))\right)=:
m^+_{\tz_2}(w),
\end{split}
\label{eq:T-L11-z2}
\end{equation}
\begin{equation}
\begin{split}
L_{12}(\tz_2(w);T,w)^2 %&= -\ee^{-2 \ii(T^{1/3} \kappa(w)+ p \mu(w))}\frac{1}{4}\left(y(z_2(w))-\frac{1}{y(z_2(w))} \right)^2 \\
&=-\ee^{-2 \ii(T^{1/3} \kappa(w)+ p \mu(w))} \left[ \frac{1}{4}\left( y(\tz_2(w);w)^2+\frac{1}{y(\tz_2(w);w)^2} \right) - \frac{1}{2}\right]\\
%&=\ee^{-2 \ii(T^{1/3} \kappa(w)+ p \mu(w))} \frac{1}{2}\left(1-\cos(\arg(\tz_2(w)-\tz_0(w)))\right)\\
&= \ee^{-2 \ii(T^{1/3} \kappa(w)+ p \mu(w))} m^-_{\tz_2}(w).%,\qquad m^-_{\tz_2}(w).
\end{split}
\label{eq:T-L12-z2}
\end{equation}
Note that $m_{\tz_2}^\pm(w)>0$ and $m_{\tz_2}^+(w)+m_{\tz_2}^-(w)=1$.
Thus, the formulae \eqref{T-Hz2-11-formula} and \eqref{T-Hz2-12-formula} can be expressed as
\begin{align}
H_{11}^{\tz_2}(\tz_2(w);T,w)^2 &=  \ee^{2\ii q \gamma(w)} m_{\tz_2}^+(w) \ee^{- \ii \phi_{\tz_2}(T,w)},\\
H_{12}^{\tz_2}(\tz_2(w);T,w)^2 &=  \ee^{2\ii q \gamma(w)} \ee^{-2 \ii(T^{1/3} \kappa(w)+ p \mu(w))}  m_{\tz_2}^-(w) \ee^{ \ii \phi_{\tz_2}(T,w)},
\end{align}
where a real phase is defined by
\begin{multline}
%\phi_{\tz_2}(T,w) := 2 T^{1/3} h(\tz_2(w);w) -3p \ln(\tz_2(w)-\tz_1(w)) - p \ln\left(\frac{|\tz_2(w)-\tz_0(w)|}{\tz_2(w)^2}\right)\\
%+2(p \newk(\tz_2(w);w) + q \newf(\tz_2(w);w)){\color{red}-p\ln(2)}.
\phi_{\tz_2}(T,w) := 2 T^{1/3} h(\tz_2(w);w) -3p \ln(\tz_2(w)-\tz_1(w)) - p \ln\left(\frac{|\tz_2(w)-\tz_0(w)|}{\tz_2(w)^2}\right)\\
+2(p \newk(\tz_2(w);w) + q \newf(\tz_2(w);w))-p\ln(2).
\label{phi-z2-large-T}
\end{multline}
Similarly, using $\varphi_{\tz_1}(\tz_1(w);w)=0$, we find from \eqref{T-Hz1-formula} that
\begin{multline}
%H_{11}^{\tz_1}(\tz_1(w);T,w)^2 = - \ee^{2\ii q \gamma(w)} L_{11+}(\tz_1(w);T,w)^2 \ee^{-2(p k^0(\tz_1(w);w) + q f^0(\tz_1(w);w))} \ee^{-2\ii T^{1/3}h_{+}(\tz_1(w);w)}\\
%\cdot (\tz_2(w)-\tz_1(w))^{-2\ii p} \left[ \frac{1}{\varphi'_{\tz_1}(\tz_1(w);w)^2}\right]^{\ii p }
H_{11}^{\tz_1}(\tz_1(w);T,w)^2 = - \ee^{2\ii q \gamma(w)} L_{11+}(\tz_1(w);T,w)^2 \ee^{-2\ii(p \newk^0(\tz_1(w);w) + q \newf^0(\tz_1(w);w))} \ee^{-2\ii T^{1/3}h_{+}(\tz_1(w);w)}\\
\cdot (\tz_2(w)-\tz_1(w))^{-2\ii p} \left[ \frac{1}{\varphi'_{\tz_1}(\tz_1(w);w)^2}\right]^{\ii p }
\label{eq:T-H11-z1-squared}
\end{multline}
and
\begin{multline}
%H_{12}^{\tz_1}(\tz_1(w);T,w)^2 = - \ee^{2\ii q \gamma(w)} L_{12+}(\tz_1(w);T,w)^2 \ee^{2(p k^0(\tz_1(w);w) + q f^0(\tz_1(w);w))} \ee^{2\ii T^{1/3}h_{+}(\tz_1(w);w)}\\
%\cdot (\tz_2(w)-\tz_1(w))^{2\ii p} \left[ \frac{1}{\varphi'_{\tz_1}(\tz_1(w);w)^2}\right]^{-\ii p }.
H_{12}^{\tz_1}(\tz_1(w);T,w)^2 = - \ee^{2\ii q \gamma(w)} L_{12+}(\tz_1(w);T,w)^2 \ee^{2\ii(p \newk^0(\tz_1(w);w) + q \newf^0(\tz_1(w);w))} \ee^{2\ii T^{1/3}h_{+}(\tz_1(w);w)}\\
\cdot (\tz_2(w)-\tz_1(w))^{2\ii p} \left[ \frac{1}{\varphi'_{\tz_1}(\tz_1(w);w)^2}\right]^{-\ii p }.
\label{eq:T-H12-z1-squared}
\end{multline}
The analogue of \eqref{eq:T-varphi-2-prime-pmip} needed here is
\begin{equation}
%\left[\frac{1}{\varphi'_{\tz_1}(\tz_1(w);w)^2}\right]^{\pm\ii p} = \ee^{\mp\ii p\ln({\color{red}2}(\tz_2(w)-\tz_1(w))\tz_1(w)^{-2}|\tz_1(w)-\tz_0(w)|)},
\left[\frac{1}{\varphi'_{\tz_1}(\tz_1(w);w)^2}\right]^{\pm\ii p} = \ee^{\mp\ii p\ln(2(\tz_2(w)-\tz_1(w))\tz_1(w)^{-2}|\tz_1(w)-\tz_0(w)|)},
\end{equation}
and those of \eqref{eq:T-L11-z2}--\eqref{eq:T-L12-z2}, using $y_+(\tz_1(w);w)=\ii\ee^{\ii\arg(\tz_1(w)-\tz_0(w))/2}$ with the principal branch of the argument, are
\begin{equation}
\begin{split}
L_{11+}(\tz_1(w);T,w)^2&=\frac{1}{2}+\frac{1}{4}\left(y_+(\tz_1(w);w)^2+\frac{1}{y_+(\tz_1(w);w))^2}\right)\\
&=
%\frac{1}{2}(1-\cos(\arg(\tz_1(w)-\tz_0(w))))=:
m_{\tz_1}^-(w),
\end{split}
\end{equation}
\begin{equation}
\begin{split}
L_{12+}(\tz_1(w);T,w)^2&=-\ee^{-2\ii (T^{1/3}\kappa(w)+p\mu(w))}\left[\frac{1}{4}\left(y_+(\tz_1(w);w)^2+\frac{1}{y_+(\tz_1(w);w)^2}\right)-\frac{1}{2}\right]\\
%&=\ee^{-2\ii (T^{1/3}\kappa(w)+p\mu(w))}\frac{1}{2}(1+\cos(\arg(\tz_1(w)-\tz_0(w))))\\
&=\ee^{-2\ii (T^{1/3}\kappa(w)+p\mu(w))}m_{\tz_1}^+(w), %,\quad m_{\tz_1}^+(w)>0.
\end{split}
\end{equation}
where we again recall the amplitude notation from \eqref{eq:m-amplitudes-intro} with $m^\pm_{\tz_1}(w)>0$ and $m^+_{\tz_1}(w)+m^-_{\tz_1}(w)=1$.
Therefore, \eqref{eq:T-H11-z1-squared} and \eqref{eq:T-H12-z1-squared} become
\begin{align}
H_{11}^{\tz_1}(\tz_1(w);T,w)^2&=-\ee^{2\ii q\gamma(w)}m^-_{\tz_1}(w)\ee^{-\ii\phi_{\tz_1}(T,w)},
\\
H_{12}^{\tz_1}(\tz_1(w);T,w)^2&=-\ee^{2\ii q\gamma(w)}\ee^{-2\ii(T^{1/3}\kappa(w)+p\mu(w))}m_{\tz_1}^+(w)\ee^{\ii\phi_{\tz_1}(T,w)},
\end{align}
where 
%$m_{\tz_1}^+(w)+m_{\tz_1}^-(w)=1$ and 
another real phase is defined by
\begin{multline}
%\phi_{\tz_1}(T,w):=2T^\frac{1}{3}h_+(\tz_1(w);w)+3p\ln(\tz_2(w)-\tz_1(w))+p\ln\left(\frac{|\tz_1(w)-\tz_0(w)|}{\tz_1(w)^2}\right)\\
%{}+2 (p\newk^0(\tz_1(w);w)+q\newf^0(\tz_1(w);w)){\color{red}+p\ln(2)}.
\phi_{\tz_1}(T,w):=2T^\frac{1}{3}h_+(\tz_1(w);w)+3p\ln(\tz_2(w)-\tz_1(w))+p\ln\left(\frac{|\tz_1(w)-\tz_0(w)|}{\tz_1(w)^2}\right)\\
{}+2 (p\newk^0(\tz_1(w);w)+q\newf^0(\tz_1(w);w))+p\ln(2).
\label{phi-z1-large-T}
\end{multline}

Using these results in \eqref{eq:T-F12-1-approx} together with the fact $s(p,\tau)=- r(p,\tau)^*$ gives
\begin{multline}
%F_{12}^{[1]}(T, w)= \frac{\ee^{2\ii q \gamma(w)}\ee^{-\ii (T^{1/3}\kappa(w)+ p\mu(w))}\ee^{\ii\frac{\pi}{2}}}{2\ii T^{\frac{1}{6}} }
%\left(
%\frac{|r(\bar{p},\bar{\tau})|}{\sqrt{-h_-''(\tz_1(w);w)}}\left( m_{\tz_1}^+(w)\ee^{\ii \Phi_{\tz_1}(T,w)} + m_{\tz_1}^-(w)\ee^{-\ii \Phi_{\tz_1}(T,w)}\right)\right.
%\\
%\left.+
% \frac{|r(p,\tau)|}{\sqrt{h''(\tz_2(w);w)}}\left( m_{\tz_2}^-(w)\ee^{\ii \Phi_{\tz_2}(T,w)} + m_{\tz_2}^+(w)\ee^{-\ii \Phi_{\tz_2}(T,w)}\right)
% \right) + O_\varepsilon(T^{-\frac{1}{3}}),
F_{12}^{[1]}(T,w)=\frac{\ee^{2\ii q\gamma(w)}\ee^{-\ii (T^{1/3}\kappa(w)+p\mu(w))}}{2T^\frac{1}{6}\sqrt{2(\tz_2(w)-\tz_1(w))}}
\left(\frac{|r(\bar{p},\bar{\tau})||\tz_1(w)|\left(m^+_{\tz_1}(w)\ee^{\ii\Phi_{\tz_1}(T,w)}+m^-_{\tz_1}(w)\ee^{-\ii\Phi_{\tz_1}(T,w)}\right)}{|\tz_1(w)-\tz_0(w)|^{\frac{1}{2}}}\right.\\
{}+\left.\frac{|r(p,\tau)|\tz_2(w)\left(m_{\tz_2}^-(w)\ee^{\ii\Phi_{\tz_2}(T,w)}+m_{\tz_2}^+(w)\ee^{-\ii\Phi_{\tz_2}(T,w)}\right)}{|\tz_2(w)-\tz_0(w)|^\frac{1}{2}}\right) + O_\varepsilon(T^{-\frac{1}{3}}).
\end{multline}
with the modified real phases
\begin{align}
\Phi_{\tz_1}(T,w)&:=\phi_{\tz_1}(T,w) - T^{\frac{1}{3}}\kappa(w) - p \mu(w) - \frac{1}{3}\bar{p}\ln(T)+\frac{\pi}{2} - \arg(r(\bar{p},\bar{\tau})),\\
\Phi_{\tz_2}(T,w)&:=\phi_{\tz_2}(T,w) - T^{\frac{1}{3}}\kappa(w) - p \mu(w) - \frac{1}{3}p \ln(T)+\frac{\pi}{2} - \arg(r(p,\tau)).
\end{align}
We note from the definition \eqref{r-def} that $\arg(r(\bar{p},\bar{\tau})) = \frac{\pi}{4} + \bar{p} \ln(2) - \arg(\Gamma(\ii \bar{p}))$ and $\arg(r(p,\tau)) = \frac{\pi}{4} + p \ln(2) - \arg(\Gamma(\ii p))$, so that the  modified real phases take the form
\begin{align}
\Phi_{\tz_1}(T,w)&:=\phi_{\tz_1}(T,w) - T^{\frac{1}{3}}\kappa(w) - p \mu(w) - \frac{1}{3}\bar{p}\ln(T)+\frac{\pi}{2} - \frac{\pi}{4} - \bar{p} \ln(2) + \arg(\Gamma(\ii \bar{p})),\\
\Phi_{\tz_2}(T,w)&:=\phi_{\tz_2}(T,w) - T^{\frac{1}{3}}\kappa(w) - p \mu(w) - \frac{1}{3}p \ln(T)+\frac{\pi}{2} - \frac{\pi}{4} - p \ln(2) + \arg(\Gamma(\ii p)).
\end{align}
Finally, substituting $|r(p,\tau)|=\sqrt{2p}$ from \eqref{r-s-polar} yields
\begin{multline}
%F_{12}^{[1]}(T, w)=  \frac{\ee^{2\ii q \gamma(w)}\ee^{-\ii (T^{1/3}\kappa(w)+ p\mu(w))} \ee^{\ii\frac{\pi}{2}} }{2\ii T^{\frac{1}{6}} }\left(
%\frac{\sqrt{2 \bar{p}}}{\sqrt{-h_-''(\tz_1(w);w)}}\left( m_{\tz_1}^+(w)\ee^{\ii \Phi_{\tz_1}(T,w)} + m_{\tz_1}^-(w)\ee^{-\ii \Phi_{\tz_1}(T,w)}\right)\right.
%\\
%\left.+
% \frac{\sqrt{2 p}}{\sqrt{h''(\tz_2(w);w)}}\left( m_{\tz_2}^-(w)\ee^{\ii \Phi_{\tz_2}(T,w)} + m_{\tz_2}^+(w)\ee^{-\ii \Phi_{\tz_2}(T,w)}\right)
% \right).
F_{12}^{[1]}(T,w)=\frac{\ee^{2\ii q\gamma(w)}\ee^{-\ii (T^{1/3}\kappa(w)+p\mu(w))}}{2T^\frac{1}{6}\sqrt{\tz_2(w)-\tz_1(w)}}
\left(\frac{\sqrt{\bar{p}}|\tz_1(w)|\left(m^+_{\tz_1}(w)\ee^{\ii\Phi_{\tz_1}(T,w)}+m^-_{\tz_1}(w)\ee^{-\ii\Phi_{\tz_1}(T,w)}\right)}{|\tz_1(w)-\tz_0(w)|^{\frac{1}{2}}}\right.\\
{}+\left.\frac{\sqrt{p}\tz_2(w)\left(m_{\tz_2}^-(w)\ee^{\ii\Phi_{\tz_2}(T,w)}+m_{\tz_2}^+(w)\ee^{-\ii\Phi_{\tz_2}(T,w)}\right)}{|\tz_2(w)-\tz_0(w)|^\frac{1}{2}}\right) + O_\varepsilon(T^{-\frac{1}{3}}).
\end{multline}
Using this in \eqref{Psi-from-W-aspymptotics} along with \eqref{eq:LeadingTerm-Large-T} yields
\begin{multline}
%\Psi(T^\frac{2}{3}w,T)=-\frac{1}{3}\ii \ee^{-\ii\arg(ab)}T^{-\frac{1}{3}}\ee^{-\ii T^{1/3}\kappa(w)}\ee^{\ii( 2q\gamma(w)-p\mu(w))}\left[\sqrt{54^\frac{2}{3}-w^2}
%\vphantom{{}-T^{-\frac{1}{6}}\left\{\frac{3\sqrt{2\overline{p}}}{\sqrt{-h''_-(\tz_1(w);w)}}\left(m_{\tz_1}^+(w)\ee^{\ii\Phi_{\tz_1}(T,w)} + m_{\tz_1}^-(w)\ee^{-\ii\Phi_{\tz_1}(T,w)}\right)\right.}
%\right.\\
%{}-T^{-\frac{1}{6}}\left\{\frac{3\sqrt{2\overline{p}}}{\sqrt{-h''_-(\tz_1(w);w)}}\left(m_{\tz_1}^+(w)\ee^{\ii\Phi_{\tz_1}(T,w)} + m_{\tz_1}^-(w)\ee^{-\ii\Phi_{\tz_1}(T,w)}\right)\right.\\
%\left.\left. + \frac{3\sqrt{2p}}{\sqrt{h''(\tz_2(w);w)}}\left(m_{\tz_2}^-(w)\ee^{\ii\Phi_{\tz_2}(T,w)} + m_{\tz_2}^+(w)\ee^{-\ii\Phi_{\tz_2}(T,w)}\right)\right\} + O(T^{-\frac{1}{3}})\right].
\Psi(T^\frac{2}{3}w,T)=-\frac{1}{3}\ii \ee^{-\ii\arg(ab)}T^{-\frac{1}{3}}\ee^{-\ii T^{1/3}\kappa(w)}\ee^{\ii( 2q\gamma(w)-p\mu(w))}\left[
\vphantom{\left\{\frac{3\sqrt{2\bar{p}}|\tz_1(w)|\left(m^+_{\tz_1}(w)\ee^{\ii\Phi_{\tz_1}(T,w)}+m^-_{\tz_1}(w)\ee^{-\ii\Phi_{\tz_1}(T,w)}\right)}{|\tz_1(w)-\tz_0(w)|^{\frac{1}{2}}\sqrt{2(\tz_2(w)-\tz_1(w))}}\right.}
\sqrt{w_\mathrm{c}^2-w^2}\right.\\
{}-3T^{-\frac{1}{6}}\frac{1}{\sqrt{\tz_2(w)-\tz_1(w)}}\left\{\frac{\sqrt{\bar{p}}|\tz_1(w)|\left(m^+_{\tz_1}(w)\ee^{\ii\Phi_{\tz_1}(T,w)}+m^-_{\tz_1}(w)\ee^{-\ii\Phi_{\tz_1}(T,w)}\right)}{|\tz_1(w)-\tz_0(w)|^{\frac{1}{2}}}\right.\\
{}+\left.\left.\frac{\sqrt{p}\tz_2(w)\left(m_{\tz_2}^-(w)\ee^{\ii\Phi_{\tz_2}(T,w)}+m_{\tz_2}^+(w)\ee^{-\ii\Phi_{\tz_2}(T,w)}\right)}{|\tz_2(w)-\tz_0(w)|^\frac{1}{2}}
\right\} + O_\varepsilon(T^{-\frac{1}{3}})\right].
\label{large-T-formula}
\end{multline}
Simplifying $2q\gamma(w)-p\mu(w)-\pi/2$ into the form of $\Phi(w)$ given in \eqref{eq:Phi-intro}, and simplifying the phases $\Phi_{\tz_1}(T,w)$ and $\Phi_{\tz_2}(T,w)$ into the forms \eqref{eq:Phi-Z1-intro} and \eqref{eq:Phi-Z2-intro} respectively, we complete the proof of Theorem~\ref{t:large-T}.

\section{Transitional asymptotics for $\Psi(X,T;\mathbf{G})$}
\label{s:transitional}
Now we analyze $\Psi(X,T;\mathbf{G}(a,b))$ for large positive $(X,T)$ in the regime that $T\approx v_\mathrm{c}X^{\frac{3}{2}}$.  Due to Proposition~\ref{prop:a-b-scaling} we will use normalized parameters \eqref{eq:normalized-ab} $\mathfrak{a},\mathfrak{b}$ and write $\Psi(X,T;\mathbf{G}(a,b))=\ee^{-\ii\arg(ab)}\Psi(X,T;\mathbf{G}(\mathfrak{a},\mathfrak{b}))$.  Therefore, in the terminology of Theorem~\ref{t:large-X}, the parameter $v:=TX^{-\frac{3}{2}}$ should be allowed to increase into a neighborhood of $v=v_\mathrm{c}:=54^{-\frac{1}{2}}$.  Assuming that for some fixed $\varepsilon>0$ arbitrarily small we have $\tau:=\mathfrak{b}/\mathfrak{a}<\varepsilon^{-1}$, we will here adapt the analysis from Section~\ref{s:large-X} to allow for $v\approx v_\mathrm{c}$.  

As $v\uparrow v_\mathrm{c}$, a third real simple critical point of $z\mapsto\vartheta(z;v)$ collides with $z_1(v)$ to form a double critical point at $z=z_\mathrm{c}=-\sqrt{6}$, while the simple critical point at $z_2(v)$ persists with the limiting value of $z_2(v_\mathrm{c})=(\frac{3}{2})^{\frac{1}{2}}$.  Following \cite[Section 4.3]{BilmanLM2020}, we begin by introducing a Schwarz-symmetric conformal mapping $z\mapsto \transconf(z;v)$ 
on a neighborhood of $z=z_\mathrm{c}$ by means of the equation 
\begin{equation}
2\vartheta(z;v)=\transconf^3+\transr\transconf-\transs
\end{equation}
where $\transr=\transr(v)$ and $\transs=\transs(v)$ 
are real-analytic functions of $v\approx v_\mathrm{c}$ determined such that the two critical points of the left-hand side near $z=z_\mathrm{c}$ correspond to the two critical points of the cubic on the right-hand side.  These functions satisfy 
\begin{equation}
\transr_\mathrm{c}:=\transr(v_\mathrm{c})=0,\quad \transr_\mathrm{c}':=\transr'(v_\mathrm{c})=4\cdot 6^{\frac{1}{2}}9^{\frac{1}{3}},\quad \transs_\mathrm{c}:=\transs(v_\mathrm{c})=2\cdot 6^{\frac{1}{2}},\quad \transs'_\mathrm{c}:=\transs'(v_\mathrm{c})=-12.
\label{eq:Transitional-XT-critical-values}
\end{equation}
The conformal map $z\mapsto \transconf(z;v)$ is locally a dilation and reflection through $z_\mathrm{c}$; indeed $\transconf_\mathrm{c}':=\transconf'(z_\mathrm{c};v_\mathrm{c})=-9^{-\frac{1}{3}}<0$.  We denote by $z_*(v)$ the preimage under $z\mapsto\transconf(z;v)$ of $\transconf=0$ for general $v\approx v_\mathrm{c}$; it is an analytic function of $v$ with $z_*(v_\mathrm{c})=z_\mathrm{c}$.
\begin{figure}[h]
\includegraphics[width=0.45\textwidth]{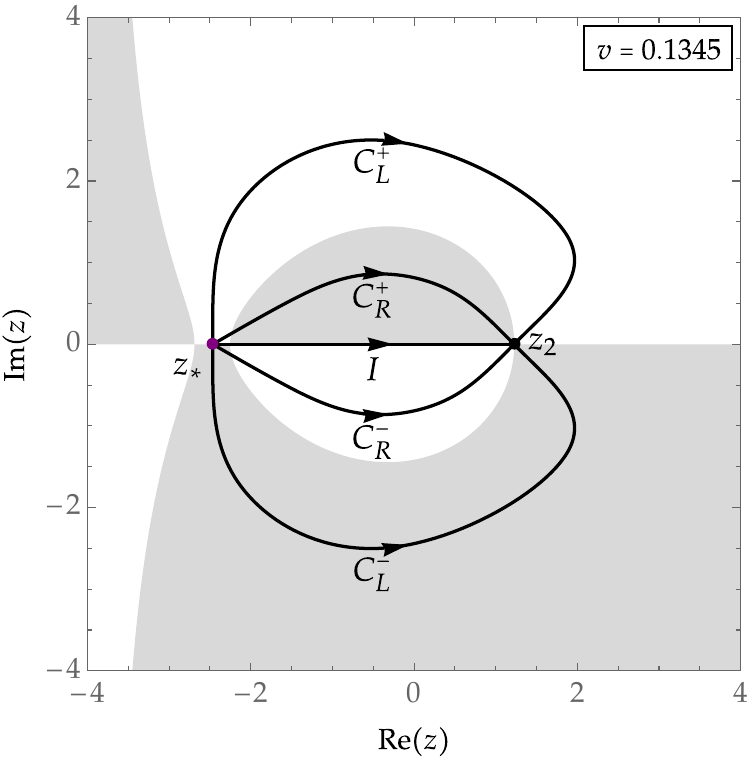}
\includegraphics[width=0.45\textwidth]{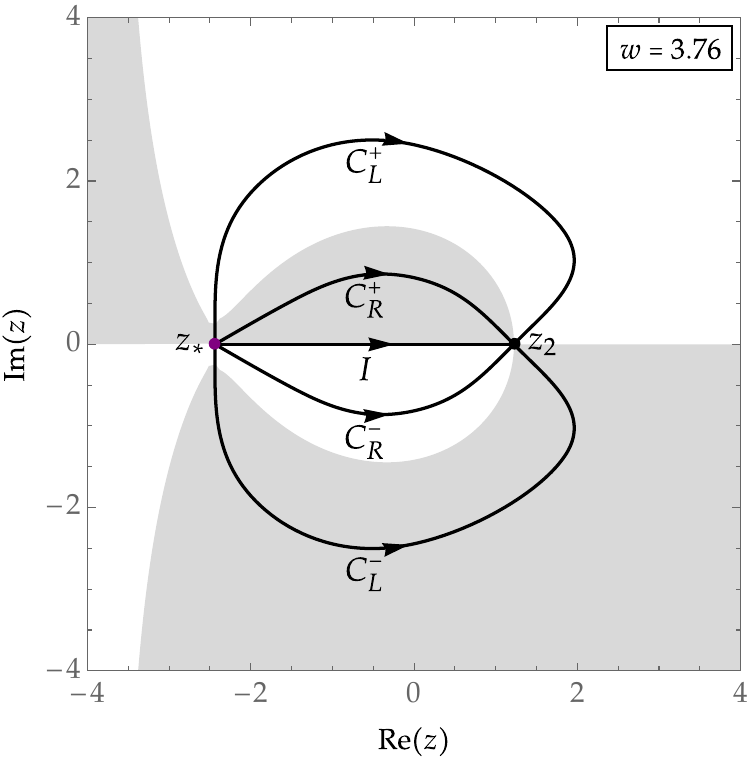}
\caption{Left: the jump contours in the $z$-plane when $v=0.1345$ near $v_{\mathrm{c}}$ using the points $z_2(v)$ and $z_{*}(v)$ overlayed with the regions where $\Im(\vartheta(z;v))$ has a definite sign. Right: the jump contours in the $z$-plane when $w=3.76$ near $w_{\mathrm{c}}$ using the points $z_2(v)$ and $z_{*}(v)$ overlayed with the regions where $\Im(h(Z;w))$ has a definite sign. This is plotted in the $z$-plane using the relation $z=Z/v^{\frac{1}{3}}$ and the points $z_2(v)$ and $z_{*}(v)$ are found using the relation $v=w^{-\frac{3}{2}}$.}
\label{fig:Painleve-contour}
\end{figure}

Next, we modify the outer parametrix defined in Section~\ref{s:large-X-outer-parametrix} simply by replacing the point $z_1(v)$ with $z_*(v)$ 
\begin{equation}
\dot{\mathbf{T}}^\mathrm{out}(z)=\dot{\mathbf{T}}^\mathrm{out}(z;v):=\left(\frac{z-z_*(v)}{z-z_2(v)}\right)^{\ii p\sigma_3},\quad p:=\frac{1}{2\pi}\ln(1+\tau^2)>0,\quad z\in \mathbb{C}\setminus [z_*(v),z_2(v)].
\end{equation}
The definition of an inner parametrix near the simple critical point $z=z_2(v)$ is given by \eqref{T-in-z2}, \eqref{T-out-near-z2}, and \eqref{T-z2-def} with only one small alteration:  the factor $\mathbf{H}^{z_2}(z;v)$ holomorphic near $z=z_2(v)$ is modified from its definition in \eqref{T-out-near-z2} only by replacing $z_1(v)$ with $z_*(v)$.  On the other hand, we will need an inner parametrix near $z=z_\mathrm{c}$ that is no longer constructed from parabolic cylinder functions at all.

We now explain how to construct such an inner parametrix.  The exact jump conditions satisfied by $\mathbf{U}^{\mathrm{c}}:=\mathbf{T}\ee^{\ii X^{1/2}\transs(v)\sigma_3/2}(\ii\sigma_2)$ near $z=z_\mathrm{c}$ can be expressed in terms of the conformal coordinate via $\zeta:=X^{\frac{1}{6}}\transconf$ and the rescaled parameter $y:=X^{\frac{1}{3}}\transr(v)$.  Locally we take the jump contours to coincide with the rays $\arg(\zeta)=\pm\frac{1}{2}\pi$, $\arg(\zeta)=\pm\frac{5}{6}\pi$, and $\arg(-\zeta)=0$, and then the jump conditions are as shown in Figure~\ref{fig:PIItritronquee}.
\begin{figure}[h]
\includegraphics[width=0.5\textwidth]{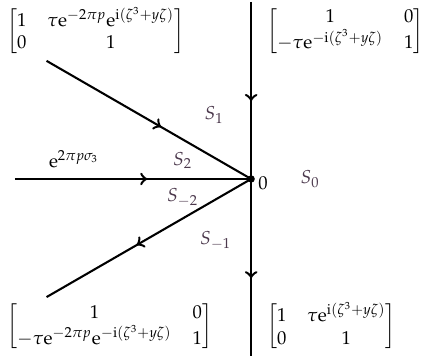}
\caption{The jump contours and jump matrices near $z=z_*(v)$ take the form shown here when expressed in terms of the rescaled conformal coordinate $\zeta$ for which $\zeta=0$ is the image of $z=z_*(v)$.  Compare with \cite[Figure 3]{Miller2018}.}
\label{fig:PIItritronquee}
\end{figure}
Now let $\mathbf{U}^\mathrm{TT}(\zeta;y,\tau)$ denote the solution of \cite[Riemann-Hilbert Problem 2.1]{Miller2018} (denoted by $\mathbf{W}(\zeta;y)$ in that reference).  By a simple generalization of the argument given in \cite[Section 5]{Miller2018} to arbitrary $\tau>0$, it exists for all $y\in\mathbb{R}$ and $\tau>0$.  It has the following properties.
\begin{itemize}
\item $\mathbf{U}^\mathrm{TT}(\zeta;y,\tau)$ is analytic in the five sectors shown in Figure~\ref{fig:PIItritronquee}, which are $S_0:  |\arg(\zeta)|<\frac{1}{2}\pi$, $S_1:  \frac{1}{2}\pi<\arg(\zeta)<\frac{5}{6}\pi$, $S_{-1}:  -\frac{5}{6}\pi<\arg(\zeta)<-\frac{1}{2}\pi$, $S_2:  \frac{5}{6}\pi<\arg(\zeta)<\pi$, and $S_{-2}:  -\pi<\arg(\zeta)<-\frac{5}{6}\pi$.
\item $\mathbf{U}^\mathrm{TT}(\zeta;y,\tau)$ takes continuous boundary values on the excluded rays and at the origin from each of the five sectors, which are related by the jump condition $\mathbf{U}^\mathrm{TT}_+(\zeta;y,\tau)=\mathbf{U}^\mathrm{TT}_-(\zeta;y,\tau)\mathbf{V}^\mathrm{TT}(\zeta;y,\tau)$, where the jump contours and the jump matrix $\mathbf{V}^\mathrm{TT}(\zeta;y,\tau)$ are given in Figure~\ref{fig:PIItritronquee}.
\item The matrix function $\mathbf{U}^\mathrm{TT}(\zeta;y,\tau)\zeta^{\ii p\sigma_3}$ has a complete asymptotic expansion in descending integer powers of $\zeta$ as $\zeta\to\infty$, with coefficients that are independent of the sector in which $\zeta\to\infty$.  In particular,
\begin{equation}
\mathbf{U}^\mathrm{TT}(\zeta;y,\tau)\zeta^{\ii p\sigma_3}=\mathbb{I} + \mathbf{U}^1(y,\tau)\zeta^{-1}+O(\zeta^{-2}),\quad\zeta\to\infty
\label{eq:Transitional-XT-UTT-expansion}
\end{equation}
holds uniformly for bounded $\tau>0$.
For each $\tau>0$, the function 
\begin{equation}
\mathcal{V}(y;\tau):=U_{21}^1(y,\tau)
\label{eq:Transitional-XT-V}
\end{equation}
is analytic for all $y\in\mathbb{R}$ and has no real zeros or critical points.
\end{itemize}
The reason for the notation ``TT'' is that $\mathcal{V}(y;\tau)$ is connected with an 
increasing tritronqu\'ee solution of the Painlev\'e-II equation \eqref{eq:PII} as explained in the introduction.  
The analogue of \eqref{T-out-near-z1} in this case is
\begin{equation}
\begin{aligned}
\dot{\mathbf{T}}^{\mathrm{out}}(z;v)\ee^{\ii X^{1/2}\transs(v)\sigma_3/2}(\ii\sigma_2)&=X^{-\ii p\sigma_3/6}\ee^{\ii X^{1/2}\transs(v)\sigma_3/2}\mathbf{H}^{\mathrm{c}}(z;v)\zeta^{-\ii p\sigma_3},\\
\mathbf{H}^\mathrm{c}(z;v)&:=(z_2(v)-z)^{-\ii p\sigma_3}\left(\frac{z_*(v)-z}{\transconf(z;v)}\right)^{\ii p\sigma_3}(\ii\sigma_2),
\end{aligned}
\label{eq:Transitional-XT-Hc}
\end{equation}
in which $\zeta=X^{\frac{1}{6}}\transconf(z;v)$.
The function $z\mapsto\mathbf{H}^\mathrm{c}(z;v)$ is analytic for $z$ near $z_\mathrm{c}$, and we use it together with $\mathbf{U}^\mathrm{TT}(\zeta;y,\tau)$ to build an inner parametrix by setting for $z\in D_{z_\mathrm{c}}(\delta)$,
\begin{equation}
\dot{\mathbf{T}}^\mathrm{c}(z;X,v):=X^{-\ii p\sigma_3/6}\ee^{\ii X^{1/2}\transs(v)\sigma_3/2}\mathbf{H}^\mathrm{c}(z;v)\mathbf{U}^{\mathrm{TT}}(X^{\frac{1}{6}}\transconf(z;v);X^{\frac{1}{3}}\transr(v))(-\ii\sigma_2)\ee^{-\ii X^{1/2}\transs(v)\sigma_3/2}.
\end{equation}
This is the correct analogue of \eqref{T-z1-def} in the situation that $v\approx v_\mathrm{c}$.

We now define a global parametrix $\dot{\mathbf{T}}(z;X,v)$ for $\mathbf{T}(z;X,v)$ by direct analogy with \eqref{eq:large-X-global-parametrix}:
\begin{equation}
\dot{\mathbf{T}}(z;X,v):=\begin{cases}
\dot{\mathbf{T}}^{\mathrm{c}}(z;X,v),&\quad z\in D_{z_\mathrm{c}}(\delta),\\
\dot{\mathbf{T}}^{z_2}(z;X,v),&\quad z\in D_{z_2}(\delta),\\
\dot{\mathbf{T}}^{\mathrm{out}}(z;v),&\quad z \in \mathbb{C} \setminus \left([z_*(v),z_2(v)] \cup D_{z_\mathrm{c}}(\delta) \cup D_{z_2}(\delta)\right).
\end{cases}
\label{eq:Transitional-XT-global-parametrix}
\end{equation}
Defining the error as $\mathbf{F}(z;X,v):=\mathbf{T}(z;X,v)\dot{\mathbf{T}}(z;X,v)^{-1}$ wherever both factors make sense, we see that $z\mapsto\mathbf{F}(z;X,v)$ is analytic on the complement of a bounded jump contour that is a version of the contour $\Sigma_\mathbf{F}$ defined in Section~\ref{s:asymptotics-X}, and at each non-self-intersection point $z\in\Sigma_\mathbf{F}$ there is a well-defined jump matrix $\mathbf{V}^\mathbf{F}(z;X,v)$ such that $\mathbf{F}_+(z;X,v)=\mathbf{F}_-(z;X,v)\mathbf{V}^\mathbf{F}(z;X,v)$.  The arguments of Section~\ref{s:asymptotics-X} then imply that 
\begin{equation}
\sup_{z\in\Sigma_\mathbf{F}\setminus\partial D_{z_\mathrm{c}}(\delta)}\|\mathbf{V}^\mathbf{F}(z;X,v)-\mathbb{I}\|=O_\varepsilon(X^{-\frac{1}{4}}),\quad X\to+\infty
\label{eq:Transitional-XT-error-away}
\end{equation}
holds uniformly for $v\approx v_\mathrm{c}$ and $\tau=\mathfrak{b}/\mathfrak{a}<\varepsilon^{-1}$.  However, using $\mathbf{V}^\mathbf{F}(z;X,v):=\dot{\mathbf{T}}^\mathrm{c}(z;X,v)\dot{\mathbf{T}}^\mathrm{out}(z;v)^{-1}$ for $z\in\partial D_{z_\mathrm{c}}(\delta)$ then shows that if $v=v_\mathrm{c}+O(X^{-\frac{1}{3}})$ so that $y$ is bounded, the sharp estimate $\mathbf{V}^\mathbf{F}(z;X,v)-\mathbb{I}=O_\varepsilon(X^{-\frac{1}{6}})$ holds uniformly for $z\in\partial D_{z_\mathrm{c}}(\delta)$, which is small but which dominates all other contributions to $\mathbf{V}^\mathbf{F}-\mathbb{I}$.  Applying the small-norm theory, it follows that the analogue of \eqref{E-minus-estimate} in the present situation is that $\mathbf{F}_-(\diamond;X,v)-\mathbb{I}=O_\varepsilon(X^{-\frac{1}{6}})$ holds in the $L^2(\Sigma_\mathbf{F})$ sense.  Using this and the fact that $\mathbf{V}^\mathbf{F}(\diamond;X,v)-\mathbb{I}=O_\varepsilon(X^{-\frac{1}{6}})$ also holds in the same topology, by Cauchy-Schwarz, the exact formula \eqref{Psi-large-X-exact} implies that the analogue of \eqref{Psi-large-X-approx-2} in this situation is that
\begin{equation}
\Psi(X,X^\frac{3}{2}v)=-\frac{\ee^{-\ii\arg(ab)}}{\pi X^\frac{1}{2}}\int_{\Sigma_\mathbf{F}}V_{12}^\mathbf{F}(z;X,v)\,\dd z + O(X^{-\frac{5}{6}}),\quad X\to+\infty,\quad v=v_\mathrm{c}+O_\varepsilon(X^{-\frac{1}{3}}).
\label{eq:Psi-trans-SigmaF}
\end{equation}
Since the integrand is uniformly exponentially small unless $z\in\partial D_{z_\mathrm{c}}(\delta)\cup\partial D_{z_2}(\delta)$ we can simplify further:
\begin{multline}
\Psi(X,X^\frac{3}{2}v)=-\frac{\ee^{-\ii\arg(ab)}}{\pi X^\frac{1}{2}}\int_{\partial D_{z_\mathrm{c}}(\delta)\cup\partial D_{z_2}(\delta)}V_{12}^\mathbf{F}(z;X,v)\,\dd z + O_\varepsilon(X^{-\frac{5}{6}}),\\ X\to+\infty,\quad v=v_\mathrm{c}+O(X^{-\frac{1}{3}}).
\label{eq:Psi-trans-two-circles}
\end{multline}

As in Section~\ref{s:asymptotics-X}, we then compute the integral over $\partial D_{z_2}(\delta)$ in \eqref{eq:Psi-trans-two-circles}, which we identify up to an error of order $O_\varepsilon(X^{-\frac{5}{4}})$ as the second term on the right-hand side of \eqref{Psi-large-X-approx-4} in which $H^{z_2}_{11}(z_2(v);v)$ is modified from its definition in \eqref{T-out-near-z2} just replacing $z_1(v)$ with $z_*(v)$.  In other words,
\begin{multline}
-\frac{\ee^{-\ii\arg(ab)}}{\pi X^\frac{1}{2}}\int_{\partial D_{z_2(\delta)}}V_{12}^\mathbf{F}(w;X,v)\,\dd w\\
{}=
\frac{\ee^{-\ii\arg(ab)}}{X^\frac{3}{4}}X^{\frac{1}{2}\ii p}\ee^{-2\ii X^{1/2}\vartheta(z_2(v);v)}
\frac{\sqrt{2p}\ee^{\ii(\frac{1}{4}\pi+p\ln(2)-\arg(\Gamma(\ii p)))}}{\sqrt{\vartheta''(z_2(v);v)}}(z_2(v)-z_*(v))^{2\ii p}\vartheta''(z_2(v);v)^{\ii p}
\\
{}+O_\varepsilon(X^{-\frac{5}{4}}).
\end{multline}
By Taylor expansion, $v=v_\mathrm{c}+O(X^{-\frac{1}{3}})$ implies
\begin{equation}
\begin{split}
\vartheta(z_2(v);v)&=\sqrt{\frac{75}{8}}+\frac{3}{2}(v-v_\mathrm{c}) + O(X^{-\frac{2}{3}})\\
\vartheta''(z_2(v);v)&= \sqrt{6}+O(X^{-\frac{1}{3}})\\
z_2(v)-z_*(v)&= \sqrt{\frac{27}{2}} + O(X^{-\frac{1}{3}}),
\end{split}
\end{equation}
so it follows that 
\begin{multline}
-\frac{\ee^{-\ii\arg(ab)}}{\pi X^\frac{1}{2}}\int_{\partial D_{z_2}(\delta)}V_{12}^\mathbf{F}(w;X,v)\,\dd w =
2^\frac{1}{4}3^{-\frac{1}{4}}p^\frac{1}{2}X^{-\frac{3}{4}}
%\frac{\sqrt{2p}}{6^\frac{1}{4}X^{\frac{3}{4}}}
\ee^{\ii\Omega_{2}(X,v)} + O_\varepsilon(X^{-\frac{11}{12}}),\\ X\to+\infty,\quad v=v_\mathrm{c}+O(X^{-\frac{1}{3}}),
\label{eq:2nd-term-trans}
\end{multline}
wherein the phase $\Omega_2(X,v)$ is defined by \eqref{eq:trans-phase-2}.

Next, we compute the integral over $\partial D_{z_\mathrm{c}}(\delta)$ in \eqref{eq:Psi-trans-two-circles}.  Using the expansion \eqref{eq:Transitional-XT-UTT-expansion} and the representation \eqref{eq:Transitional-XT-Hc} of the outer parametrix within $D_{z_\mathrm{c}}(\delta)$, we obtain that as $X\to+\infty$,
\begin{equation}
\begin{split}
V^\mathbf{F}_{12}(z;X,v)&=
\frac{ X^{-\frac{1}{3}\ii p}\ee^{\ii X^{1/2}\transs(v)}}{X^{\frac{1}{6}}\transconf(z;v)}\left[\mathbf{H}^\mathrm{c}(z;v)\mathbf{U}^1(X^{\frac{1}{3}}\transr(v),\tau)\mathbf{H}^\mathrm{c}(z;v)^{-1}\right]_{12} + O_\varepsilon(X^{-\frac{1}{3}})\\
&=-\frac{X^{-\frac{1}{3}\ii p}\ee^{\ii X^{1/2}\transs(v)}U^1_{21}(X^{\frac{1}{3}}\transr(v),\tau)}{X^{\frac{1}{6}}\transconf(z;v)}(z_2(v)-z)^{-2\ii p}\left(\frac{z_*(v)-z}{\transconf(z;v)}\right)^{2\ii p} + O_\varepsilon(X^{-\frac{1}{3}})\\
&=-\frac{X^{-\frac{1}{3}\ii p}\ee^{\ii X^{1/2}\transs(v)}\mathcal{V}(X^{\frac{1}{3}}\transr(v);\tau)}{X^{\frac{1}{6}}\transconf(z;v)}(z_2(v)-z)^{-2\ii p}\left(\frac{z_*(v)-z}{\transconf(z;v)}\right)^{2\ii p} + O_\varepsilon(X^{-\frac{1}{3}}).
\end{split}
\end{equation}
Therefore,
\begin{multline}
-\frac{\ee^{-\ii\arg(ab)}}{\pi X^\frac{1}{2}}\int_{\partial D_{z_\mathrm{c}}(\delta)}V_{12}^\mathbf{F}(w;X,v)\,\dd w\\=\frac{\ee^{-\ii\arg(ab)}X^{-\frac{1}{3}\ii p}\ee^{\ii X^{1/2}\transs(v)}\mathcal{V}(X^{\frac{1}{3}}\transr(v);\tau)}{\pi X^{\frac{2}{3}}}\int_{\partial D_{z_\mathrm{c}}(\delta)}
(z_2(v)-w)^{-2\ii p}\left(\frac{z_*(v)-w}{\transconf(w;v)}\right)^{2\ii p}\frac{\dd w}{\transconf(w;v)}\\{} + O_\varepsilon(X^{-\frac{5}{6}}),\quad X\to+\infty.
\label{eq:Psi-trans-half}
\end{multline}
We evaluate the integral by residues, using the fact that $w\mapsto \transconf(w;v)$ has a simple zero at $w=z_*(v)\in D_{z_\mathrm{c}}(\delta)$ and the first two factors in the integrand are analytic at $w=z_*(v)$ with value
\begin{equation}
\left.(z_2(v)-w)^{-2\ii p}\left(\frac{z_*(v)-w}{\transconf(w;v)}\right)^{2\ii p}\right|_{w=z_*(v)} = (z_2(v)-z_*(v))^{-2\ii p}(-\transconf'(z_*(v);v))^{-2\ii p}.
\end{equation}
Further expanding the result in $v$ about $v=v_\mathrm{c}$ using \eqref{eq:Transitional-XT-critical-values}, $\transconf'(z_\mathrm{c};v_\mathrm{c})=-9^{-\frac{1}{3}}$,  $z_2(v_\mathrm{c})=(\frac{3}{2})^{\frac{1}{2}}$, and the assumption that $v-v_\mathrm{c}=O(X^{-\frac{1}{3}})$, we produce no error terms larger than already present in \eqref{eq:Psi-trans-half} and hence
\begin{multline}
-\frac{\ee^{-\ii\arg(ab)}}{\pi X^\frac{1}{2}}\int_{\partial D_{z_\mathrm{c}}(\delta)}V_{12}^\mathbf{F}(w;X,v)\,\dd w=2\cdot 3^{\frac{2}{3}}X^{-\frac{2}{3}}\mathcal{V}(2^{\frac{5}{2}}3^{\frac{7}{6}}X^{\frac{1}{3}}(v-v_\mathrm{c});\tau)\ee^{\ii\Omega_\mathrm{c}(X,v)} + O_\varepsilon(X^{-\frac{5}{6}}),\\ X\to+\infty,\quad v=v_\mathrm{c}+O(X^{-\frac{1}{3}}),
\label{eq:Psi-trans-1}
\end{multline}
where the phase $\Omega_\mathrm{c}(X,v)$ is as defined in \eqref{eq:trans-phase-c}.

Using \eqref{eq:2nd-term-trans} and \eqref{eq:Psi-trans-1} in \eqref{eq:Psi-trans-two-circles}, we obtain
\eqref{eq:Psi-transitional-final}.

\section{The \texttt{RogueWaveInfiniteNLS.jl} software package for \texttt{Julia}}
\label{s:Numerics}
In this section we introduce the software package \texttt{RogueWaveInfiniteNLS.jl} \cite{RogueWaveInfiniteNLS} for \texttt{Julia} that is developed as part of this work to accurately compute $\Psi(X,T;\mathbf{G},\bg)$ at virtually any given point in the $(X,T)$-plane for given $\mathbf{G}$ and $\bg>0$. 
The method for computing $\Psi(X,T;\mathbf{G},\bg)$ is based on numerically solving a suitably regularized Riemann-Hilbert problem that is selected depending on the location of the point $(X,T)$ in $\mathbb{R}^2$.  Basically,
\begin{itemize}
\item
if $X$ is deemed to be sufficiently large and $|T|<v_\mathrm{c}|X|^\frac{3}{2}$, then we numerically solve for $\mathbf{T}$ described in Section~\ref{s:large-X};
\item
if $T$ is deemed to be sufficiently large and $|X|<w_\mathrm{c}|T|^\frac{2}{3}$, then we numerically solve for the different matrix $\mathbf{T}$ described in Section~\ref{s:large-T};
\item
if either $X$ or $T$ is deemed to be sufficiently large and $|T|\approx v_\mathrm{c}|X|^\frac{3}{2}$ or equivalently $|X|\approx w_\mathrm{c}|T|^\frac{2}{3}$, then we numerically solve for $\mathbf{T}$ defined as in Section~\ref{s:large-X} but with more suitable contour choices as described in Section~\ref{s:transitional};
\item
otherwise, we numerically solve a version of Riemann-Hilbert Problem~\ref{rhp:near-field} for $\mathbf{P}$.
\end{itemize}
Since the selected Riemann-Hilbert problem 
has then been appropriately deformed (via employing noncommutatitve steepest descent techniques) to be suitable for asymptotic analysis in one of the asymptotic regimes considered in this work, it is also good for computation, modulo some details due to numerical considerations. 
This final Riemann-Hilbert problem is solved using the routines available in the \texttt{OperatorApproximation.jl} package, see \cite{OperatorApproximation}.  \texttt{OperatorApproximation.jl} is a framework for approximating functions and operators, and for solving equations involving such objects.

Numerical solution of a Riemann-Hilbert problem posed on a suitable oriented contour $\Gamma$ (may be open, closed, or unbounded) for an unknown $\mathbf{\Phi}(z)\in\mathbb{C}^{2 \times 2}$ satisfying a jump condition
\begin{equation}
\mathbf{\Phi}_+(z) = \mathbf{\Phi}_-(z)\mathbf{V}(z),\qquad z\in\Gamma,
\label{basic-rhp-jump}
\end{equation}
and normalized to satisfy $\mathbf{\Phi}(z) \to \mathbb{I}$ as $z\to\infty$, essentially involves seeking a solution of the form
\begin{equation}
\mathbf{\Phi}(z) =\mathbb{I} +\mathcal{C}_{\Gamma,W}[\mathbf{F}](z),\qquad \mathcal{C}_{\Gamma,W}[\mathbf{F}](z):=\frac{1}{2\pi \ii}\int_{\Gamma}\frac{\mathbf{F}(s)W(s)\dd s}{s-z},
\end{equation}
where $W$ is a suitably chosen weight (possibly different on each arc of $\Gamma$) and rephrasing \eqref{basic-rhp-jump} as a singular integral equation for the unknown density $\mathbf{F}(z)$ in the form
\begin{equation}
\mathcal{C}^+_{\Gamma,W}[\mathbf{F}](z) - \mathcal{C}^-_{\Gamma,W}[\mathbf{F}](z)\mathbf{V}(z) = \mathbf{V}(z)-\mathbb{I},\quad z\in \Gamma.
\label{sie}
\end{equation}
This singular integral equation is discretized via collocation separately on each arc of $\Gamma$ using a basis of suitable  polynomials  orthogonal with respect to the weight $W$. In practice, a basis of orthogonal polynomials on the unit interval $[-1,1]$ with positive weight $W$ is mapped to
each arc of $\Gamma$. 
The linear system resulting from the employed collocation is solved for the coefficients of $\mathbf{F}_{\pm}(z)$ for $z\in\Gamma$. All of this machinery is readily implemented and available in a black-box manner in \texttt{OperatorApproximation.jl} \cite{OperatorApproximation}, and \cite{RogueWaveInfiniteNLS} uses those capabilities to numerically solve various Riemann-Hilbert problem representations of $\Psi(X,T;\mathbf{G},\bg)$.
The theoretical framework behind the computational approach described above is due to S.\@ Olver and T.\@ Trogdon, see \cite{TrogdonO2015} (and also \cite{Olver2011,Olver2012}) and the references therein.
An in-depth description and analysis of the accuracy of the numerical method employed can be found in \cite{OlverT2014} and \cite[Chapters 2 and 7]{TrogdonO2015}.

\subsection{Basic use of the software package{} \texttt{RogueWaveInfiniteNLS.jl}}
The installation of the software package to the user's \texttt{Julia} environment follows the standard package installation:
\begin{lstlisting}
[julia> using Pkg
[julia> Pkg.add("RogueWaveInfiniteNLS")
\end{lstlisting}
Then at the \texttt{Julia} prompt one can activate the package via
\begin{lstlisting}
[julia> using RogueWaveInfiniteNLS
\end{lstlisting}
and to compute $\Psi(X,T;\mathbf{G}(a,b),\bg)$ at $X=-1.8$, $T=0.6$ with $\mathbf{G}=\mathbf{G}(a=2-3\ii, b=1+0.5\ii)$ and $\bg=1.2$, one just calls:
\begin{lstlisting}
[julia> psi(-1.8, 0.6, 2-3im, 1+0.5im, 1.2)
-0.283757397147 - 0.8166685877725581im
\end{lstlisting}
The syntax for using the main routine \texttt{psi} is \texttt{psi(X, T, a, b, B)}.  See also the user's guide in Appendix~\ref{a:julia-summary}.

\subsection{Details of the implementation and the regions of the $(X,T)$-plane}
The routines we develop to compute $\Psi(X,T; \mathbf{G}(a,b), \bg)$ make use of the elementary symmetry properties given in Section~\ref{s:introduction}. Thanks to Proposition~\ref{p:scaling}, Proposition~\ref{prop:X-symmetry}, and Proposition~\ref{prop:T-symmetry}, 
to compute the value of $\Psi(X,T; \mathbf{G}(a,b), \bg)$ at a given point $(X,T)\in\mathbb{R}^2$ for given $\bg>0$ and $\mathbf{G}=\mathbf{G}(a,b)$, it suffices to compute $\Psi(\widetilde{X},\widetilde{T}; \widetilde{\mathbf{G}}, \bg=1)$ for 
\begin{equation}
(\widetilde{X},\widetilde{T}) = (\bg |X|, \bg^2 |T|)
\end{equation} 
and
\begin{equation}
\widetilde{\mathbf{G}}:= \begin{cases}
\mathbf{G}(a,b),&\quad\text{if}~X\geq 0~\text{and}~T\geq 0\\
\mathbf{G}(b,a),&\quad\text{if}~X< 0~\text{and}~T\geq 0\\
\mathbf{G}(a,b)^*,&\quad\text{if}~X \geq 0~\text{and}~T<0\\
\mathbf{G}(b,a)^*,&\quad\text{if}~X < 0~\text{and}~T<0.
\end{cases}
\end{equation}

The method underlying the routine \texttt{psi} for computing $\Psi(X,T;\mathbf{G}(a,b),\bg)$ consists of the following steps.
\begin{itemize}
\item[\textbf{Step 1:}]  Choose a computationally-appropriate Riemann-Hilbert problem to solve based on the location of the point $(X,T)$ using Algorithm~\ref{alg:regions} (see below). Thus the $(X,T)$-plane is written as the disjoint union of four regions denoted \texttt{NoDeformation}, \texttt{LargeX}, \texttt{LargeT}, and \texttt{Painleve}.
\item[\textbf{Step 2:}] Construct data structures representing the jump contours and jump matrices of the selected Riemann-Hilbert problem.  
%Implement the Riemann-Hilbert problem that is well-conditioned asymptotically in the region chosen.
Because of the choice made in Step 1, the jump matrix differs little from the identity 
%for large values of $X$ or $T$ thanks to the steepest descent deformations in place, away from the 
except on certain arcs where it takes a (piecewise) constant value and near self-intersection points of the jump contour.  Depending on details of the problem, small circles centered at the self-intersection points may be added to the jump contour at this stage in order to remove singularities, and the radii of these circles is chosen so that for the given coordinates $(X,T)$ the jump matrices supported on the circles remain bounded in norm.
\item[\textbf{Step 3:}] Solve the relevant Riemann-Hilbert problem numerically by passing the data structures built in Step 2 to suitable routines in the package \texttt{OperatorApproximation.jl}.  The quantity 
\begin{equation}
P_{12}^{[1]}(\widetilde{X},\widetilde{T},\widetilde{\mathbf{G}}):= \lim_{\Lambda\to\infty} 2\ii \Lambda P_{12}(\Lambda;\widetilde{X},\widetilde{T},\widetilde{\mathbf{G}})
\end{equation}
is then extracted from the numerically computed solution by a contour integration of the returned weighted Cauchy density.  
\item[\textbf{Step 4:}] Recover $\Psi(X,T;\mathbf{G}(a,b),\bg)$ from $P_{12}^{[1]}(\widetilde{X},\widetilde{T},\widetilde{\mathbf{G}})$ using \eqref{Psi-def} and the symmetry relations in Proposition~\ref{p:scaling}, Proposition~\ref{prop:X-symmetry}, and Proposition~\ref{prop:T-symmetry}.
\end{itemize}
Since the computation of $\Psi(X,T;\mathbf{G},\bg)$ is based on the numerical solution of Riemann-Hilbert problems that depend on $(X,T)$ explicitly and parametrically, the computations for different pairs $(X,T)$ are independent and can be immediately parallelized over a large range of the coordinates. 
Some of the computations for this work were performed in parallel on the 48-core Pitzer nodes of the Ohio Supercomputer Center \cite{OhioSupercomputerCenter1987}. 
For instance, the solution shown in the plots in Figure~\ref{f:Psi-a1-b1}
% and Figure~\ref{f:Psi-elliptic} 
is computed over the domain $\{(X,T)\colon -16\leq X \leq 16,~ -8\leq T \leq 8\}$ with grid spacings $\texttt{dX}=\texttt{dT}=0.05$. Therefore, to obtain the data for these plots, 205,761 Riemann-Hilbert problems were solved in parallel on the supercomputer as $(X,T)$ ranges over the 205,761 points on the discretized domain (in batches, of course, due to memory limitations and the number of nodes available).
%(v>VCRIT) && (abs(v-VCRIT)>VDISTANCEPAINLEVE)
\begin{algorithm}
\caption{Asymptotic Region Algorithm}\label{alg:regions}
\KwData{$X\in\mathbb{R}$, $T\in\mathbb{R}$, $\bg>0$}
\KwResult{The computationally-appropriate region: \texttt{region}}
initialization\;
$\widetilde{X}\gets \bg|X|$; $\widetilde{T}\gets \bg^2 |T|$; $\widetilde{T}_{\mathrm{max}}\gets 8$; $R\gets 2$\;
$v\gets \widetilde{T}\widetilde{X}^{-\frac{3}{2}}$; $v_{\mathrm{c}}\gets 54^{-\frac{1}{2}}$; $w\gets\widetilde{X}\widetilde{T}^{-\frac{2}{3}}$; $w_{\mathrm{c}}\gets 54^{\frac{1}{3}}$; $\varepsilon_{v} \gets 0.00025$; $\varepsilon_{w} \gets 0.02*w_{\mathrm{c}}$\;
%\Comment{$v_{\mathrm{c}} - v < \varepsilon_v$ and $w_{\mathrm{c}} - w < \varepsilon_w$ capture $|z_{-\infty}-z_1|<0.125$ and $\Im(z_0)<0.125$, respectively}\;
%@everywhere VDISTANCEPAINLEVE = 0.0001
%vfXT(X,T) = T/X^(3/2)
%# These choices roughly guarantees:
%# abs(zintfy - z1) < zpert
%# imag(z0) < zpert
%# with zpert = 0.125
%Good else if
\uIf{$\widetilde{X}^2 + \widetilde{T}^2 \leq R^2$}{
$\texttt{region} \gets \texttt{NoDeformation}$\;
\Return \texttt{region}\;
}
\uElseIf{$\widetilde{T}\leq \widetilde{T}_{\mathrm{max}}$}{
\uIf{$v>v_\mathrm{c}$ \textbf{\bf\rm and} $|v-v_{\mathrm{c}}|>\varepsilon_v$}{$\texttt{region}\gets \texttt{NoDeformation}$\;}
\uElseIf{$v<v_\mathrm{c}$ {\bf\rm and} $|v-v_{\mathrm{c}}|>\varepsilon_v$}{$\texttt{region}\gets \texttt{LargeX}$\;}
\uElse{$\texttt{region}\gets \texttt{Painleve}$\;}
\Return \texttt{region}\;
}
\Else{
\uIf{$v>v_\mathrm{c}$ {\bf\rm and} $|w-w_{\mathrm{c}}|>\varepsilon_w$}{$\texttt{region}\gets \texttt{LargeT}$\;}
\uElseIf{$v<v_\mathrm{c}$ {\bf\rm and} $|v-v_{\mathrm{c}}|>\varepsilon_v$}{$\texttt{region}\gets \texttt{LargeX}$\;}
\uElse{$\texttt{region}\gets \texttt{Painleve}$\;}
\Return \texttt{region}\;
}
\end{algorithm}

The routine \texttt{psi} calls Algorithm~\ref{alg:regions}, and based on the region determined to contain $(X,T)$, calls one of the programs \texttt{psi\_undeformed}, \texttt{psi\_largeX}, \texttt{psi\_largeT}, or \texttt{psi\_Painleve} to perform Steps 2--4.  These programs are in turn ``wrappers'' for corresponding lower-level programs \texttt{rwio\_undeformed}, \texttt{rwio\_largeX}, \texttt{rwio\_largeT} and \texttt{rwio\_Painleve}.  We will describe the syntax of these programs and give some details about the choices of contours and jump matrices needed to build the data structures in Step 2 below.  Some users of \texttt{RogueWaveInfiniteNLS.jl} may like to use these routines directly to compute $\Psi(X,T;\mathbf{G},\bg)$ by using a specific Riemann-Hilbert representation as they bypass Algorithm~\ref{alg:regions} and just compute $\Psi(X,T;\mathbf{G},\bg)$ from the indicated problem, whether or not that is a good idea given the values of $(X,T)$.  However we wish to emphasize that the casual user of the package need not be concerned with any of these programs and can reliably compute general rogue waves of infinite order in most of the $(X,T)$-plane just by using the main routine \texttt{psi}.

It is our intention that \texttt{psi} return an accurate numerical evaluation of $\Psi(X,T;\mathbf{G},\bg)$ for coordinates $(X,T)$ lying in a very large, but bounded region of the $(X,T)$-plane.  Indeed, the main utility of the routines is to allow the reliable computation of $\Psi(X,T;\mathbf{G},\bg)$ for values of $(X,T)$ that are definitely not in the regime of applicability of any of the theorems in Section~\ref{s:asymptotics-intro}.  However, we also want the region of accurate computability to be large enough to allow for substantial overlap with the various regions of validity of those theorems.
This allows one to validate the analytical results as shown in Figures~\ref{fig:LargeX-a1-b1}, \ref{fig:LargeX-a0p5EIPiOver4-b1}, \ref{fig:LargeT-a1-b1}, \ref{fig:LargeT-a0p5EIPiOver4-b1}, \ref{fig:TransitionalComparison}, and \ref{fig:TransitionalLogLogError}. 

In principle the analytical asymptotics make numerical calculations unnecessary for extreme values of the variables.  Nonetheless, it is of some interest to push the envelope of applicability of the numerical methods beyond the limits of the current version of the software, and here we point out that as the variables become larger, even a Riemann-Hilbert problem adapted to asymptotic analysis in the relevant regime becomes challenging to solve numerically.  The reason is that the deviation of the jump matrix from a piecewise-constant matrix becomes increasingly concentrated near the contour self-intersection points, and the rapid variation requires an increasingly large number of collocation points as the variables grow in size.  There are strategies for dealing with this phenomenon. The first technique is to remove the non-identity limiting jump matrices using the analytical outer parametrix.  This yields a modified Riemann-Hilbert problem with jump matrices that are different from the identity only in small neighborhoods of the self-intersection points.  One may think that an optimal approach would be to then deal with the self-intersection points using analytical local/inner parametrices constructed from special functions (e.g., Airy, Bessel, or parabolic cylinder), and then reducing the problem to a small-norm problem --- exactly as in the proofs of the theorems ---  for the computer to solve.  However, just calculating the jump matrices for the small-norm problem accurately requires reliable evaluation of the relevant special functions for extreme values of the arguments, which is again a computational problem of a similar nature.  An alternative is to construct a numerical parametrix for a given intersection point by truncating jump contours away from it, imposing identity asymptotics at infinity, and rescaling to obtain a model requiring fewer collocation points for accuracy.  Then one conjugates the full problem by the parametrix, which removes all difficulties near the selected point and conjugates the jumps near the remaining points by near-identity factors.  Iterating this procedure to take care of the intersection points one-by-one can yield excellent results.  The actual procedure has additional technical details, but this is the main idea.  We plan to continue to update the software in \texttt{RogueWaveInfiniteNLS.jl} as such improvements come to light, with the aim of making the accurate computation of $\Psi(X,T;\mathbf{G},\bg)$ available in increasingly larger domains of the $(X,T)$-plane.

In the remainder of this section we redefine $v$ (defined in Section~\ref{s:large-X-results}) and $w$ (defined in Section~\ref{s:large-T-results}) in terms of the rescaled and reflected coordinates $(\widetilde{X},\widetilde{T})$:
\begin{equation}
v:= \widetilde{T} \widetilde{X}^{-\frac{3}{2}},\qquad w:= \widetilde{X}\widetilde{T}^{-\frac{2}{3}}.
\label{eq:vw-rescaled}
\end{equation}
We also recall the critical values of $v$ and $w$: $v_{\mathrm{c}}:=54^{-\frac{1}{2}}$ and $w_{\mathrm{c}}:=54^{\frac{1}{3}}$.


\begin{remark}[Scaling of arguments]
Like the main program \texttt{psi}, the programs \texttt{psi\_undeformed}, \texttt{psi\_largeX}, \texttt{psi\_largeT}, and \texttt{psi\_Painleve} take the unscaled variables $(X,T)\in\mathbb{R}^2$ as arguments, along with the value of $\bg$.  However, the lower-level programs assume that $\bg=1$ and take the rescaled coordinates $(\widetilde{X},\widetilde{T})$ as arguments (both nonnegative).  To simplify the notation in describing the latter routines below, we will drop the tildes.  This also makes it easier for the reader to match with the notation in the rest of the paper where the Riemann-Hilbert problems that are solved numerically by these routines are formulated in terms of variables denoted $(X,T)$ and the derived quantities $v$ and $w$ given by \eqref{eq:vw-rescaled}.  Of course if $B=1$ and $X,T>0$, there is no difference between the scaled and unscaled coordinates.
%
%As described in Algorithm~\ref{alg:regions}, given $\bg>0$ with $\bg\neq 1$, we introduce the rescaled variables $\widetilde{X}:= \bg|X|$ and $\widetilde{T}:= \bg^2 |T|$ which are both nonnegative and work with the deformed Riemann-Hilbert problems analyzed in the paper for $\bg=1$. In the material following this remark we perform
%\begin{equation}
%X\gets \bg|X| \quad\text{and}\quad T \gets \bg^2 |T|.
%\end{equation}
%In other words, we denote the rescaled coordinates $(\widetilde{X},\widetilde{T})$ by $(X,T)$ so that  references to the Riemann-Hilbert problems from Sections~\ref{s:introduction}, \ref{s:large-X}, \ref{s:large-T}, and \ref{s:transitional} (with $B=1$) is more clear.
\end{remark}

\subsection{The region \texttt{NoDeformation}} If  $(X,T)$ lies in the region \texttt{NoDeformation} according to Algorithm~\ref{alg:regions}, then \texttt{psi} calls the wrapper \texttt{psi\_undeformed} which in turn calls the low-level program \texttt{rwio\_undeformed}.  The latter program solves numerically the basic \rhref{rhp:near-field} which does not leverage any deformation or opening of lenses.  Although the jump contour is stated to be $|\Lambda|=1$ in \eqref{P-jump}, one can actually take any Jordan curve enclosing the origin $\Lambda=0$ to be the jump contour.  \texttt{rwio\_undeformed} leverages this freedom and uses the circle $|\Lambda|=T^{-\frac{1}{2}}$ as the jump contour if $1< T\leq T_{\mathrm{max}}$ and uses the original contour $|\Lambda|=1$ if $0 \leq T \leq 1$. In practice we take $T_{\mathrm{max}}:=8$ (see Algorithm~\ref{alg:regions}). With this choice the jump contour stays away from the singularity of the exponential factors in \eqref{P-jump} at $\Lambda=0$ (with a distance at least $1/\sqrt{8}$) while, under the conditions that $(X,T)$ is assigned to the region \texttt{NoDeformation} by Algorithm~\ref{alg:regions}, the matrix norm of the jump matrix is uniformly of moderate size on the jump contour. 
%We employ this strategy in the case $X^2 + T^2 \leq R$, where we take $R=2$ in practice, or the case $T \leq  T_{\mathrm{max}} = 8$ and $v>v_{\mathrm{c}}$, see Algorithm~\ref{alg:regions}. The latter conveniently delays the use of the more complicated Riemann-Hilbert problem considered in the \texttt{LargeT} region (see Section~\ref{s:num-large-T}). 
%
The low-level routine for computing $\Psi(X,T;\mathbf{G},\bg=1)$ via solving the undeformed problem \rhref{rhp:near-field} with the given parameters and for $X\geq 0$ and $T\geq 0$ is \texttt{rwio{\_}undeformed} with arguments $X$, $T$ (rescaled and nonnegative), $a$, $b$, and an integer $n$, the number of collocation points to use on each straight-line arc of the polygonal jump contour.  For instance, the command
%\begin{lstlisting}
%[julia> rwio_nodeformation_rescaled(0.8, 1.5, 1, 2im, 800)
%\end{lstlisting}
\begin{lstlisting}
[julia> rwio_undeformed(0.8, 1.5, 1, 2im, 400)
\end{lstlisting}
returns $\Psi(X,T;\mathbf{G},\bg)$ at $X=0.8$, $T=1.5$, with $\mathbf{G}=\mathbf{G}(a=1,b=2\ii)$ and $\bg=1$.  The original circular jump contour is modeled as a square, each side of which is resolved using 400 collocation points.  
%It overrides Algorithm~\ref{alg:regions} and just computes $\Psi(X,T;\mathbf{G},1)$ as indicated, whether or not that is a good idea given the values of $(X,T)$.
The corresponding wrapper
\texttt{psi{\_}undeformed} takes the same arguments as does \texttt{psi} except for an additional argument allowing the user to specify the number of collocation points on each of the four sides of the square jump contour.  These are the unscaled coordinates, which can take any signs, and the value of $\bg$ is specified in the argument list.  Thus for instance
\begin{lstlisting}
[julia> psi_undeformed(-0.8,1.5,1,2im,1.2,400)
\end{lstlisting}
returns $\Psi(X,T;\mathbf{G},\bg)$ at $X=-0.8$, $T=1.5$, $\mathbf{G}=\mathbf{G}(a=1,b=2\ii)$ and $\bg=1.2$ using $400$ collocation points on each edge of the square.
%with the arguments \texttt{psi{\_}undeformed(X, T, a, b, B, n)}
%to compute for $(X,T)$ outside the first quadrant using $\texttt{n}$ collocation points on each subarc (segment) of the contour.
%We emphasize that end user need not solve this Riemann-Hilbert problem specifically,  and we are just giving the description of an available routine here. The choice of the appropriate Riemann-Hilbert problem to compute $\Psi(X,T;\mathbf{G},\bg)$ is done automatically (in a black-box manner) when the main routine \texttt{psi(X, T, a, b, B)} is called depending on the value of $(X,T)$.


\subsection{The region \texttt{LargeX}} 
\label{s:num-large-X}
%This region corresponds to the case where $X^2+T^2\geq R$ and $0\leq v < v_{\mathrm{c}}$.
If Algorithm~\ref{alg:regions} determines that $(X,T)$ lies in the region \texttt{LargeX}, then \texttt{psi} calls the wrapper \texttt{psi\_largeX} which calls the low-level program \texttt{rwio\_largeX}.  
This low-level program solves the Riemann-Hilbert problem satisfied by $\mathbf{T}(z;X,v)$ defined by the substitutions \eqref{L-plus-sub}--\eqref{L-minus-sub} with jump conditions given by \eqref{jump-C-L-plus-X}--\eqref{jump-C-L-minus-X}. See Section~\ref{s:large-X} and in particular Figure~\ref{fig:largeX-contour} for the jump contour of this Riemann-Hilbert problem.
This is exactly what was done in our first paper \cite{BilmanLM2020} where $\Psi(X,T;\mathbf{G},\bg=1)$ was computed for the first time for $\mathbf{G}=\mathbf{Q}^{-1}$ for $T$ small: $|T|\leq 1$. The numerical routine \texttt{rwio\_largeX} implemented in the package \texttt{RogueWaveInfiniteNLS.jl}  is very similar, the main difference being the use of the open-source \texttt{Julia} programming language.

As described in Section~\ref{s:large-X}, the jump contour for the deformed problem is independent of $X$ but depends on $v\in[0,v_{\mathrm{c}})$. \texttt{rwio\_largeX} adaptively chooses a polygonal model for the jump contour that varies as $v$ ranges over this interval. This variation becomes especially important as $v$ approaches $v_\mathrm{c}\approx 0.136$. See Figure~\ref{f:num-contours-X} for the numerical jump contours used by \texttt{rwio\_largeX} for different values of $v$.

\begin{figure}[h]
\includegraphics[width=0.31\textwidth]{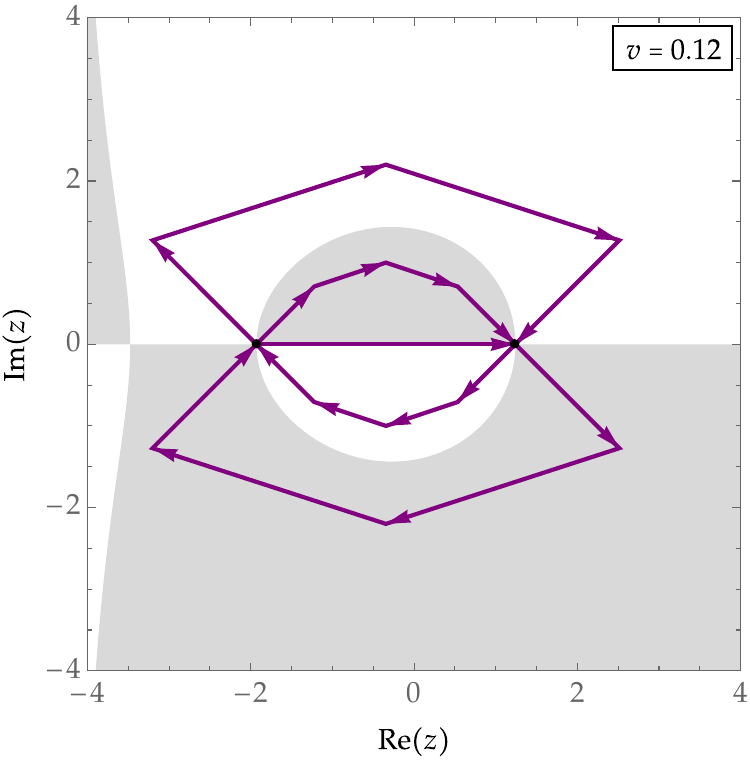}
\includegraphics[width=0.31\textwidth]{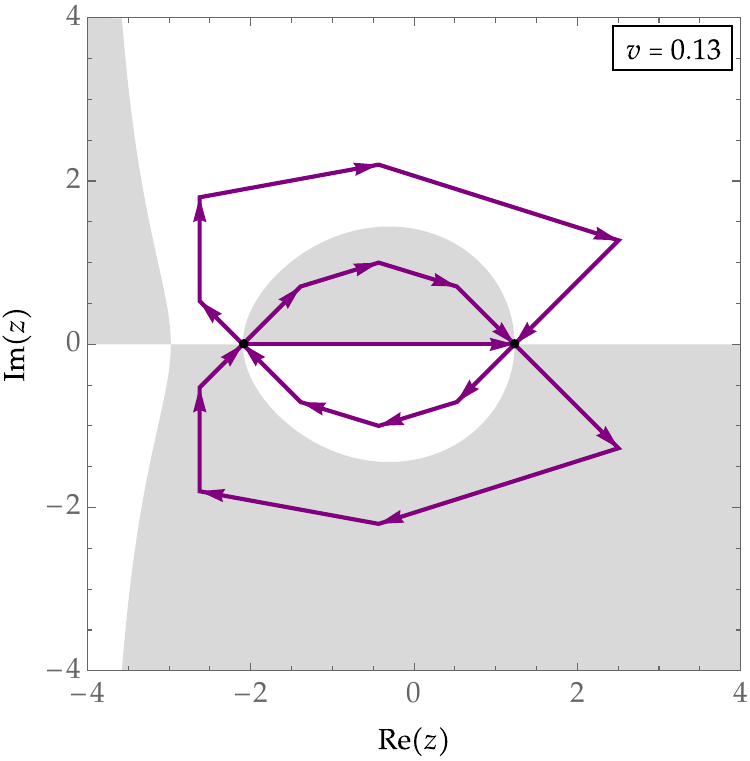}
\includegraphics[width=0.31\textwidth]{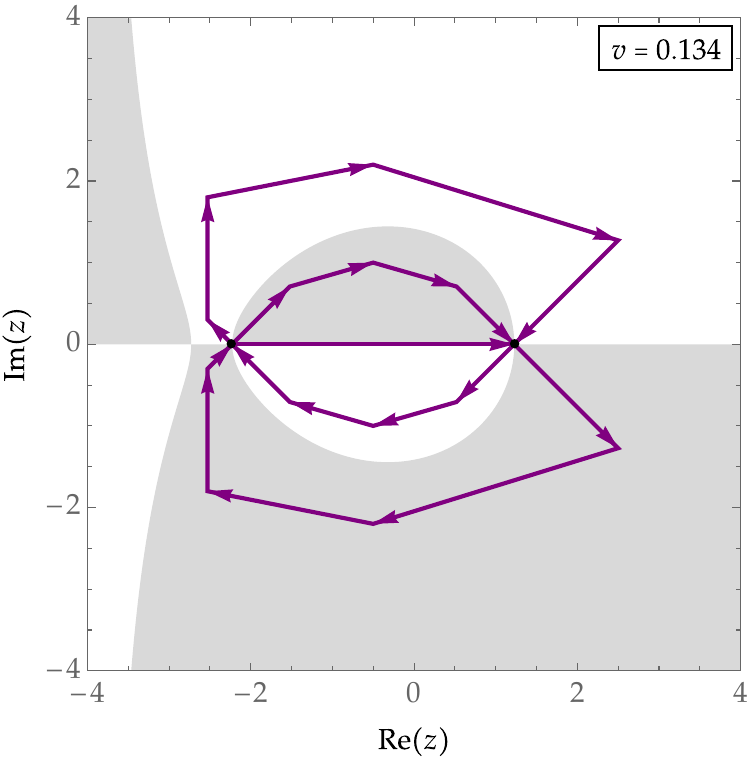}
\caption{The numerical contours used by \texttt{rwio\_largeX} for increasing values of $v\in[0,v_{\mathrm{c}})$.}
\label{f:num-contours-X}
\end{figure}

The arguments taken by \texttt{rwio\_largeX} are $X$ and $v$ (the natural parameters for $\mathbf{T}(z;X,v)$ as described in Section~\ref{s:large-X}), $a$, $b$, and $n$ (the number of collocation points per segment of the polygonal jump contour).  
%To compute $\Psi(X,T;\mathbf{G},\bg)$ by solving the deformed Riemann-Hilbert problem satisfied by $\mathbf{T}(z;X,v)$ (see the associated jump conditions \eqref{jump-C-L-plus-X}--\eqref{jump-C-L-minus-X}) with the given parameters for $X\geq 0$ and fixed $v\in[0,v_{\mathrm{c}})$ one can also use the routine
%\texttt{rwio{\_}largeX} of the package as:
%\begin{lstlisting}
%[julia> rwio_largeX(25, 0.1, 1, 2im, 280)
%\end{lstlisting}
For instance, 
\begin{lstlisting}
[julia> rwio_largeX(25, 0.1, 1, 2im, 140)
\end{lstlisting}
returns $\Psi(X,T;\mathbf{G},\bg=1)$ at $X=25$, $v=TX^{-\frac{3}{2}}=0.1$, with $\mathbf{G}=\mathbf{G}(a=1,b=2\ii)$ by using 
%. (and $\bg=1$ is the hard-coded default) by solving the deformed Riemann-Hilbert problem satisfied by $\mathbf{T}(z; X, v)$ using 
$140$ collocation points in each segment.
%subarc (segment) of the jump contour, which is depicted in Figure~\ref{f:num-contours-X}.  
The auxiliary routine \texttt{vfromXT(X,T)} included in \texttt{RogueWaveInfiniteNLS.jl} can be used to compute the value of $v$ given $(X,T)$ if needed.

In this case, the wrapper \texttt{psi{\_}largeX} takes different arguments, as it allows the user to directly compute $\Psi(X,T;\mathbf{G},\bg)$ at arbitrary given $(X,T)$ coordinates (unscaled, and in any quadrant of the plane) using the numerical approach underpinning \texttt{rwio\_largeX}.  The arguments of \texttt{psi\_largeX} are again the same as those of \texttt{psi} with an additional integer argument for specifying the number of points used on each contour segment.
Thus,
\begin{lstlisting}
[julia> psi_largeX(-25,-0.1,1,2im,1.2,140)
\end{lstlisting}
returns $\Psi(-25,-0.1,\mathbf{G}(1,2\ii),1.2)$ computed by scaling the variables by $\bg=1.2$,  computing $v$ from the scaled variables using \texttt{vfromXT}, then calling \texttt{rwio\_largeX} with 140 collocation points, and finally scaling the returned value by $\bg=1.2$.

%\textcolor{green}{[DB: Here is a new remark that clarifies possible questions of a numerically-inclined referee.]}
%
%\begin{remark}
%\label{rem:num-largeX}
%We are not claiming that our \emph{initial release} \cite{RogueWaveInfiniteNLS} would work for extremely large values of $X$ for given fixed $T$ or $v$.
%As X increases unboundedly, one eventually needs more collocation points in the implementation described above. This is likely due to the density of points being insufficient near the points $z=z_1$ and $z=z_2$, where the jump matrix (its perturbation from the identity) localizes. 
%A way to overcome this could be as follows: 
%One can truncate the four arcs $C^{\pm}_{R}$ and $C^{\pm}_{L}$ connecting the points $z_1(v)$ and $z_2(v)$ by eliminating the segments on which the jump matrices' discrepancy from the identity matrix equals $0$ up to machine precision (recall the exponential decay of the off-diagonal elements of the jump matrices on those arcs).
%Then, one can implement the outer parametrix and use it in the numerical routines to remove the constant jump matrix supported on $I$, which is also independent of the complex independent variable $z$. 
%Since the outer parametrix has singularities at $z=z_1(v)$ and $z=z_2(v)$, now one has to introduce little disks centered at those points and apply the outer parametrix in the exterior of those disks. 
%Following this, one can transfer the jump conditions supported on segments in the interior of these disks to the boundary of the disks, but then one has to shrink the radii of those little disks at the rate $|X|^{-\frac{1}{4}}$ as $X$ becomes large to ensure that the jump matrices on the disk boundaries remain $O(1)$. 
%This practice is common in the numerical inverse-scattering transform methods in the literature. 
%Doing so, one ends up with two sets of jump conditions that are supported on disjoint contours: one on the disk boundary at $z=z_1(v)$ along with the segments of the truncated contours attached to the circle, and the other in an analogous configuration at $z=z_2(v)$.
%At this point one faces the problem that the disks cannot be shrunk indefinitely as $X$ grows and inaccuracies may eventually be formed for extremely large values of $X$. 
%At that point, a strategy could be to consider the now-disconnected jump conditions separately, formulate and solve the Riemann-Hilbert problem consisting of the jump conditions near, say, $z=z_1(v)$, and use that computed solution as a numerical parametrix to reduce the Riemann-Hilbert problem to a problem consisting of solely the jump conditions near $z=z_2(v)$.
%The advantage of doing this could be that since each of the problems localized on the small circles, they can be shifted to the origin and rescaled appropriately since they are treated separately. One hopes that this would help combat the inaccuracies arising from shrinking of the disks for extremely large values of $X$.
%This procedure may be implemented in a future version of the software package \cite{RogueWaveInfiniteNLS}. 
%
%The above-mentioned ``truncate-shift-and-rescale'' strategy was implemented, for example, in the numerical inverse-scattering \cite{TrogdonO2013} for the NLS equation \eqref{nls} to solve the initial-value problem on the full line. In that setting there is only one critical point of the controlling phase in the Riemann-Hilbert problem and the whole problem localizes around that point, which can be shifted and rescaled (upon eliminating a jump on an infinite ray via implementing an outer parametrix and truncating the jump contour in the sense described above).
%\end{remark}

%In light of Remark~\ref{rem:num-largeX}, we recall our Theorem~\ref{t:large-X}. 
%The main utility we have in mind for the initial release of \texttt{RogueWaveInfiniteNLS.jl} is to be able to compute and plot general rogue waves of infinite order for moderate and moderately large values of $(X,T)$ where our analytical asymptotic results (such as Theorem~\ref{t:large-X}) may not give accurate information. 
%This is especially important for small and moderate values of $(X,T)$ since these solutions tend to zero in all directions in the $(X,T)$-plane and they are not small near the origin.
%Figure~\ref{fig:LargeX-a1-b1} and Figure~\ref{fig:LargeX-a0p5EIPiOver4-b1} also demonstrate that the accuracy of the numerical routines employed in the \texttt{largeX} region for moderately large values of $X$.
%
%\textcolor{green}{[------DB: End of new material.------]}

\subsection{The region \texttt{LargeT}}
\label{s:num-large-T}
When Algorithm~\ref{alg:regions} assigns $(X,T)$ to the region \texttt{LargeT}, the main program \texttt{psi} calls the wrapper \texttt{psi\_largeT} which in turn calls the low-level program \texttt{rwio\_largeT}.  The latter routine computes the solution by via the Riemann-Hilbert problem satisfied by the matrix $\mathbf{T}(\tz;T,w)$ described in Section~\ref{s:large-T}.  The numerical solution of the latter problem is substantially more complicated than is solving for $\mathbf{T}(z;X,v)$ and it was not 
considered in our earlier work in \cite{BilmanLM2020}.  Its implementation in \texttt{RogueWaveInfiniteNLS.jl} is a significant new contribution.

To describe the method used by \texttt{rwio\_largeT}, we first rewrite the jump conditions satisfied by $\mathbf{T}(\tz;T,w)$ by assigning new names to the constant matrices that appear in the jump conditions that will be convenient in describing certain local transformations later on. Thus the material starting with \eqref{T-from-S-L-plus-Gamma} is reformulated as follows:
\begin{equation}
\mathbf{T}_+(\tz;T,w) = \mathbf{T}_-(\tz;T,w) \ee^{-\ii T^{1/3}h(\tz;w)\sigma_3} \mathbf{V}_{L}\ee^{\ii T^{1/3}h(\tz;w)\sigma_3}
,\quad \mathbf{V}_{L}:= \begin{bmatrix}1 & 0 \\-\dfrac{\mathfrak{b}}{\mathfrak{a}} & 1 \end{bmatrix},\quad \tz\in C^+_{\Gamma,L},
\label{num-jump-T-Gamma-L}
\end{equation}
\begin{equation}
\mathbf{T}_+(\tz;T,w) = \mathbf{T}_-(\tz;T,w) \ee^{-\ii T^{1/3}h(\tz;w)\sigma_3} \mathbf{V}_{R}\ee^{\ii T^{1/3}h(\tz;w)\sigma_3},
\quad \mathbf{V}_{R}:=
\begin{bmatrix}1 &\mathfrak{a b} \\ 0 & 1 \end{bmatrix},\quad \tz\in  C^+_{\Gamma,R},
\end{equation}
\begin{equation}
\mathbf{T}_+(\tz;T,w) = \mathbf{T}_-(\tz;T,w)\mathfrak{a}^{-2\sigma_3},\quad \tz\in  I,
\label{num-jump-T-I}
\end{equation}
\begin{equation}
\mathbf{T}_+(\tz;T,w) = \mathbf{T}_-(\tz;T,w) \ee^{-\ii T^{1/3}h(\tz;w)\sigma_3} \mathbf{Y}_{R}\ee^{\ii T^{1/3}h(\tz;w)\sigma_3},
\quad \mathbf{Y}_{R}:=
\begin{bmatrix}1 & 0 \\ -\mathfrak{a b}  & 1 \end{bmatrix},\quad \tz\in  C^-_{\Gamma,R},
\end{equation}
\begin{equation}
\mathbf{T}_+(\tz;T,w) = \mathbf{T}_-(\tz;T,w)\ee^{-\ii T^{1/3}h(\tz;w)\sigma_3} \mathbf{Y}_{L}\ee^{\ii T^{1/3}h(\tz;w)\sigma_3},\quad
\mathbf{Y}_{L}:=\begin{bmatrix}1 &  \dfrac{\mathfrak{b}}{\mathfrak{a}} \\ 0 & 1 \end{bmatrix},\quad \tz\in  C^-_{\Gamma,L},
\end{equation}
\begin{equation}
\mathbf{T}_+(\tz;T,w) = \mathbf{T}_-(\tz;T,w)  \ee^{-\ii T^{1/3}h(\tz;w)\sigma_3} \mathbf{W}_{L}\ee^{\ii T^{1/3}h(\tz;w)\sigma_3},\quad
\mathbf{W}_{L} := \begin{bmatrix}1 & -\dfrac{\mathfrak{a}}{\mathfrak{b}} \\ 0 & 1 \end{bmatrix},\quad \tz\in  C^+_{\Sigma,L},
\end{equation}
\begin{equation}
\mathbf{T}_+(\tz;T,w) = \mathbf{T}_-(\tz;T,w)\ee^{-\ii T^{1/3}h(\tz;w)\sigma_3} \mathbf{W}_{R}\ee^{\ii T^{1/3}h(\tz;w)\sigma_3},\quad
\mathbf{W}_{R}:=\begin{bmatrix}1 & -\dfrac{\mathfrak{a}^3}{\mathfrak{b}}  \\ 0 & 1 \end{bmatrix},\quad \tz\in  C^+_{\Sigma,R},
\end{equation}
\begin{equation}
\mathbf{T}_+(\tz;T,w) = \mathbf{T}_-(\tz;T,w)\ee^{-\ii T^{1/3}h(\tz;w)\sigma_3} \mathbf{X}_{R}\ee^{\ii T^{1/3}h(\tz;w)\sigma_3},\quad
\mathbf{X}_{R}:=\begin{bmatrix}1 & 0 \\  \dfrac{\mathfrak{a}^3}{\mathfrak{b}} & 1 \end{bmatrix},\quad \tz\in  C^-_{\Sigma,R},
\end{equation}
\begin{equation}
\mathbf{T}_+(\tz;T,w) = \mathbf{T}_-(\tz;T,w)\ee^{-\ii T^{1/3}h(\tz;w)\sigma_3} \mathbf{X}_{R}\ee^{\ii T^{1/3}h(\tz;w)\sigma_3},\quad
\mathbf{X}_{L}:=
\begin{bmatrix}1 & 0 \\ \dfrac{\mathfrak{a}}{\mathfrak{b}}& 1 \end{bmatrix},\quad \tz\in  C^-_{\Sigma,L}.
\end{equation}
Finally, on $\Sigma=\Sigma^+\cup\Sigma^-$ we have
\begin{alignat}{2}
\mathbf{T}_+(\tz;T,w) &= \mathbf{T}_-(\tz;T,w)\ee^{-\ii T^{1/3}h_-(\tz;w)\sigma_3} \mathbf{W}\ee^{\ii T^{1/3}h_+(\tz;w)\sigma_3}, \quad
\mathbf{W}:=
\begin{bmatrix} 0 & \dfrac{\mathfrak{a}}{\mathfrak{b}} \\ - \dfrac{\mathfrak{b}}{\mathfrak{a}} & 0\end{bmatrix},&&\quad \tz\in\Sigma^+,
\label{num-jump-T-Sigma-p}\\
\mathbf{T}_+(\tz;T,w) &= \mathbf{T}_-(\tz;T,w) \ee^{-\ii T^{1/3}h_-(\tz;w)\sigma_3} \mathbf{X}\ee^{\ii T^{1/3}h_+(\tz;w)\sigma_3},\quad
\mathbf{X}:=\begin{bmatrix} 0 &\dfrac{\mathfrak{b}}{\mathfrak{a}} \\ - \dfrac{\mathfrak{a}}{\mathfrak{b}}  & 0\end{bmatrix},&&\quad \tz\in\Sigma^-.\label{num-jump-T-Sigma-m}
\end{alignat}
Note that the jump matrices on $\Sigma^+$ and $\Sigma^-$ are constants in $\tz$ since $h_+(\tz)+h_-(\tz) = \kappa(w)$ on those arcs. For numerical purposes, we implement $h(z;w)$ as follows. We first define $R^{N}(z;w)$ as the function (the numerical implementation of $R(z;w)$, hence the superscript) that has branch cuts on the line segments from $\tz_0^*$ to $\tz_1$ and from $\tz_1$ to $\tz_0$. It can be written in terms of principal branch square roots as
%Depending on the location of $\tz$ in the complex plane, there are various ways to implement this, but it can essentially be implemented as
\begin{equation}
R^{N}(\tz;w):=\left( \frac{\tz -\tz_1(w)}{\tz -\tz_0^*(w)}\right)^\frac{1}{2}\left( \frac{\tz -\tz_0(w)}{\tz -\tz_1(w)}\right)^\frac{1}{2}(\tz-\tz_0^*(w)),
\end{equation}
%for the principal branch of the power function.
however, this formulation is not stable near $\tz=\tz_1(w)$ or $\tz=\tz_0^*(w)$. In practice, we work with the following functions:
\begin{equation}
R^{N}_{\leftarrow}(\tz;w) := (\tz-\tz_0(w))^{\frac{1}{2}}(\tz-\tz_0^*(w))^{\frac{1}{2}},
\end{equation}
which has horizontal branch cuts from $\tz=\tz_0(w)$ to $\tz=(-\infty)+ \ii \Im(\tz_0)$ and from $\tz=\tz_0^*(w)$ to $\tz=(-\infty) - \ii \Im(\tz_0(w))$,
\begin{equation}
R^{N}_{\rightarrow}(\tz;w) := - (\tz_0(w)-\tz)^{\frac{1}{2}}(\tz_0^*(w) -  \tz)^{\frac{1}{2}},
\end{equation}
which has horizontal branch cuts from $\tz=\tz_0(w)$ to $\tz=(+\infty)+ \ii \Im(\tz_0(w))$ and from $\tz=\tz_0^*(w)$ to $\tz=(+\infty) - \ii \Im(\tz_0(w))$, and
\begin{equation}
R^{N}_{\uparrow}(\tz;w) := \left(\frac{\tz-\tz_0(w)}{\tz-\tz_0^*(w)}\right)^{\frac{1}{2}}(\tz-\tz_0^*(w)),
\end{equation}
which has a vertical branch cut from $\tz=\tz_0^*$ to $\tz=\tz_0$. Recalling that $\Sigma$ is oriented upward, $R^{N}_{\leftarrow}(\tz;w)$ is the continuation of $R^N_-(\tz;w)$ to the strip $-\Im(\tz_0)<\Im(\tz)<\Im(\tz_0)$ lying to the left-hand side of $\Sigma$ and $R^{N}_{\rightarrow}(\tz;w)$ is the continuation of $R^N_+(\tz;w)$ to the strip $-\Im(\tz_0)<\Im(\tz)<\Im(\tz_0)$ lying to the right-hand side of $\Sigma$. The values of the functions $R^{N}_{\leftarrow}(\tz;w)$ and $R^{N}_{\rightarrow}(\tz;w)$ match with the values of $R^N(\tz;w)$ directly above and below the relevant semi-infinite strips emanating from $\Sigma$.
Using these numerical versions of $R(\tz;w)$, we define corresponding numerical versions $h^{N}_{\leftarrow}(\tz;w)$, $h^{N}_{\rightarrow}(\tz;w)$, and $h^{N}_{\uparrow}(\tz;w)$ of $h(\tz;w)$ via the formula
\begin{equation}
h(\tz;w) = \frac{R(\tz;w)^3}{\tz} - 3\cdot2^{-\frac{1}{3}}-\frac{w^2}{6},
\end{equation}
as in \eqref{g-def-large-T}.
%\textcolor{green}{[Maybe include a figure describing on which branch is used on which contour.]}

Before we do anything, we locally collapse the jump conditions supported on $C_{\Gamma,R}^+$ and $C_{\Sigma,R}^+$ near $\tz=\tz_0$ to a single common arc $C^+_{\Gamma\Sigma,R}$ oriented towards $\tz_0$ by a local transformation and we still call the resulting unknown matrix $\mathbf{T}(\tz)=\mathbf{T}(\tz;T,w)$ due to the simplicity of the transformation. Since $C_{\Sigma,R}^+$ is oriented towards $\tz_0$ but $C_{\Gamma,R}^+$ is oriented away from $\tz_0$, it follows that the collapsed jump condition satisfied by the redefined $\mathbf{T}(\tz;T,w)$ is
\begin{equation}
\mathbf{T}_+(\tz;T,w) =  \mathbf{T}_-(\tz;T,w)   \ee^{-\ii T^{1/3}h(\tz;w)\sigma_3}\mathbf{C}_{\tz_0}  \ee^{\ii T^{1/3}h(\tz;w)\sigma_3}, \quad \mathbf{C}_{\tz_0}:=\mathbf{V}_R^{-1}\mathbf{W}_R, \quad \tz\in C^+_{\Gamma\Sigma,R}.
\end{equation}
%See Figure~\ref{f:num-z0-collapse}.
\begin{figure}[h]
\includegraphics[width=0.48\textwidth]{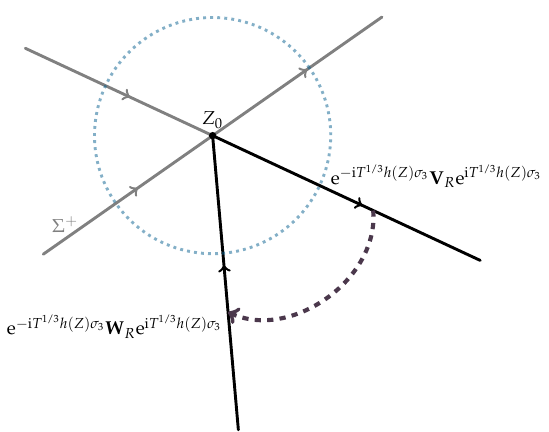}\,\,
\includegraphics[width=0.48\textwidth]{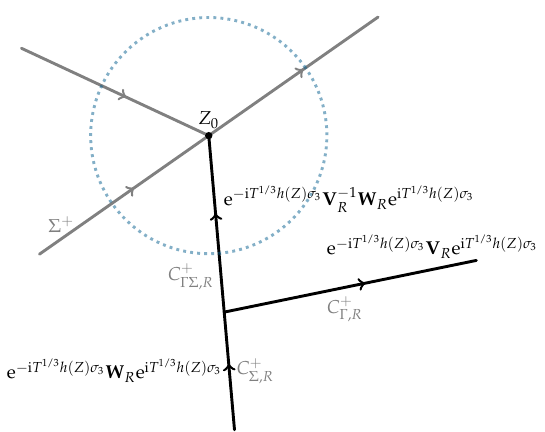}
\caption{Collapsing the jump conditions supported on $C_{\Gamma,R}^+$ and $C_{\Sigma,R}^+$ to a common arc near $\tz_0$.}
\label{f:num-z0-collapse}
\end{figure}
The analogue of this transformation is also carried out near $\tz=\tz_0^*$ by locally collapsing the jump conditions supported on $C_{\Gamma,R}^-$ and $C_{\Sigma,R}^-$ near $\tz=\tz_0^*$ to a single common arc $C^-_{\Gamma\Sigma,R}$ oriented away from $\tz_0^*$ by a similar transformation.
Since $C_{\Sigma,R}^-$ is oriented away from $\tz_0^*$ but $C_{\Gamma,R}^-$ is oriented towards $\tz_0^*$, it follows that the collapsed jump condition satisfied by the redefined $\mathbf{T}(\tz;T,w)$ is
\begin{equation}
\mathbf{T}_+(\tz;T,w) =  \mathbf{T}_-(\tz;T,w)   \ee^{-\ii T^{1/3}h(\tz;w)\sigma_3}\mathbf{C}_{\tz_0^*}  \ee^{\ii T^{1/3}h(\tz;w)\sigma_3}, \quad \mathbf{C}_{\tz_0^*}:=\mathbf{Y}_R^{-1}\mathbf{X}_R, \quad \tz\in C^-_{\Gamma\Sigma,R}.
\end{equation}
We place the collapsed contours $C^+_{\Gamma\Sigma,R}$ and $C^-_{\Gamma\Sigma,R}$ where $C^+_{\Sigma,R}$ and $C^-_{\Sigma,R}$ were placed near $\tz_0$ and $\tz_0^*$, respectively. See Figure~\ref{f:num-z0-collapse} for the arrangement of contours near $\tz_0$. We omit the figure for the configuration near $\tz_0^*$. 

We now let $D_{\tz_0}(\delta_0(T))$ and $D_{\tz_0^*}(\delta_0(T))$ denote disks centered at $\tz=\tz_0$ and $\tz=\tz_0^*$ with common radii $\delta_0(T)$, whose dependence on $T$ will be determined later. We introduce 
\begin{equation}
\mathbf{A}(\tz;T,w):=\begin{cases} 
\mathbf{T}(\tz;T,w) \ee^{-\ii T^{1/3}h(\tz;w)\sigma_3},&\quad\tz\in D_{\tz_0}(\delta_0(T)) \cup D_{\tz_0}(\delta_0(T)),\\
\mathbf{T}(\tz;T,w),&\quad\text{everywhere else}.
\end{cases}
\end{equation}
This transformation introduces jump conditions on the boundary of the disks, which we take to be clockwise oriented. It also modifies the jump matrices on the existing arcs of the jump contour for $\mathbf{T}(\tz;T,w)$. 
The jump matrices associated with the jump conditions satisfied by $\mathbf{A}(\tz;T,w)$ for $\tz$ near $\tz_0$ are illustrated in Figure~\ref{f:num-T-z0-pre} in blue color. See for Figure~\ref{f:num-T-z0conj-pre} for the analogous jump conditions for $\tz$ near $\tz_0^*$.
For the purposes of numerics, the disks are modeled by two polygons related by Schwarz reflection to preserve that of the original jump contour.

\begin{figure}[h!]
\includegraphics[width=0.5\textwidth]{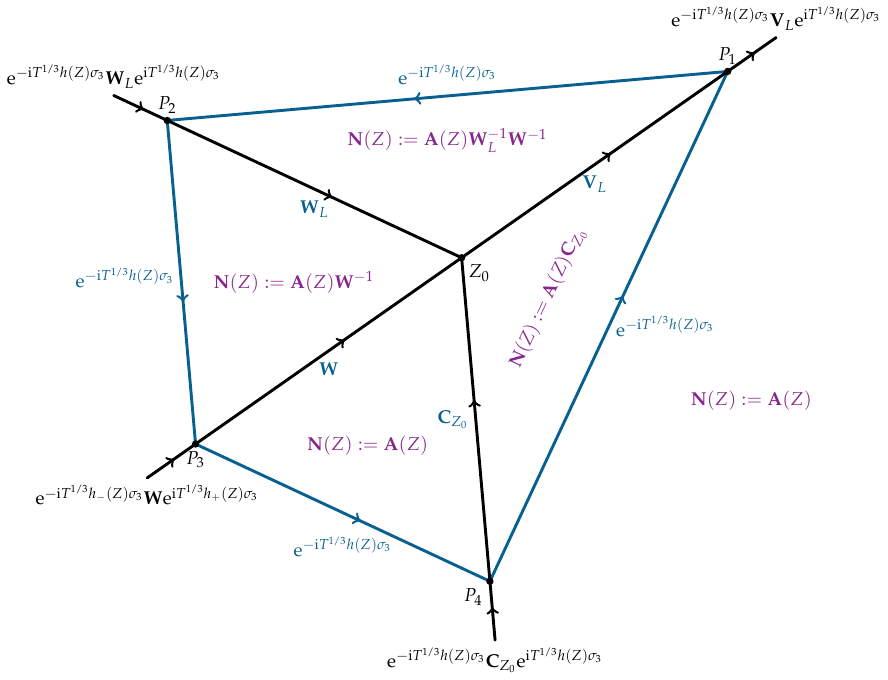}
\caption{The transformation $\mathbf{T}(\tz;T,w)\mapsto \mathbf{A}(\tz;T,w)$ augments the jump contour for $\mathbf{T}(\tz;T,w)$ (black segments) by the blue-colored segments. The jump matrices (modified or new) associated with $\mathbf{A}(\tz;T,w)$ are given in blue. The substitutions shown in fuchsia define the transformation $\mathbf{A}(\tz;T,w)\mapsto \mathbf{N}(\tz;T,w)$ inside the disk (modeled by a polygon in \texttt{rwio\_largeT}).}
\label{f:num-T-z0-pre}
\end{figure}

\begin{figure}[h!]
\includegraphics[width=0.5\textwidth]{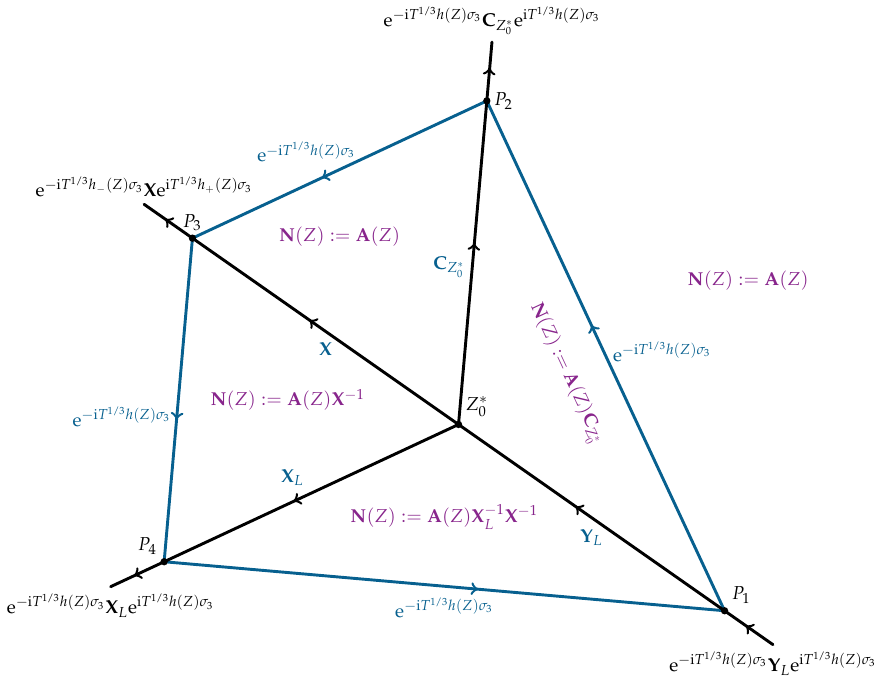}
\caption{As in Figure~\ref{f:num-T-z0-pre} but for the neighborhood of $\tz_0^*$.}
\label{f:num-T-z0conj-pre}
\end{figure}

We then make the local substitutions shown in fuchsia in Figures~\ref{f:num-T-z0-pre}--\ref{f:num-T-z0conj-pre} to transform $\mathbf{A}(\tz;T,w)$ to a new unknown $\mathbf{N}(\tz;T,w)$, and we define $\mathbf{N}(\tz;T,w):= \mathbf{A}(\tz;T,w)$ for $\tz$ outside the polygonal disks. The transformation $\mathbf{A}(\tz;T,w)\mapsto \mathbf{N}(\tz;T,w)$ results in the new unknown $\mathbf{N}(\tz;T,w)$ being analytic inside the disks.  The jump conditions satisfied by $\mathbf{N}(\tz;T,w)$ near $\tz=\tz_0$ and $\tz=\tz_0^*$ are described in Figure~\ref{f:num-T-z0-post}. 

\begin{figure}[h!]
\includegraphics[width=0.45\textwidth]{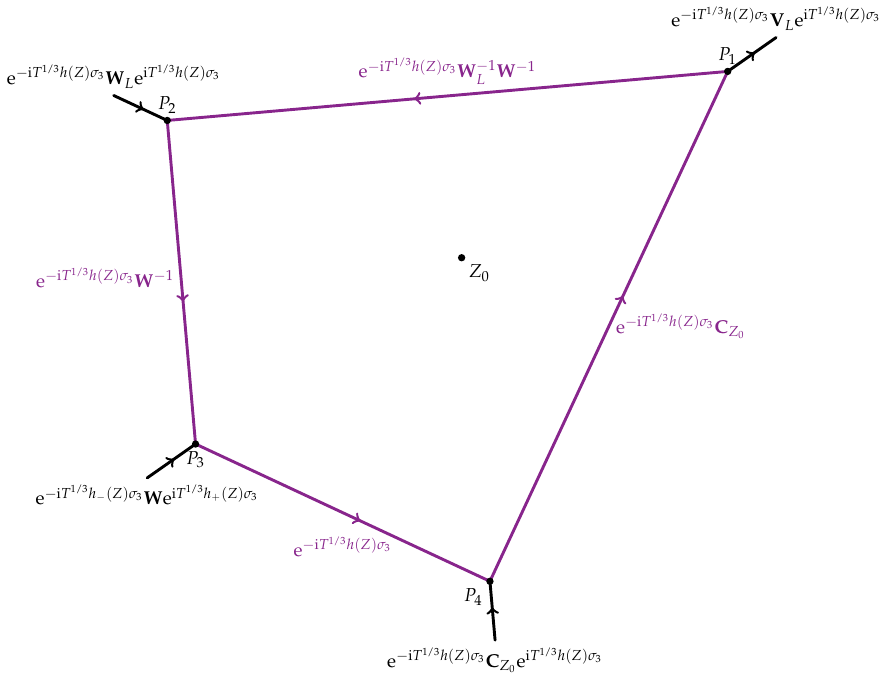}
\includegraphics[width=0.45\textwidth]{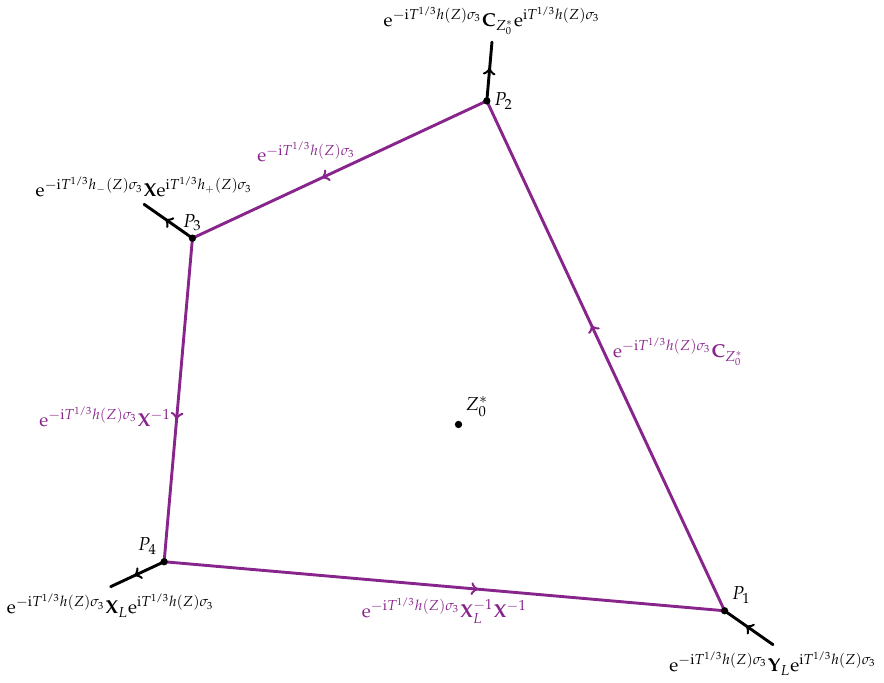}
\caption{The transformation $\mathbf{A}(\tz;T,w)\mapsto \mathbf{N}(z;T,w)$ removes the jump discontinuities inside the disk (polygon). The jump matrices (modified) associated with $\mathbf{N}(z;T,w)$ are given in fuchsia.}
\label{f:num-T-z0-post}
\end{figure}

Although the jump contour for the problem satisfied by $\mathbf{N}(\tz;T,w)$ (or by $\mathbf{T}(\tz;T,w)$) is independent of $T$, it depends on $w$. So the numerical jump contour is chosen adaptively according to the value of $w$. The final Riemann-Hilbert problem that is satisfied by $\mathbf{N}(\tz;T,w)$  has the numerical jump contours as shown in Figure~\ref{fig:num-contours-T}.
\begin{figure}[h]
\includegraphics[width=0.31\textwidth]{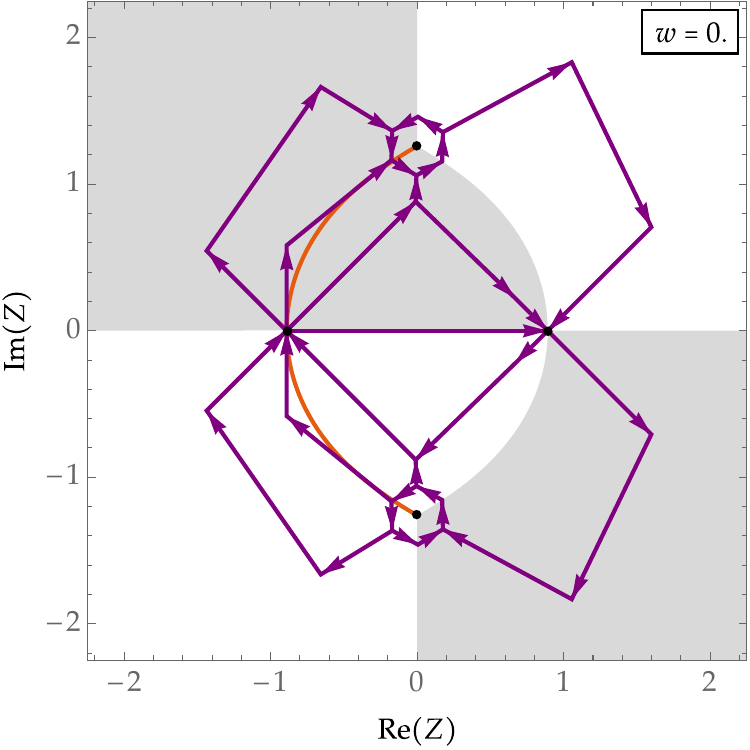}
\includegraphics[width=0.31\textwidth]{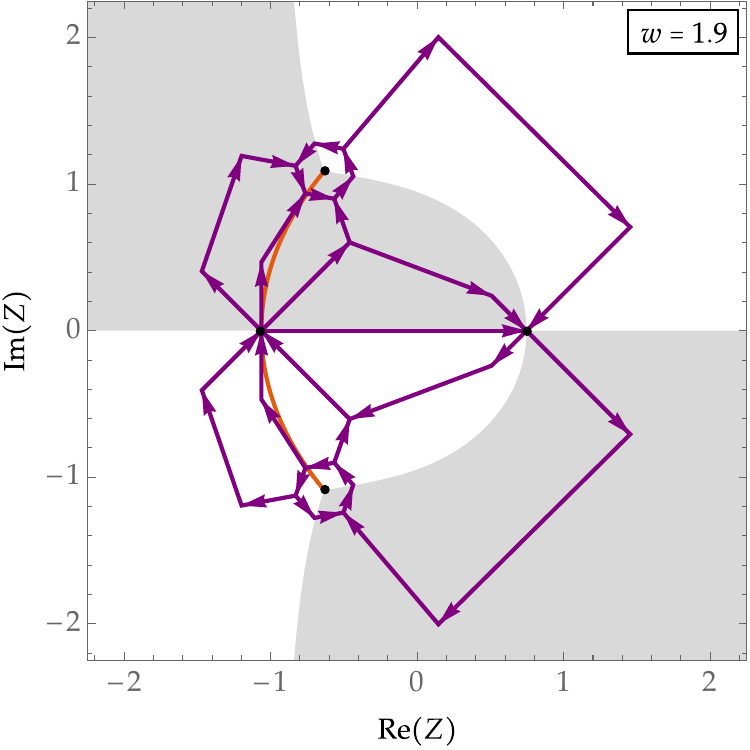}
\includegraphics[width=0.31\textwidth]{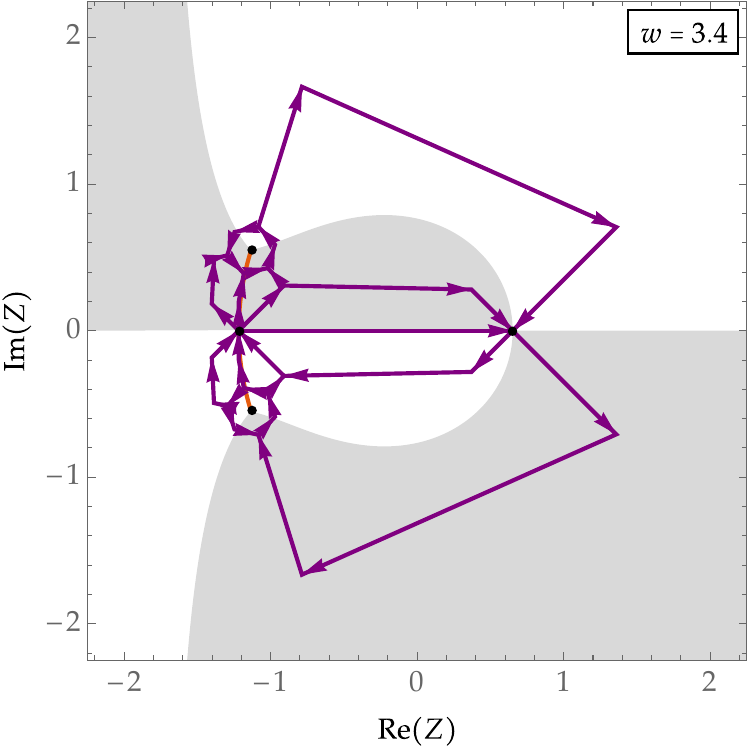}
\caption{The numerical contours used in the region \texttt{largeT} for increasing values of $w\in[0,w_{\mathrm{c}})$. The orange arc is $\Sigma$ (see Figure~\ref{fig:largeT-contour}) and it is included just for reference in the plots. $\Sigma$ is also modeled by line segments.}
\label{fig:num-contours-T}
\end{figure}
Except for those on the polygons modeling the disks centered at $\tz=\tz_0$ and $\tz=\tz_0^*$, the jump matrices for $\mathbf{N}(\tz;T,w)$ on the subarcs of the numerical contour shown in Figure~\ref{fig:num-contours-T} coincide with the jump matrices for $\mathbf{T}(\tz;T,w)$ as described in \eqref{num-jump-T-Gamma-L}--\eqref{num-jump-T-Sigma-m}. These jump matrices (except for the constant jump matrices supported on $I$ and $\Sigma^+\cup\Sigma^-$) are close to the identity away from the points $\tz=\tz_1$ and $\tz=\tz_2$ when Algorithm~\ref{alg:regions} assigns $(X,T)$ to the region \texttt{LargeT}.
The jump matrices for $\mathbf{N}(\tz;T,w)$ on the polygons centered at $\tz=\tz_0$ and $\tz=\tz_0^*$ as shown in Figure~\ref{f:num-T-z0-post} have elements that grow as $T$ increases. To combat this growth when $T$ is large, \texttt{rwio\_largeT} chooses the radii $\delta_0(T)$ to be smaller when $T$ is larger.
Indeed, noting from
\begin{equation}
h(\tz ; w)-h(\xi ; w)=O((\tz-\xi)^{3 / 2}), \quad \tz \rightarrow \xi,\quad\text{for}~ \xi=\tz_0(w), \tz_0(w)^*,
\end{equation}
that 
\begin{equation}
\ee^{\ii T^{1 / 3} h(\tz ; w) \sigma_3}=O(1), \quad T \rightarrow+\infty,
\end{equation}
if $|\tz-\xi| T^{\frac{2}{9}}=O(1)$, for $\xi=\tz_0(w), \tz_0(w)^*$ as $T \rightarrow+\infty$. Therefore, \texttt{rwio\_largeT} scales the common radius of the circles (polygons) centered at $z=\tz_0(w),\tz_0(w)^*$ as $|T|^{-\frac{2}{9}}$.
With these choices and the numerical implementation of $h(\tz;w)$ discussed earlier to compute the jump matrices, \texttt{rwio\_largeT} solves for $\mathbf{N}(\tz;T,w)$ using the routines in \texttt{OperatorApproximation.jl}. %This is the Riemann-Hilbert problem we solve numerically.

The low-level routine \texttt{rwio\_largeT} takes as arguments $T$, $w$, $a$, $b$, and an integer specifying the number of collocation points per polygonal segment of the jump contour.  The use of $T$ and $w$ appears because these are natural coordinates for the description of the jump conditions, however $w=XT^{-\frac{2}{3}}$ is an explicit function of $(X,T)$ and the routine \texttt{wfromXT(X,T)} provided with \texttt{RogueWaveInfiniteNLS.jl} could be used to find $w$ from given $(X,T)$ if needed.  For instance, 
%
%If one would like to compute $\Psi(X,T;\mathbf{G},\bg=1)$ by specifically solving this Riemann-Hilbert problem, one could call, for example:
%\begin{lstlisting}
%[julia> rwio_largeT(40, 1.9, 1, 2im, 300)
%\end{lstlisting}
\begin{lstlisting}
[julia> rwio_largeT(40, 1.9, 1, 2im, 150)
\end{lstlisting}
computes $\Psi(X,T;\mathbf{G},\bg=1)$ at $T=40$ , $w=XT^{-\frac{2}{3}}=1.9$, with $\mathbf{G}=\mathbf{G}(a=1,b=2\ii)$, with 150 collocation points on each segment of the polygonal numerical jump contour.  The corresponding wrapper \texttt{psi\_largeT} has the same arguments as $\texttt{psi}$ except for an additional argument it passes directly to \texttt{rwio\_largeT} to determine the number of collocation points.  Thus
\begin{lstlisting}
[julia> psi_largeT(40,-1.9,1,2im,1.2,150)
\end{lstlisting}
computes $\Psi(40,-1.9,\mathbf{G}(1,2\ii),1.2)$ by scaling the variables by $\bg=1.2$, computing $w$ from the scaled variables using \texttt{wfromXT}, and then calling \texttt{rwio\_largeT} with 150 collocation points, after which the returned value is scaled by $\bg=1.2$.

%\textcolor{green}{[DB: New material.]}
%\begin{remark}\label{rem:num-largeT}
%Considerations in Remark~\ref{rem:num-largeX} also apply to the routines employed in the \texttt{largeT} region.  Here the situation is more complicated due to the existence of the branch points at $\tz=\tz_0(w)$ and $\tz=\tz_0^*(w)$ off the real axis. After contour truncation and the implementation of the outer parametrix, the resulting jump conditions on the circles centered at $\tz=\tz_0(w)$ and $\tz=\tz_0^*(w)$ would have to be treated simultaneously in order to have a well-posed Riemann-Hilbert problem. This is challenging in the implementation of a ``truncate-shift-and-rescale'' strategy described in Remark~\ref{rem:num-largeX}. One way to combat this issue could be a reliable implementation of local parametrices built out of the Airy functions as constructed in Section~\ref{sec:Airy-parametrices}. 
%\end{remark}
%
%\textcolor{green}{[DB: End of new material]}


\subsection{The region \texttt{Painleve}} 
\label{s:num-transition}
Finally, if Algorithm~\ref{alg:regions} determines that $(X,T)$ lies in the region \texttt{Painleve}, then \texttt{psi} calls the wrapper \texttt{psi\_Painleve} that then calls the low-level routine \texttt{rwio\_Painleve}.  The latter implements the numerical solution of the modification described in Section~\ref{s:transitional} of the Riemann-Hilbert jump conditions for $\mathbf{T}(z;X,v)$ given in Section~\ref{s:large-X} with the aim of improving accuracy as
$v=TX^{-\frac{3}{2}}\uparrow v_{\mathrm{c}}=54^{-\frac{1}{2}}$.  The method for computing $\Psi(X,T;\mathbf{G},\bg)$ implemented in \texttt{rwio\_Painleve} is also a new contribution of the package \texttt{RogueWaveInfiniteNLS.jl} and such a computation was not attempted in our earlier work \cite{BilmanLM2020}.

% for the same reason the analysis in Section~\ref{s:large-X} fails; and the routines employed in the computational region \texttt{largeT} fail as $w(X,T)\uparrow w_{\mathrm{c}}=54^{\frac{1}{3}}$ for the same reason the analysis in Section~\ref{s:large-T} fails. 
%This computational region follows the analytical approach in Section~\ref{s:transitional} to capture these scenarios and it was not considered in our earlier work \cite{BilmanLM2020}.

Recall that for $0<v<v_{\mathrm{c}}$, the controlling exponent function $\vartheta(z;v)$ in the Riemann-Hilbert problem analyzed in Section~\ref{s:large-X} has three real simple critical points $z_{-\infty}(v)<z_1(v)<0<z_2(v)$. 
As $v<v_\mathrm{c}$ gets close to $v_\mathrm{c}$, the third critical point $z_{-\infty}(v)$ (not playing a role in the large-$X$ analysis) gets close to $z_1(v)$, and at $v=v_\mathrm{c}$ these two critical points collide at $z=z_\mathrm{c}:=-\sqrt{6}=z_1(v_{\mathrm{c}})$, forming a double critical point.

On the other side of the critical curve $v=TX^{-\frac{3}{2}}=v_{\mathrm{c}}$ in the $(X,T)$-plane (equivalent to the curve $w=XT^{-\frac{2}{3}}=w_{\mathrm{c}}$), for $0<w<w_{\mathrm{c}}$, the controlling exponent function $h(Z;w)$ has two simple critical points $Z_1(w)<0<Z_2(w)$ and two non-real branch points $Z_0(w)$ and $Z_0(w)^*$. As $w$ approaches $w_\mathrm{c}$, the branch points $Z_0(w)$ and $Z_0(w)^*$ approach to the real critical point $Z_1(w)$, and at $w=w_\mathrm{c}$ three points collide at a point $\tz=\tz_\mathrm{c}$ related to $z=z_\mathrm{c}$ by the scaling relation $\tz=w^{-\frac{1}{2}}z=v^{\frac{1}{3}}z$.

Algorithm~\ref{alg:regions} assigns $(X,T)$ to the region \texttt{Painleve} when $|v-v_{\mathrm{c}}|$ is small so that one of the two collision scenarios described above is about to occur. The pair of nearby critical points 
%that are about to collide as $v\uparrow v_{\mathrm{c}}$ as described in the second paragraph and the points that are about to collide as $w\uparrow w_{\mathrm{c}}$ (i.e. as $v\downarrow v_{\mathrm{c}}$) as described in the third paragraph above are all 
is modeled by the double critical point $z=z_\mathrm{c}=-\sqrt{6}$ of the exponent function $\vartheta(z;v_\mathrm{c})$. Consequently, the Riemann-Hilbert problem solved numerically by \texttt{rwio\_Painleve} coincides mostly with that solved by \texttt{rwio\_largeX} but with critical points $z=z_\mathrm{c}$ (fixed, double) and $z=z_2(v)$, varying slightly with $v\approx v_{\mathrm{c}}$. The main new feature accounted for in \texttt{rwio\_Painleve} is an adjustment of the angles with which the jump contours exit the point $z=z_\mathrm{c}$. See Figure~\ref{fig:num-contours-Painleve} for the numerical jump contours used by \texttt{rwio\_Painleve} (and compare with the right-hand panel of Figure~\ref{fig:Painleve-contour}) to formulate and solve a numerically-tractable Riemann-Hilbert problem for the relevant values of $(X,T)$.
\begin{figure}
\includegraphics[width=0.3\textwidth]{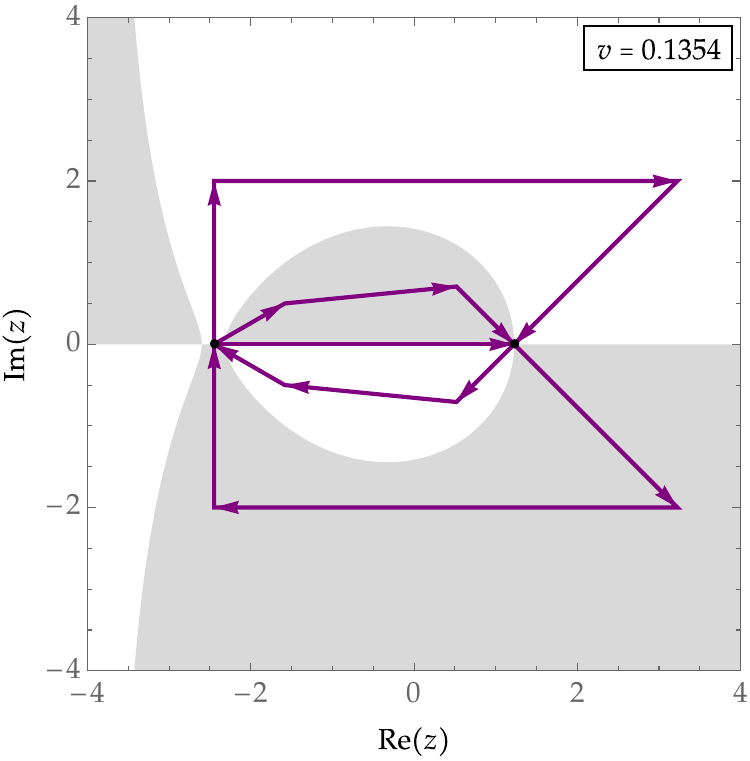}
\includegraphics[width=0.3\textwidth]{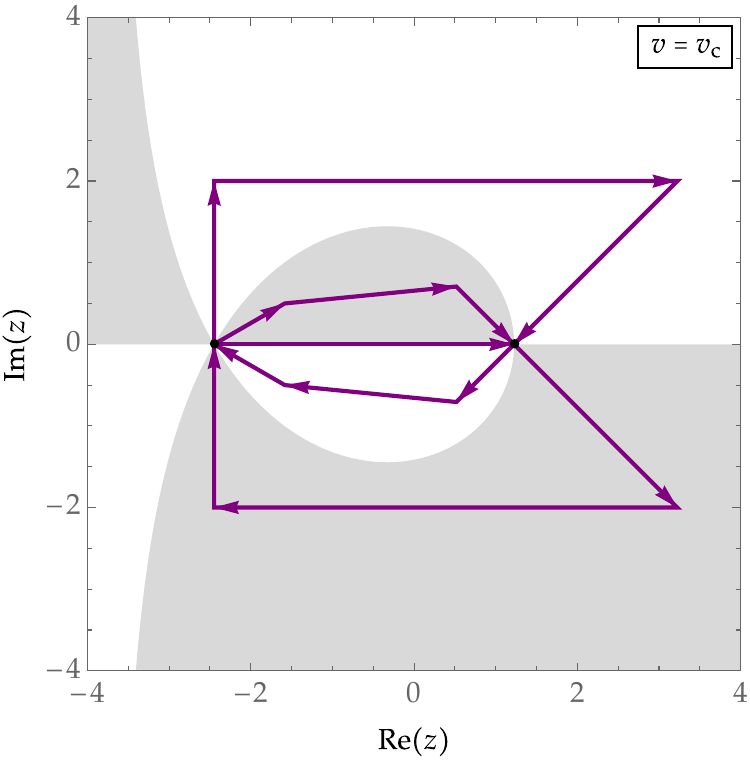}
\includegraphics[width=0.3\textwidth]{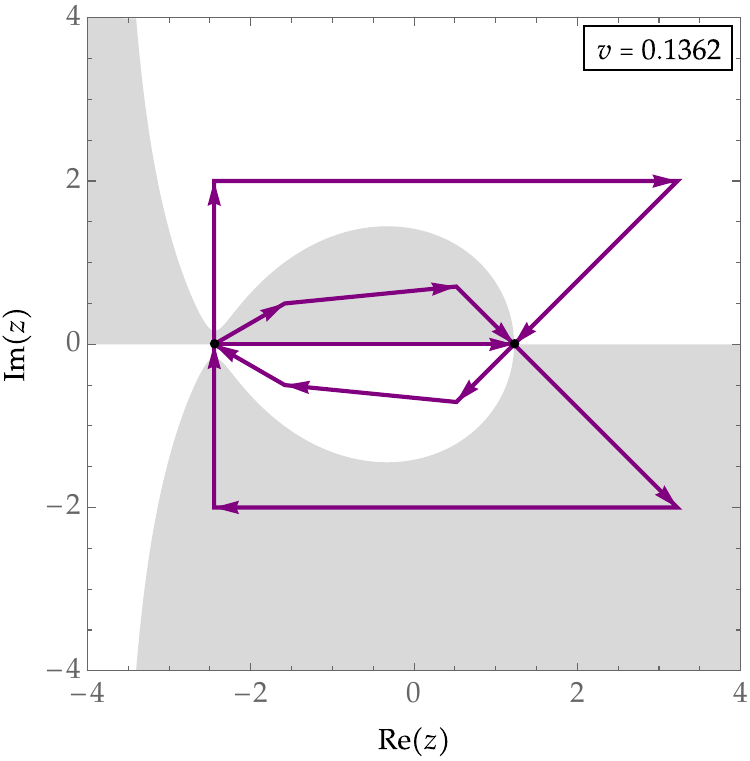}
\caption{The numerical contours used by \texttt{rwio\_Painleve} as $v$ increases when $|v-v_{\mathrm{c}}|$ remains small. $v_{\mathrm{c}}\approx 0.136083$.}
\label{fig:num-contours-Painleve}
\end{figure}

Since \texttt{rwio\_Painleve} is quite similar to \texttt{rwio\_largeX}, the two routines take the same arguments.  
%To compute $\Psi(X,T;\mathbf{G},\bg=1)$ using the Riemann-Hilbert problem treated numerically considered in this region, one can use the following command
%\begin{lstlisting}
%[julia> rwio_Painleve(80, 0.13, 1, 2im, 300)
%\end{lstlisting}
So, for instance,
\begin{lstlisting}
[julia> rwio_Painleve(80, 0.13, 1, 2im, 150)
\end{lstlisting}
returns the value of $\Psi$ at $X=80$ and $v=TX^{-\frac{3}{2}}=0.13$, with $a=1$ and $b=2\ii$ using 150 collocation points at each segment of the numerical jump contour.  The corresponding wrapper has the same arguments as \texttt{psi}, except for the number of collocation points to use that gets passed directly to \texttt{rwio\_Painleve}.  Thus, for instance
\begin{lstlisting}
[julia> psi_Painleve(80,-93,1,2im,1,150)
\end{lstlisting}
returns $\Psi(80,-93,\mathbf{G}(1,2\ii),1)$ using 150 collocation points per contour segment.

%Note that just as in the region \texttt{largeX}, the routine starting with \texttt{rwio\_} requires the value of $v$ as input, not the value of $T$.
%We repeat that the end user doesn't need to solve this Riemann-Hilbert problem specifically. The choice of appropriate Riemann-Hilbert problems to compute $\Psi(X,T;\mathbf{G},\bg)$ is done automatically when the main routine \texttt{psi(X, T, a, b, B)} is called.

\subsection{Consistency of the low-level programs near region boundaries}
In this section we cross-validate the numerically computed solution by comparing the output of the different low-level routines, which we remind the reader are actually designed to compute exactly the same quantity:  $\Psi(X,T;\mathbf{G}(a,b),\bg=1)$.  We first compare the solution computed by each of the routines \texttt{rwio\_largeX}, \texttt{rwio\_largeT}, and $\texttt{rwio\_Painleve}$ with the solution computed by the simplest routine \texttt{rwio\_undeformed}. This is done by taking $(X,T)$ near the origin. We then take $X$ and $T$ large and near the critical curve $v=TX^{-\frac{3}{2}}=v_{\mathrm{c}}$, and compare the solution computed by \texttt{rwio\_largeX} and \texttt{rwio\_largeT} with the solution computed by  \texttt{rwio\_Painleve}. We demonstrate that the computed solutions match to at least 13 digits of accuracy in all the cases mentioned above. This indicates that Algorithm~\ref{alg:regions} operates in a ``seamless'' fashion.
See the notebook \texttt{Paper-Code.ipynb} in the repository \cite{PaperCode} for the sample codes that produced the examples below along with the codes for performing the computations presented in Section~\ref{s:introduction}.

\subsubsection{Comparing \emph{\texttt{rwio\_largeX}} and \emph{\texttt{rwio\_undeformed}} near the origin} We set $X=1$ and $v=0.1 v_{\mathrm{c}}$. Then the value of $T$ is determined via the code:
\begin{lstlisting}
[julia> Xval = 1.
[julia> vval = 0.1*VCRIT
[julia> Tval = TfromXv(Xval,vval)
	0.013608276348795434
\end{lstlisting}
The code below shows that the solution computed using \texttt{rwio\_undeformed} and that computed using \texttt{rwio\_largeX} match with high accuracy:
\begin{lstlisting}
[julia> abs(rwio_undeformed(Xval,Tval,1,2im,400)-rwio_largeX(Xval,vval,1,2im,140))
	3.5749182362651225e-13
\end{lstlisting}
Similarly, we set $X=1$ and $v=0.8 v_{\mathrm{c}}$, and find:
\begin{lstlisting}
[julia> Xval = 1.
[julia> vval = 0.8*VCRIT
[julia> Tval = TfromXv(Xval,vval)
[julia> abs(rwio_undeformed(Xval,Tval,1,2im,400)-rwio_largeX(Xval,vval,1,2im,140))
	3.784797839067842e-13
\end{lstlisting}

\subsubsection{Comparing \emph{\texttt{rwio\_largeT}} and \emph{\texttt{rwio\_undeformed}} near the origin} We set $T=1$ and $w=0.1w_{\mathrm{c}}$. Then the value of $X$ is determined via the code:
\begin{lstlisting}
[julia> Tval = 1.
[julia> wval = 0.1*WCRIT
[julia> Xval = XfromTw(Tval,wval)
	0.37797631496846196
\end{lstlisting}
The code below shows that the solution computed using \texttt{rwio\_undeformed} and that computed using \texttt{rwio\_largeT} match with high accuracy:
\begin{lstlisting}
[julia> abs(rwio_undeformed(Xval,Tval,1,2im,400)-rwio_largeT(Tval,wval,1,2im,140))
	1.5685785522111634e-13
\end{lstlisting}
Similarly, we set $T=1$ and $w=0.8w_{\mathrm{c}}$, and find:
\begin{lstlisting}
[julia> Tval = 1.
[julia> wval = 0.8*WCRIT
[julia> Xval = XfromTw(Tval,wval)
[julia> abs(rwio_undeformed(Xval,Tval,1,2im,400)-rwio_largeT(Tval,wval,1,2im,140)
	6.467885900866133e-14
\end{lstlisting}

\subsubsection{Comparing \emph{\texttt{rwio\_Painleve}} and \emph{\texttt{rwio\_undeformed}} near the origin} We set $X=1$ and $v=v_{\mathrm{c}}$. Again, these choices determine the value of $T$ via the code:
\begin{lstlisting}
[julia> Xval = 1.
[julia> vval = VCRIT
[julia> Tval = TfromXv(Xval,vval)
	0.13608276348795434
\end{lstlisting}
The code below shows that the solution computed using  \texttt{rwio\_undeformed} and that computed using \texttt{rwio\_Painleve} match with high accuracy:
\begin{lstlisting}
[julia> abs(rwio_undeformed(Xval,Tval,1,2im,140)-rwio_Painleve(Xval,vval,1,2im,140))
	2.4340993547656853e-13
\end{lstlisting}

\subsubsection{Comparing \emph{\texttt{rwio\_Painleve}} with \emph{\texttt{rwio\_largeX}} or \emph{\texttt{rwio\_largeT}} for $(X,T)$ large near the critical curve}
To compare {\texttt{rwio\_Painleve}} with {\texttt{rwio\_largeX}} we set $X=2000$ and consider $v<v_\mathrm{c}$ but also $v\approx v_\mathrm{c}$ by setting $v=0.98v_\mathrm{c}$.
The code below shows that the solution computed using  \texttt{rwio\_largeX}  and that computed using \texttt{rwio\_Painleve} match near the critical curve $v=v_{\mathrm{c}}$ with high accuracy:
\begin{lstlisting}
[julia> Xval = 2000
[julia> vval = 0.98*VCRIT
[julia> abs(rwio_largeX(Xval,vval,1,2im,140)-rwio_Painleve(Xval,vval,1,2im,140))
	1.5436906930782975e-15
\end{lstlisting}

To compare {\texttt{rwio\_Painleve}} with {\texttt{rwio\_largeT}} we again set $X=2000$ and consider $v>v_\mathrm{c}$ (i.e. $w<w_\mathrm{c}$) but also $v\approx v_\mathrm{c}$ by setting $v=1.05v_\mathrm{c}$. This choice determines the value of $T$ which we find numerically and then obtain the value of $w$ from these determined values of $X$ and $T$. 
The code below performs these initial computations and then shows that the solution computed using  \texttt{rwio\_largeT}  and that computed using \texttt{rwio\_Painleve} match near the critical curve $v=v_{\mathrm{c}}$ with high accuracy:
\begin{lstlisting}
[julia> Xval = 2000
[julia> vval = 1.05*VCRIT
[julia> Tval = TfromXv(Xval,vval)
[julia> wval = wfromXT(Xval,Tval)
[julia> abs(rwio_largeT(Tval,wval,1,2im,140)-rwio_Painleve(Xval,vval,1,2im,140))
	1.0412793661344947e-14
\end{lstlisting}
%\textcolor{green}{The subsection below is also new material}

\subsection{Effect of increasing the number of collocation points}
We now demonstrate how the accuracy of the various routines improves as the number of collocation points is increased. 
%We let $\texttt{Psi}(X,T,a,b, m)$ denote the value of the computed solution at $(X,T)$ with 
We fix parameters $a=b=1$ and $\bg=1$.
% and compare  using $m$ collocation points on each segment of the relevant numerical jump contour.
First, to study \texttt{rwio\_largeX}, we fix a relatively large reference number of collocation points $N_\mathrm{X}=80$ and fix $v=0.5v_\mathrm{c}$.  Then varying $X$ and the number $m$ of collocation points per contour segment, we define
%For the method in \texttt{largeX}, we set $N_{\mathrm{X}}=80$, fix $v=0.5 v_{\mathrm{c}}$, and define the pointwise error
%\begin{equation}
%\mathcal{E}^{\mathrm{X}}_m(X):= | \texttt{Psi}(X, 0.5 v_{\mathrm{c}}X^{\frac{3}{2}} , a=1, b=1, m) -  \texttt{Psi}(X, 0.5 v_{\mathrm{c}}X^{\frac{3}{2}} , a=1, b=1, N_{\mathrm{X}}) |.
%\end{equation}
\begin{equation}
\mathcal{E}^{\mathrm{X}}_m(X):= \left|\mathtt{rwio\_largeX}(X,0.5v_\mathrm{c},1,1,m)-\mathtt{rwio\_largeX}(X,0.5v_\mathrm{c},1,1,N_{\mathrm{X}})\right|.
\end{equation}
Next, to study \texttt{rwio\_largeT}, we fix a relatively large reference number of collocation points $N_\mathrm{T}=80$ and fix $w=0.5w_\mathrm{c}$.  Then varying $T$ and the number $m$ of collocation points per contour segment, we define
%For the method in \texttt{largeT}, we set $N_{\mathrm{T}}=80$, fix $w=0.5 w_{\mathrm{c}}$, and define the pointwise error
%\begin{equation}
%\mathcal{E}^{\mathrm{T}}_m(T):= | \texttt{Psi}(0.5 w_{\mathrm{c}}T^{2/3}, T , a=1, b=1, m) -  \texttt{Psi}(0.5 w_{\mathrm{c}}T^{2/3}, T , a=1, b=1, N_{\mathrm{T}}) |.
%\end{equation}
\begin{equation}
\mathcal{E}^{\mathrm{T}}_m(T):=\left|\mathtt{rwio\_largeT}(T,0.5w_\mathrm{c},1,1,m)-\mathtt{rwio\_largeT}(T,0.5w_\mathrm{c},1,1,N_\mathrm{T})\right|.
\end{equation}
Finally, to study \texttt{rwio\_Painleve}, we fix a relatively large reference number of collocation points $N_\mathrm{P}=160$ and fix $v=v_\mathrm{c}$ to be on the critical curve.  Then varying $X$ and the number $m$ of collocation points per contour segment, we define
%For the method in \texttt{Painleve}, we set $N_{\mathrm{P}}=160$, fix $v=v_{\mathrm{c}}$, and define the pointwise error
%\begin{equation}
%\mathcal{E}^{\mathrm{P}}_m(X):= | \texttt{Psi}(X, v_{\mathrm{c}}X^{\frac{3}{2}} , a=1, b=1, m) -  \texttt{Psi}(X, v_{\mathrm{c}}X^{\frac{3}{2}} , a=1, b=1, N_{\mathrm{P}}) |.
%\end{equation}
\begin{equation}
\mathcal{E}_m^\mathrm{P}(X):=\left|\mathtt{rwio\_Painleve}(X,v_\mathrm{c},1,1,m)-\mathtt{rwio\_Painleve}(X,v_\mathrm{c},1,1,N_\mathrm{P})\right|.
\end{equation}

In Figure~\ref{fig:convergence} we plot these pointwise errors over different intervals (over the $X$-axis for $\mathcal{E}_m^\mathrm{X}(X)$ and $\mathcal{E}_m^\mathrm{P}(X)$, and over the $T$-axis for $\mathcal{E}_m^\mathrm{T}(T)$)  for three values of $m$ increasing towards $N_{\mathrm{X}}$, $N_{\mathrm{P}}$ and $N_{\mathrm{T}}$. 
\begin{figure}[h]
\includegraphics[width=0.3\textwidth]{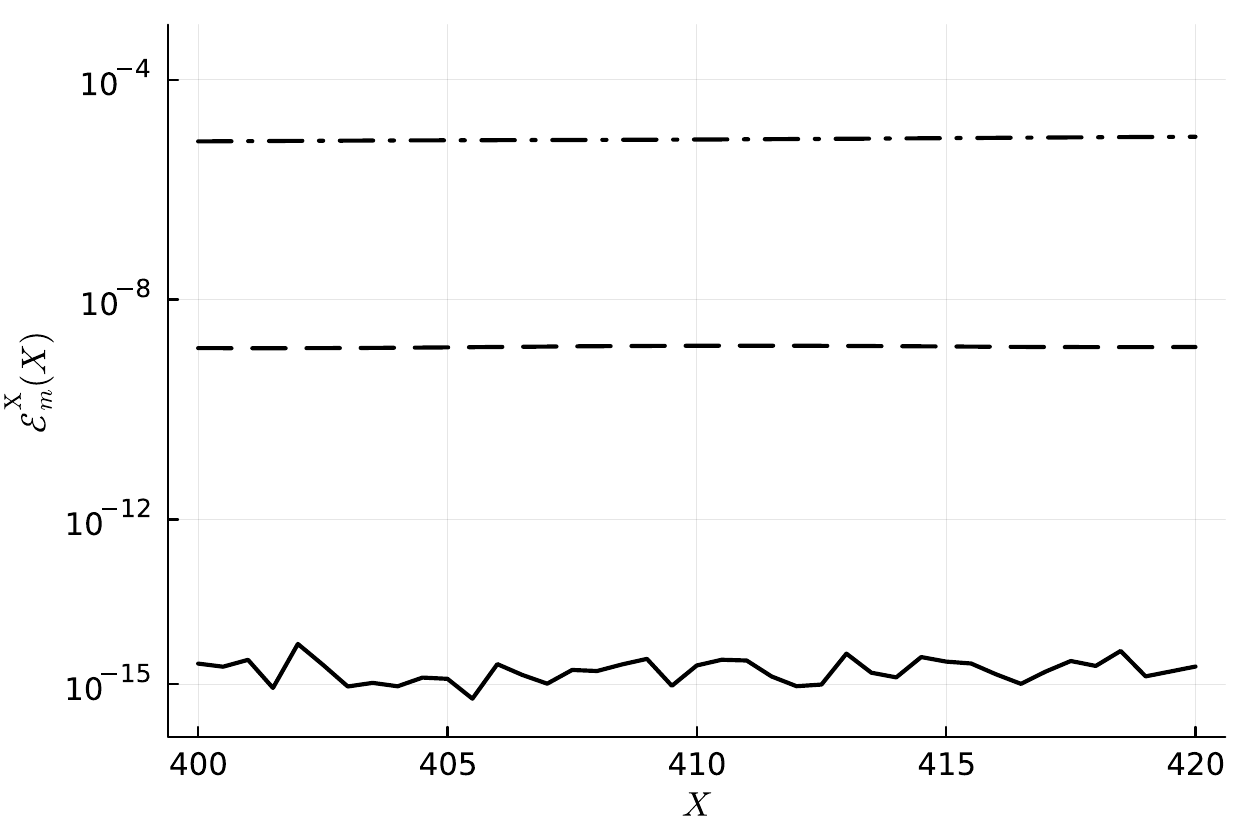}\,
\includegraphics[width=0.3\textwidth]{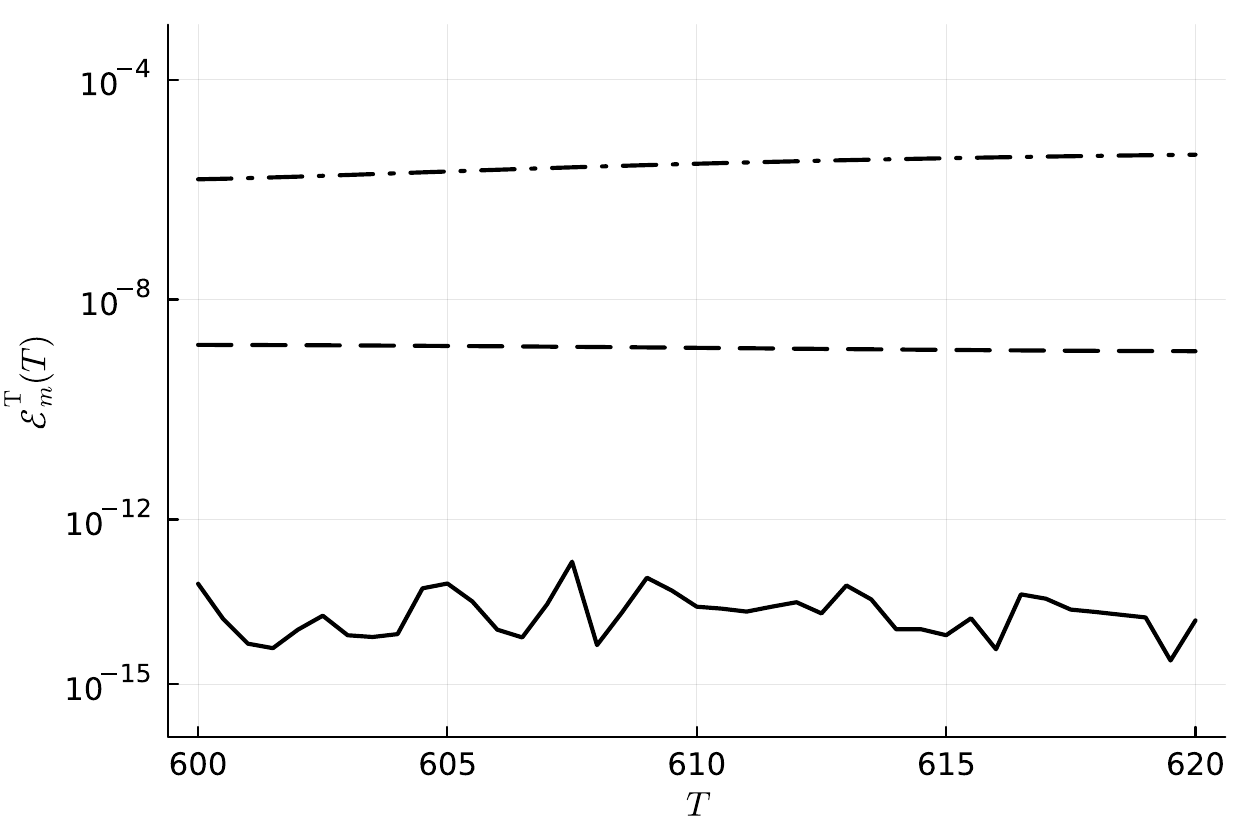}\,
\includegraphics[width=0.3\textwidth]{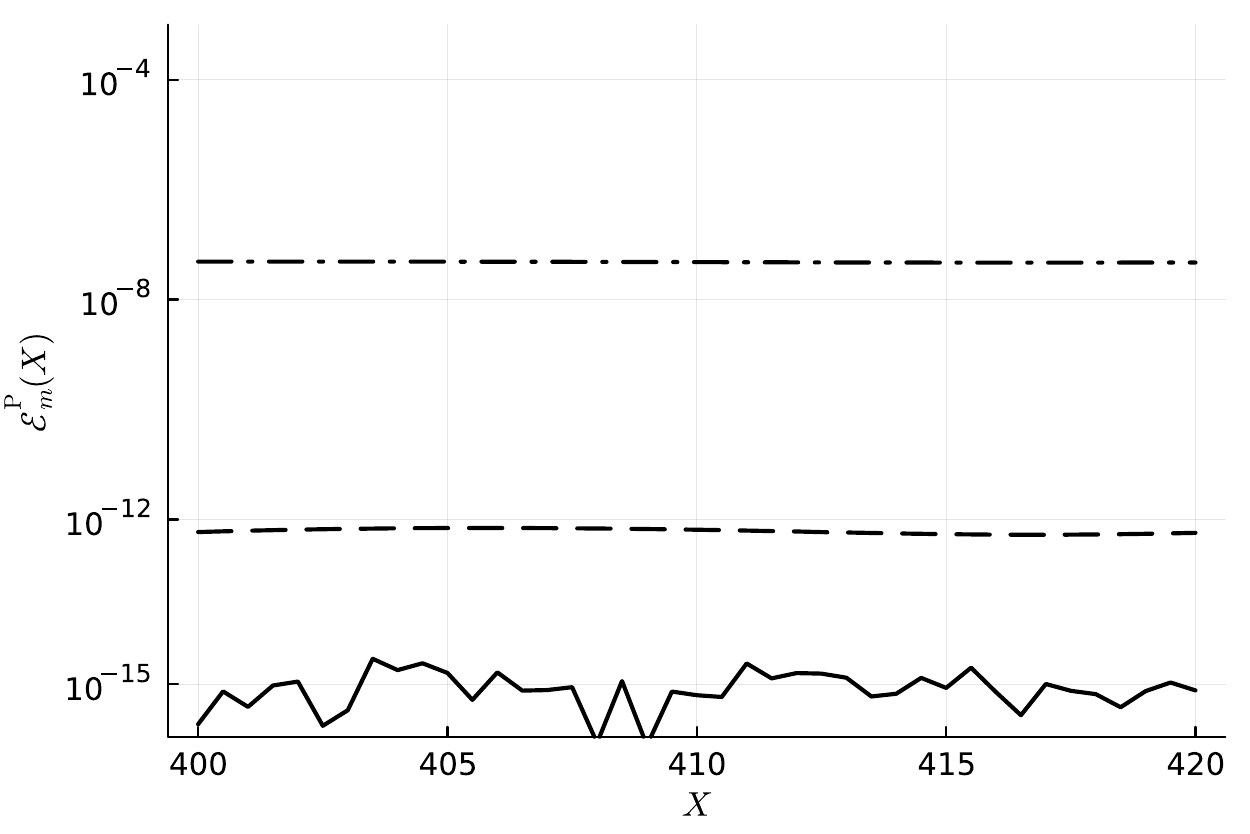}
\caption{Left: $\mathcal{E}^{\mathrm{X}}_m(X)$ for $m=10$ (dashed-dotted), $m=20$ (dashed), and $m=40$ (solid) over $400\leq X \leq 420$. 
Center: $\mathcal{E}^{\mathrm{T}}_m(T)$ for $m=10$ (dashed-dotted), $m=20$ (dashed), and $m=40$ (solid) over $600\leq T \leq 620$.
Right: $\mathcal{E}^{\mathrm{P}}_m(X)$ for $m=20$ (dashed-dotted), $m=40$ (dashed), and $m=80$ (solid) over $400\leq X \leq 420$.}
\label{fig:convergence} 
\end{figure}
These plots show that once $m$ has increased to half of the reference value in each case, the difference has decreased to machine precision.  This demonstrates how quickly the accuracy improves as the number of collocation points per contour segment is increased.

\appendix

\section{Elementary properties of $\Psi(X,T;\mathbf{G},\bg)$}
\label{a:elementary}
This appendix is devoted to the proofs of several basic symmetries of the function $\Psi(X,T;\mathbf{G},\bg)$.  First, we prove the scaling invariance of the solution $\Psi(X,T;\mathbf{G},\bg)$ with respect to $\bg>0$.
\begin{proof}[Proof of Proposition~\ref{p:scaling}]
Suppose that $(X,T)\in\mathbb{R}^2$, $\bg>0$, and $\mathbf{G}$ satisfying $\det(\mathbf{G})=1$ and $\mathbf{G}=\sigma_2\mathbf{G}^*\sigma_2$ are given.  
Let ${\mathbf{P}}(\Lambda;{X},{T},\mathbf{G},\bg)$ be the solution of \rhref{rhp:near-field} with these given parameters. On the other hand, let $\widetilde{\mathbf{P}}(\Lambda;\widetilde{X},\widetilde{T},\mathbf{G},1)$ be the solution of \rhref{rhp:near-field} for given $(\widetilde{X},\widetilde{T})\in\mathbb{R}^2$ with $\bg=1$. 
Since the jump matrix in \rhref{rhp:near-field} is invariant under $\Lambda\mapsto \bg^{-1}\Lambda$, $X\mapsto \bg X$, and $T\mapsto \bg^2 T$, by uniqueness we find that
\begin{equation}
\mathbf{P}(\Lambda;X,T,\mathbf{G},\bg) \equiv \widetilde{\mathbf{P}}(\bg^{-1} \Lambda; \bg X,\bg^2  T,\mathbf{G},1)
\label{P-identity}
\end{equation}
since the radius of the circular jump contour in \rhref{rhp:near-field} can be taken arbitrary.
We deduce from \eqref{P-identity} that
\begin{equation}
\begin{aligned}
\Psi(X,T; \mathbf{G}, \bg) &= 2\ii \lim_{\Lambda\to \infty} \Lambda P_{12}(\Lambda;X,T,\mathbf{G},\bg)\\
&= 2\ii \lim_{\Lambda\to \infty} \Lambda \widetilde{P}_{12}(\bg^{-1}\Lambda; \bg X, \bg^2 T,\mathbf{G}, 1)\\
&= \bg \left(2\ii \lim_{\Lambda\to \infty} \bg^{-1}\Lambda \widetilde{P}_{12}(
\bg^{-1}\Lambda; \bg X, \bg^2 T,\mathbf{G}, 1)\right)\\
&= \bg \Psi(\bg X, \bg^2 T; \mathbf{G}, 1),
\end{aligned}
\end{equation}
which is the claimed scaling symmetry.
\end{proof}
As in the rest of the paper, from this point onwards in this Appendix we take $\bg=1$ and omit $\bg$ from all argument lists.

We now work towards proving the symmetries of the solution $\Psi(X,T;\mathbf{G})$ with respect to $X\mapsto -X$ and $T\mapsto -T$. Given the solution $\mathbf{P}(\Lambda;X,T,\mathbf{G}(a,b))$ of \rhref{rhp:near-field}, define
\begin{equation}
\mathbf{X}(\Lambda;X,T,\mathbf{G}(a,b)) :=
\begin{cases}
\sigma_3 \mathbf{P}(\Lambda;X,T, \mathbf{G}(a,b)) \ee^{-4\ii  \Lambda^{-1}\sigma_3}\sigma_3, &\quad |\Lambda|>1,\\
\sigma_3 \mathbf{P}(\Lambda;X,T, \mathbf{G}(a,b)) \ee^{-2\ii (\Lambda X + \Lambda^2 T)\sigma_3}(\ii \sigma_2)\sigma_3, &\quad |\Lambda|<1.
\end{cases}
\end{equation}

\begin{proposition} $\mathbf{X}(- \Lambda; -X, T,\mathbf{G}(b,a)) = \mathbf{P}(\Lambda;X,T,\mathbf{G}(a,b))$.
\label{P-symmetry-X}
\end{proposition}
\begin{proof}
Observe that for $|\Lambda| = 1$ we have
\begin{equation}
\begin{aligned}
\mathbf{X}_+(\Lambda;X,T, \mathbf{G}(a,b)) =& \sigma_3 \mathbf{P}_+(\Lambda;X,T, \mathbf{G}(a,b)) \ee^{-4\ii \Lambda^{-1}\sigma_3}\sigma_3\\
=& \sigma_3 \mathbf{P}_-(\Lambda;X,T, \mathbf{G}(a,b)) \ee^{-\ii (\Lambda X+\Lambda^2 T + 2 \Lambda^{-1})\sigma_{3}}\mathbf{G}(a,b)\sigma_3 \ee^{\ii (\Lambda X+\Lambda^2 T - 2 \Lambda^{-1})\sigma_{3}}\\
%&\quad\times \ee^{-4\ii \beta \Lambda^{-1}\sigma_3}\sigma_3
=& \mathbf{X}_-(\Lambda;X,T, \mathbf{G}(a,b))\sigma_3 (-\ii\sigma_2) \ee^{\ii(\Lambda X + \Lambda^2 T - 2 \Lambda^{-1})\sigma_3} \mathbf{G}(a,b)\sigma_3  \\
&{}\cdot \ee^{\ii (\Lambda X+\Lambda^2 T - 2 \Lambda^{-1})\sigma_{3}}\\
=& \mathbf{X}_-(\Lambda;X,T, \mathbf{G}(a,b))\ee^{-\ii(\Lambda X + \Lambda^2 T - 2 \Lambda^{-1})\sigma_3}
[\sigma_3 (-\ii\sigma_2) \mathbf{G}(a,b)\sigma_3]\\
&{}\cdot  \ee^{\ii (\Lambda X+\Lambda^2 T - 2 \Lambda^{-1})\sigma_{3}}.\\
=& \mathbf{X}_-(\Lambda;X,T, \mathbf{G}(a,b))\ee^{-\ii(\Lambda X + \Lambda^2 T - 2 \Lambda^{-1})\sigma_3}
\mathbf{G}(b,a) \ee^{\ii (\Lambda X+\Lambda^2 T - 2 \Lambda^{-1})\sigma_{3}},
\end{aligned}
\end{equation}
since  $\sigma_3 (-\ii\sigma_2) \mathbf{G}(a,b)\sigma_3 = \mathbf{G}(b,a)$. Thus, $\mathbf{X}(-\Lambda; -X, T, \mathbf{G}(b,a))$ satisfies exactly the same jump condition as $\mathbf{P}(\Lambda;X,T, \mathbf{G}(a,b))$. Since they satisfy the same normalization and analyticity properties, by uniqueness of the solutions of \rhref{rhp:near-field} the result follows.
\end{proof}
The proof of Proposition~\ref{prop:X-symmetry} is now a simple consequence.
%\begin{proposition} $\Psi(X,T; \mathbf{G}(a,b),\beta) = \Psi(-X,T; \mathbf{G}(b,a),\beta)$.
%\label{prop:X-symmetry}
%\end{proposition}
\begin{proof}[Proof of Proposition~\ref{prop:X-symmetry}] 
We make use of Proposition~\ref{P-symmetry-X} and compute
\begin{equation}
\begin{aligned}
\Psi(X,T;\mathbf{G}(a,b)) &= 2\ii \lim_{\Lambda\to \infty} \Lambda P_{12}(\Lambda; X,T, \mathbf{G}(a,b))\\
&= 2\ii \lim_{\Lambda\to \infty} \Lambda \left[ \sigma_3 \mathbf{P}(- \Lambda; -X, T, \mathbf{G}(b,a))\ee^{-4\ii \Lambda^{-1}\sigma_3} \sigma_3 \right]_{12}\\
&= 2\ii \lim_{\Lambda\to \infty} \Lambda P_{12}(\Lambda; X,T, \mathbf{G}(a,b))\\
&= 2\ii \lim_{\Lambda\to \infty} \Lambda \left[  - P_{12}(- \Lambda; -X, T, \mathbf{G}(b,a)) \right]\\
&= 2\ii \lim_{\Lambda\to \infty}  \left[ ( -\Lambda) P_{12}(- \Lambda; -X, T, \mathbf{G}(b,a)) \right]\\
&=\Psi(-X,T;\mathbf{G}(b,a)),
\end{aligned}
\end{equation}
which is the claimed symmetry.
\end{proof}
An easier observation is the following.
\begin{proposition} $\mathbf{P}(\Lambda;X,-T,\mathbf{G}(a,b)) =\mathbf{P}(-\Lambda^*;X,T,\mathbf{G}(a,b)^*)^*$
\label{prop:P-symmetry-T}
\end{proposition}
\begin{proof}
It is straightforward to verify that the two matrix functions satisfy the same analyticity and normalization properties, and satisfy the same jump condition. The result follows from uniqueness.
\end{proof}
We now prove Proposition~\ref{prop:T-symmetry} as a consequence of this result.
%\begin{proposition} $\Psi(X,-T; \mathbf{G}(a,b),\beta) = \Psi(X,T; \mathbf{G}(a,b)^*,\beta)^*$.
%\label{prop:T-symmetry}
%\end{proposition}
\begin{proof}[Proof of~Proposition~\ref{prop:T-symmetry}]
We make use of Proposition~\ref{prop:P-symmetry-T} and compute:
\begin{equation}
\begin{aligned}
\Psi(X,-T;\mathbf{G}(a,b)) &= 2\ii \lim_{\Lambda\to \infty} \Lambda P_{12}(\Lambda; -X,T, \mathbf{G}(a,b))\\
&= 2\ii \lim_{\Lambda\to \infty} \Lambda \left[P_{12}(- \Lambda^*; X, T, \mathbf{G}(a,b)^*)^* \right]\\
&= 2\ii \lim_{\Lambda\to \infty}  \left[\Lambda^* P_{12}(- \Lambda^*; X, T, \mathbf{G}(a,b)^*) \right]^*\\
%&=  \lim_{\Lambda\to \infty}  \left[-2\ii \Lambda^* P_{12}(- \Lambda^*; X, T, \mathbf{G}(a,b)^*,\beta) \right]^*\\
&=  \lim_{\Lambda\to \infty}  \left[2\ii (-\Lambda^*) P_{12}(- \Lambda^*; X, T, \mathbf{G}(a,b)^*) \right]^*\\
&=  \lim_{\Lambda\to \infty}  \left[2\ii \Lambda P_{12}(\Lambda ; X, T, \mathbf{G}(a,b)^*) \right]^*\\
&=\Psi(X,T;\mathbf{G}(a,b)^*)^*,
\end{aligned}
\end{equation}
which is the claimed symmetry.
\end{proof}

Next, we turn to the proof of Proposition~\ref{prop:a-b-scaling}, which concerned a normalization of the parameters $a,b$ in $\mathbf{G}(a,b)$.
\begin{proof}[Proof of Proposition~\ref{prop:a-b-scaling}]
It is clear from the structure of the matrix $\mathbf{G}(a,b)$ in \eqref{G-form} that $\mathbf{G}(a,b)=\mathbf{G}(c a, c b)$ for any positive
scalar $c >0$. Since the dependence on $a$ and $b$ of \rhref{rhp:near-field} enters only via the matrix $\mathbf{G}$, we have
\begin{equation}
\Psi(X,T,\mathbf{G}(a,b)) = \Psi(X,T,\mathbf{G}(c a, c b)),\qquad\text{for any $c>0$}.
\end{equation}
%Therefore, from now on in the paper we assume without any loss of generality that
%\begin{equation}
%|a|^2+|b|^2 = 1.
%\end{equation}
%\textcolor{red}{[Placing the above choice here somewhat tricky. I still want to present the results for general $a,b$. I would like to place it at the end of the next section: ``Preliminaries and elementary properties'']}
%\textcolor{teal}{I think we need $c>0$ to have $\mathbf{G}(ca,cb)=\mathbf{G}(a,b)$.  

An additional identity following from \eqref{G-form} is $\mathbf{G}(\ee^{\ii\theta}a,\ee^{-\ii\theta}b)=\ee^{\ii\theta\sigma_3}\mathbf{G}(a,b)$ for all $\theta\in\mathbb{R}$.  Since $\ee^{\ii\theta\sigma_3}$ commutes with $\ee^{-\ii(\Lambda X+\Lambda^2T+2 \bg \Lambda^{-1})\sigma_3}$ it can be absorbed into $\mathbf{P}_-(\Lambda;X,T,\mathbf{G})$, which has no bearing on $\Psi(X,T;\mathbf{G})$. Taken together, we can say that for any $c\in\mathbb{C}\setminus\{0\}$, the matrices $\mathbf{G}(ca,c^*b)$ and $\mathbf{G}(a,b)$ yield exactly the same solution $\Psi(X,T;\mathbf{G})$, or put another way, 
\begin{equation}
\Psi(X,T;\mathbf{G}(ca,c^*b))=\Psi(X,T;\mathbf{G}(a,b)).
\label{eq:identity-one}
\end{equation}  
%This is really the $\mathbb{CP}^1$ parameter manifold again; this complex manifold can be covered by two charts:  if $a\neq 0$ then we can pick $c=a^{-1}$ and it suffices to study the solution with $\mathbf{G}(1,b)$ for general $b\in\mathbb{C}$, and if $b\neq 0$ then we can pick $c^*=b^{-1}$ and it suffices to study the solution with $\mathbf{G}(a,1)$ for general $a\in\mathbb{C}$.  

Finally, one can remove phase factors on $b$ by diagonal conjugation of $\mathbf{P}$, leading to the formula 
\begin{equation}
\Psi(X,T;\mathbf{G}(a,b\ee^{\ii\theta}))=\ee^{-\ii\theta}\Psi(X,T;\mathbf{G}(a,b)).  
\label{eq:identity-two}
\end{equation}
Composing these identities by first taking $c=\ee^{-\ii\arg(a)}/\sqrt{|a|^2+|b|^2}$ in \eqref{eq:identity-one} so that
\begin{equation}
\Psi(X,T;\mathbf{G}(a,b))=\Psi\left(X,T;\mathbf{G}\left(\frac{|a|}{\sqrt{|a|^2+|b|^2}},\frac{|b|}{\sqrt{|a|^2+|b|^2}}\ee^{\ii\arg(ab)}\right)\right)
\end{equation}
and then taking $\theta=\arg(ab)$ in \eqref{eq:identity-two}, we arrive at \eqref{ab-symmetry}.
%The last formula is the one we are using to study the ``escape'' asymptotics, in which we are writing 
%\[
%\frac{|a|}{\sqrt{|a|^2+|b|^2}}=\ee^{-2M}\quad\text{and}\quad \frac{|b|}{\sqrt{|a|^2+|b|^2}}=\sqrt{1-\ee^{-4M}}.
%\]
%}
\end{proof}

Finally, recall the parameterization $\mathbf{G}=\mathbf{G}(a,b)$ by complex numbers $a,b$ not both zero as given in \eqref{G-form}.  Proposition~\ref{p:degeneration} concerned the special case that either $a=0$ or $b=0$, and we give its proof now.
\begin{proof}[Proof of Proposition~\ref{p:degeneration}]
First, if $b=0$, but $a\neq 0$, then $\mathbf{G}=\mathbf{G}(a,b)$ given in \eqref{G-form} is a diagonal matrix and it is easy to verify that the matrix function
\begin{equation}
\mathbf{P}(\Lambda;X,T,\mathbf{G}) = 
\begin{cases} 
\begin{bmatrix} \dfrac{a^*}{|a|} & 0 \\0 & \dfrac{a}{|a|} \end{bmatrix},&\quad |\Lambda|<1,\\
\mathbb{I},&\quad |\Lambda|>1,
\end{cases}
\end{equation}
is the solution of \rhref{rhp:near-field}, which produces $\Psi(X,T;\mathbf{G})\equiv 0$ by \eqref{Psi-def}. 
Similarly, if $a=0$, but $b\neq 0$, then $\mathbf{G}$ is an off-diagonal matrix, and the jump matrix \eqref{P-jump} may be expressed as
\begin{equation}
\ee^{-\ii(\Lambda X+\Lambda^2T+2 \Lambda^{-1})\sigma_3}\mathbf{G}
\ee^{\ii(\Lambda X+\Lambda^2T+2 \Lambda^{-1})\sigma_3} 
= \frac{1}{|b|}\begin{bmatrix}
0 & b^* \ee^{-2\ii (\Lambda X +\Lambda^2 T)}\\
-b \ee^{2\ii (\Lambda X +\Lambda^2 T)} & 0
\end{bmatrix}\ee^{4\ii  \Lambda^{-1}\sigma_3}.
\end{equation}
In this case, one can verify that
\begin{equation}
\mathbf{P}(\Lambda;X,T,\mathbf{G}) = 
\begin{cases} 
\displaystyle\frac{1}{|b|}\begin{bmatrix}
0 &-b^* \ee^{-2\ii (\Lambda X +\Lambda^2 T)}\\
b \ee^{2\ii (\Lambda X +\Lambda^2 T)} & 0
\end{bmatrix},&\quad |\Lambda|<1,\\
\ee^{4\ii  \Lambda^{-1}\sigma_3},&\quad |\Lambda|>1,
\end{cases}
\label{eq:answer-a-zero}
\end{equation}
is the solution of \rhref{rhp:near-field}, which again produces $\Psi(X,T;\mathbf{G})\equiv 0$ by \eqref{Psi-def}. 
\end{proof}

\section{Computing $\mathcal{V}(y;\tau)$ related to the increasing tritronqu\'ee solution of Painlev\'e-II}
\label{a:P2-numerics}
In this appendix we provide the details concerning the computation of $\mathcal{V}(y;\tau)$ for real bounded values of $y$, which is used in the numerical validation of Theorem~\ref{t:transition} as shown in Figure~\ref{fig:TransitionalLogLogError}.
We recall that  $\mathcal{V}(y;\tau)$ is characterized by the conditions \eqref{V-P2-def} in terms of the (unique) increasing tritronqu\'ee solution $u(x)$ of \eqref{eq:PII} with $\alpha=\frac{1}{2}+\ii p$, where $p=\frac{1}{2\pi}\ln(1+\tau^2)$ and $\tau=|b/a|$.
As explained in Section~\ref{s:transitional}, $\mathcal{V}(y;\tau)$ is obtained via \eqref{eq:Transitional-XT-V} from the unique solution $\mathbf{U}^{\mathrm{TT}}(\zeta;y,\tau)$ of the Riemann-Hilbert problem arising from a Lax pair for the Painlev\'e-II equation due to Jimbo and Miwa \cite{JimboM81}.

As discussed in Section~\ref{s:Numerics}, the numerical framework developed in \cite{TrogdonO2015} and implemented in \texttt{OperatorApproximation.jl} \cite{OperatorApproximation} concerns Riemann-Hilbert problems posed on a suitable oriented contour $\Gamma$ and normalized such that the solution is of the form $\mathbf{C}+\mathcal{C}^{\Gamma}[\mathbf{F}](\zeta)$, where $\mathbf{C}$ is a constant $2\times 2$ matrix and $\mathcal{C}^{\Sigma}[\mathbf{F}](\zeta)$ is the Cauchy transform
\begin{equation}
\mathcal{C}^{\Gamma}[\mathbf{F}](\zeta) = \frac{1}{2\pi \ii} \int_{\Gamma} \frac{\mathbf{F}(s)}{s-\zeta}\dd s.
\end{equation}
Therefore, we consider the following Riemann-Hilbert problem satisfied by the renormalized function
\begin{equation}
\mathbf{W}^{\mathrm{TT}}(\zeta;y,\tau):=
\begin{cases} 
\mathbf{U}^{\mathrm{TT}}(\zeta;y,\tau),&\quad |\zeta|<1,\\
\mathbf{U}^{\mathrm{TT}}(\zeta;y,\tau)\zeta^{\ii p \sigma_3},&\quad |\zeta|>1.
\end{cases}
\end{equation}
\begin{figure}
\includegraphics[width=0.75\textwidth]{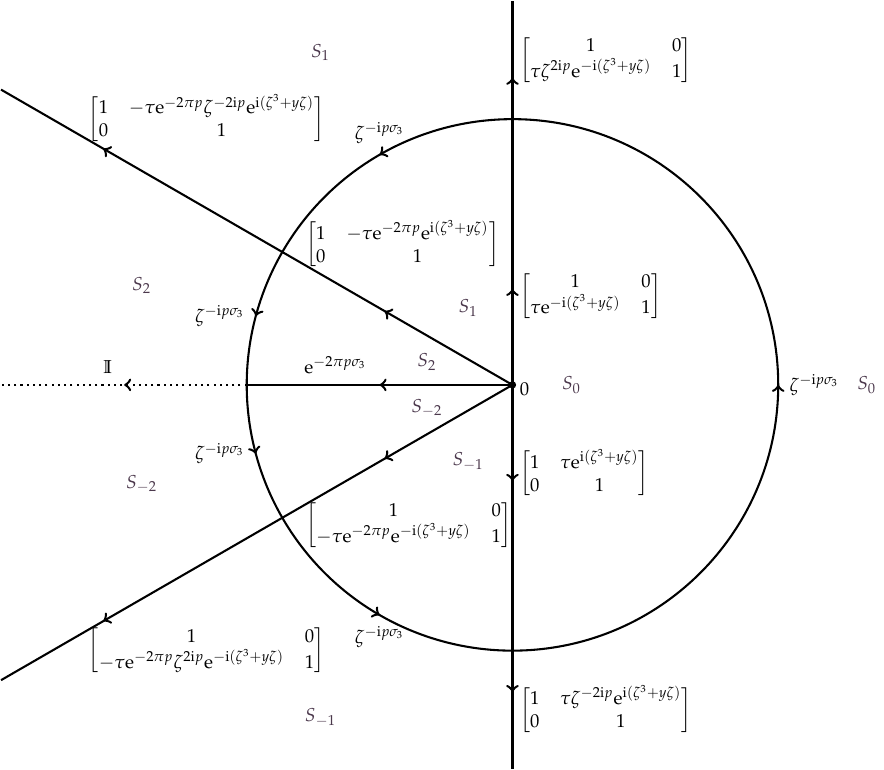}
\caption{Jump contours and conditions associated with \rhref{rhp:normalized-P2} in the $\zeta$-plane satisfied by $\mathbf{W}^{\mathrm{TT}}(\zeta;y,\tau)$.}
\label{fig:P2-jumps}
\end{figure}

\begin{rhp}[Renormalized Jimbo-Miwa Painlev\'e-II problem]
Let $y, p, \tau \in \mathbb{C}$ be related by $\tau^2=\ee^{2 \pi p}-1$. Seek a $2 \times 2$ matrix-valued function $\mathbf{W}^{\mathrm{TT}}(\zeta;y,\tau)$ with the following properties.
\begin{itemize}
\item \emph{Analyticity:}
$\mathbf{W}^{\mathrm{TT}}(\zeta;y,\tau)$ is analytic for $\zeta$ in the complement of the unit circle in the five sectors $S_0:|\arg (\zeta)|<\frac{1}{2} \pi, S_1: \frac{1}{2} \pi<$ $\arg (\zeta)<\frac{5}{6} \pi, S_{-1}:-\frac{5}{6} \pi<\arg (\zeta)<-\frac{1}{2} \pi, S_2: \frac{5}{6} \pi<\arg (\zeta)<\pi$, and $S_{-2}:-\pi<\arg (\zeta)<$ $-\frac{5}{6} \pi$. It takes continuous boundary values on the excluded rays and at the origin from each sector.
\item \emph{Jump conditions:} $\mathbf{W}_+^{\mathrm{TT}}(\zeta;y,\tau)=\mathbf{W}^{\mathrm{TT}}_-(\zeta;y,\tau) \mathbf{V}^{\mathrm{PII}}(\zeta ; y,\tau)$, where $\mathbf{V}^{\mathrm{PII}}(\zeta ; y,\tau)$ is the matrix defined on the jump contour shown in Figure~\ref{fig:P2-jumps}.
\item \emph{Normalization:} $\mathbf{W}^{\mathrm{TT}}(\zeta;y,\tau)  \rightarrow \mathbb{I}$ as $\zeta \rightarrow \infty$ uniformly in all directions.
\end{itemize}
\label{rhp:normalized-P2}
\end{rhp}
The function $\mathcal{V}(y;\tau)$ is then given by
\begin{equation}
\mathcal{V}(y;\tau) = \lim_{\zeta\to\infty} \zeta {W}^{\mathrm{TT}}_{12}(\zeta; y,\tau).
\label{V-P2-recover}
\end{equation}

For the purposes of verifying Theorem~\ref{t:transition}, one only needs to obtain $\mathcal{V}(y;\tau)$ for fairly small values of $|y|$.
For $y\in\mathbb{R}$ close to $y=0$, the jump matrix $\mathbf{V}^{\mathrm{PII}}(\zeta ; y,\tau)$ is bounded and tends to the identity matrix as $\zeta\to\infty$ on any arc of the jump contour described in Figure~\ref{fig:P2-jumps}. For such $y\in\mathbb{R}$, we numerically solve \rhref{rhp:normalized-P2} as is (without employing any steepest descent deformations) using \cite{OperatorApproximation}.
The routines we developed to compute $\mathcal{V}(y;\tau)$ can be found in the repository associated with this paper \cite{PaperCode}. See the Jupyter notebook \texttt{Painleve2TT.ipynb} in the repository \cite{PaperCode}.

In order to verify Theorem~\ref{t:transition}, we consider a dense grid $\mathbb{Y}$ on the closed interval $-1\leq y \leq 1$ with a mesh size of $0.005$. We numerically solve \rhref{rhp:normalized-P2} for each $y\in\mathbb{Y}$ and obtain the data for $\mathcal{V}(y;\tau)$ via \eqref{V-P2-recover} over $\mathbb{Y}$.
We then interpolate over $\mathbb{Y}$ to obtain a continuous function $\mathcal{V}(y;\tau)$ of $y$. We use this interpolant for each value of $X\in\mathbb{X}$, where
\begin{multline}
\mathbb{X} := \{200,400,600,800,1000,1200,1400,1600,1800,2000,2200,2400,2600,2800,3000,\\
3200,3400,3600,3800,4000,5000,6000,7000,8000,9000,10000\},
\end{multline}
to compute $E(y)$ as defined in \eqref{E-transition} and then take the supremum over $y\in\mathbb{R}$ as described immediately thereafter.

\section{User's guide for the package \texttt{RogueWaveInfiniteNLS.jl}}
\label{a:julia-summary}
In this Appendix we list all of the commands defined in the package \texttt{RogueWaveInfiniteNLS.jl} written in the \texttt{Julia} programming language.  

\subsection{Main command}
Most users will only need the command \texttt{psi}.
\begin{lstlisting}
[julia> psi(X,T,a,b,B)
\end{lstlisting}
This command returns a numerical approximation of $\Psi(X,T;\mathbf{G}(a,b),\bg)$ computed in a black-box fashion.  It determines which of the routines \texttt{psi\_undeformed}, \texttt{psi\_largeX}, \texttt{psi\_largeT}, and \texttt{psi\_Painleve} to call based on the coordinates $(X,T)$.

\subsection{Commands for computing $\Psi$ based on specific deformed Riemann-Hilbert problems}
\subsubsection{Using the undeformed Riemann-Hilbert problem}
The following commands implement a direct solution of Riemann-Hilbert Problem~\ref{rhp:near-field}.
\begin{lstlisting}
[julia> rwio_undeformed(X,T,a,b,n)
\end{lstlisting}
This solves a version of Riemann-Hilbert Problem~\ref{rhp:near-field} assuming $X>0$ and $T>0$, and returns $\Psi(X,T;\mathbf{G}(a,b),1)$ with $n$ collocation points per contour segment.
\begin{lstlisting}
[julia> psi_undeformed(X,T,a,b,B,n)
\end{lstlisting}
This wrapper for \texttt{rwio\_nodeformation\_rescaled} allows for variables $(X,T)$ of any signs, and also takes an additional argument representing the value of $\bg$.  It calls \texttt{rwio\_undeformed} after rescaling the variables by $\bg$ and determining the appropriate parameters from the matrix $\widetilde{\mathbf{G}}(a,b)$, and then rescales the returned value again by $\bg$.

\subsubsection{Using the Riemann-Hilbert problem deformed for large-$X$ asymptotics}
The following commands implement a numerical solution of the Riemann-Hilbert Problem satisfied by $\mathbf{T}(z;X,v)$ described in Section~\ref{s:large-X}.
\begin{lstlisting}
[julia> rwio_largeX(X,v,a,b,n)
\end{lstlisting}
This uses the native variables $(X,v)$ for $\mathbf{T}(z;X,v)$. The underlying assumptions are $X>0$ and $0\leq v < v_\mathrm{c}$. There are useful routines \texttt{TfromXv} and \texttt{vfromXT} for switching back and forth between the coordinates $(X,v)$ and $(X,T)$ that are described below in Section~\ref{s:change-coordinates}.
\begin{lstlisting}
[julia> psi_largeX(X,T,a,b,B,n)
\end{lstlisting}
This wrapper allows for variables $(X,T)$ of any signs, and also takes an additional argument representing the value of $\bg$.
For $X\geq 0$ and $T\geq 0 $, it calls \texttt{rwio\_largeX} with rescaled $(X,T)$ coordinates so that $B=1$ and with $v$ obtained from the rescaled $(X,T)$ via the routine \texttt{vfromXT}. If $X<0$ or $T<0$, a symmetry is used to map (X,T) to a point in the first quadrant, and then \texttt{rwio\_largeX} is called using that point.

\subsubsection{Using the Riemann-Hilbert problem deformed for large-$T$ asymptotics}
The following commands implement a numerical solution of the Riemann-Hilbert Problem satisfied by $\mathbf{T}(\tz;T,w)$ described in Section~\ref{s:large-T}.
\begin{lstlisting}
[julia> rwio_largeT(T,w,a,b,n)
\end{lstlisting}
This uses the native variables $(T,w)$ for $\mathbf{T}(\tz;T,w)$.   The underlying assumptions are $T>0$ and $0\leq w < w_\mathrm{c}$. There are  useful routines \texttt{XfromTw} and \texttt{wfromXT} for switching back and forth between the coordinates $(T,w)$ and $(X,T)$ that are described below in Section~\ref{s:change-coordinates}.
\begin{lstlisting}
[julia> psi_largeT(X,T,a,b,B,n)
\end{lstlisting}
This wrapper allows for variables $(X,T)$ of any signs, and also takes an additional argument representing the value of $\bg$. For $X\geq 0$ and $T\geq 0 $, it calls \texttt{rwio\_largeT} with rescaled $(X,T)$ coordinates so that $\bg=1$ and with $w$ obtained from the rescaled $(X,T)$ via the routine \texttt{wfromXT}. If $X<0$ or $T<0$, a symmetry is used to map (X,T) to a point in the first quadrant, and then \texttt{rwio\_largeT} is called using that point.

\subsubsection{Using the Riemann-Hilbert problem deformed for large $X$ and $T$ near the critical curve}
The following commands implement a modification of the Riemann-Hilbert problem satisfied by $\mathbf{T}(z;X,v)$ as described in Section~\ref{s:transitional}, accounting for the effect of $v\approx v_\mathrm{c}$.  
\begin{lstlisting}
[julia> rwio_Painleve(X,v,a,b,n)
\end{lstlisting}
This again uses the native variables $(X,v)$ for $\mathbf{T}(z;X,v)$.   The underlying assumptions are $X>0$ and $v\approx v_\mathrm{c}$.
\begin{lstlisting}
[julia> psi_Painleve(X,T,a,b,B,n)
\end{lstlisting}
This wrapper allows for variables $(X,T)$ of any signs, and also takes an additional argument representing the value of $\bg$.
For $X\geq 0$ and $T\geq 0 $, it calls \texttt{rwio\_Painleve} with rescaled $(X,T)$ coordinates so that $B=1$ and with $v$ obtained from the rescaled $(X,T)$ via the routine \texttt{vfromXT}. If $X<0$ or $T<0$, a symmetry is used to map (X,T) to a point in the first quadrant, and then \texttt{rwio\_Painleve} is called using that point.

\subsection{Routines for changing coordinates}
\label{s:change-coordinates}
The package \texttt{RogueWaveInfiniteNLS.jl} defines for the user the two important constants \texttt{VCRIT} representing $v_\mathrm{c}:=54^{-\frac{1}{2}}$ and \texttt{WCRIT} representing $w_\mathrm{c}:=54^\frac{1}{3}$.  The following commands allow the user to easily move between the coordinates $(X,T)$, $(X,v)$, and $(T,w)$.
\begin{lstlisting}
[julia> TfromXv(X,v)
\end{lstlisting}
This returns the value $T = X^{\frac{3}{2}}v$ determined by a given $(X,v)$, in case one would like to extract the $T$-coordinate of a point on the curve in the $(X,T)$ plane determined by fixing the value of $v$.
\begin{lstlisting}
[julia> vfromXT(X,T)
\end{lstlisting}
In a similar fashion, this returns $v = T X^{-\frac{3}{2}}$, in case one would like to determine the value of $v$ to use in the routine \texttt{rwio\_largeX} from given $(X,T)$.
\begin{lstlisting}
[julia> XfromTw(T,w)
\end{lstlisting}
This returns the value $X = T^{\frac{2}{3}} w$ determined by a given $(T,w)$, in case one would like to extract the $X$-coordinate of a point on the curve in the $(X,T)$ plane determined by fixing the value of $w$.
\begin{lstlisting}
[julia> wfromXT(X,T)
\end{lstlisting}
In a similar fashion, this returns $w = X T^{-\frac{2}{3}}$, in case one would like to determine the value of $w$ to use in the routine \texttt{rwio\_largeT} from given $(X,T)$.

\end{document}